\documentclass{article}
\usepackage{amssymb,amsmath,amsthm}
\usepackage{mathtools}
\mathtoolsset{showonlyrefs}
\usepackage{booktabs}
\usepackage{graphicx,subfigure,epstopdf}
\usepackage[usenames]{color}
\usepackage{url}
\usepackage{algorithm,algorithmic}
\usepackage{dsfont,verbatim}
\usepackage{varwidth}
\usepackage{caption}
\graphicspath{{./pics/}}
\usepackage{amsthm}
\usepackage{enumerate}

\allowdisplaybreaks

\newtheorem{theorem}{Theorem}[section]
\newtheorem{lemma}{Lemma}[section]
\newtheorem{definition}[theorem]{Definition}

\newtheorem{remark}{Remark}[section]

\numberwithin{equation}{section}

\title{Expectation Propagation for Poisson Data}
\author{Chen Zhang\thanks{Department of Computer Science, University College London, London WC1E 6BT, UK (\texttt{s.arridge, b.jin,
chen.zhang.16@ucl.ac.uk})}  \and Simon Arridge\footnotemark[1] \and Bangti Jin\footnotemark[1] }
\date{}
\begin{document}
\maketitle
\begin{abstract}
The Poisson distribution arises naturally when dealing with data involving counts, and it has found many applications
in inverse problems and imaging. In this work, we develop an approximate Bayesian inference technique based on
expectation propagation for approximating the posterior distribution formed from the Poisson likelihood function and
a Laplace type prior distribution, e.g., the anisotropic total variation prior. The approach iteratively yields a
Gaussian approximation, and at each iteration, it updates the Gaussian approximation to one factor of the posterior
distribution by moment matching. We derive explicit update formulas in terms of one-dimensional integrals, and also discuss
stable and efficient quadrature rules for evaluating these integrals. The method is showcased on two-dimensional PET images.


\end{abstract}

\section{Introduction}
The Poisson distribution is widely employed to describe inverse and imaging problems involving count data, e.g.,
emission computed tomography \cite{vardi1985statistical,shepp1982maximum}, including positron emission tomography
and single photon emission computed tomography. The corresponding likelihood function is a Poisson distribution
with its parameter given by an affine transform (followed by a suitable link function). Over the past few decades,
the mathematical theory and numerical algorithms for image reconstruction with Poisson data have witnessed
impressive progresses. We refer interested readers to \cite{HohageWerner:2016} for a comprehensive overview on
variational regularization techniques for Poisson data and \cite{BerteroBoccacci:2009} for mathematical modeling
and numerical methods for Poisson data. It is worth noting that the Poisson model is especially
important in the low-count regime, e.g., $[0,10]$ photons, whereas in the moderate (or high) count regime, heteroscedastic
normal approximations can be employed in the reconstruction, leading to a weighted Gaussian likelihood function (e.g.,
via the so-called Anscombe transform \cite{Anscombe:1948}). In this work, we focus on the Poisson model.

To cope with the inherent ill-posed nature of the imaging problem, regularization plays an important role in image
reconstruction. This can be achieved implicitly via early stopping during an iterative reconstruction procedure (e.g.,
EM algorithm or Richardson-Lucy iterations) or explicitly via suitable penalties, e.g., Sobolev penalty,
sparsity and total variation. The penalized maximum likelihood (or equivalently maximum a posteriori (MAP))
is currently the most popular way for image reconstruction with Poisson models \cite{de1995modified,sotthivirat2002image}.
However, these approaches can only provide point estimates, and the important issue of uncertainty quantification,
which provides crucial reliability assessment on point estimates, is not fully addressed. The full Bayesian approach
provides a principled yet very flexible framework for uncertainty quantification of inverse and imaging problems
\cite{KaipioSomersalo:2005,Stuart:2010}. The prior distribution acts as a regularizer, and the ill-posedness of
the imaging problem is naturally dealt with. Due to the imprecise prior knowledge of the solution and the presence
of the data noise, the posterior distribution contains an ensemble of inverse solutions consistent with the observed
data, which can be used to quantify the uncertainties associated with a point estimator, via, e.g., credible interval
or highest probability density regions.

For imaging problems with Poisson data, a full Bayesian treatment is challenging, due to the nonnegativity constraint
and high-dimensionality of the parameter / data space. There are several possible strategies from the computational perspective. One
idea is to use general-purposed sampling methods to explore the posterior state space, predominantly Markov chain Monte
Carlo (MCMC) methods \cite{Liu:2001,RobertCasella:2004}. Recent scalable variants, e.g., stochastic gradient
Langevin dynamics \cite{WellingTeh:2011}, are very promising, although these techniques have not been applied to the
Poisson model. Then the constraints on the signal can be incorporated directly by discarding
samples violating the constraint. However, in order to obtain accurate statistical estimates, sampling methods
generally require many samples and thus tend to suffer from high computational cost, due to the high problem
dimensionality. Further, the MCMC convergence is challenging to diagnose. These observations have motivated intensive
research works on developing approximate inference techniques (AITs).  In the machine learning literature, a large
number of AITs have been proposed, e.g., variational inference \cite{jordan1999introduction,BleiKucukelbir:2017,
ChallisBarber:2013,Jin:2012,ArridgeItoJinZhang:2018}, expectation propagation \cite{minka2001expectation,minka2001family}
and more recently Bayesian (deep) neural network \cite{gal2016dropout}; see the survey \cite{ZhangButepage:2019}
for a comprehensive overview of recent developments on variational inference. In all AITs, one aims at finding an
optimal approximate yet tractable distribution within a family of parametric/nonparametric probability distributions
(e.g., Gaussian), by minimizing the error in a certain probability metric, prominently Kullback-Leibler divergence.
Empirically they can often produce reasonable approximations but at a much reduced computational cost than MCMC. However,
there seem no systematic strategies for handling constraints in these approaches. For example, a straightforward
truncation of the distribution due to the constraint often leads to elaborated distributions, e.g., truncated normal
distribution, which tends to make the computation tedious or even completely intractable in variational Bayesian
inference.

In this work, we develop a computational strategy for exploring the posterior distribution for Poisson data (with two
popular nonnegativity constraints) with a Laplace type prior based on expectation propagation \cite{minka2001expectation,
minka2001family}, in order to deliver a Gaussian approximation. Laplace prior promotes the sparsity of the image in a
transformed domain, which is a valid assumption on most natural images. The main contributions of the work are as
follows. First, we derive explicit update formulas using one-dimensional integrals. It essentially exploits the
rank-one projection form of the factors to reduce the intractable high-dimensional integrals to tractable one-dimensional
ones. In this way, we arrive at two approximate inference algorithms, parameterized by either moment or natural parameters.
Second, we derive stable and efficient quadrature rules for evaluating the resulting one-dimensional integrals, i.e., a
recursive scheme for Poisson sites with large counts and an approximate expansion for Laplace sites, and discuss different
schemes for the recursion, dependent of the integration interval, in order to achieve good numerical stability. Last, we
illustrate the approach with comprehensive numerical experiments with the posterior distribution formed by Poisson
likelihood and an anisotropic total variation prior, clearly showcasing the feasibility of the approach.

Last, we put the work in the context of Bayesian analysis of Poisson data. The predominant body of literature in
statistics employs a log link function, commonly known as Poisson regression in statistics and machine learning (see,
e.g., \cite{CameronTrivedi:2013,ArridgeItoJinZhang:2018}). This differs substantially from the one frequently
arising in medical imaging, e.g., positron emission tomography, and in particular the crucial nonnegativity constraint
becomes vacuous. The only directly relevant work we are aware of is the recent work \cite{ko2016expectation}. The work
\cite{ko2016expectation} discussed a full Bayesian exploration with EP, by modifying the posterior distributions
using a rectified linear function on the transformed domain of the signal, which induces singular measures on the
region violating the constraint. However, the work  \cite{ko2016expectation} does not consider the background.

The rest of the paper is organized as follows. In Section \ref{sec:form} we describe the posterior distribution
for the Poisson likelihood function and a Laplace type prior. Then we give explicit expressions of the integrals
involved in EP update and describe two algorithms in Section \ref{sec:EP}. In Section \ref{sec:int} we present
stable and efficient numerical methods for evaluating one-dimensional integrals. Last, in Section \ref{sec:numer}
we present numerical results for three benchmark images. In the appendices, we describe two useful
parameterizations of a Gaussian distribution, Laplace approximation and additional comparative numerical results
for a one-dimensional problem with MCMC and Laplace approximation to shed further insights into the performance
of EP algorithms.

\section{Problem formulation}\label{sec:form}
In this part, we give the Bayesian formulation for Poisson data, i.e., the likelihood function $p(y|x)$ and prior
distribution $p(x)$, and discuss the nonnegativity constraint.

Let $x\in\mathbb{R}^n$ be the (unknown) signal/image of interest, $y\in\mathbb{R}^{m_1}_+$ be the observed Poisson data, and
$A=[a_{ij}]=[a^t_i]_{i=1}^{m_1}\in\mathbb{R}^{m_1\times n}_+$ be the forward map, where the superscript $t$ denotes matrix / vector
transpose. The entries of the matrix $A$ are assumed to be nonnegative. For example, in emission computed tomography, it
can be a discrete analogue of Radon transform, or probabilistically, the entry $a_{ij}$ of the matrix $A$ denotes the probability
that the $i$th sensor pair records the photon emitted at the $j$th site.

The conditional probability density $p(y_i|x)$ of observing $y_i\in\mathbb{N}$ given the signal $x$ is given by
\begin{equation*}
	p(y_i|x) = \frac{(a^t_ix+r_i)^{y_i}e^{-(a^t_ix+r_i)}}{y_i!},
\end{equation*}
where $r=[r_i]_i\in\mathbb{R}^{m_1}_+$ is the background. That is, the entry $y_i$ follows a Poisson distribution with a parameter
$a^t_ix+r_i$. The Poisson model of this form is popular in the statistical modeling of inverse and imaging problems involving counts, e.g.,
positron emission tomography \cite{vardi1985statistical}. If the entries of $y$ are independent and identically
distributed (i.i.d.), then the likelihood function $p(y|x)$ is given by
\begin{equation}
	p(y|x) = \prod_{i=1}^{m_1} p(y_i|x).
\end{equation}

Note that the likelihood function $p(y|x)$ is not well-defined for all $x\in\mathbb{R}^n$, and
suitable constraints on $x$ are needed in order to ensure the well-definedness of the factors
$p(y_i|x)$'s. In the literature, there are three popular constraints:
\begin{enumerate}
	\item $ \mathcal{C}_1=\{x|x>0\}  :=\cap_i\{x|x_i>0\}$;
	\item $ \mathcal{C}_2=\{x|Ax>0\} :=\cap_i\{x|[Ax]_i=a^t_ix>0\}$;
	\item $ \mathcal{C}_3=\{x|Ax+r>0\} :=\cap_i\{x|[Ax+r]_i=a^t_ix+r_i>0\}$.
\end{enumerate}
Since the entries of $A$ are nonnegative, there holds $\mathcal{C}_1\subset \mathcal{C}_2\subset \mathcal{C}_3$.
In practice, the first assumption is most consistent with the physics in that it reflects the physical
constraint that emission counts are non-negative. The last two assumptions were proposed to reduce positive
bias in the cold region \cite{lim2018pet}, i.e., the region that has zero count. In this work, we shall focus on the last two constraints.

The constraints $\mathcal{C}_2$ and $\mathcal{C}_3$ can be unified, which is useful for the discussions below.
\begin{definition}
	For each likelihood factor $p(y_i|x)$ with the constraint $\mathcal{C}_2$, let
	\begin{equation*}
		V_i^+ = \{x|[Ax]_i=a^t_ix>0\}
	\quad
	\mbox{and}\quad
		V_i^- = \mathbb{R}^n\backslash V_i^+.
	\end{equation*}
	For each likelihood factor $p(y_i|x)$ with the constraint $\mathcal{C}_3$, let
	\begin{equation}
		V_i^+ = \{x|[Ax+r]_i=a^t_ix+r_i>0\}
	\quad\mbox{and}\quad
		V_i^- = \mathbb{R}^n\backslash V_i^+.
	\end{equation}
	Then the constraints $\mathcal{C}_2$ and $\mathcal{C}_3$ are both given by $V^+ = \cap_iV^+_i$
		and $V^- = \mathbb{R}^n\backslash V^+$.
\end{definition}

With the indicator function $\mathbf{1}_{V^+}(x)$ of the set $V^+$, we modify the likelihood function $p(y|x)$ by
\begin{equation}
	\ell(x) = p(y|x)\mathbf{1}_{V^+}(x). 
\end{equation}
This extends the domain of $p(y|x)$ from $V^+$ to $\mathbb{R}^n$, and it facilitates
a full Bayesian treatment. Since the indicator function $\mathbf{1}_{V^+}(x)$ admits a separable form, i.e.,
$\mathbf{1}_{V^+}(x) = \prod_{i=1}^{m_1} \mathbf{1}_{V^+_i}(x)$,  $\ell(x)$ factorizes into
\begin{equation}
	\ell(x) = \prod_{i=1}^{m_1} \ell_i(x)\quad \mbox{with } \ell_i(x) = p(y_i|x)\mathbf{1}_{V^+_i}(x).
\end{equation}

To fully specify the Bayesian model, we have to stipulate the prior $p(x)$. We
focus on a Laplace type prior.
Let $L\in\mathbb{R}^{m_2\times n}$ and $L_i^t\in\mathbb{R}^{n\times 1}$ be the $i$th row of $L$.
Then a Laplace type prior $p(x)$ is given by
\begin{equation}
    p(x)=\prod_{i=1}^{m_2} p_i(x) \quad\mbox{with } 	p_i(x) = \tfrac{\alpha}{2}e^{-\alpha|L_i^tx|}.
\end{equation}
The parameter $\alpha>0$ determines the strength of the prior, playing the role of a regularization parameter
in variational regularization \cite{ItoJin:2015}. The choice of the hyperparameter $\alpha$ in the prior
$p(x)$ is notoriously challenging \cite{ItoJin:2015}. One may apply hierarchical Bayesian modeling in order
to estimate it from the data simultaneously with $q(x)$ \cite{WangZabaras:2005,JinZou:2009,ArridgeItoJinZhang:2018}.
The prior $p(x)$ is commonly known as a sparsity prior (in the transformed domain), which favors a candidate
with many small elements and few large elements in the vector $Lx$. The canonical total variation prior is
recovered when the matrix $L$ computes the discrete gradient. It is well known that the total variation
penalty can preserve well edges in the image/signals, and hence it has been very popular for various imaging
tasks \cite{RudinOsherFatemi:1992,ChanShen:2005}.

By Bayes' formula, we obtain the Bayesian solution to the Poisson inverse problem, i.e., the posterior
probability density function:
\begin{equation}\label{eqn:posterior}
	p(x|y) = Z^{-1}\prod^{m_1}_{i=1}\ell_i(x)\prod^{m_2}_{i=1}p_i(x),
\end{equation}
where $Z$ is the normalizing constant, defined by
$Z = \int_{\mathbb{R}^n}\prod^{m_1}_{i=1}\ell_i(x)\prod^{m_2}_{i=1}p_i(x) {\rm d}x.$
The computation of $Z$ is generally intractable for high-dimensional problems, and
$p(x|y)$ has to be approximated.

\section{Approximate inference by expectation propagation}\label{sec:EP}
In this section, we describe the basic concepts and algorithms of expectation propagation (EP),
for exploring the posterior distribution \eqref{eqn:posterior}. EP due to
Minka \cite{minka2001expectation,minka2001family} is a popular variational type approximate
inference method in the machine learning literature. It is especially suitable for approximating
a distribution formed by a product of functions, with each factor being of projection form. Since its first appearance in 2001, EP has
found many successful applications in practice, and it is reported to be very accurate, e.g., for Gaussian
processes \cite{rasmussen2004gaussian}, and electrical impedance tomography with sparsity prior
\cite{gehre2014expectation}. However, the theoretical understanding
of EP remains quite limited \cite{dehaene2018expectation, dehaene2015bounding}.

EP looks for an approximate Gaussian distribution $q(x)$ to a target distribution by means of an iterative algorithm.
It relies on the following factorization of the posterior distribution $p(x|y)$ (with $m=m_1+m_2$ being the total
number of factors):
\begin{equation}\label{eqn:post-factor}
	p(x|y) = Z^{-1}\prod_{i=1}^{m}t_i(x),\quad \mbox{with } t_i(x) =
	\begin{cases}
		\ell_i(x),&\quad i=1,\ldots,m_1,\\
		p_{i-m_1}(x), &\quad i = m_1+1,\ldots,m.
	\end{cases}
\end{equation}
Note that each factor $t_i(x)$ is a function defined on the whole space $\mathbb{R}^n$.
Likewise, we denote the Gaussian approximation $q(x)$ to the posterior distribution $p(x|y)$ by
\begin{equation*}
	q(x) = \tilde{Z}^{-1}\prod_{i=1}^{m}\tilde{t}_i(x),
\end{equation*}
with each factor $\tilde{t}_i(x)$ being a Gaussian distribution $\mathcal{N}(x|\mu_i,C_i)$,
and $\tilde{Z}$ is the corresponding normalizing constant. Below we use two different
parameterizations of a Gaussian distribution, i.e., moment parameters (mean and covariance)
$(\mu,C)$ and  natural parameters $(h,\Lambda)$; see Appendix \ref{app:Gaussian}. Both
parameterizations have their pros and cons: the moment one does not require solving
linear systems, and the natural one allows singular covariances for the Gaussians $\tilde t_i(x)$. The
rest of this section is devoted to the derivation of the algorithms and their complexity.

\subsection{Reduction to one-dimensional integrals}

There are two  main steps of one EP iteration: (a) form a tilted distribution $\hat{q}_i(x)$, and (b) update the
Gaussian approximation $q(x)$ by matching its moments with that of $\hat{q}_i(x)$. The moment matching step can be
interpreted as minimizing Kullback-Leilber divergence $\mathrm{KL}(\hat{q}_i||q)$ \cite{minka2001expectation,
minka2001family,gehre2014expectation}. Recall that the Kullback-Leibler divergence from one probability distribution
$p(x)$ to another $q(x)$ is defined by \cite{KullbackLeibler:1951}
\begin{equation}
	\mathrm{D}_{\rm KL}(p||q) = \int_{\mathbb{R}^n} p(x)\log\frac{p(x)}{q(x)}\mathrm{d}x.
\end{equation}
By Jensen's inequality, the divergence ${\rm D}_{\rm KL}(p||q)$ is always nonnegative, and it
vanishes if and only if $p(x)=q(x)$ almost everywhere.

The task at step (a) is to construct the $i$th tilted distribution $\hat q_i(x)$.
Let $q_{\backslash i}(x)$ be the $i$th cavity distribution, i.e., the product of all but the $i$th factor,
and defined by
\begin{equation}
	q_{\backslash i}(x) = Z^{-1}_i\prod_{j\neq i}\tilde{t}_i(x)
\end{equation}
with $Z_i = \int_{\mathbb{R}^n} \prod_{j\neq i}\tilde{t}_i(x)\text{d}x$.
It is Gaussian, i.e., $q_{\backslash i}(x)=\mathcal{N}(x|\mu_{\backslash i},C_{\backslash i})$,
whose moment and natural parameters are denoted by $(\mu_{\backslash i},C_{\backslash i})$ and
$(h_{\backslash i},\Lambda_{\backslash i})$, respectively.
Then the $i$th tilted distribution $\hat{q}_{i}(x)$ of the approximation $q(x)$ is given by
\begin{equation}
	\hat{q}_{i}(x) = \hat{Z}^{-1}_i t_i(x)\prod_{j\neq i}\tilde{t}_i(x),
\end{equation}
where $\hat{Z}_i=\int_{\mathbb{R}^n} t_i(x)\prod_{j\neq i}\tilde{t}_i(x)\text{d}x$ is the corresponding normalizing constant.
With the exclusion-inclusion step, one replaces the $i$th factor $\tilde t_i(x)$ in the
approximation $q$ with the exact one $t_i(x)$.

The task at step (b) is to compute moments of the $i$th tilde distribution $\hat q_i(x)$, which are then
used to update the approximation $q(x)$. This requires integration over $\mathbb{R}^n$, which is generally
numerically intractable, if $\hat q_i(x)$ were arbitrary. Fortunately, each factor $t_i(x)$ in
\eqref{eqn:post-factor} is of projection form and depends only on the scalar $u^tx$, with the vector
$u\in\mathbb{R}^n$ being either $a_i$ or $L_i$. This is the key fact rendering relevant high-dimensional
integrals numerically tractable. Below we write the factor $t_i(x)$ as $t_i(u^t_ix)$ and accordingly,
the $i$th cavity function $\hat q_i(x)$ as
\begin{equation}
	\hat{q}_{i}(x) = \hat{Z}^{-1}_i t_i(u^t_ix)\mathcal{N}(x|\mu_{\backslash i},C_{\backslash i}),
\end{equation}
upon replacing $\prod_{j\neq i}\tilde{t}_i(x)$ with its normalized version
$\mathcal{N}(x|\mu_{\backslash i},C_{\backslash i})$, and accordingly the normalizing
constant $\hat{Z}_i$.

Since a Gaussian is determined by its mean and covariance, it suffices to evaluate
the 0th to 2nd moments of $\hat q_i(x)$. The projection form of the factor $t_i$ allows
reducing the moment evaluation of $\hat q_i(x)$ to 1D integrals. Theorem \ref{lem:int} gives the
update scheme for  $q(x)$ from $\hat q_{i}(x)$.

\begin{theorem}\label{lem:int}
The normalizing constant $\hat Z_i:=\int_{\mathbb{R}^n} t_i(u^t_ix)\mathcal{N} (x|\mu_{\backslash i},C_{\backslash i})\mathrm{d}x$ is given by
\begin{align*}
\hat{Z}_i&= \int_{\mathbb{R}} t_i(s)\mathcal{N}(s|u^t_i\mu_{\backslash i},u^t_iC_{\backslash i}u_i)\mathrm{d}s =: Z_s
\end{align*}
Then with the auxiliary variables $\bar{s}\in\mathbb{R}$ and $C_s$ defined by
\begin{equation}\label{eqn:sCs}
  \bar{s} = Z_s^{-1}\int_{\mathbb{R}}t_i(s)\mathcal{N}(s|u^t_i\mu_{\backslash i},u^t_iC_{\backslash i}u_i)s\mathrm{d}s\quad \mbox{and}\quad
   C_s =Z_s^{-1}\int_{\mathbb{R}}t_i(s)\mathcal{N}(s|u^t_i\mu_{\backslash i},u^t_iC_{\backslash i}u_i)s^2\mathrm{d}s-\bar{s}^2,
\end{equation}
the mean $\mu=\mathbb{E}_{\hat{q}_i}[x]$ and covariance $C=\mathbb{V}_{\hat{q}_i}[x]$ are given respectively by
\begin{align*}
  \mu & = \mu_{\backslash i}+C_{\backslash i}u_i(u^t_iC_{\backslash i}u_i)^{-1}(\bar{s}-u^t_i\mu_{\backslash i}),\\
  C &= C_{\backslash i} + (u_i^tC_{\backslash i}u_i)^{-2}(C_s-u_i^tC_{\backslash i}u_i)C_{\backslash i}u_iu_i^tC_{\backslash i}.
\end{align*}
Similarly, the precision mean $h_{\hat q_i}$ and precision $\Lambda_{\hat q_i}$ are given respectively by
\begin{align*}
    h_{\hat{q}_i}     & = h_{\backslash i}+\lambda_{1,i}u_i\quad \mbox{with } \lambda_{1,i} = \frac{\bar{s}}{C_s}-\frac{u_i^t\mu_{\backslash i}}{u^t_iC_{\backslash i}u_i},\\
   \Lambda_{\hat{q}_i}&= \Lambda_{\backslash i}+\lambda_{2,i} u_iu^t_i\quad \mbox{with } \lambda_{2,i} = \frac{1}{C_s}-\frac{1}{u^t_iC_{\backslash i}u_i}.
\end{align*}
\end{theorem}
\begin{proof}
The expressions for $\hat{Z}_i$, $\mu$ and $C$ were given in
\cite[Section 3]{gehre2014expectation}. Thus it suffices to derive the formulas for $(h,\Lambda)$. Recall the Sherman-Morrison formula
\cite[p. 65]{golub2012matrix}: for any invertible $B\in\mathbb{R}^{n\times n}$, $u,v\in\mathbb{R}^n$, there holds
\begin{equation}\label{eqn:sherman}
  (B+uv^t)^{-1} = B^{-1} - \frac{B^{-1}uv^tB^{-1}}{1+v^tB^{-1}u}.
\end{equation}
Let $\lambda = (u_i^tC_{\backslash i}u_i)^{-2}(C_s-u_i^tC_{\backslash i}u_i)$. Then the precision matrix $\Lambda$ is given by
\begin{align*}
	\Lambda &= (C_{\backslash i}+C_{\backslash i}u_i\lambda u_i^tC_{\backslash i})^{-1}\\
			&= C_{\backslash i}^{-1}-u_i(\lambda^{-1}+u_i^tC_{\backslash i}u_i)^{-1}u_i^t\\
			&= \Lambda_{\backslash i}+\Big(\frac{1}{C_s}-\frac{1}{u^t_iC_{\backslash i}u_i}\Big)u_iu^t_i.
\end{align*}
Similarly, the precision mean $h:=\Lambda\mu$ is given by
\begin{align*}
	h 	  &= \Big[\Lambda_{\backslash i}+\Big(\frac{1}{C_s}-\frac{1}{u^t_iC_{\backslash i}u_i}\Big)u_iu^t_i\Big][\mu_{\backslash i}+C_{\backslash i}u_i(u^t_iC_{\backslash i}u_i)^{-1}(\bar{s}-u^t_i\mu_{\backslash i})]\\
	  &= \Lambda_{\backslash i}\mu_{\backslash i}+u_i\Big(\frac{\bar{s}}{C_s}-\frac{u_i^t\mu_{\backslash i}}{u^t_iC_{\backslash i}u_i}\Big)= h_{\backslash i}+u_i\Big(\frac{\bar{s}}{C_s}-\frac{u_i^t\mu_{\backslash i}}{u^t_iC_{\backslash i}u_i}\Big).
\end{align*}
This completes the proof of the theorem.
\end{proof}

In both approaches, the 1D integrals $(Z_s,\bar{s},C_s)$ are needed, which depend on $u^t_i\mu_{
\backslash i}$ and $u^t_iC_{\backslash i}u_i$. A direct approach is first to downdate (the Cholesky factor
of) $\Lambda$ and then to solve a linear system. In practice, this can be expensive and the cost can be
mitigated. Indeed, they can be computed without the downdating step; see Lemma \ref{lm:1d_moments} below.
Below we use the super- or subscript \textit{n} and \textit{o} to denote a
variable updated at current iteration from that of the last iteration.
\begin{lemma}\label{lm:1d_moments}
Let $c = u^t_i\Lambda^{-1}_{o}u_i=u^t_iC_{o}u_i$, $(h,\Lambda)$ be the natural parameter of $q(x)$ and
$(\lambda_{1,i}, \lambda_{2,i})$ be defined in Theorem \ref{lem:int}. Then the mean $ u^t_i\mu_{\backslash i}$ and variance $u^t_iC_{\backslash i}u_i$ of the
Gaussian distribution $\mathcal{N}(s|u^t_i\mu_{\backslash i},u^t_iC_{\backslash i}u_i)$ are respectively given by
	\begin{equation}\label{eq:1_mean_var}
		u^t_i\mu_{\backslash i} = \frac{u^t_i\Lambda^{-1}_oh-c\lambda^{o}_{1,i}}{1-c\lambda^{o}_{2,i}}
    \quad\mbox{and}\quad
		u^t_iC_{\backslash i}u_i = \frac{c}{1-c\lambda^{o}_{2,i}}.
	\end{equation}
\end{lemma}
\begin{proof}
We suppress the sub/superscript $o$. By the definition of $u^t_iC_{\backslash i}u_i  $ and the Sherman-Morrison formula \eqref{eqn:sherman}, we have
\begin{align*}
	u^t_iC_{\backslash i}u_i &= u^t_i(\Lambda-\lambda_{2,i}u_iu^t_i)^{-1}u_i\\
		&= u^t_i[\Lambda^{-1}-\Lambda^{-1}u_i(-\lambda_{2,i}^{-1}+c)^{-1}u^t_i\Lambda^{-1}]u_i\\
		&= c - c(-\lambda_{2,i}^{-1}+c)^{-1}c
		= \frac{c}{1-c\lambda_{2,i}}, 
\end{align*}
and similarly, we have
\begin{align*}
		u^t_i\mu_{\backslash i} &= u^t_i(\Lambda-\lambda_{2,i}u_iu^t_i)^{-1}(h-\lambda_{1,i}u_i)\\
		&= u^t_i[\Lambda^{-1}-\Lambda^{-1}u_i(-\lambda_{2,i}^{-1}+c)^{-1}u^t_i\Lambda^{-1}](h-\lambda_{1,i}u_i)
        = \frac{u^t_i\Lambda^{-1}h-c\lambda_{1,i}}{1-c\lambda_{2,i}}.
\end{align*}
This completes the proof of the lemma.
\end{proof}

Since the quantities for the 1D integrals can be calculated from variables updated in the last iteration,
it is unnecessary to form cavity distributions. Indeed, the cavity precision is formed by $\Lambda_{\backslash
i} = \Lambda_o - \lambda_{2,i}^o u_iu_i^t,$ and the updated precision is given by $\Lambda_n = \Lambda_{\backslash
i} + \lambda_{2,i}^n u_iu_i^t;$ and similarly for $h$. Thus, we can update $\Lambda$ directly with $(\lambda_{2,i}^o,
\lambda_{2,i}^n)$ and $h$ with $(\lambda_{1,i}^o, \lambda_{1,i}^n)$; this is summarized in the next remark.
\begin{remark}\label{rm:lbd_difference}
The differences $\lambda_{k,i}^{n}-\lambda_{k,i}^{o}$, $k=1,2$, can be used to update
the natural parameter $(h,\Lambda)$:
	\begin{equation}
		\lambda_{1,i}^{n}-\lambda_{1,i}^{o} = \frac{\bar{s}}{C_s} - \frac{u^t_i\mu_{o}}{u^t_iC_{o}u_i}
	\quad\mbox{and}\quad
		\lambda_{2,i}^{n}-\lambda_{2,i}^{o} = \frac{1}{C_s} - \frac{1}{u^t_iC_{o}u_i}.
	\end{equation}
Moreover, the sign of $\lambda_{2,i}^{n}-\lambda_{2,i}^{o}$ determines whether to update
or downdate the Cholesky factor of $\Lambda$. 
\end{remark}

\subsection{Update schemes and algorithms}
Now we state the direct update scheme, i.e. without explicitly constructing the intermediate cavity distribution
$q_{\backslash i}(x)$, for both natural and moment parameterizations.
\begin{theorem}\label{thm:update}
Let $(h,\Lambda)$ and $(\mu,C)$ be the natural and moment parameters of the Gaussian approximation $q(x)$, respectively.
The following update schemes hold.
\begin{itemize}
 \item[$\rm(i)$] The precision mean $h$ and precision $\Lambda$ can be updated by
\begin{equation}\label{eq:update_natural}
\begin{split}
	h_{n} &= h_{o} + \Big(\frac{\bar{s}}{C_s} - \frac{u^t_i\Lambda^{-1}_{o}h_{o}}{u^t_i\Lambda^{-1}_{o}u_i}\Big)u_i\quad \mbox{and}
	\quad \Lambda_{n} = \Lambda_{o} + \Big(\frac{1}{C_s} - \frac{1}{u^t_i\Lambda^{-1}_{o}u_i}\Big)u_iu_i^t.
\end{split}
\end{equation}
\item[$\rm(ii)$] The mean $\mu$ and covariance $C$ can be updated by
\begin{equation}\label{eq:update_moment}
\begin{split}
	\mu_{n} &= \mu_{o} + \frac{\bar{s}-u^t_i\mu_{o}}{u^t_iC_{o}u_i} C_{o}u_i\quad\mbox{and}\quad
	C_{n} = C_{o} +\Big( \frac{C_s}{(u^t_iC_{o}u_i)^2}-\frac{1}{u^t_iC_{o}u_i}\Big)(C_{o}u_i)(u^t_iC_{o}).
\end{split}
\end{equation}
\end{itemize}
\end{theorem}
\begin{proof}
The first assertion is direct from Theorem \ref{lem:int} and Remark \ref{rm:lbd_difference}, and it can be rewritten as
\begin{equation}
		\Lambda_{n} = \Lambda_{o} + (\lambda_{2,i}^{n}-\lambda_{2,i}^{o})u_iu_i^t\quad \mbox{and}\quad
	h_{n} = h_{o} + (\lambda_{1,i}^{n}-\lambda_{1,i}^{o})u_i.
\end{equation}
By Sherman-Morrison formula \eqref{eqn:sherman}, the covariance $C_{n}=\Lambda_{n}^{-1}$ is given by
\begin{align*}
	C_{n} & = (\Lambda_{o}+(\lambda^{n}_{2,i}-\lambda^{o}_{2,i})u_iu^t_i)^{-1}\\
	&= \Lambda_{o}^{-1} - \Lambda_{o}^{-1}u_i\Big(\frac{1}{\lambda^{n}_{2,i}-\lambda^{o}_{2,i}}+u^t_iC_{o}u_i\Big)^{-1}u^t_i\Lambda_{o}^{-1}\\
	&=: C_{o} + \eta_2(C_{o}u_i)(u^t_iC_{o}),
\end{align*}
where the scalar $\eta_2 := -(\frac{1}{\lambda^{n}_{2,i}-\lambda^{o}_{2,i}}+u^t_iC_{o}u_i)^{-1}$
can be simplified to
\begin{align*}
		\eta_2 = -\frac{\lambda^{n}_{2,i}-\lambda^{o}_{2,i}}{1+(\lambda^{n}_{2,i}-\lambda^{o}_{2,i})u^t_iC_{o}u_i}= -\frac{1}{u^t_iC_{o}u_i} + \frac{C_s}{(u^t_iC_{o}u_i)^2},
\end{align*}
where the second identity follows from Remark \ref{rm:lbd_difference}.
Similarly, the mean $\mu_{n} := C_{n}h_{n} $ is given by
\begin{align*}
\mu_{n}&= [C_{o} +\eta_2(C_{o}u_i)(u^t_iC_{o})][h_{o}+(\lambda_{1,i}^{n}-\lambda_{1,i}^{o})u_i]\\
		&= \mu_{o} + (\lambda_{1,i}^{n}-\lambda_{1,i}^{o})C_{o}u_i +\eta_2 u^t_i\mu_{o}C_{o}u_i+\eta_2(\lambda_{1,i}^{n}-\lambda_{1,i}^{o})u^t_iC_{o}u_iC_{o}u_i=: \mu_{o} + \eta_1 C_{o}u_i,
\end{align*}
where, in view of Remark \ref{rm:lbd_difference}, $\eta_1:=(\lambda_{1,i}^{n}-\lambda_{1,i}^{o})
+\eta_2 u^t_i\mu_{o}+\eta_2 (\lambda_{1,i}^{n}-\lambda_{1,i}^{o})u^t_iC_{o}u_i$ can be simplified to
\begin{align*}
	\eta_1
	&= \frac{(\lambda^{n}_{1,i}-\lambda^{o}_{1,i})-(\lambda^{n}_{2,i}-\lambda^{o}_{2,i})u^t_i\mu_{o}}{1+(\lambda^{n}_{2,i}-\lambda^{o}_{2,i})u^t_iC_{o}u_i}
	= \frac{\bar{s}-u^t_i\mu_{o}}{u^t_iC_{o}u_i}.
\end{align*}
This completes the proof of the theorem.
\end{proof}

All matrix operations in Theorem \ref{thm:update} are of rank one type, which can be implemented stably and efficiently
with the Cholesky factors and their update / downdate; see Section \ref{ssec:complexity} for details. Thus, in practice,
we employ Cholesky factors of the precision $\Lambda$ and covariance $C$, denoted by $\Lambda_{chol}$ and
$C_{chol}$, respectively, instead of $\Lambda$ and $C$. Further, we also use the auxiliary variables $(\lambda_{1,i},
\lambda_{2,i})$ defined in Theorem \ref{lm:1d_moments}, and stack $\{(\lambda_{1,i},\lambda_{2,i})\}_{i=1}^{m_1+m_2}$ into two vectors
\begin{equation*}
\lambda_1=[\lambda_{1,i}]_i,\quad \lambda_2=[\lambda_{2,i}]_i\in\mathbb{R}^{m_1+m_2},
\end{equation*}
which are initialized to zeros. Thus, we obtain two inference procedures for Poisson data with a
Laplace type prior in Algorithms \ref{alg:ep_natural} and \ref{alg:ep_moment}.

The rigorous convergence analysis of EP is outstanding. Nonetheless, empirically, it often converges very fast, which
is also observed in our numerical experiments in Section \ref{sec:numer}. In practice, one can terminate the iteration
by monitoring the relative change of the parameters or fixing the maximum number $K$ of iterations. The important task
of computing 1D integrals will be discussed in Section \ref{sec:int} below.

\begin{algorithm}[hbt!]
\centering
\caption{Expectation propagation for Poisson data (natural parametrization)\label{alg:ep_natural}}
\begin{algorithmic}[1]
	\STATE  Input: $(A,y)$, hyper-parameter $\alpha$, and maximum number $K$ of iterations
    \STATE  Initialize $h$, $\Lambda_{chol}$, $\lambda_1$ and $\lambda_2$;
	\FOR{$k=1,2,\ldots,K$}
		\STATE Randomly choose an index $i$ to update;
		\STATE Compute the mean and variance for 1D Gaussian integral by Lemma \ref{lm:1d_moments};
		\STATE Evaluate $\bar{s}$ and $C_s$ in \eqref{eqn:sCs};
		\STATE Calculate and update $\lambda_{1,i}$ and $\lambda_{2,i}$;
		\STATE Update $h$ and $\Lambda_{chol}$ by Theorem \ref{thm:update};
    \STATE Check the stopping criterion.
	\ENDFOR
    \STATE Output: $(h,\Lambda_{chol})$
\end{algorithmic}
\end{algorithm}

\begin{algorithm}[hbt!]
\centering
\caption{Expectation propagation for Poisson data (moment parametrization)\label{alg:ep_moment}}
\begin{algorithmic}[1]
	\STATE  Input: $(A,y)$, hyper-parameter $\alpha$, and maximum number $K$ of iterations
    \STATE  Initialize $\mu$, $C_{chol}$, $\lambda_1$ and $\lambda_2$;
	\FOR{$k=1,2,\ldots,K$}
		\STATE Randomly choose an index $i$ to update;
		\STATE Compute the mean and variance for 1D Gaussian integral by Lemma \ref{lm:1d_moments};
		\STATE Evaluate $\bar{s}$ and $C_s$ in \eqref{eqn:sCs};
		\STATE Calculate and update $\lambda_{1,i}$ and $\lambda_{2,i}$;
		\STATE Update $\mu$ and $C_{chol}$ by Theorem \ref{thm:update};
    \STATE Check the stopping criterion.
	\ENDFOR
    \STATE Output: $(\mu,C_{chol})$
\end{algorithmic}
\end{algorithm}

\subsection{Efficient implementation and complexity estimate}\label{ssec:complexity}

The rank-one matrix update $A\pm \beta uu^t$, for $A\in\mathbb{R}^{n\times n}$, $u\in\mathbb{R}^n$ and $\beta>0$,
can be stably and efficiently updated / downdated with the Cholesky factor of $A$ with $\sqrt{\beta}u$.
The update step of $A$ can be viewed as an iteration from $A_k$ to $A_{k+1}$. Let the upper triangular matrices
$R_k$ and $R_{k+1}$ be the Cholesky factors of $A_k$ and $A_{k+1}$ respectively, i.e., $A_k = R_k^tR_k$ and
$A_{k+1}=R_{k+1}^tR_{k+1}$. There are two possible cases:
\begin{enumerate}
\item[(i)] If $A_{k+1}=A_k+\beta uu^t$, $R_{k+1}$ is the Cholesky rank one update of $R_k$ with $\sqrt{\beta}u$.
\item[(ii)]If $A_{k+1}=A_k-\beta uu^t$, $R_{k+1}$ is the Cholesky rank one downdate of $R_k$ with $\sqrt{\beta}u$.
\end{enumerate}
The update/downdate is available in several packages. For example, in \texttt{MATLAB},
the function \texttt{cholupdate} implements the update/downdate of Cholesky factors,
based on \texttt{LAPACK} subroutines \texttt{ZCHUD} and \texttt{ZCHDD}.

Next, we discuss the computational complexity per inner iteration. The first step picks one
index $i$, which is of constant complexity. For the second step, i.e., computing the mean and variance for
1D integrals, the dominant part is linear solve involving upper triangular matrices and matrix-vector product
for natural and moment parameters. For either parameterization, it incurs $\mathcal{O}(n^2)$ operations.
The third step computes $\bar{s}$ and $C_s$ from the one dimensional integrals. For Poisson site, the
complexity is $\mathcal{O}(y_i)$, and for Laplace site, it is $\mathcal{O}(1)$. Last, the fourth step is
dominated by Cholesky factor modifications, and its complexity is $\mathcal{O}(n^2)$.
Overall, the computational complexity per inner iteration is $\mathcal{O}(n^2+y_i)$.
In a large data setting, $y_i\ll n$, and thus the complexity is about $\mathcal{O}(n^2)$.

In passing, we note that in practice, the covariance / precision matrix may admit additional structures, e.g.,
sparsity, which translate into structures on Cholesky factors. For
the general sparsity assumption, it seems unclear how to effectively exploit it
for Cholesky update/downdate for enhanced efficiency, except the diagonal case,
which can be incorporated into the algorithm straightforwardly.

\section{Stable evaluation of 1d integrals}\label{sec:int}

Now we develop a stable implementation for the three 1D integrals: $Z_s$, $\bar s$ and $C_s$
in Theorem \ref{lem:int}. These integrals form the basic components of Algorithms \ref{alg:ep_natural} and
\ref{alg:ep_moment}, and their stable, accurate and efficient evaluation is crucial to the performance of the
algorithms. By suppressing the subscript $i$, we can write the integrals in a unified way:
\begin{equation*}
	J_j = \int_{\mathbb{R}} t(s)\mathcal{N}(s|m,\sigma^2)s^j\text{d}s,\quad j=0,1,2,
\end{equation*}
where the factor $t(s)$ is either Poisson likelihood or Laplace prior.
Then we can express $\bar{s}$ and $C_s$ in terms of $J_j$ by
\begin{equation}
	\bar{s} = \frac{J_1}{J_0} \quad\text{and}\quad C_s = \frac{J_2}{J_0} - \bar{s}^2.
\end{equation}
Note that the normalizing constants in $J_j$ cancel out in $\bar{s}$ and $C_s$, and thus they can be ignored
when evaluating the integrals. In essence, the computation boils down to stable evaluation of moments of a (truncated) Gaussian
distribution. This task was studied in several works \cite{cunningham2011gaussian, seeger2008bayesian}:
\cite{cunningham2011gaussian} focuses on Gaussian moments, and \cite{seeger2008bayesian} discusses also evaluating
the integrals involving Laplace distributions. Below we derive the formulas for the (constrained) Poisson likelihood and Laplace
prior separately.

\subsection{Poisson likelihood}
Throughout, we suppress the subscript $i$, write $V_+$ etc in place of $V^+_{i}$ etc and introduce
the scaler variable $s=a^tx$. Then the constraint on $x$ transfers to that on $s$: $a^tx>0$ corresponds to
$s>0$ and $a^tx+r>0$ to $s>-r$, respectively. We shall slightly abuse the notation and use
$\mathbf{1}_{V_+}(s)$ as the indicator for the constraint on $s$. Then the Poisson likelihood $t(x)$
can be equivalently written in either $x$ or $s$ as
\begin{equation}
	t(x) = \frac{(a^tx+r)^{y}e^{-(a^tx+r)}}{y!}\mathbf{1}_{V_+}(x) \quad \mbox{and}\quad
	t(s) = \frac{(s+r)^{y}e^{-(s+r)}}{y!}\mathbf{1}_{V_+}(s).
\end{equation}
Note that the factorial $y!$ cancels out when computing $\bar{s}$ and $C_s$, so it is omitted in the derivation
below. For a fixed $\mathcal{N}(s|m,\sigma^2)$, the integrals $J_{y,j}$ depend on the
observed count data $y$ and moment order $j$:
\begin{equation}
	J_{y,j}	= \int_{b}^\infty (s+r)^{y}s^je^{-(s+r)}\mathcal{N}(s|m,\sigma^2) \text{d}s.
\end{equation}
where the lower integral bound $b=0$ or $b=-r$, which is evident from the context. Note that the terms
$e^{-(s+r)}$ and $\mathcal{N}(s|m,\sigma^2)$ in $J_{y,j}$ together give an unormalized Gaussian density.
This allows us to reduce the integrals $J_{y,j}$ into (truncated) Gaussian moment evaluations of the type:
\begin{equation}
	I_y = \int_{b}^\infty (s+r)^y\mathcal{N}(s|m-\sigma^2,\sigma^2)\text{d}s,
\end{equation}
and accordingly $\bar s$ and $C_s$. This is given in the next result.

\begin{theorem}
The scalars $\bar{s}$ and $C_s$ can be computed by
\begin{equation}
\bar{s} = \frac{I_{y+1}}{I_y} - r\quad \mbox{and}\quad C_s = \frac{I_{y+2}}{I_y} - \Big(\frac{I_{y+1}}{I_y}\Big)^2.
\end{equation}
\end{theorem}
\begin{proof}
First, we claim that with $ \alpha=e^{\frac{\sigma^2}{2}-m-r}$, there hold the following identities
\begin{align}\label{eqn:rec-J}
	J_{y,0} = \alpha I_y,\quad
	J_{y,1} = \alpha (I_{y+1} - rI_y),\quad\mbox{and}\quad
	J_{y,2} = \alpha (I_{y+2} - 2rI_{y+1} + r^2I_y).
\end{align}
Let $c_\sigma=(2\pi\sigma^2)^{-\frac12}$. Then by completing the square, we obtain
\begin{equation}
\begin{split}
	e^{-(s+r)}\mathcal{N}(s|m,\sigma^2)
	&= c_\sigma e^{-r-s-\frac{(s-m)^2}{2\sigma^2}}
	= c_\sigma e^{\frac{\sigma^2}{2}-m-r}e^{-\frac{(s-(m-\sigma^2))^2}{2\sigma^2}}.
\end{split}
\end{equation}
The claim follows directly from the trivial identities
\begin{align*}
  	(s+r)^ys &= (s+r)^{y+1}-r(s+r)^y,\\
	(s+r)^ys^2 &= (s+r)^{y+2}-2r(s+r)^{y+1}+r^2(s+r)^y.
\end{align*}
The desired identities follow from the definitions and the recursions in \eqref{eqn:rec-J} by
\begin{align*}
	\bar{s} &= \frac{J_{y,1}}{J_{y,0}}
	        = \frac{\alpha(I_{y+1}-rI_y)}{\alpha I_y}
	        = \frac{I_{y+1}}{I_y} - r,\\
	C_s &= \frac{J_{y,2}}{J_{y,0}} - \bar{s}^2
	    = \frac{\alpha(I_{y+2}-2rI_{y+1}+r^2I_y)}{\alpha I_y} -\Big (\frac{I_{y+1}}{I_y} - r\Big)^2
	    = \frac{I_{y+2}}{I_y} - \Big(\frac{I_{y+1}}{I_y}\Big)^2.
\end{align*}
This completes the proof.
\end{proof}

However, directly evaluating $I_y$ can still be numerically unstable for large $y$. To
avoid the potential instability, we develop a stable recursive scheme on $I_y$.
\begin{lemma}\label{lm:recursive}
For $y\geq 2$, the following recursion holds
\begin{equation}
	I_y = (m-\sigma^2+r)I_{y-1} + \sigma^2(y-1)I_{y-2} + \frac{\sigma^2(b+r)^{y-1}}{\sqrt{2\pi\sigma^2}}e^{-\frac{(b-m+\sigma^2)^2}{2\sigma^2}}.
\end{equation}	
\end{lemma}
\begin{proof}
Let $c=m-\sigma^2$, $d=\sigma^2$ and $f(s)=\frac{1}{\sqrt{2\pi\sigma^2}}e^{-\frac{(s-c)^2}{2d}}$.
The definition of $I_y$ implies
\begin{equation}
\begin{split}
	I_y &= \int_b^\infty(s+r)^yf(s)\mathrm{d}s
		= \int_b^\infty(s+r)^{y-1}\Big(d\frac{s-c}{d}+c+r\Big)f(s)\mathrm{d}s\\
		&= -d\int_b^\infty(s+r)^{y-1}\Big(-\frac{s-c}{d}\Big)f(s)\mathrm{d}s
		+(c+r)\int_b^\infty(s+r)^{y-1}f(s)\mathrm{d}s.
\end{split}
\end{equation}
Next we employ the trivial identity $\frac{\mathrm{d}}{\mathrm{d}s}f(s)= -\frac{s-c}{d}f(s)$ and
apply integration by parts to the first term
\begin{equation}
\begin{split}
	&\hspace{1em}\int_b^\infty(s+r)^{y-1}\Big(-\frac{s-c}{d}\Big)f(s)\mathrm{d}s\\
	&=(s+r)^{y-1}f(s)|_b^\infty - \int_b^\infty (y-1)(s+r)^{y-2}f(s)\mathrm{d}s\\
	&= -(b+r)^{y-1}f(b) - (y-1)I_{y-2}.
\end{split}
\end{equation}
Collecting the terms shows the desired recursion on the integral $I_y$.
\end{proof}

For $b=-r$, we have a simplified recursive formula for $I_y$:
\begin{equation*}
  I_y = (m-\sigma^2+r)I_{y-1} + \sigma^2(y-1)I_{y-2}.
\end{equation*}

Lemma \ref{lm:recursive} uses a two-term linear recurrence relation for $I_y$'s. The coefficients of $I_{y-1}$
and $I_{y-2}$ are raised by power when expanding $I_y$ in terms of $I_0$ and $I_1$, and thus the computation of $I_y$
is susceptible to the evaluation errors of $I_0$ and $I_1$ for large $y$. This motivates a reciprocal recursive scheme
by introducing a ratio sequence $\{L_y\}_y$ defined by $L_y = \frac{yI_{y-1}}{I_y}$, for $r=0$ or $b=-r$, in order to
restore the numerical stability. Note that $L_y$ also admits a recursive scheme $L_y = \frac{y}{(m-\sigma^2+r)+\sigma^2
L_{y-1}}$, and further $I_y$ can be recovered from $\{L_y\}$ by $\ln I_y = \ln y! + \ln I_0 - \sum_{i=1}^y L_i.$
	
We can compute $\bar{s}$ and $C_s$ directly from $L_y$. The identities follow from straightforward computation.
\begin{theorem}
If $r=0$ or $b=-r$, the ratios for calculating $\bar{s}$ and $C_s$ are given by
\begin{equation}
	\frac{I_{y+1}}{I_y}  = (m-\sigma^2+r)+\sigma^2 L_y\quad\mbox{and}\quad
		\frac{I_{y+2}}{I_y} = e^{\ln(y+1)+\ln(y+2)-\ln L_{y+1}-\ln L_{y+2}}.
	\end{equation}
\end{theorem}

Last, we discuss the computation of the first three integrals $I_0$, $I_1$ and $I_2$, which are needed for the
recursion. We employ three different forms according to the integration range with respect to the auxiliary variable
\begin{equation*}
  \eta=\frac{\sigma^2-m+b}{\sqrt{2\sigma^2}}.
\end{equation*}
The formulas are listed in Table \ref{tab:scheme}, where \texttt{erf} and \texttt{erfc} denote the error
function and complementary error function, respectively, and $\text{erfcx}(\eta)=e^{\eta^2}(1-\text{erf}(\eta)).$
Since the value of $1-\text{erf}(\eta)$ is vanishingly small for large $\eta$ value, we use Scheme 2 to avoid
underflow. Scheme 3 is useful when the $\eta$ value is large, since both $1-\text{erf}(\eta)$ and $\text{erfc}
(\eta)$ suffer from numerical underflow. Note that when $\eta$ is small, Scheme 3 is not as accurate as Scheme
2, so we use Scheme 2 in the intermediate range. In our experiments, we use Scheme 1 for $\eta\in(-\infty,5)$,
Scheme 2 for $\eta\in[5,26)$ and Scheme 3 for $\eta\in(26,\infty)$. To use Scheme 3, we construct $\tilde{I}_i
=\frac{I_i}{I_0}$, $i=0,1,2$, and $\tilde{L}_y=\frac{y\tilde{I}_{y-1}}{\tilde{I}_y}$, $y\in\mathbb{N}_+$. Then similar
identities for computing $\bar{s}$ and $C_s$ hold, i.e., $\bar{s} = \frac{\tilde{I}_{y+1}}{\tilde{I}_y} - r$
and $C_s = \frac{\tilde{I}_{y+2}}{\tilde{I}_y} - (\frac{\tilde{I}_{y+1}}{\tilde{I}_y})^2$, with
$\frac{\tilde{I}_{y+1}}{\tilde{I}_y}  = (m-\sigma^2+r)+\sigma^2 \tilde{L}_y$ and
$\frac{\tilde{I}_{y+2}}{\tilde{I}_y} = e^{\ln(y+1)+\ln(y+2)-\ln \tilde{L}_{y+1}-\ln \tilde{L}_{y+2}}$.

\begin{table}[hbt!]
  \centering
  \caption{Three schemes for evaluating $I_0$, $I_1$ and $I_2$, with $c_1=m-\sigma^2+b+2r$ and $c_2=m-\sigma^2+r$. \label{tab:scheme}}
  \begin{tabular}{|c|l|c|}
   \hline
   scheme & formulae & $\eta$\\
    \hline
    1   & $\displaystyle I_0 = \frac{1}{2}(1-\text{erf}(\eta)),\qquad  	I_1 = \sqrt{\frac{\sigma^2}{2\pi}}e^{-\eta^2} + \frac{c_2}{2}(1-\text{erf}(\eta))$ & $(-\infty,5)$ \\
	    &  $\displaystyle I_2 = \sqrt{\frac{\sigma^2}{2\pi}}c_1e^{-\eta^2} + \frac{c_2^2+\sigma^2}{2}(1-\text{erf}(\eta))$& \\
   \hline
   2   & $\displaystyle I_0 = \frac{1}{2}\text{erfc}(\eta),\qquad
	I_1 = \sqrt{\frac{\sigma^2}{2\pi}}e^{-\eta^2} + \frac{c_2}{2}\text{erfc}(\eta)$ & $[5,26]$\\
	   & $\displaystyle I_2 = \sqrt{\frac{\sigma^2}{2\pi}}c_1e^{-\eta^2} + \frac{c_2^2+\sigma^2}{2}\text{erfc}(\eta)$&\\
    \hline
    3  & 	$\displaystyle \tilde{I}_0 = 1,\qquad
	\tilde{I}_1 = \sqrt{\frac{2\sigma^2}{\pi}}\frac{1}{\text{erfcx}(\eta)} + c_2$ & $(26,\infty)$\\
	   & $\displaystyle \tilde{I}_2 = \sqrt{\frac{2\sigma^2}{\pi}}\frac{c_1}{\text{erfcx}(\eta)}+c_2^2+\sigma^2$ &\\
    \hline
  \end{tabular}
\end{table}

\subsection{Laplace potential}
Now we derive the formulas for evaluating the 1D integrals for the Laplace potential $t(x)= \frac{\alpha}{2}e^{-\alpha |\ell^tx|}$.
For any fixed $\ell\in\mathbb{R}^n$, we divide the whole space $\mathbb{R}^n$ into two disjoint half-spaces $V_+$ and $V_-$, i.e.,
$\mathbb{R}^n = V_+ \cup V_-$, with $V_+ = \{x|\ell^t x > 0\}$ and $V_- = \{x|\ell^t x \le 0\}$.
Then we split the Laplace potential $t(x)$ into
\begin{equation}
	t(x)= \frac{\alpha}{2}e^{-\alpha \ell^tx}\mathbf{1}_{V_+}(x) +
	       \frac{\alpha}{2}e^{\alpha \ell^tx}\mathbf{1}_{V_-}(x).
\end{equation}
The integrals involving $t(x)\mathcal{N}(s|\mu,\sigma^2)$ (slightly abusing $\mu$) can be divided into two parts:
\begin{equation}
\begin{split}
\int_{\mathbb{R}_+} \frac{\alpha}{2}s^ie^{-\alpha s}\mathcal{N}(s|\mu,\sigma^2)\text{d}s
	&= \frac{\alpha}{2}e^{\frac{\alpha^2\sigma^2}{2}}\underbrace{e^{-\alpha \mu}\int_{\mathbb{R}_+} s^i\mathcal{N}(s|\mu-\alpha\sigma^2,\sigma^2)\mathrm{d}s}_{:=I^+_i},\\
	\int_{\mathbb{R}_-} \frac{\alpha}{2}s^ie^{\alpha s}\mathcal{N}(s|\mu,\sigma^2)\text{d}s
	&= \frac{\alpha}{2}e^{\frac{\alpha^2\sigma^2}{2}}\underbrace{e^{\alpha \mu}\int_{\mathbb{R}_-} s^i\mathcal{N}(s|\mu+\alpha\sigma^2,\sigma^2)\mathrm{d}s}_{:= I^-_i}.
\end{split}
\end{equation}
By the change of variable $t=\frac{s-\mu\pm\alpha\sigma^2}{\sigma}$ for $I^{\pm}_i$ respectively, we have
\begin{align}
I^+_i &= \frac{e^{-\alpha \mu}}{\sqrt{2\pi}}\int_{-\frac{\mu}{\sigma}+\alpha\sigma}^{+\infty}(\sigma t+\mu-\alpha\sigma^2)^i e^{-\frac{t^2}{2}}\mathrm{d}t,\\
I^-_i &= \frac{(-1)^i e^{\alpha \mu}}{\sqrt{2\pi}}\int_{\frac{\mu}{\sigma}+\alpha\sigma}^{+\infty}(\sigma t-\mu-\alpha\sigma^2)^i e^{-\frac{t^2}{2}}\mathrm{d}t.
\end{align}
These integrals can be expressed using the cumulative distribution function $\Phi$ of the standard Gaussian
distribution. We shall view $I^{\pm}_i$ as functions of $\mu$ and let $I_i = I^+_i(\mu)+(-1)^iI^+_i(-\mu)$. Then we have
\begin{equation}
	\bar{s} = \frac{I_1}{I_0}
\quad \mbox{and}\quad
	C_s = \frac{I_2}{I_0} - \left(\frac{I_1}{I_0}\right)^2.
\end{equation}

To avoid the potential underflow of direct evaluation of $\Phi$,
we use the following well known (divergent) asymptotic expansion \cite[item 7.1.23]{abramowitz1965handbook}
\begin{align*}
  1 - \Phi(\eta) & = \int_\eta^\infty e^{-\frac{t^2}{2}}{\rm d} t
                  = e^{-\frac{\eta^2}{2}}\left(\eta^{-1} +\sum_{k=1}^\infty \frac{(-1)^k(2k-1)!}{2^k(k-1)!}\eta^{-(2k+1)}\right)\\
                 & = \mathcal{N}(\eta|0,1)\eta^{-1}\underbrace{\sum_{n=0}^\infty(-1)^n(2n-1)!!\eta^{-2n}}_{:=g(\eta)}.
\end{align*}
This formula follows by integration by parts, and allows accurate evaluation for large positive $\eta$.
It was shown in \cite{gehre2013rapid} that the error of evaluating $1 - \Phi(\eta)$ with a truncation
of the asymptotic expansion is less than $10^{-11}$ for $\eta>5$ with more than $8$ terms in the
summation of $g(\eta)$. For $\eta\le 5$, $1 - \Phi(\eta)$ can be accurately evaluated directly.
Then we introduce a ratio
\begin{align*}
	\beta &= \frac{I^+_0(-|\mu|)}{I^+_0(|\mu|)}  = e^{2\alpha|\mu|}\frac{(\alpha\sigma^2-|\mu|)g(\alpha\sigma+\frac{|\mu|}{\sigma})}{(\alpha\sigma^2+|\mu|)g(\alpha\sigma-\frac{|\mu|}{\sigma})}.
\end{align*}
With the ratio $\beta$, the two fractions $\frac{I_1}{I_0}$ and $\frac{I_2}{I_0}$ can be evaluated by
\begin{align*}
\frac{I_1}{I_0} &= \mu + \alpha\sigma^2\mathrm{sgn}(\mu)\Big(1-\frac{2}{1+\beta}\Big),\\
\frac{I_2}{I_0} &= -\frac{2\alpha\sigma^3}{\sqrt{2\pi}}e^{-(\frac{\mu^2}{2\sigma^2}+\frac{\alpha^2\sigma^2}{2})}I^{-1}_0+(\sigma^2+\alpha^2\sigma^4-\mu^2)+2\mu\frac{I_1}{I_0}.
\end{align*}
To avoid potential numerical instability of the first term in $\frac{I_2}{I_0}$, we use the identity
\begin{equation*}
	-\frac{2\alpha\sigma^3}{\sqrt{2\pi}}e^{-(\frac{\mu^2}{2\sigma^2}+\frac{\alpha^2\sigma^2}{2})}I^{-1}_0 = \frac{-2\alpha\sigma^2(-|\mu|+\alpha\sigma^2)}{g(-\frac{|\mu|}{\sigma}+\alpha\sigma)(1+\beta)}.
\end{equation*}
To avoid potential numerical instability of the term $\sigma^2+\alpha^2\sigma^4$, we use the exp-log trick
\begin{equation*}
	\sigma^2+\alpha^2\sigma^4 = \exp\Big(-2\log\frac{1}{\alpha\sigma^{2}}+\log\big(1+\frac{1}{\alpha^2\sigma^{2}}\big)\Big),
\end{equation*}
where $\log(1+\frac{1}{\alpha^2\sigma^{2}})$ is evaluated by the $\texttt{MATLAB}$ builtin function $\texttt{log1p}$.
Thus, $\bar{s}$ and $C_s$ can be evaluated by $\bar{s} = \frac{I_1}{I_0}$ and $C_s = -\frac{2\alpha\sigma^2(-|\mu|+\alpha\sigma^2)}{g(-\frac{|\mu|}{\sigma}+\alpha\sigma)(1+\beta)}+\exp[-2\log\frac{1}{\alpha\sigma^{2}}+\log(1+\frac{1}{\alpha^2\sigma^{2}})]-(\mu-\frac{I_1}{I_0})^2$.

\section{Numerical experiments and discussions}\label{sec:numer}

Now we numerically illustrate one EP algorithm on realistic images. In the implementation, we employ the natural parameter
parameterization, i.e., Algorithm \ref{alg:ep_natural}, which appears to be numerically more robust. We measure
the accuracy of a reconstruction $x^*$ relative to the ground truth $x^\dag$ by the standard $L^2$-error $||x^*-
x^\dagger||_2$, the structural similarity (SSIM) index (by \texttt{MATLAB} built-in \texttt{ssim}), and peak
signal-to-noise ratio (PSNR) (by \texttt{MATLAB} built-in \texttt{psnr} with peak value 1 for \texttt{Shepp-Logan} and \texttt{PET} phantom, and 5 for \texttt{IRT} phantom). For comparison, we
also present MAP, computed by a limited-memory BFGS algorithm \cite{LiuNocedal:1989} with constraint $\mathcal{C}_1$.
The hyperparameter $\alpha$ in the prior $p(x)$ is determined in a trial-and-error manner. Unless otherwise stated,
the EP algorithm is run for four sweeps through the sites.

\subsection{Simulated data with two benchmark images}

\begin{figure}[hbt!]
\centering
\begin{tabular}{ccccccccc}
\includegraphics[scale=0.2]{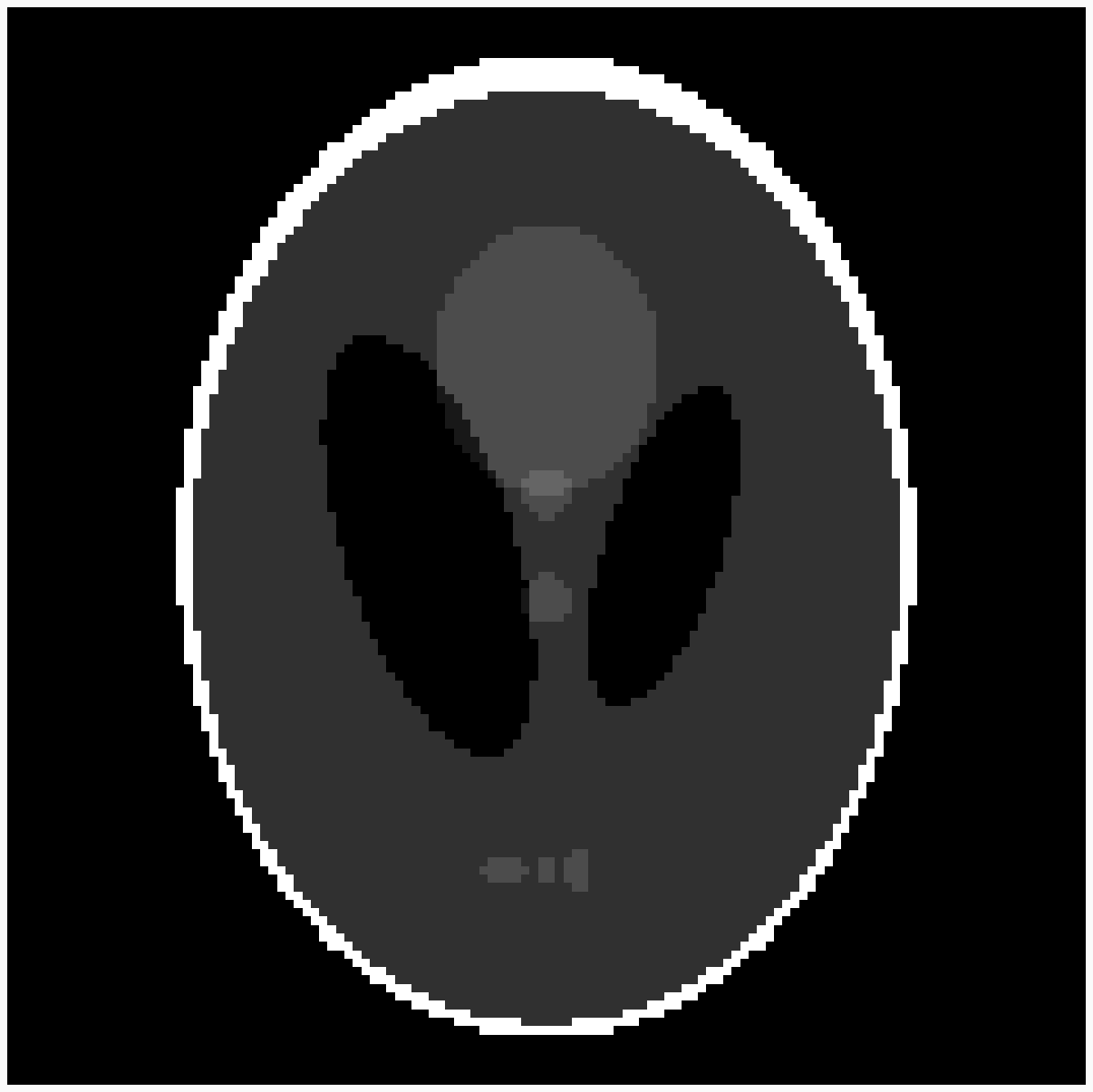} \includegraphics[scale=0.2]{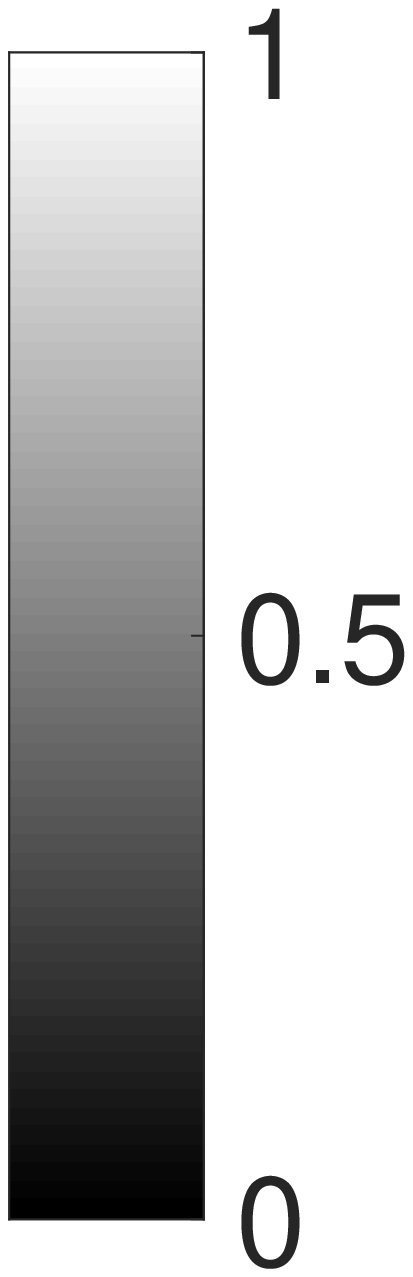} & \includegraphics[scale=0.2]{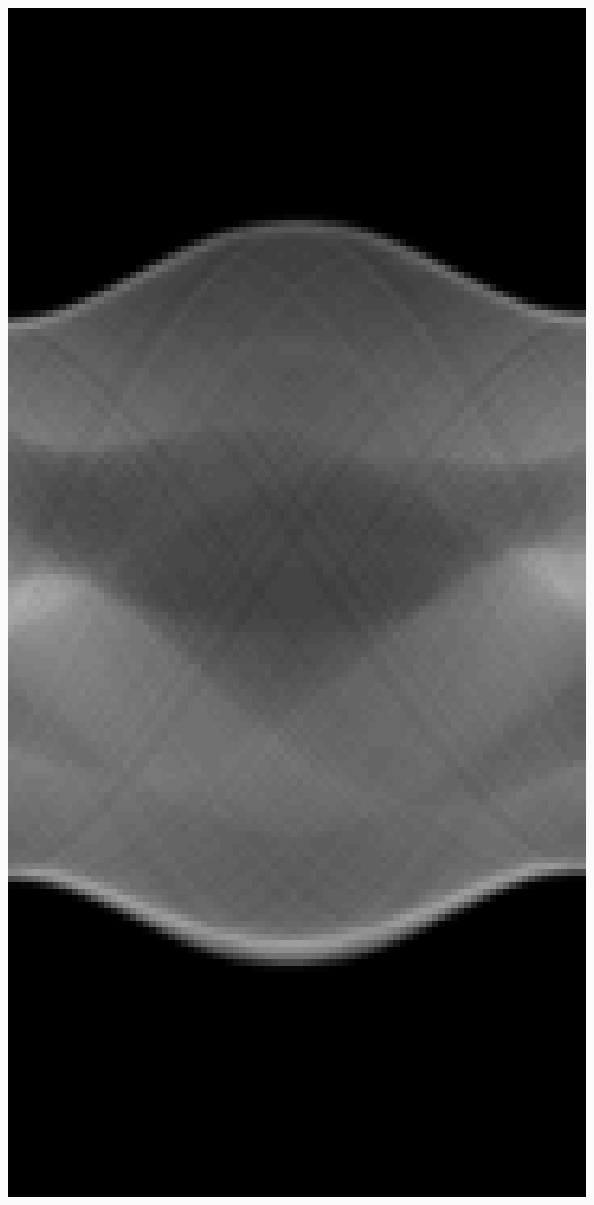} \includegraphics[scale=0.2]{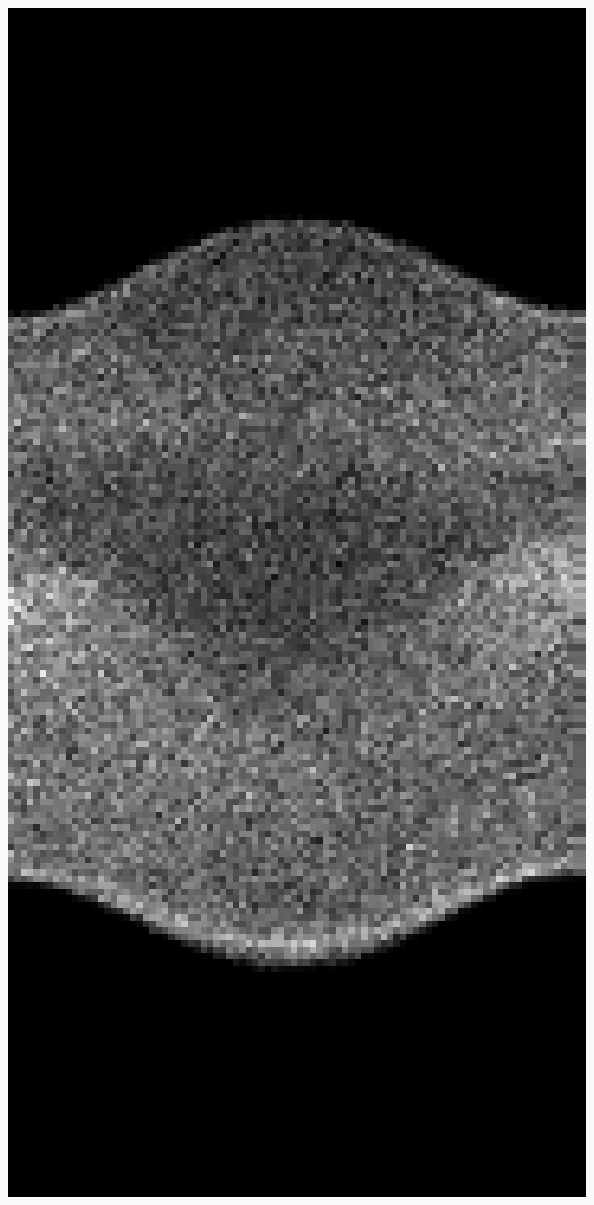}
& \includegraphics[scale=0.2]{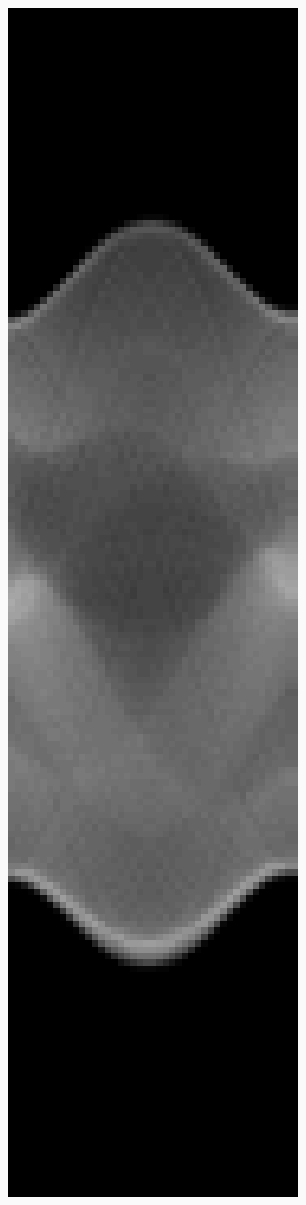}  \includegraphics[scale=0.2]{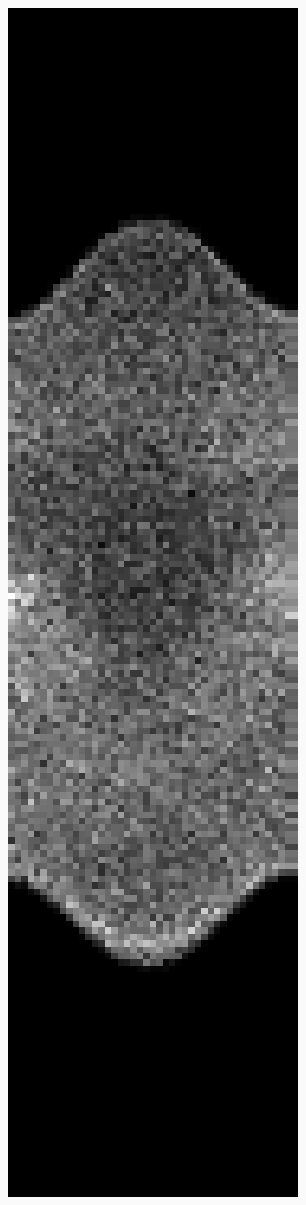} & \includegraphics[scale=0.2]{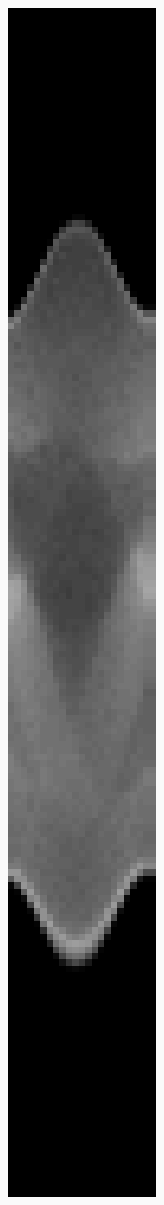}  \includegraphics[scale=0.2]{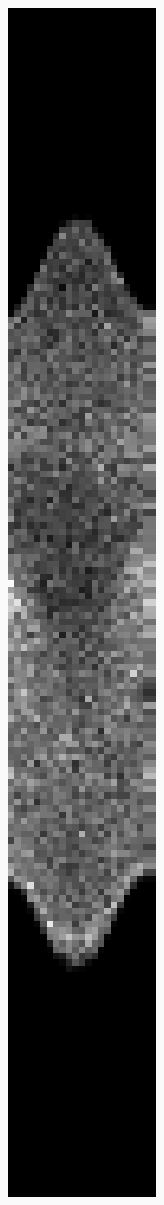}
& \includegraphics[scale=0.2]{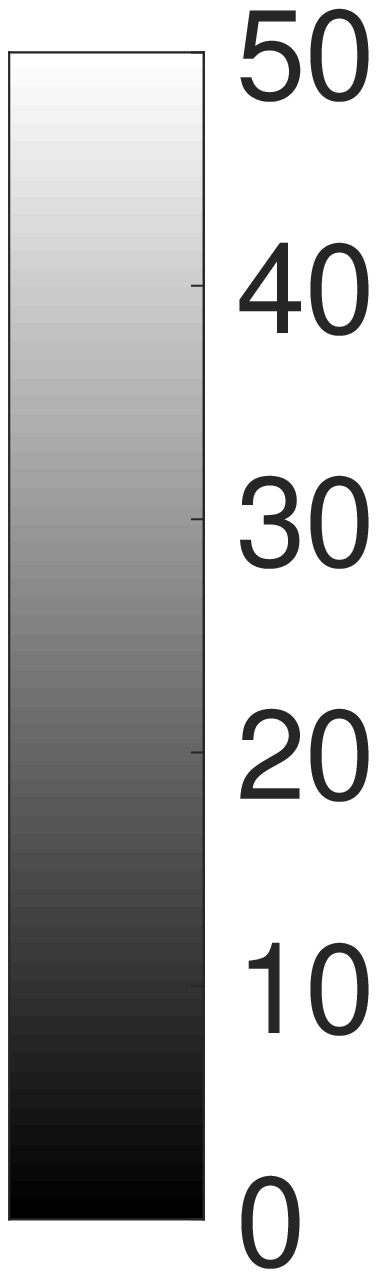}\\
\includegraphics[scale=0.2]{SL_true} \includegraphics[scale=0.2]{bar_1} & \includegraphics[scale=0.2]{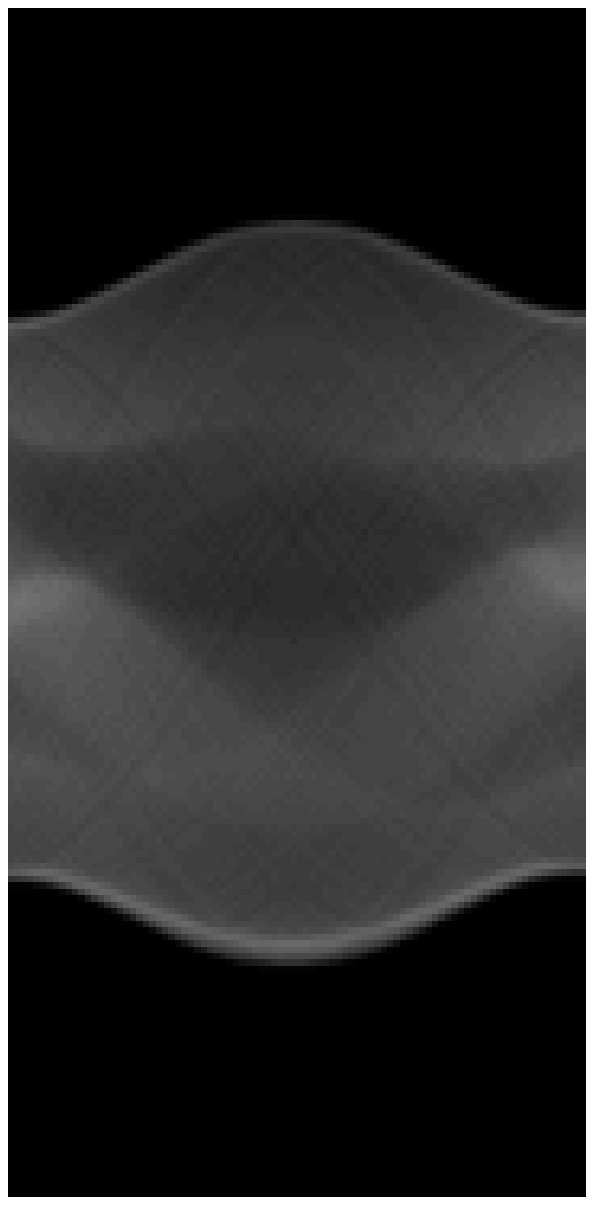} \includegraphics[scale=0.2]{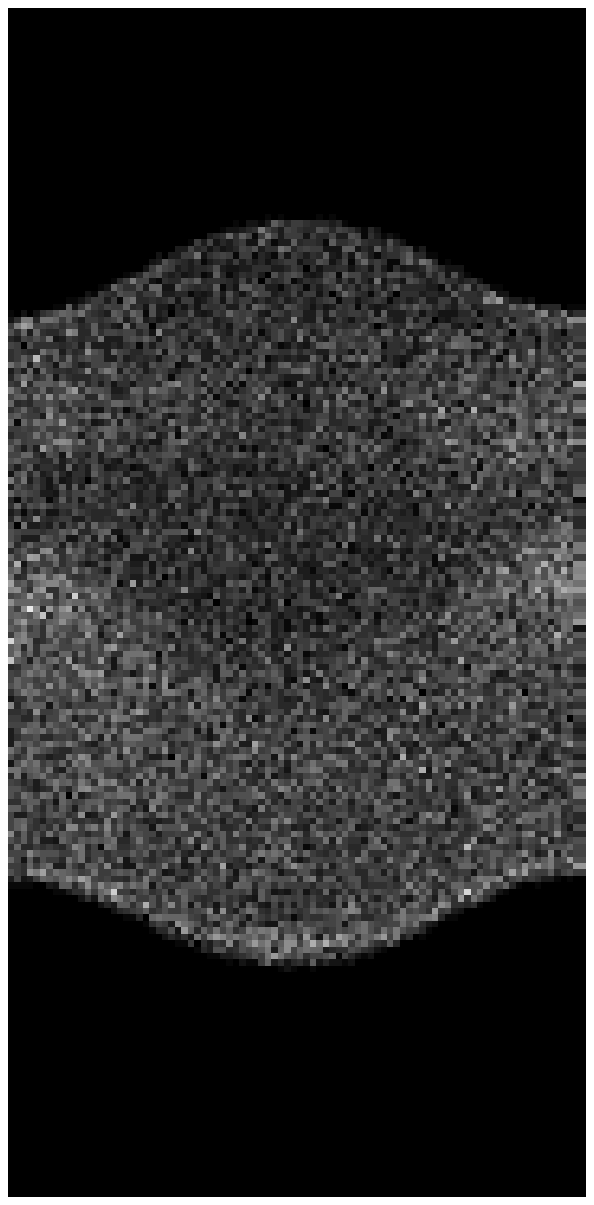}
& \includegraphics[scale=0.2]{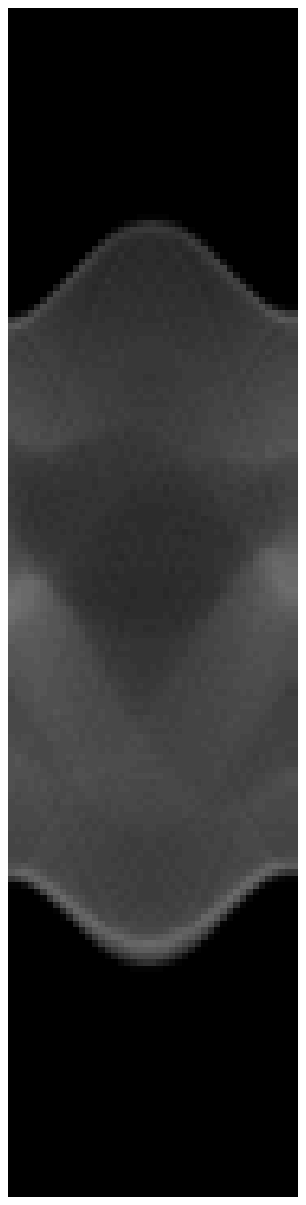}  \includegraphics[scale=0.2]{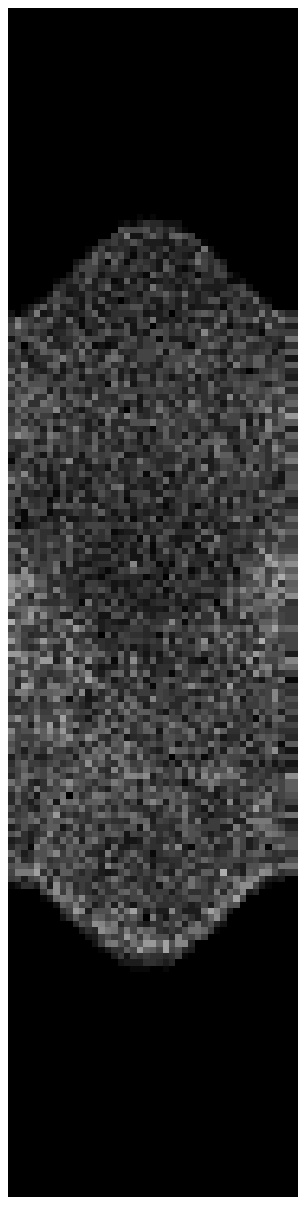} & \includegraphics[scale=0.2]{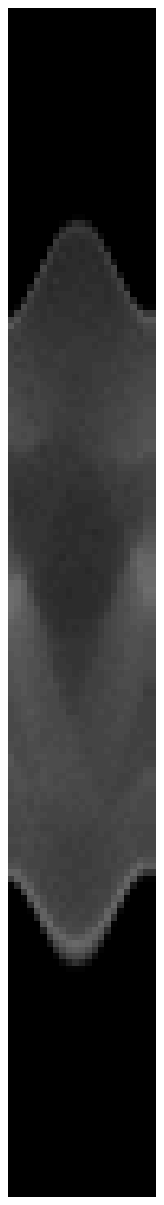}  \includegraphics[scale=0.2]{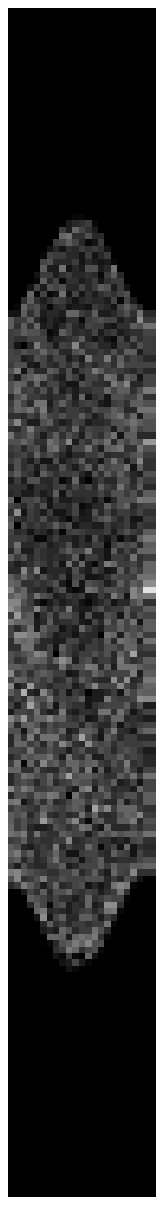}
& \includegraphics[scale=0.2]{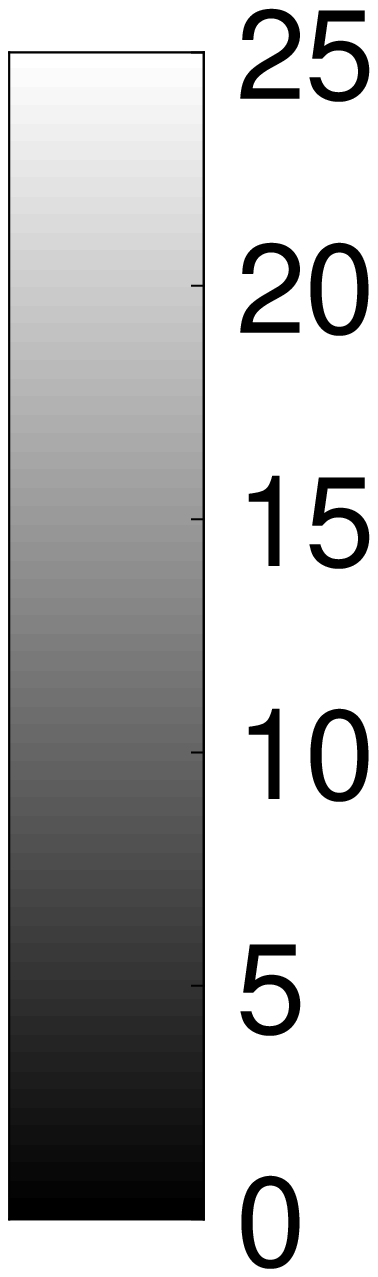}
\end{tabular}
\caption{The exact image, sinograms and data with three different $A$s for \texttt{Shepp-Logan} phantom. The top
and bottom rows refer to the moderate count and low count cases, respectively. \label{fig:SL_xby}}	
\end{figure}

\begin{figure}[hbt!]
\centering
\begin{tabular}{rccccccc}
 & \includegraphics[scale=0.2]{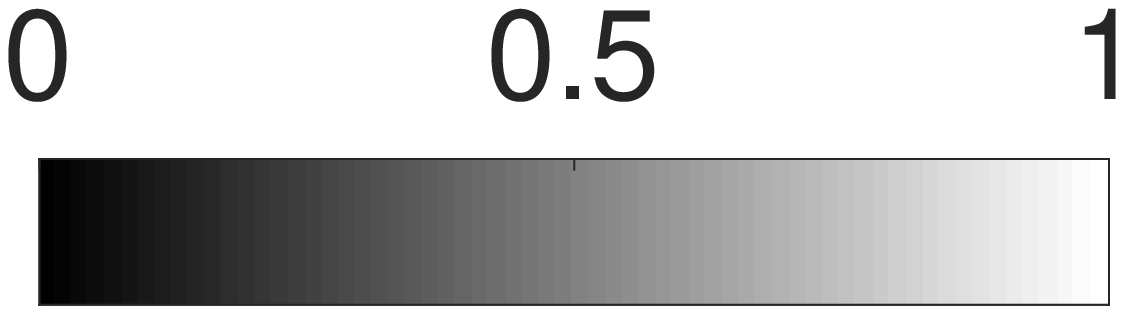} & \includegraphics[scale=0.2]{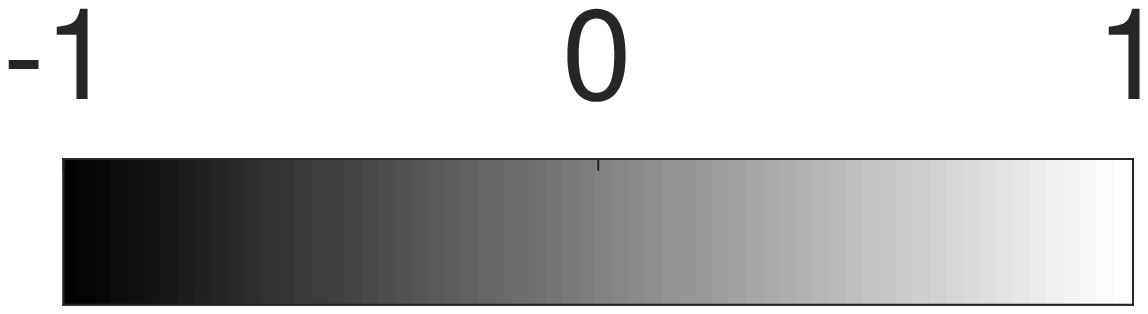} & \includegraphics[scale=0.2]{bar_1_h} & \includegraphics[scale=0.2]{bar_1_1_h} & \includegraphics[scale=0.2]{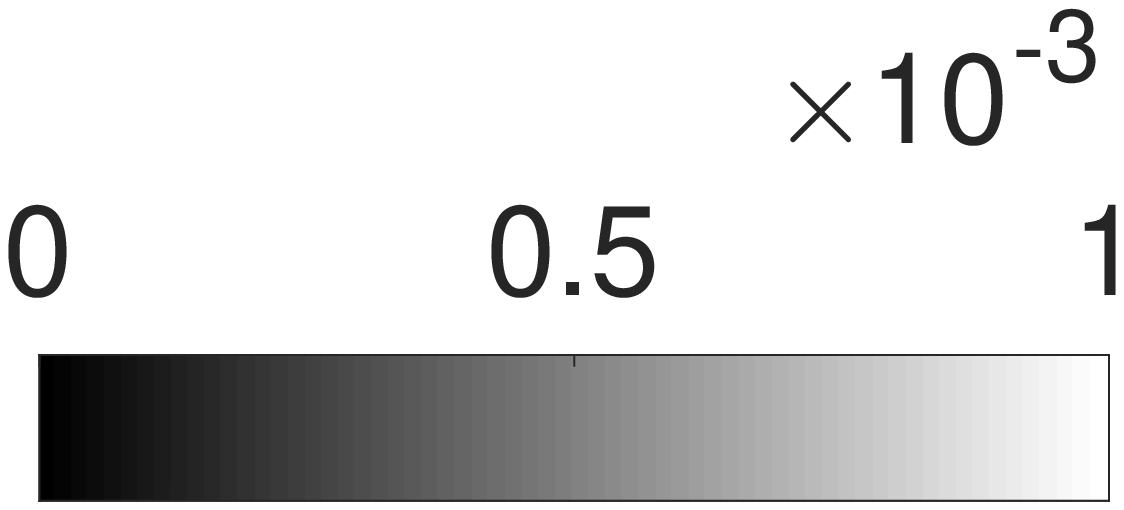}\\
\rotatebox{90}{\quad[0:2:179]}&\includegraphics[scale=0.2]{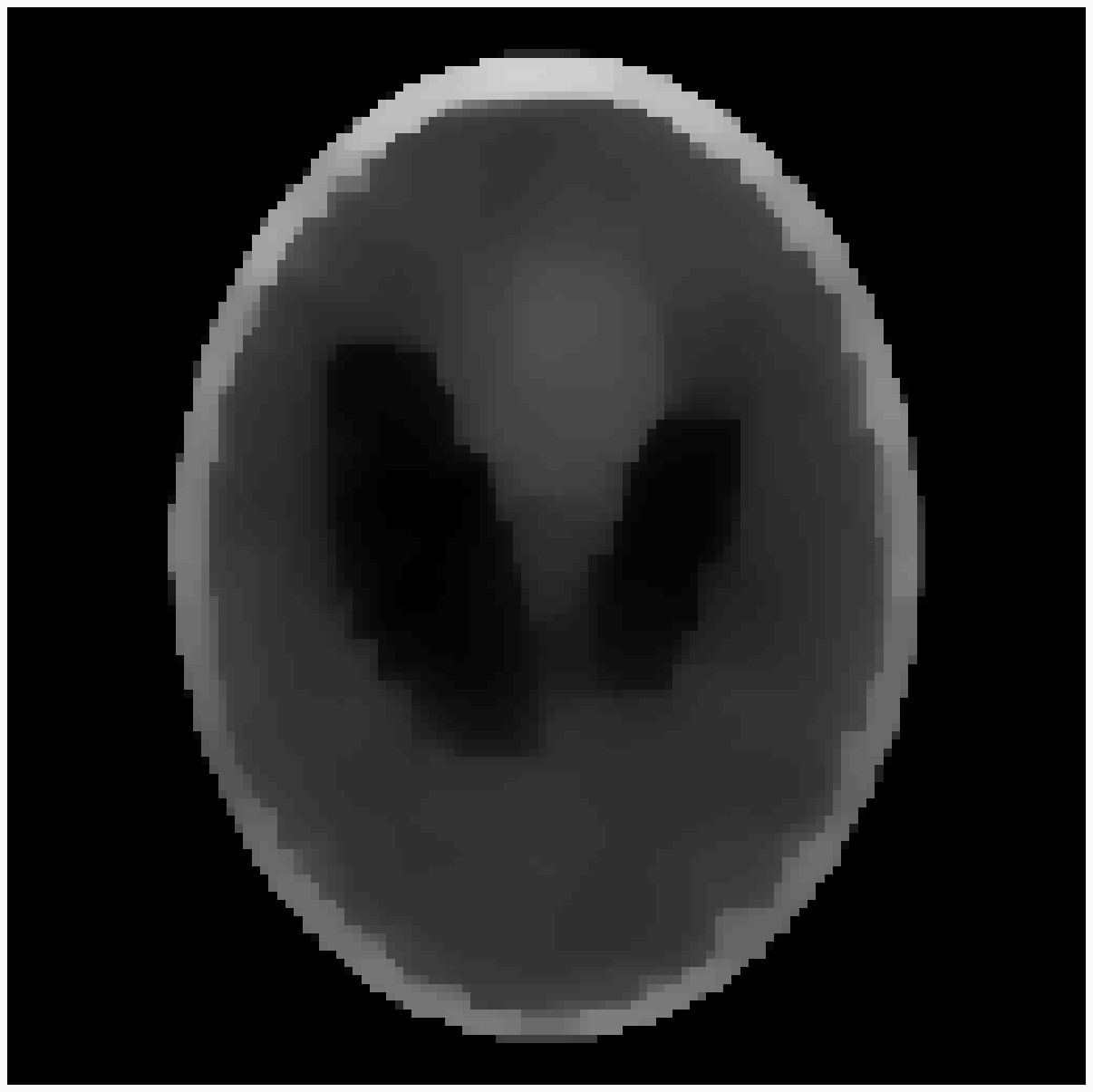} & \includegraphics[scale=0.2]{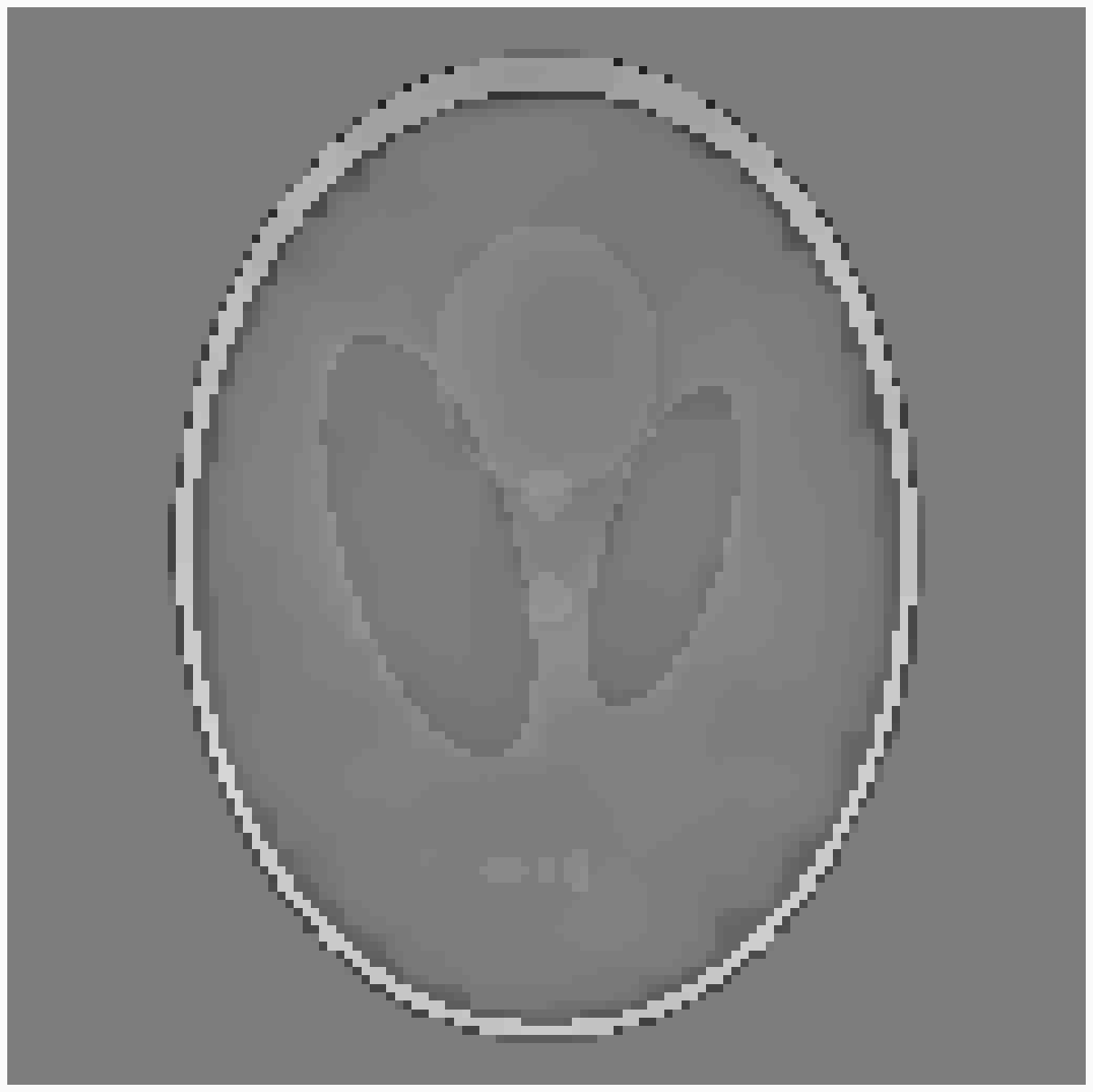} & \includegraphics[scale=0.2]{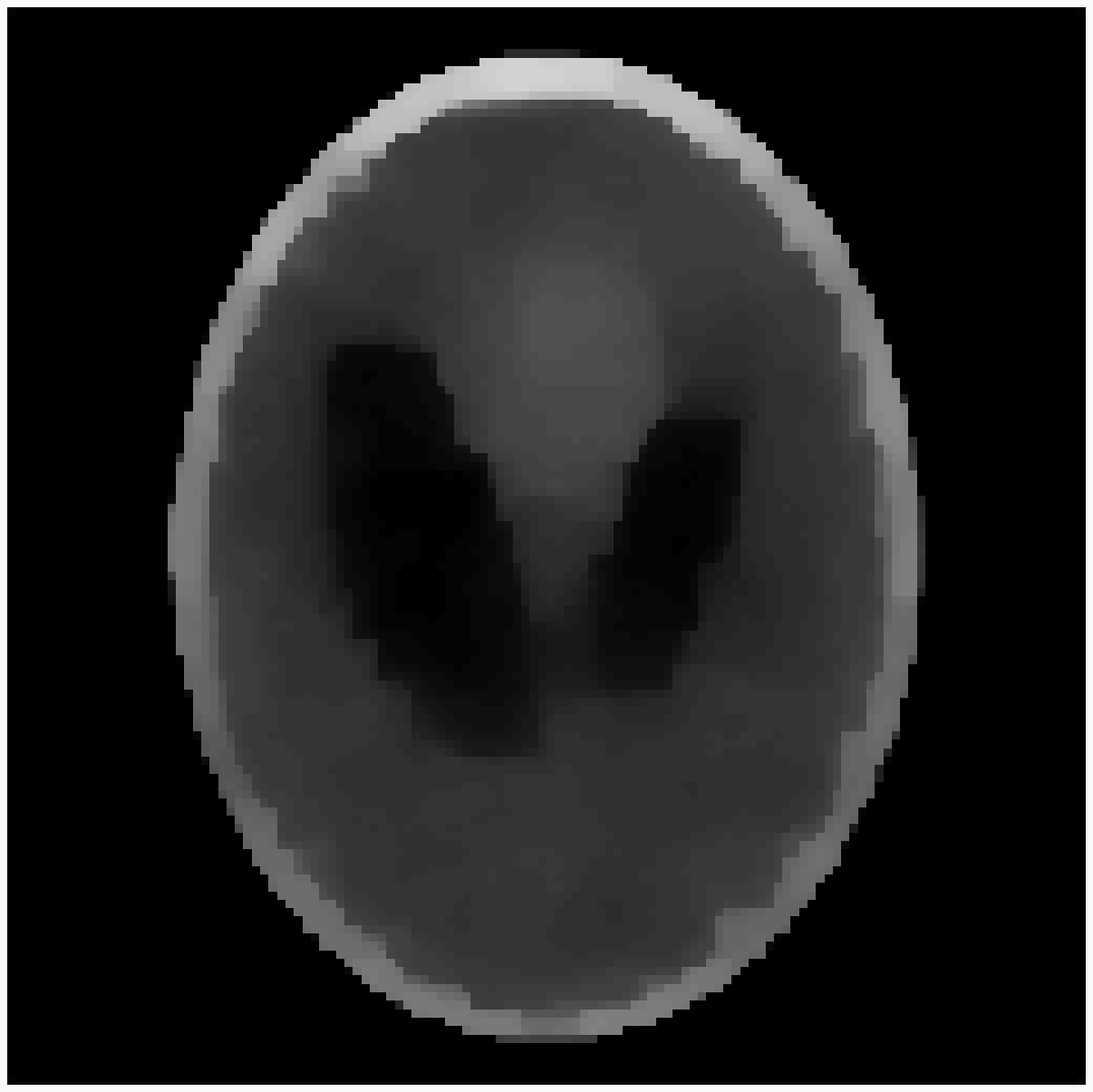} & \includegraphics[scale=0.2]{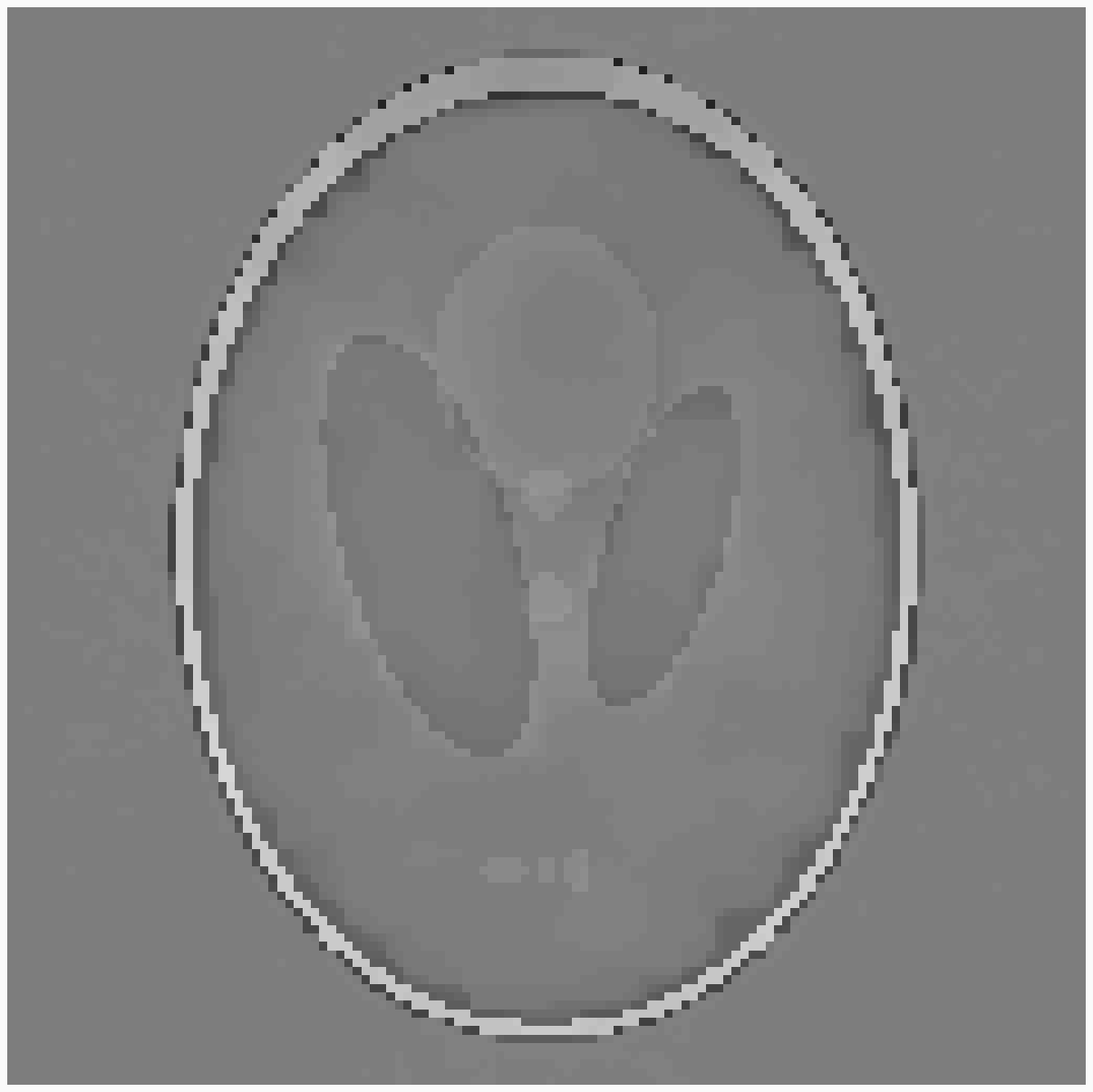} & \includegraphics[scale=0.2]{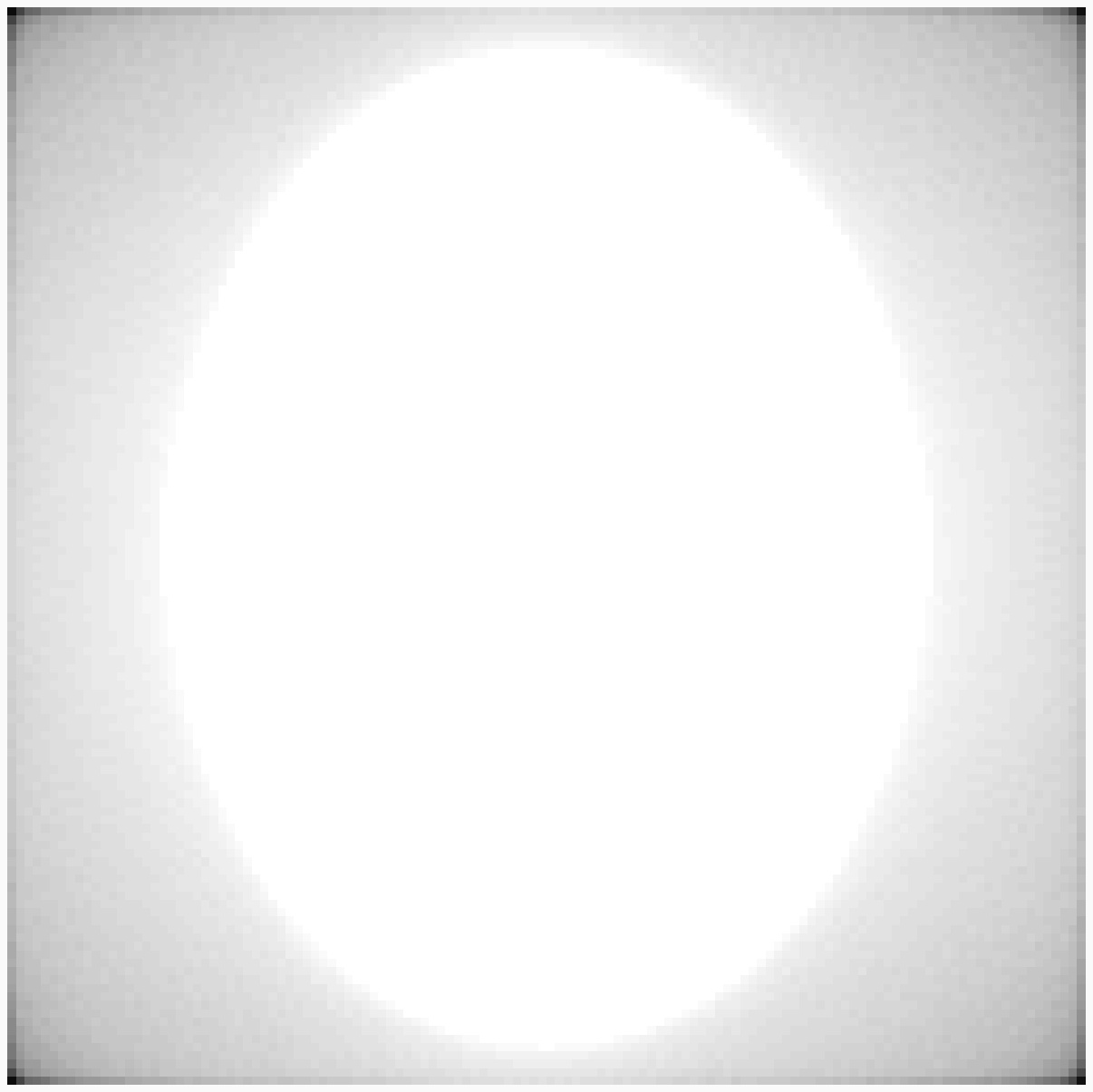}\\
\rotatebox{90}{\quad[0:4:179]}&\includegraphics[scale=0.2]{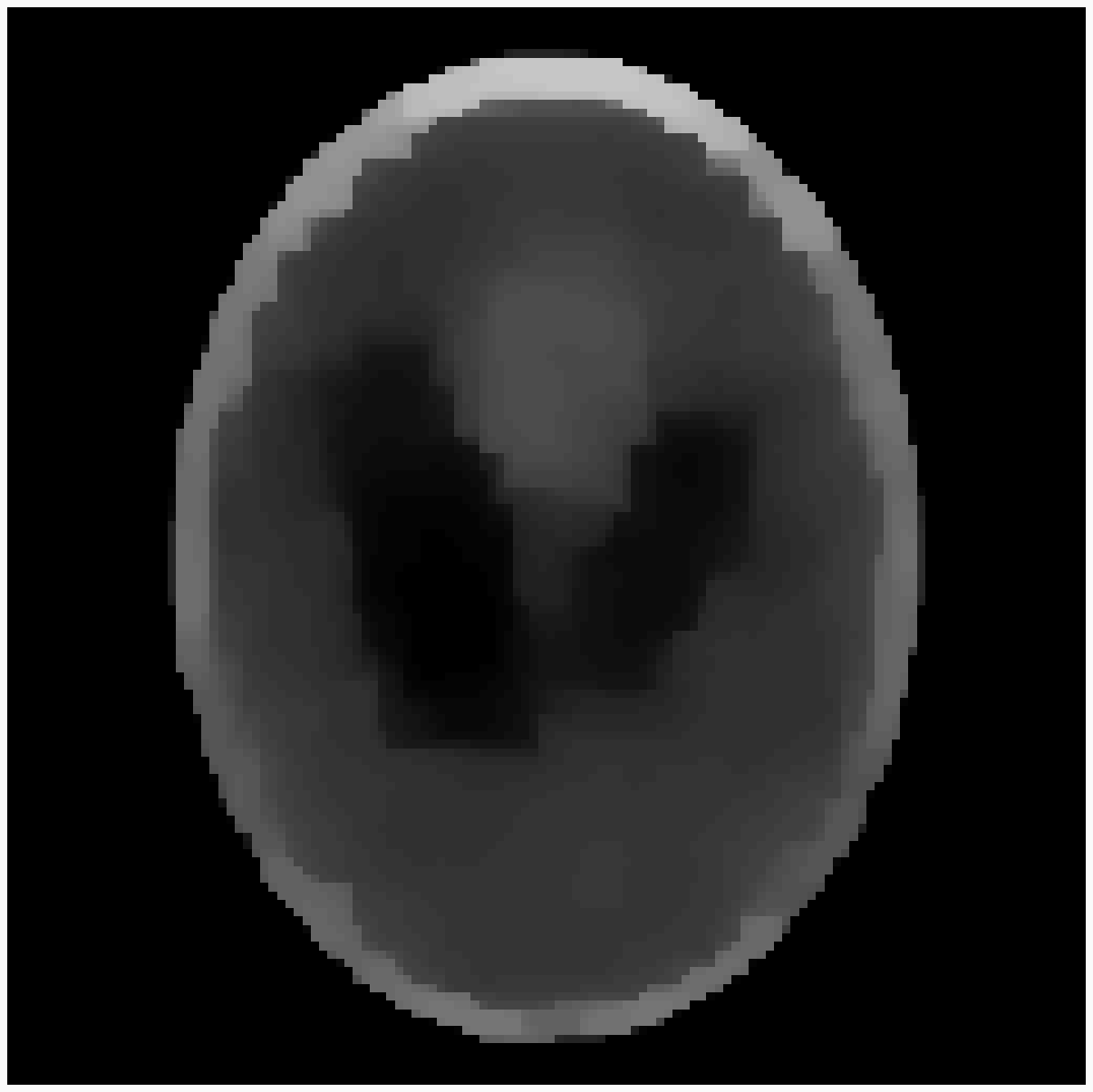} & \includegraphics[scale=0.2]{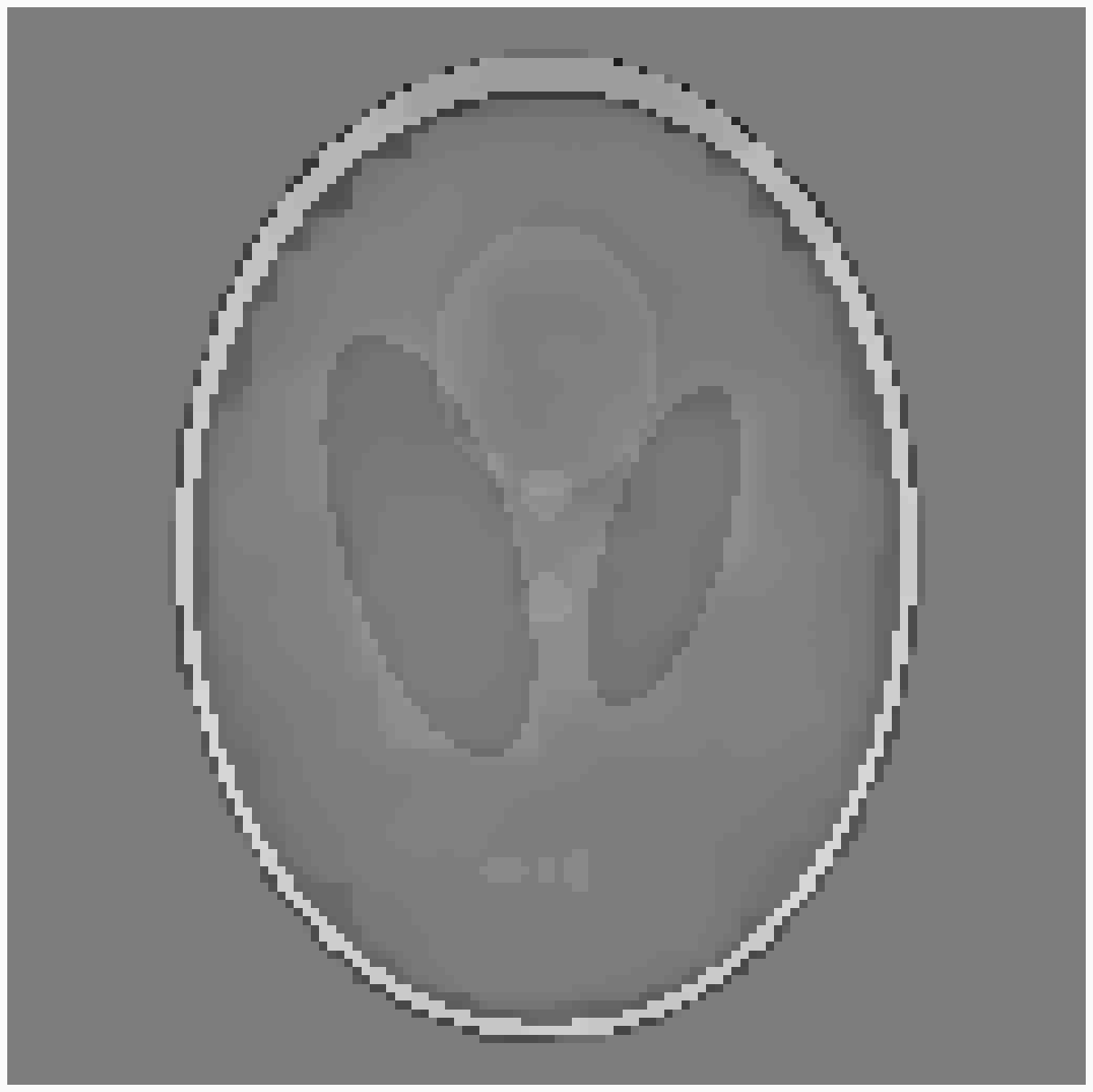} & \includegraphics[scale=0.2]{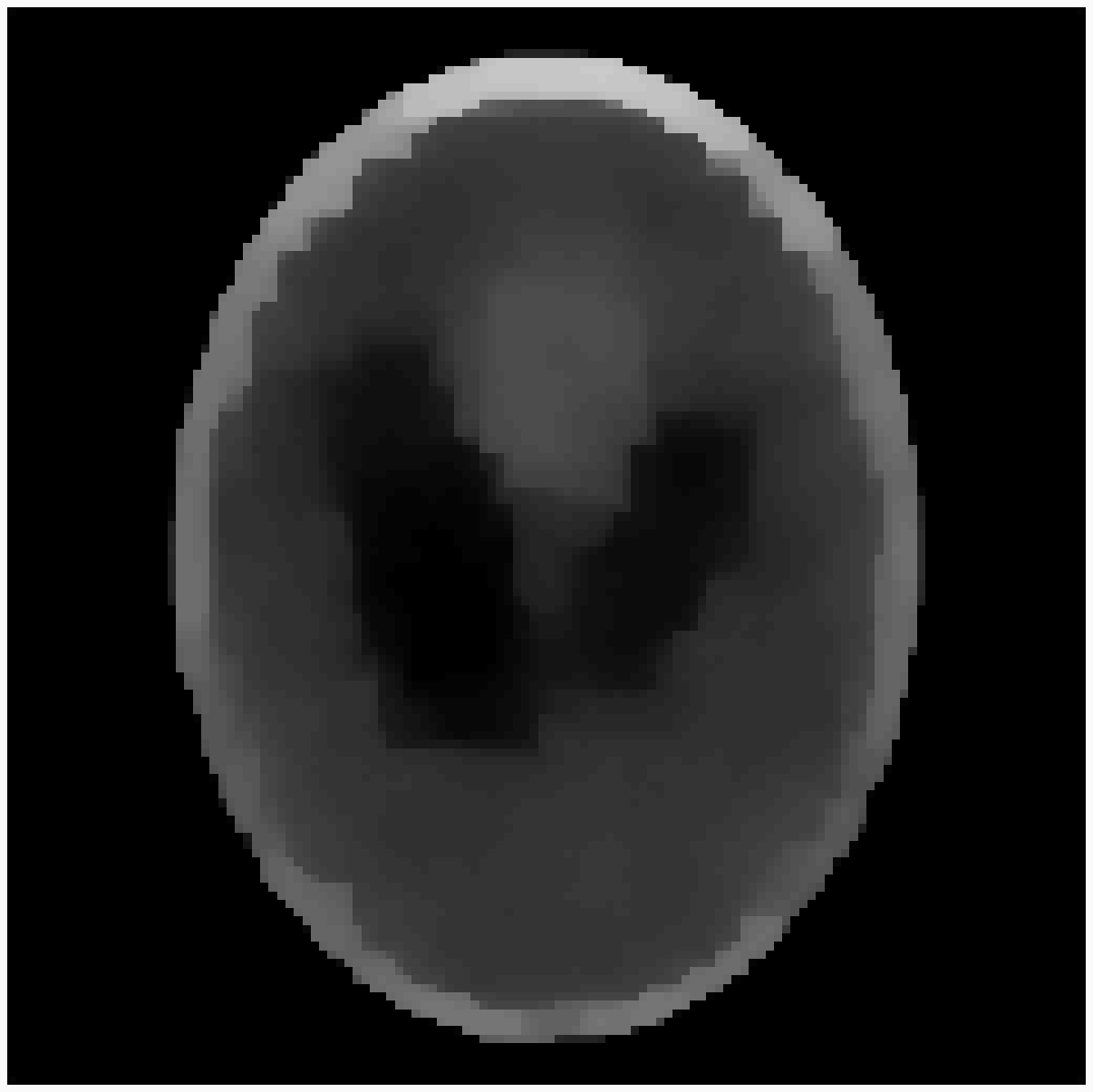} & \includegraphics[scale=0.2]{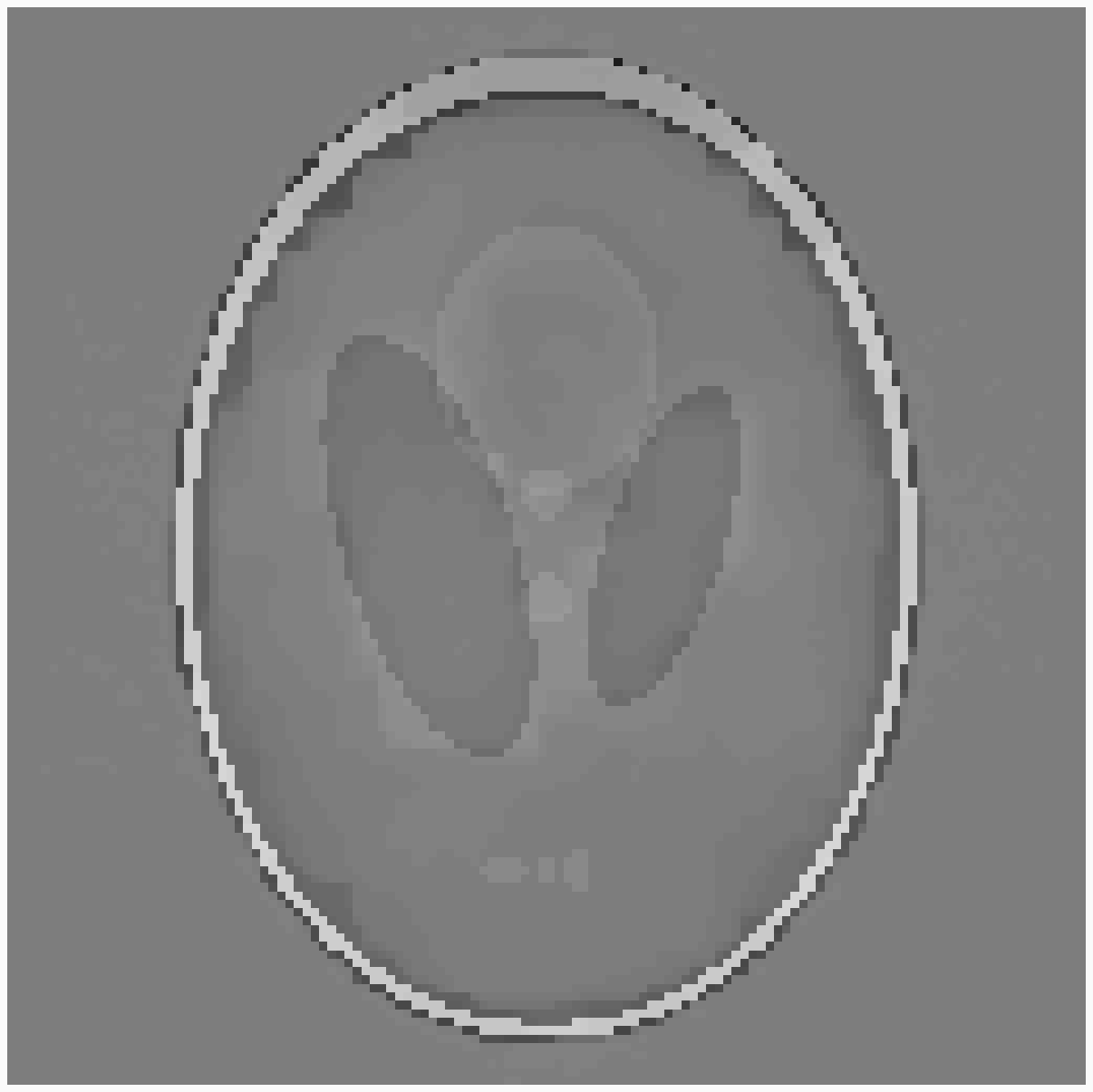} & \includegraphics[scale=0.2]{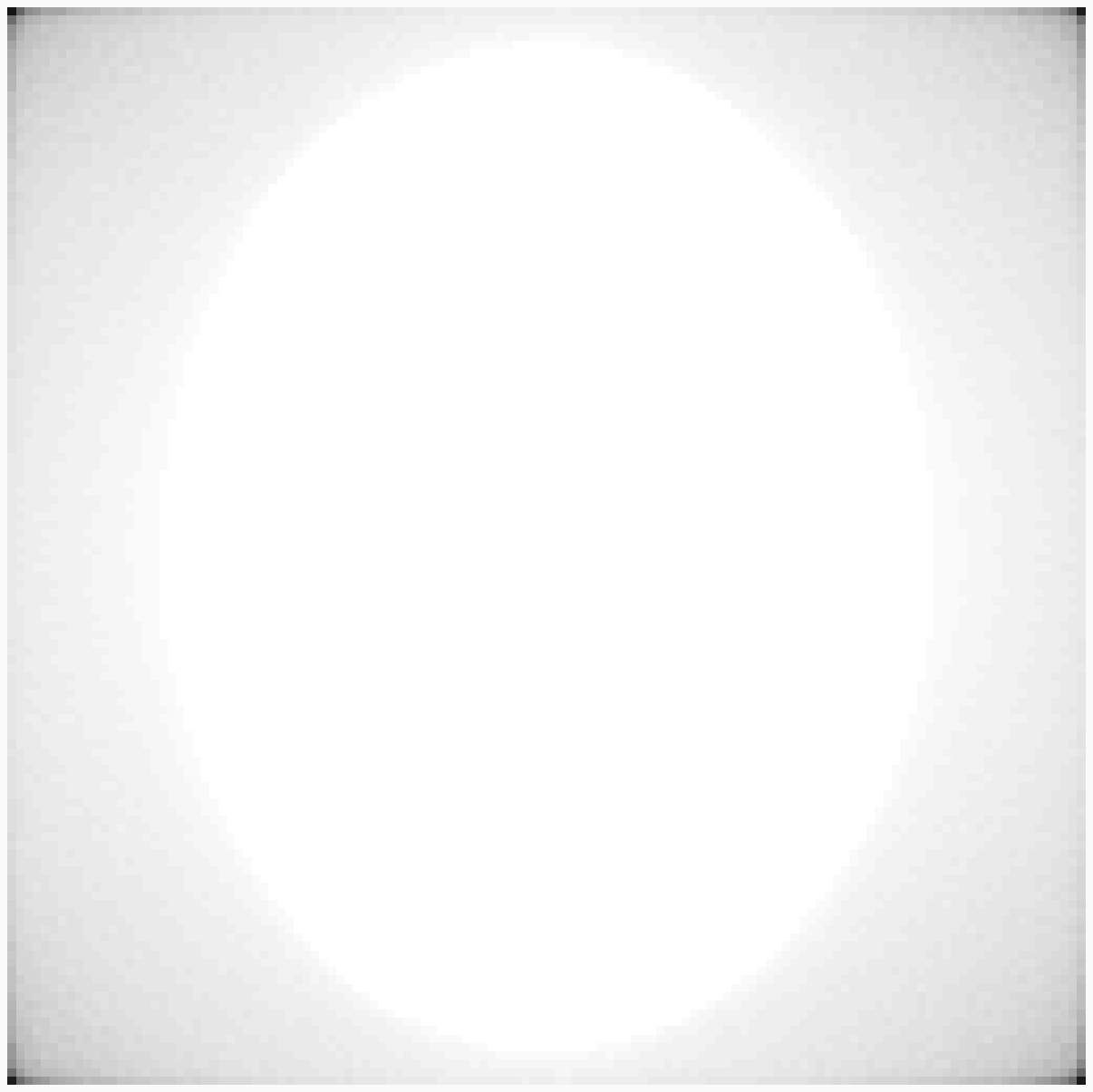}\\
\rotatebox{90}{\quad[0:8:179]}&\includegraphics[scale=0.2]{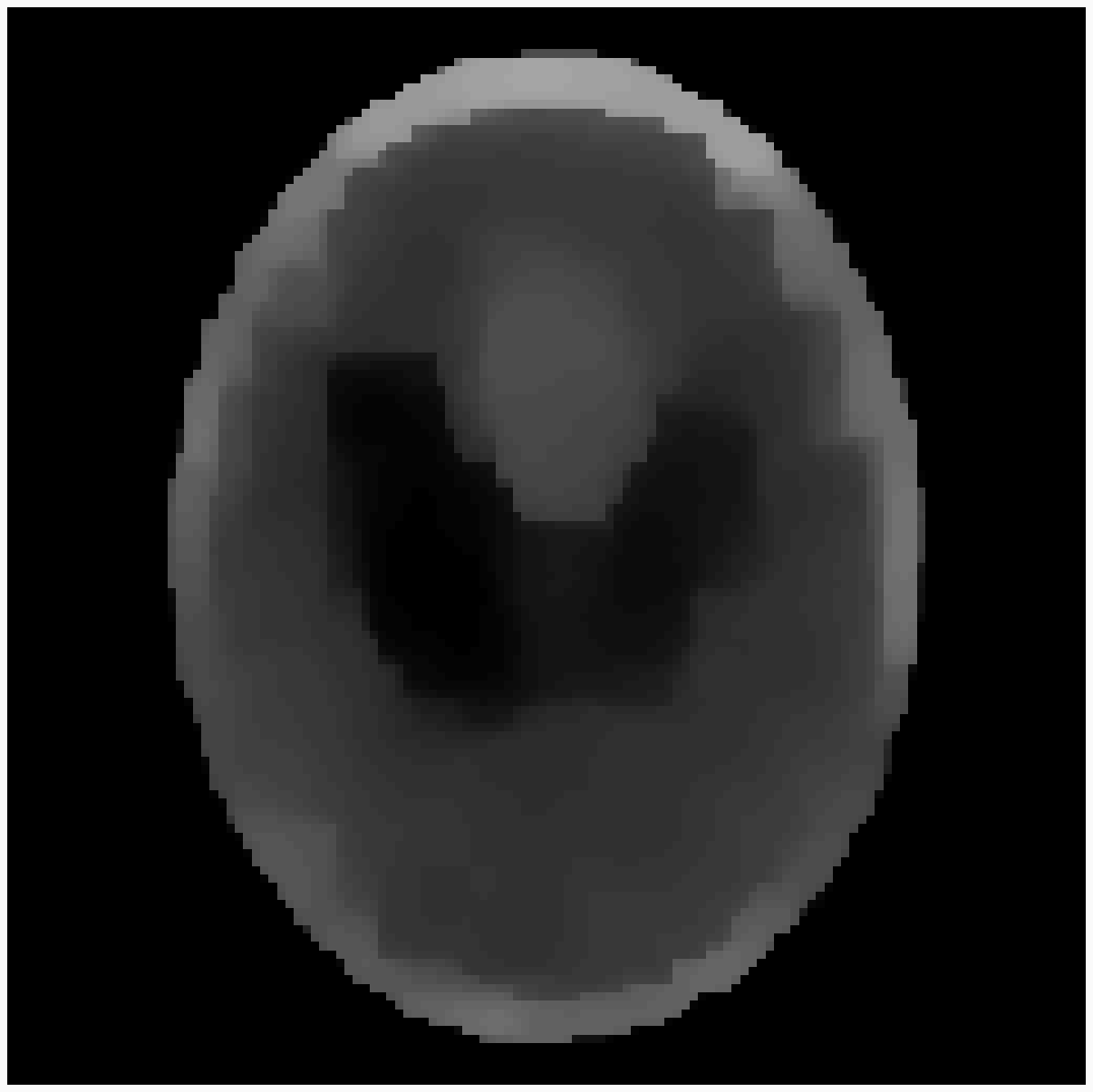} & \includegraphics[scale=0.2]{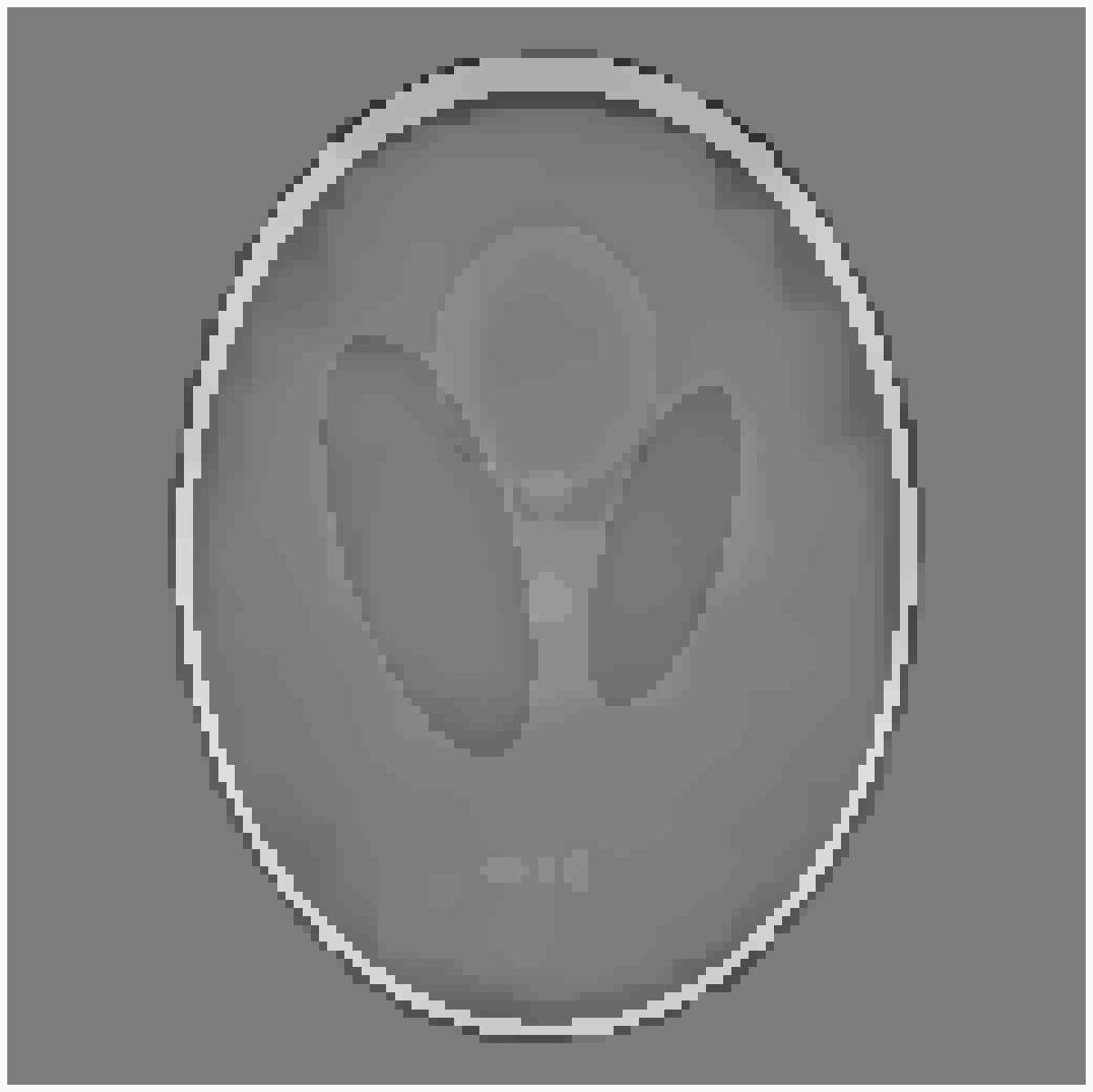} & \includegraphics[scale=0.2]{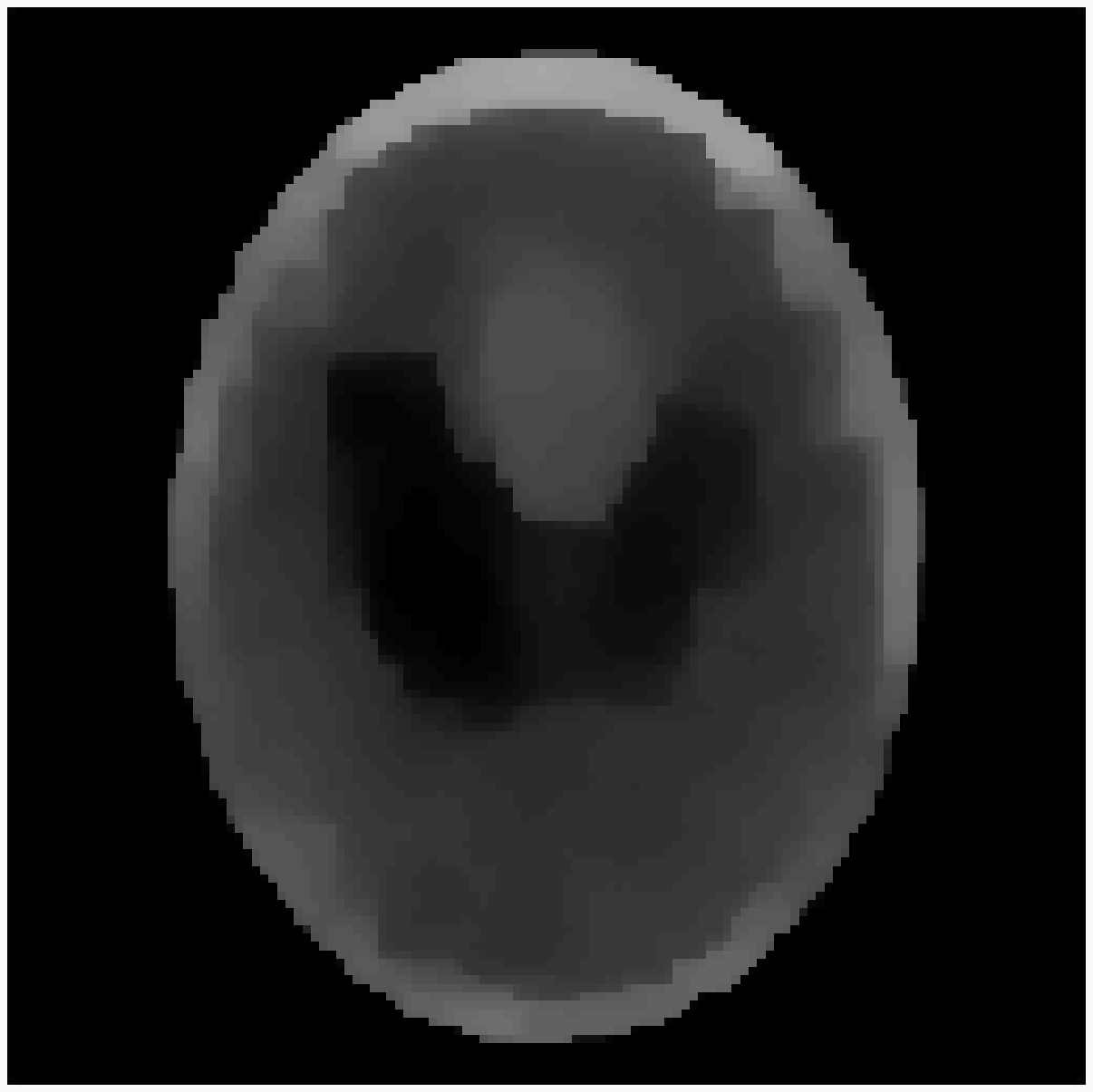} & \includegraphics[scale=0.2]{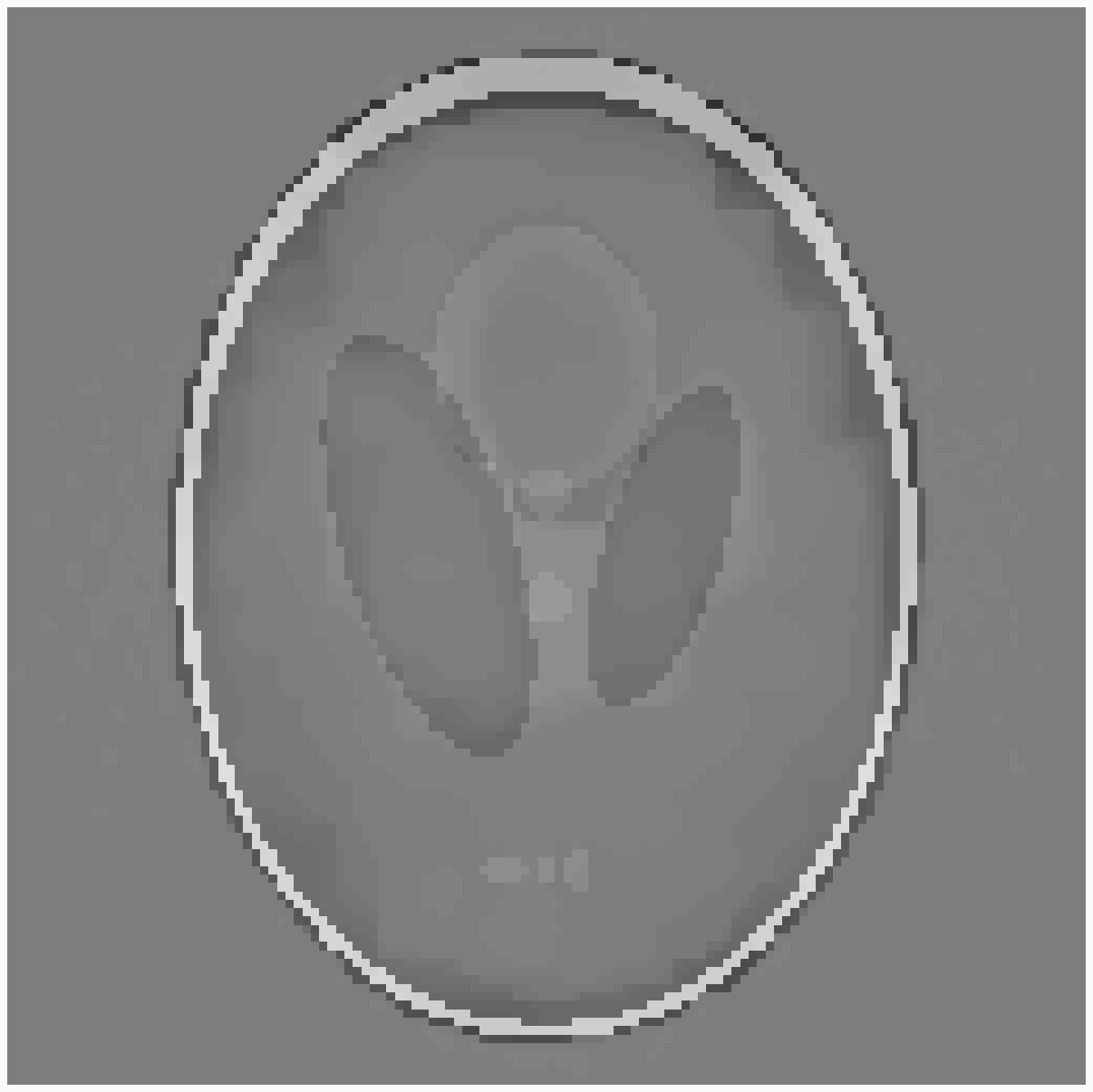} & \includegraphics[scale=0.2]{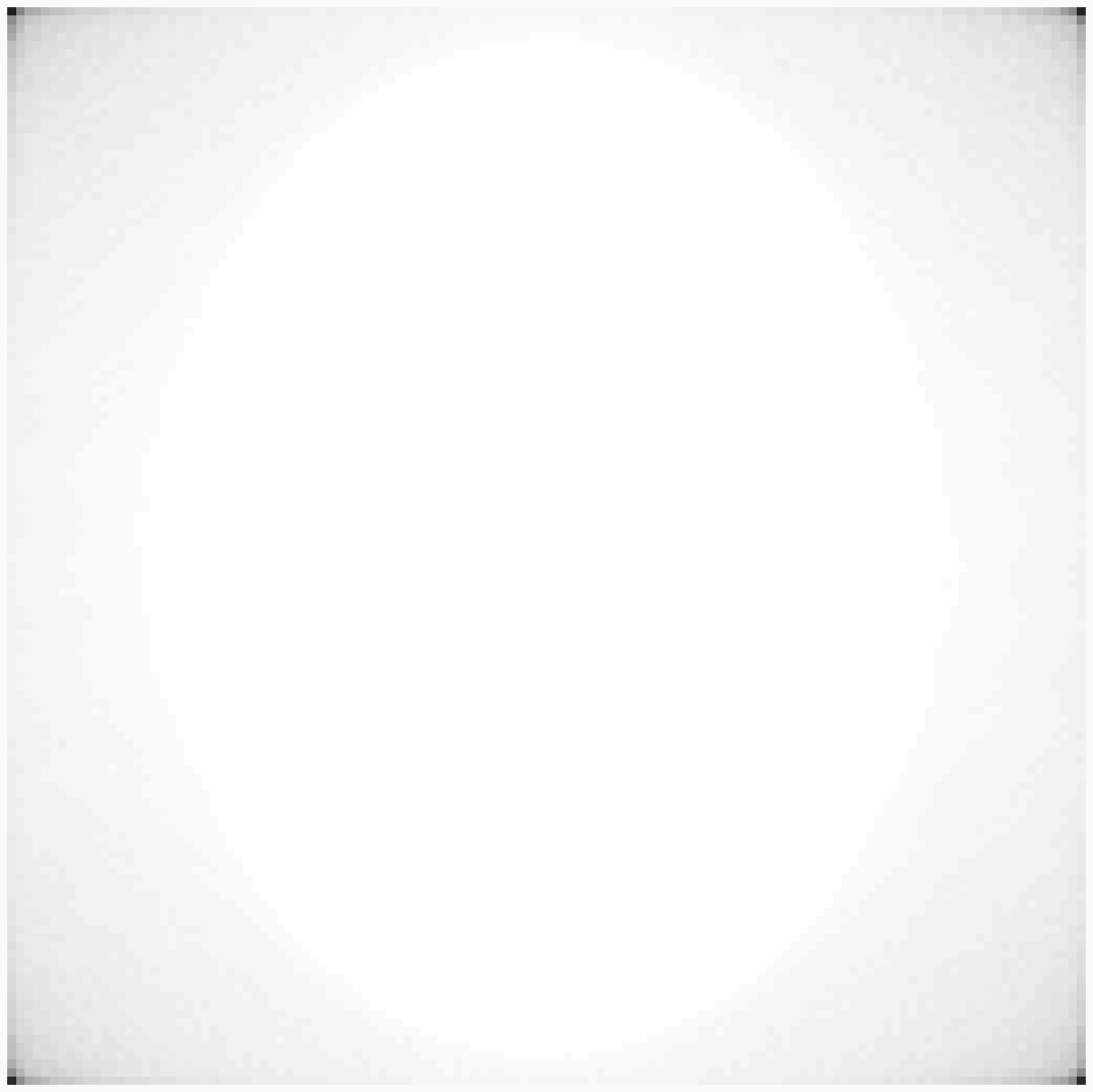}\\
&MAP & MAP error & EP mean & EP error & EP variance
\end{tabular}
\caption{MAP vs EP with anisotropic TV prior for the \texttt{Shepp-Logan} phantom, moderate count case.\label{fig:SL_248}}	
\end{figure}

\begin{figure}[hbt!]
\centering
\begin{tabular}{rccccccc}
 & \includegraphics[scale=0.2]{bar_1_h} & \includegraphics[scale=0.2]{bar_1_1_h} & \includegraphics[scale=0.2]{bar_1_h} & \includegraphics[scale=0.2]{bar_1_1_h} & \includegraphics[scale=0.2]{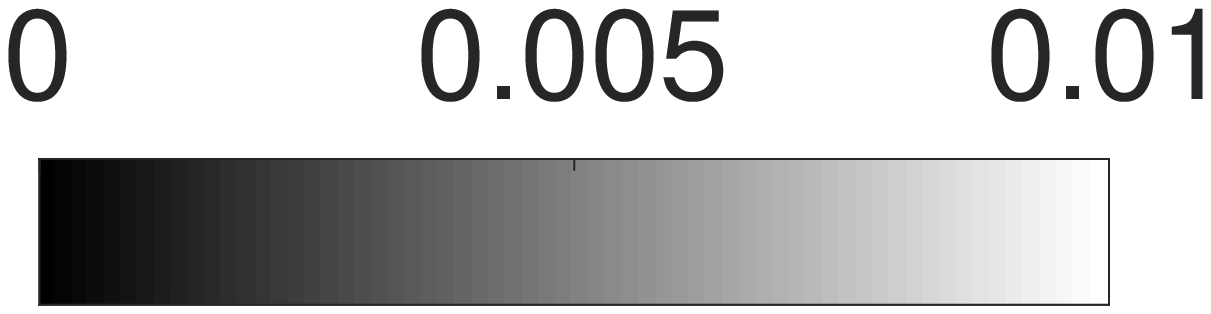}\\
\rotatebox{90}{\quad[0:2:179]}&\includegraphics[scale=0.2]{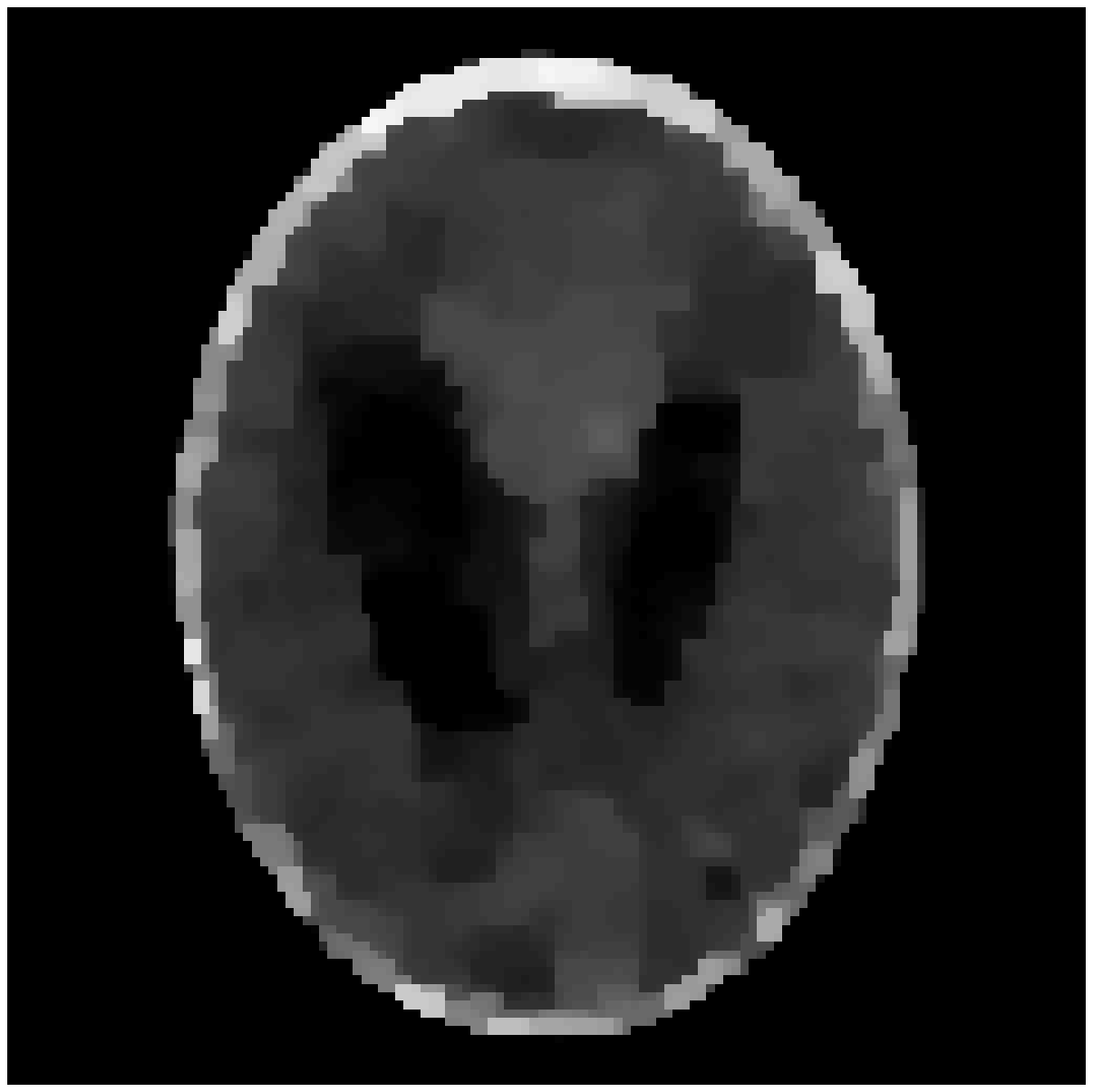} & \includegraphics[scale=0.2]{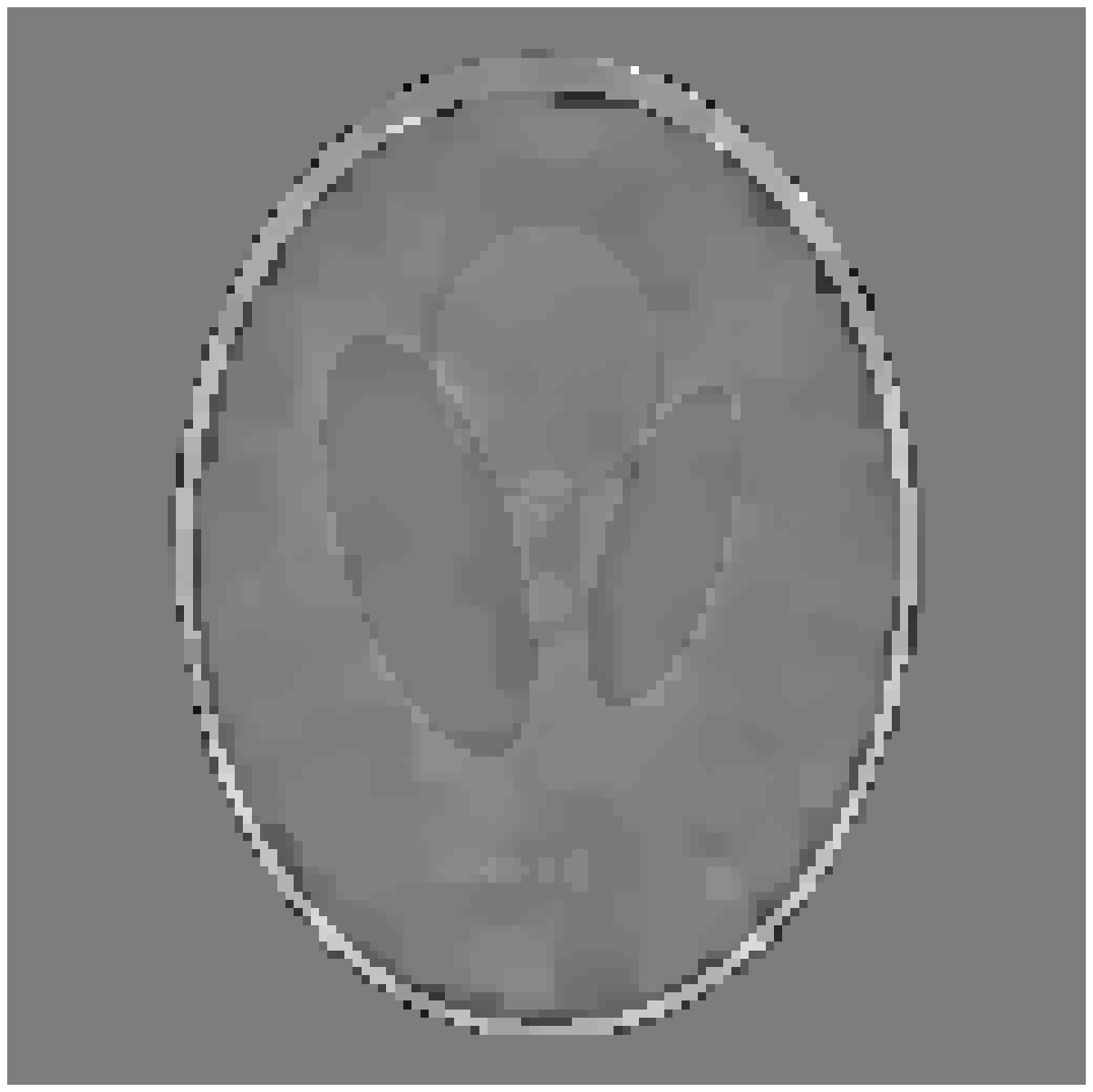} & \includegraphics[scale=0.2]{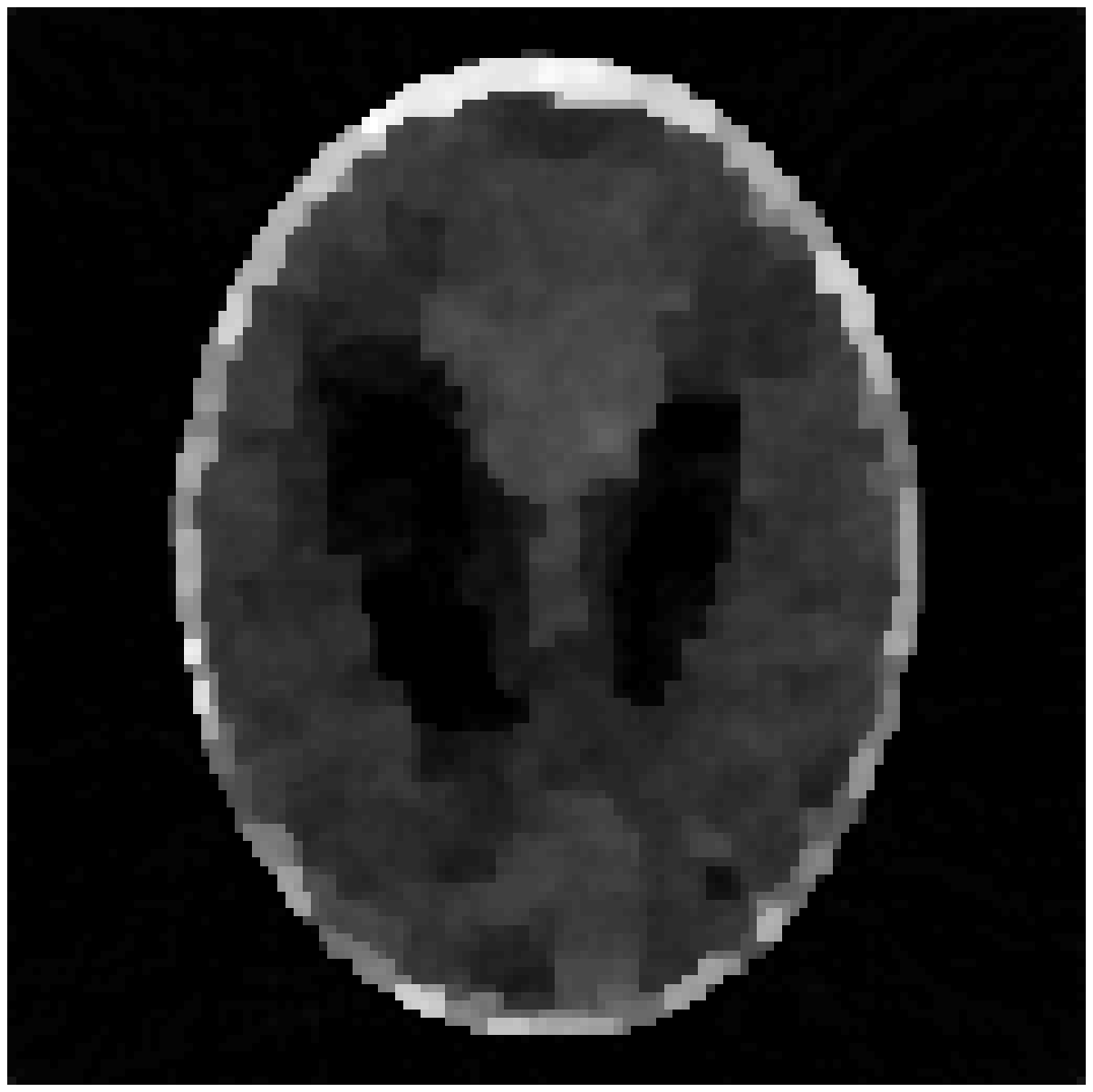} & \includegraphics[scale=0.2]{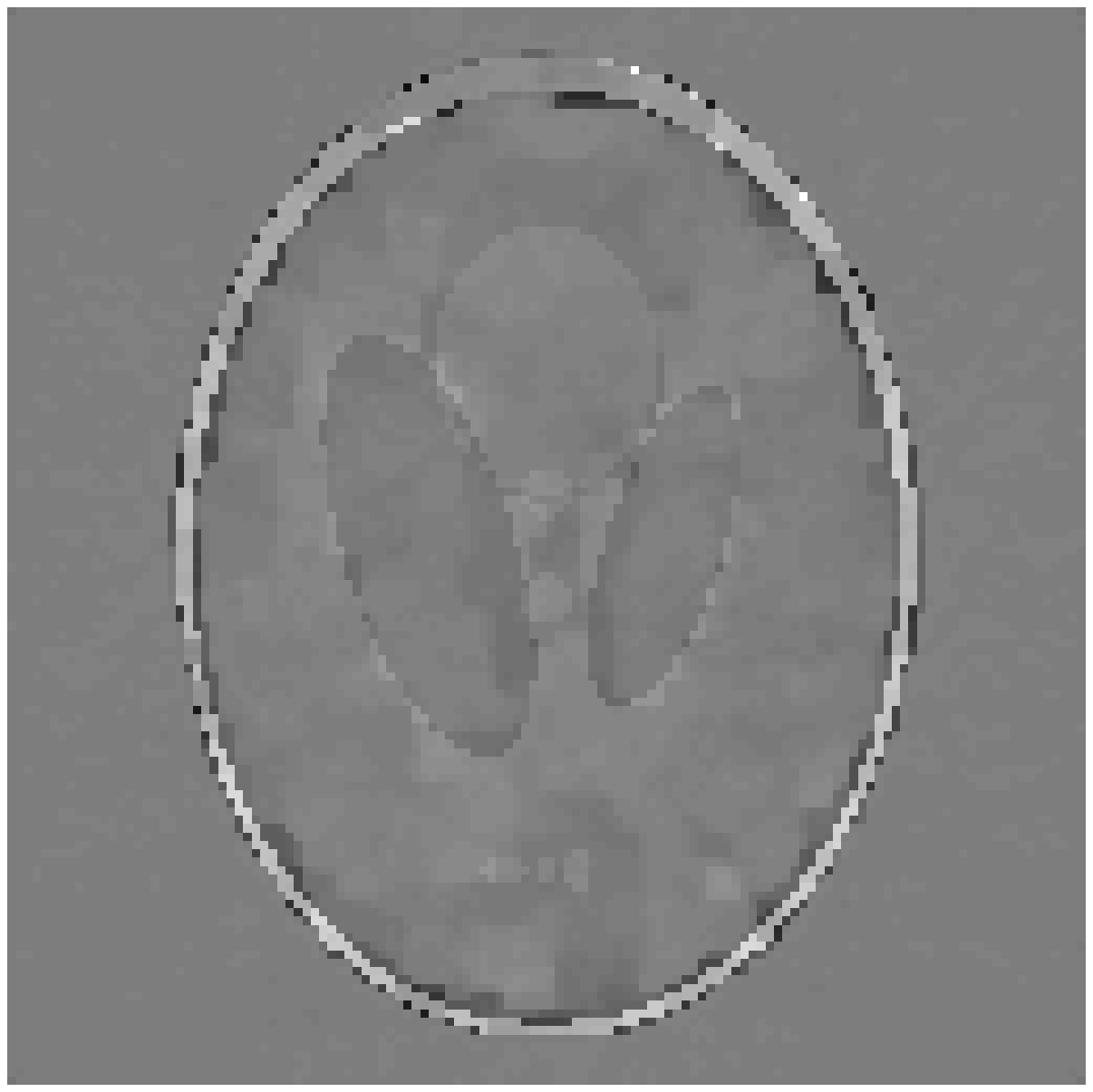} & \includegraphics[scale=0.2]{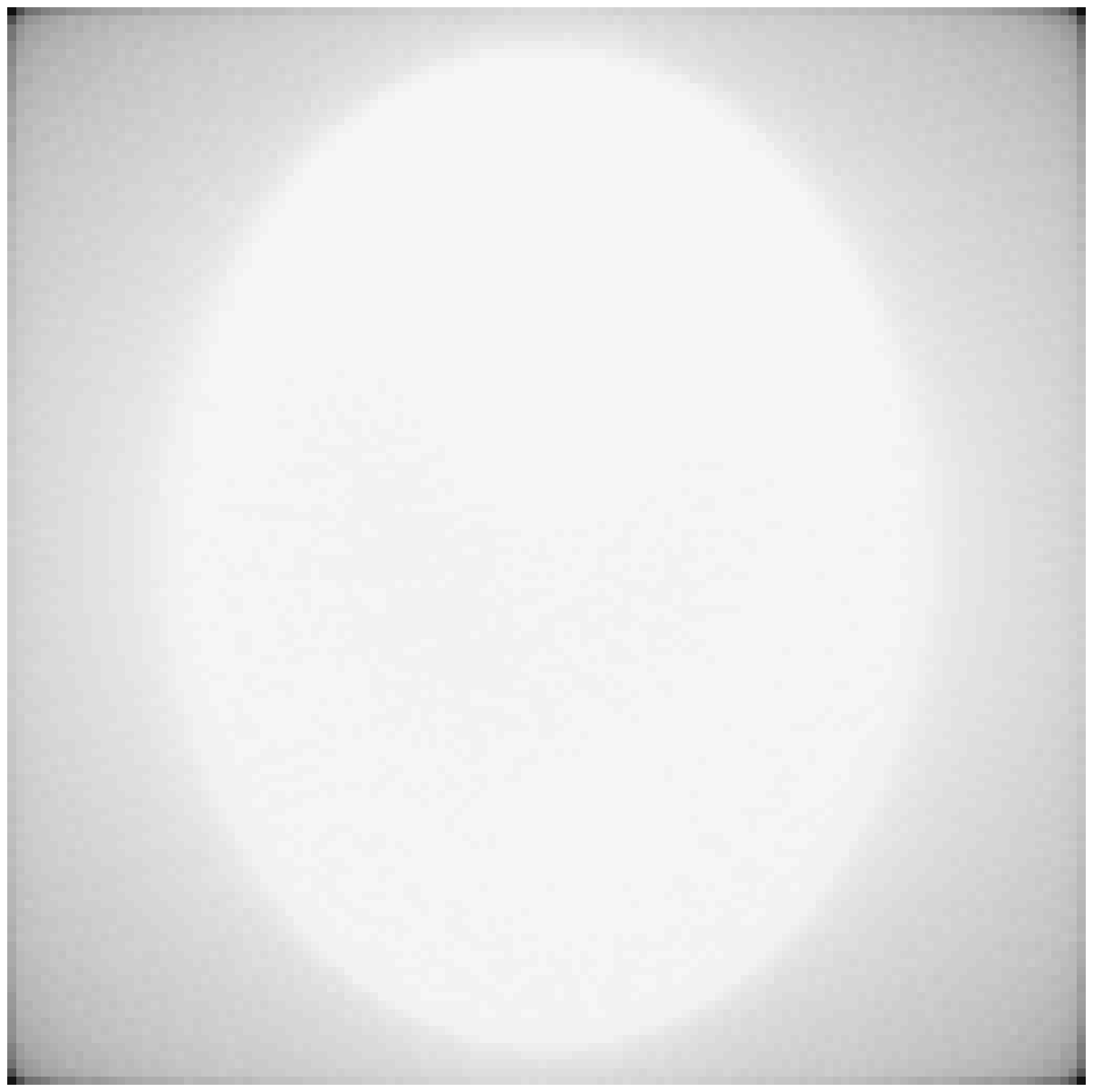}\\
\rotatebox{90}{\quad[0:4:179]}&\includegraphics[scale=0.2]{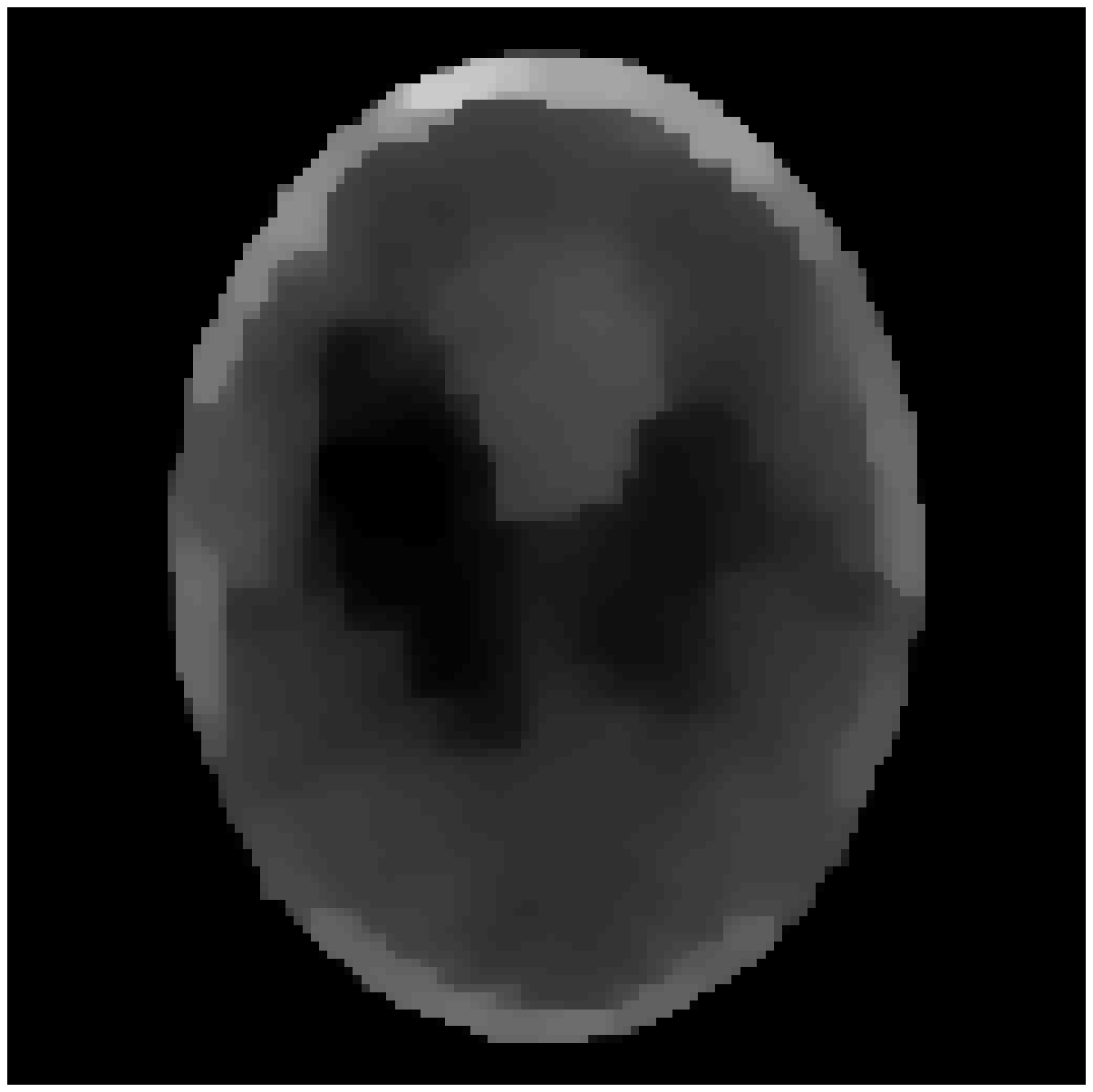} & \includegraphics[scale=0.2]{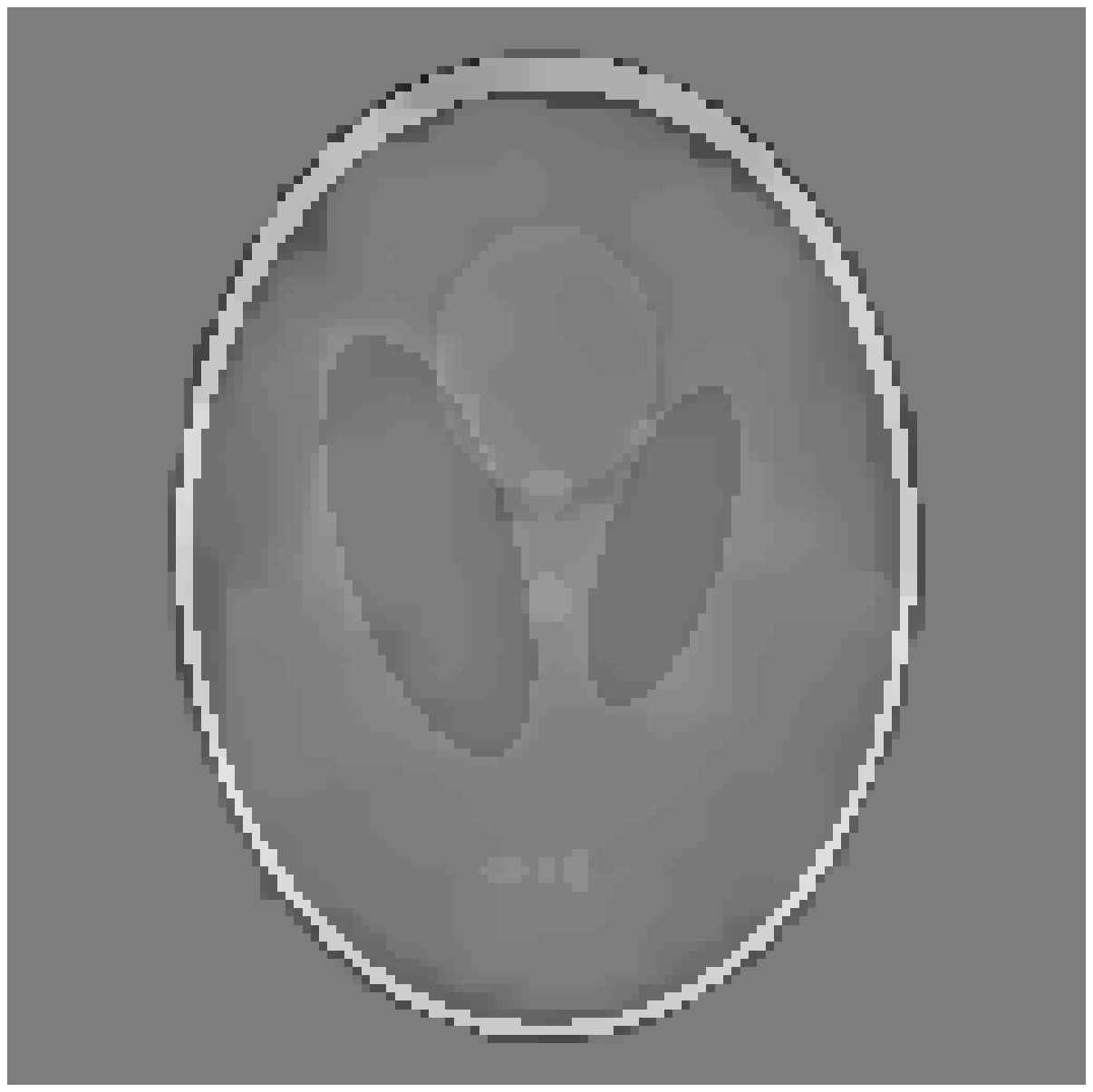} & \includegraphics[scale=0.2]{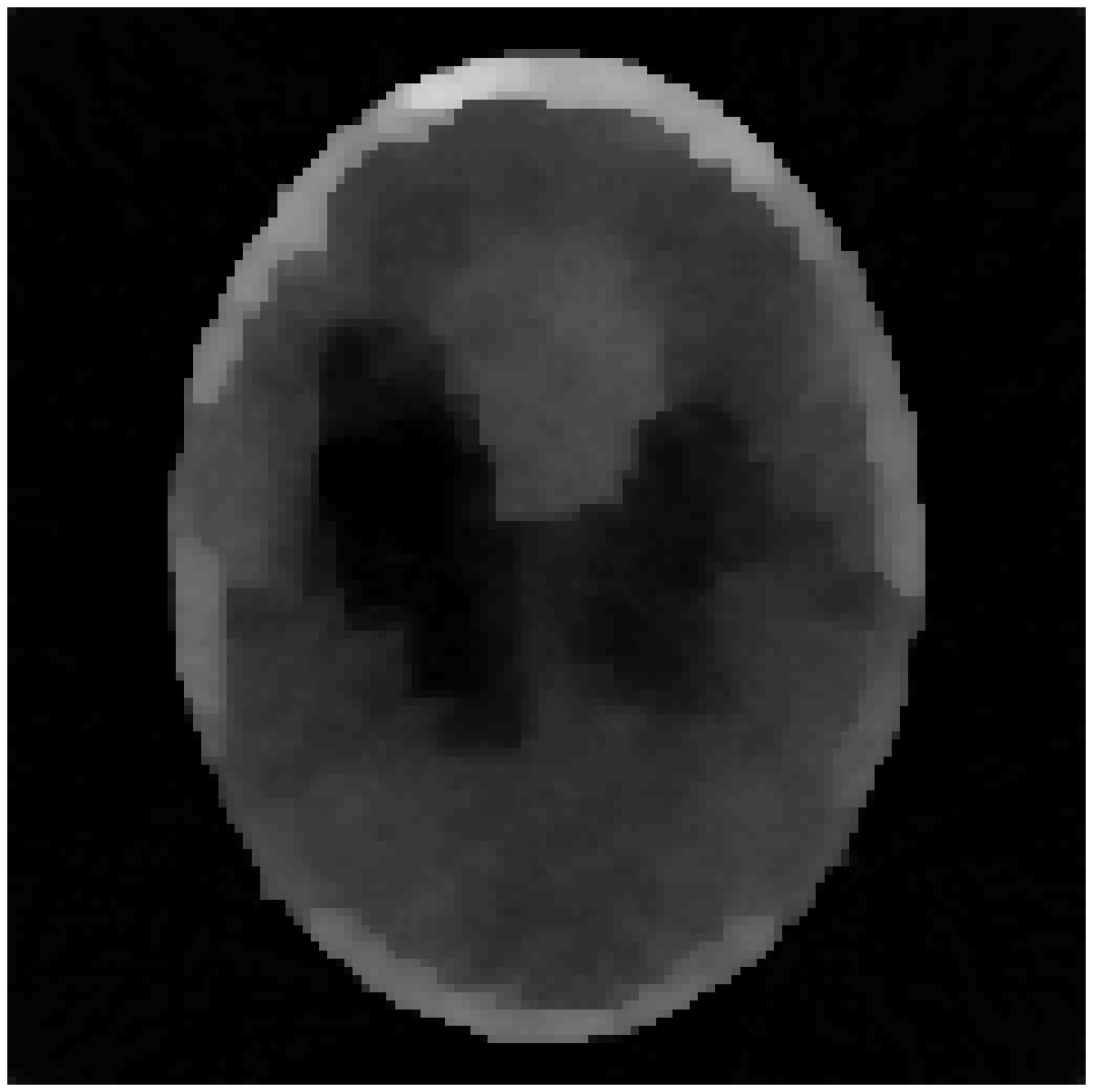} & \includegraphics[scale=0.2]{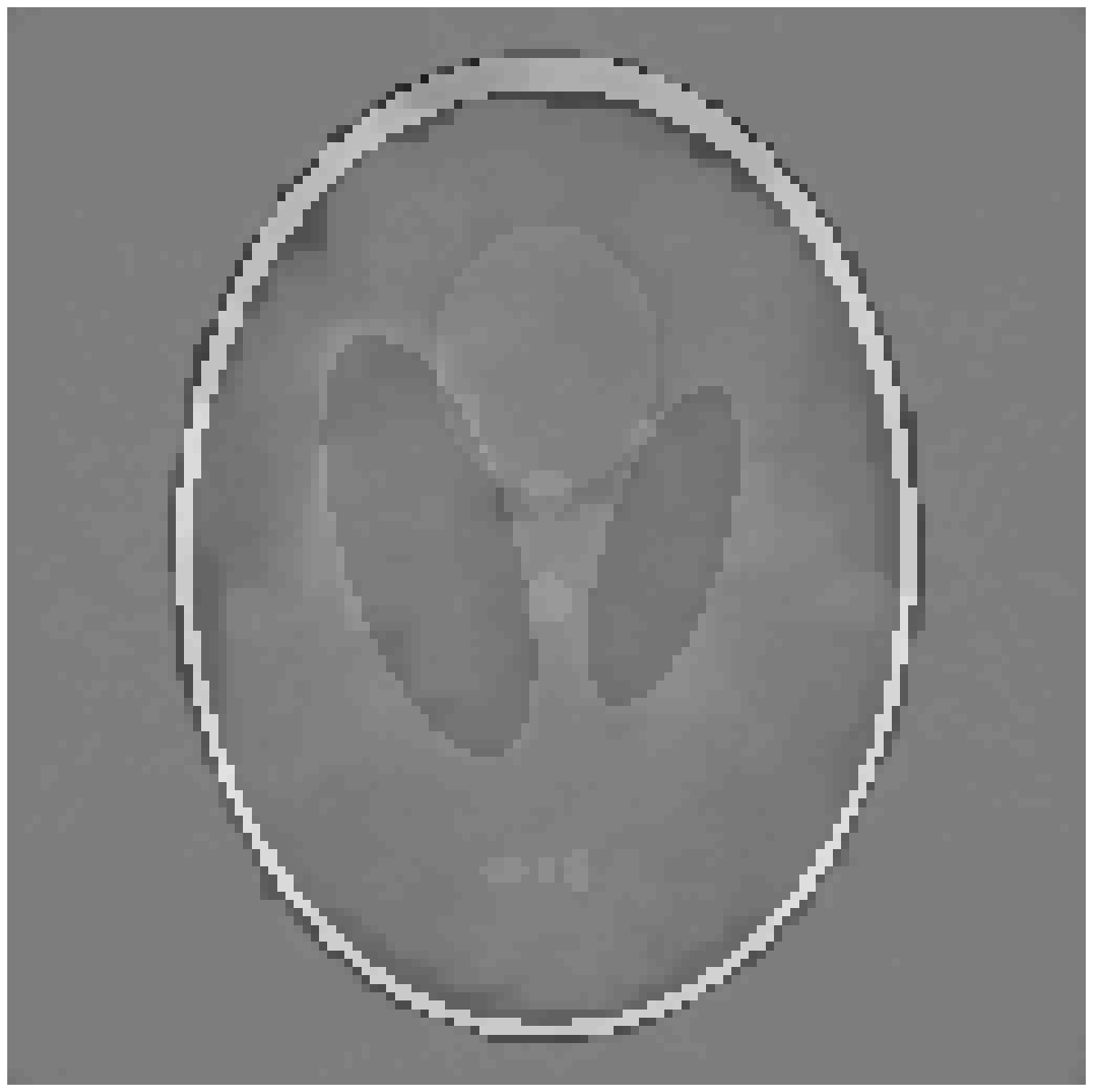} & \includegraphics[scale=0.2]{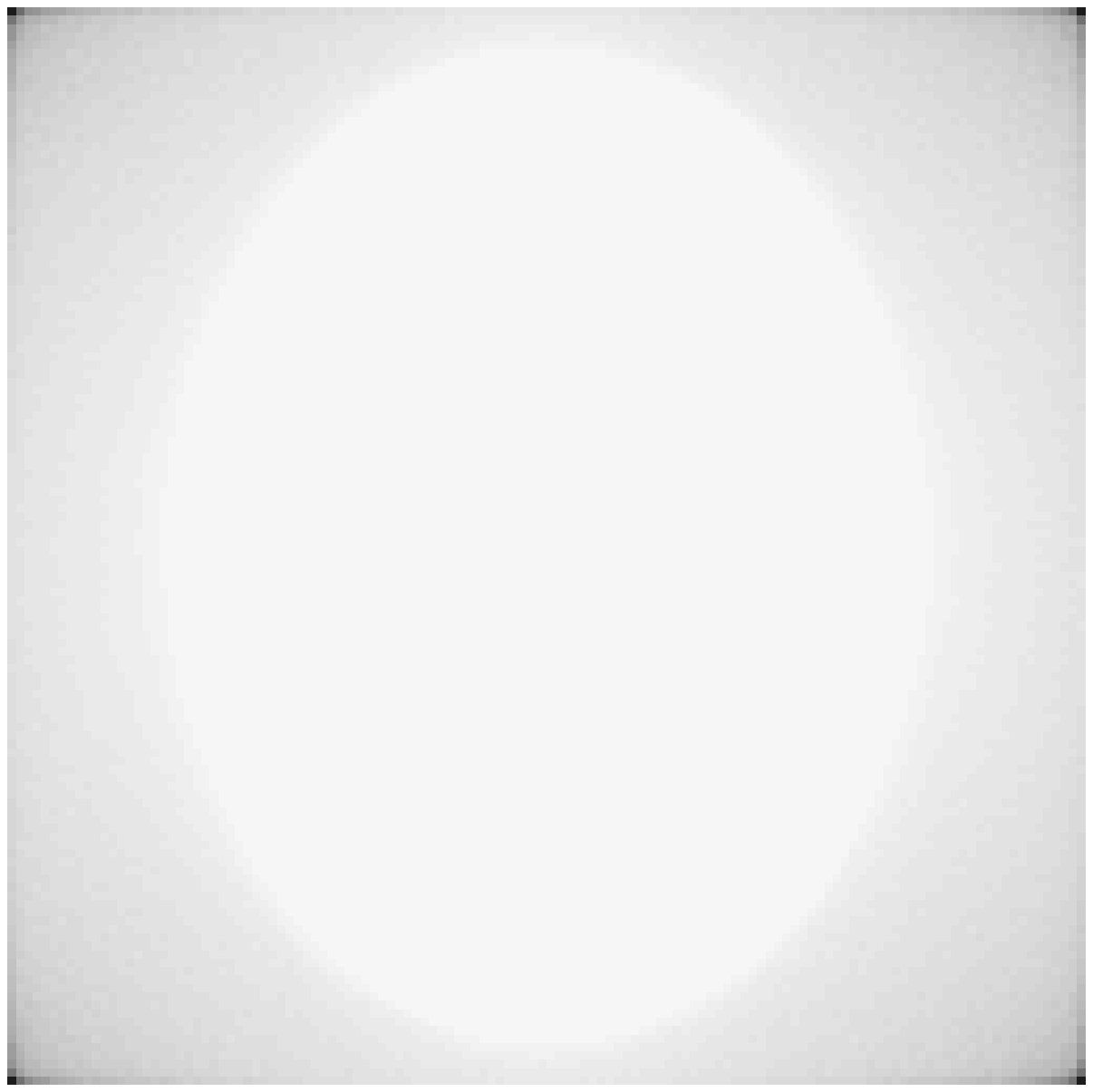}\\
\rotatebox{90}{\quad[0:8:179]}&\includegraphics[scale=0.2]{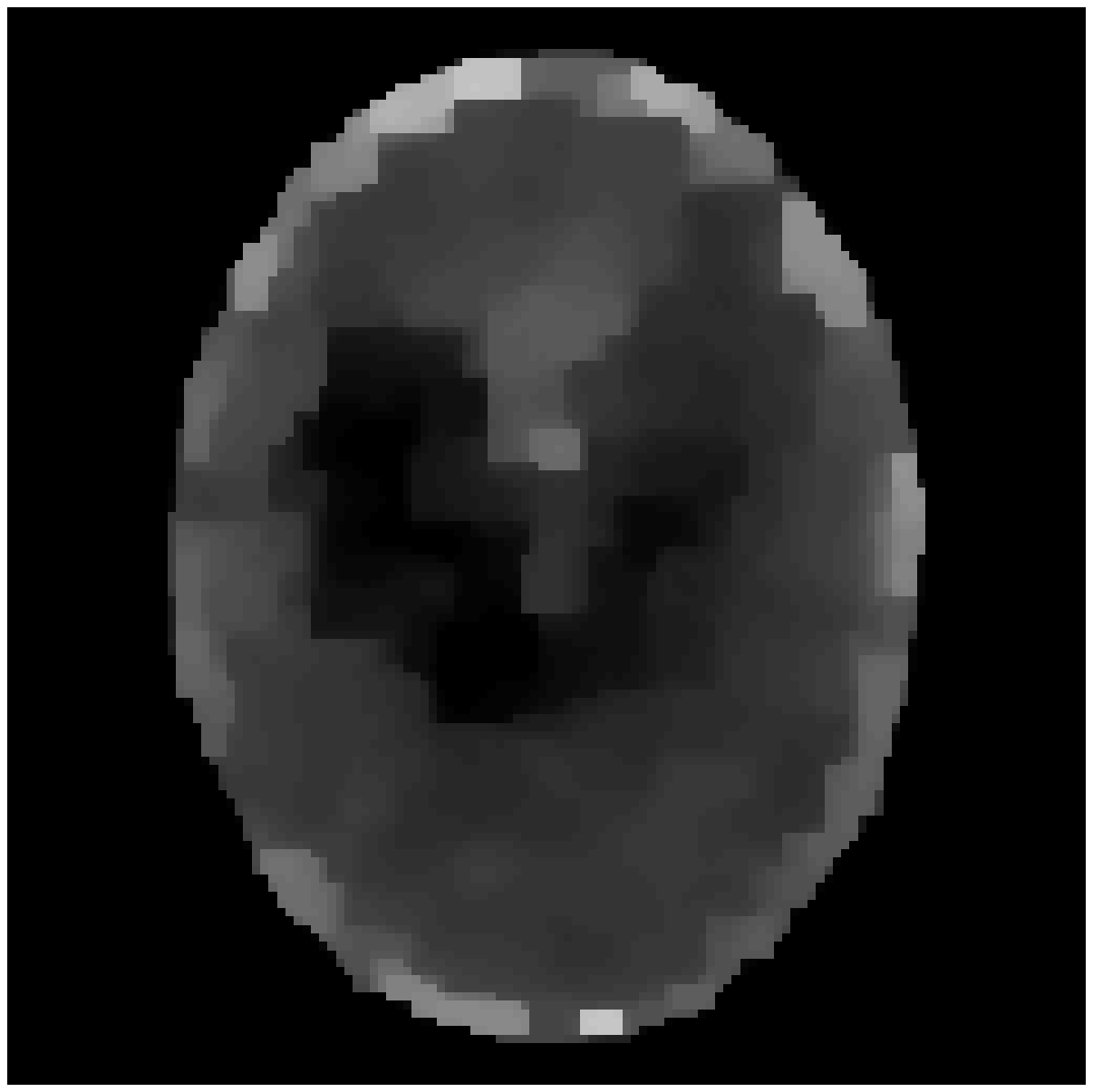} & \includegraphics[scale=0.2]{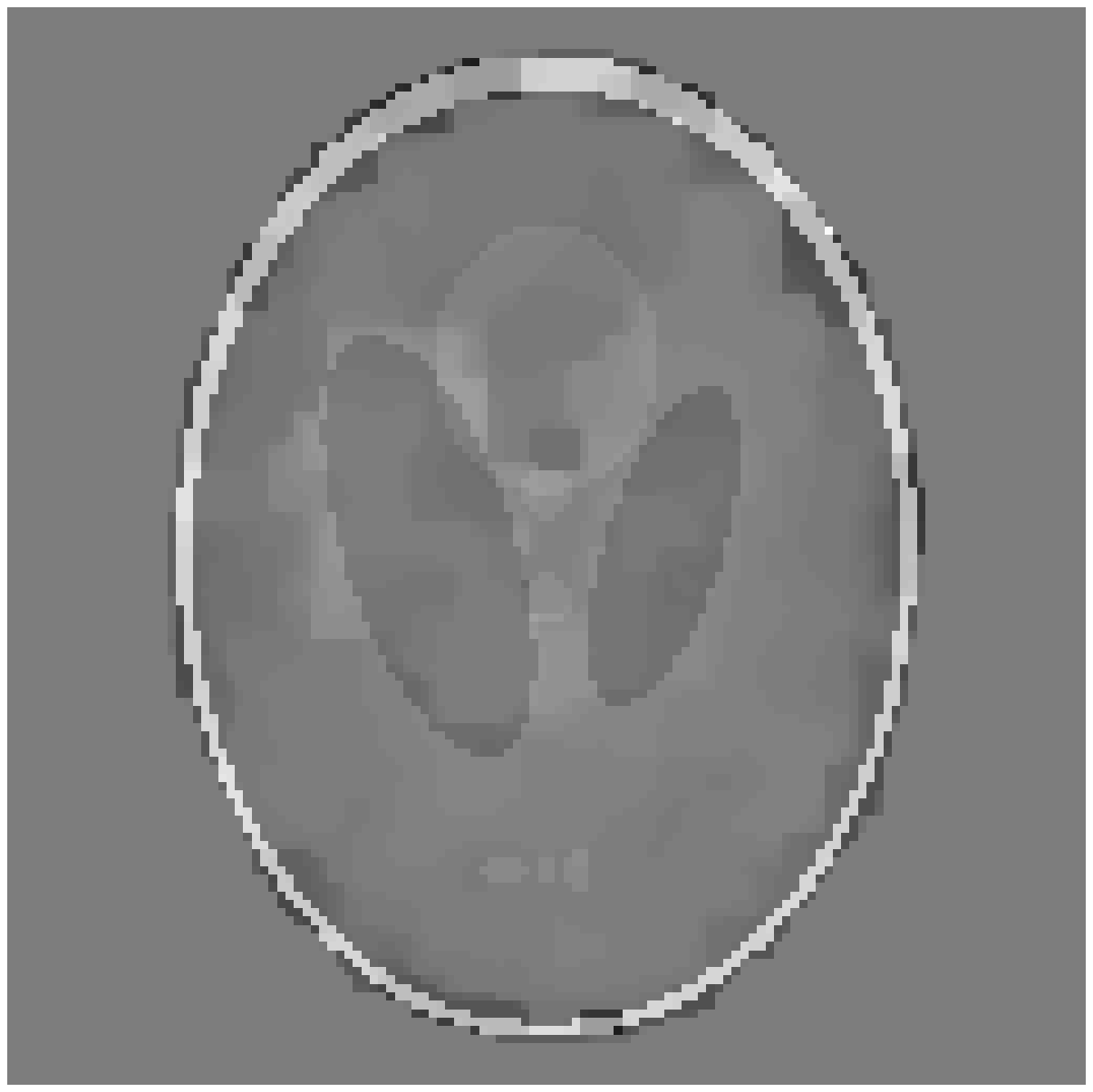} & \includegraphics[scale=0.2]{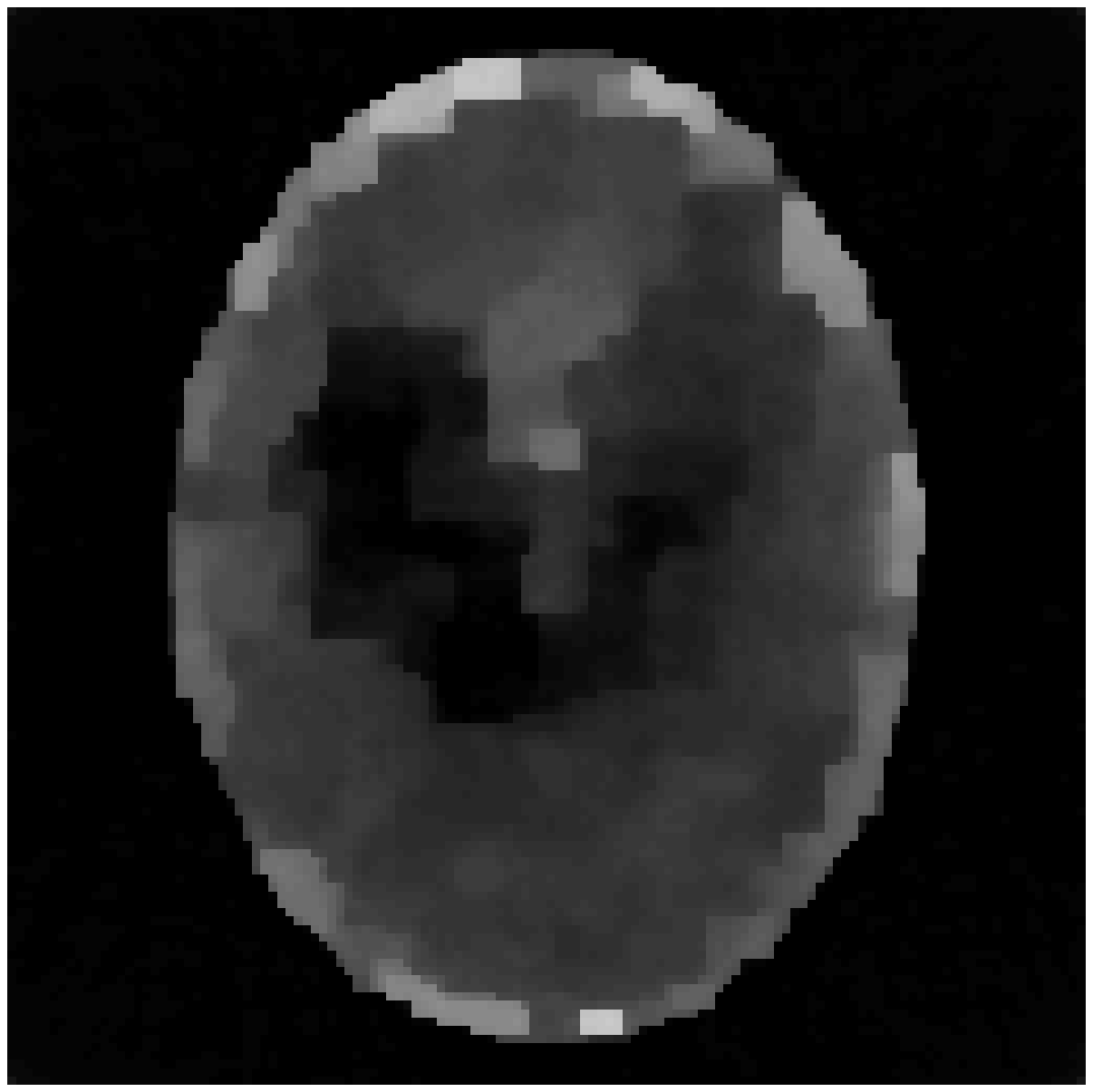} & \includegraphics[scale=0.2]{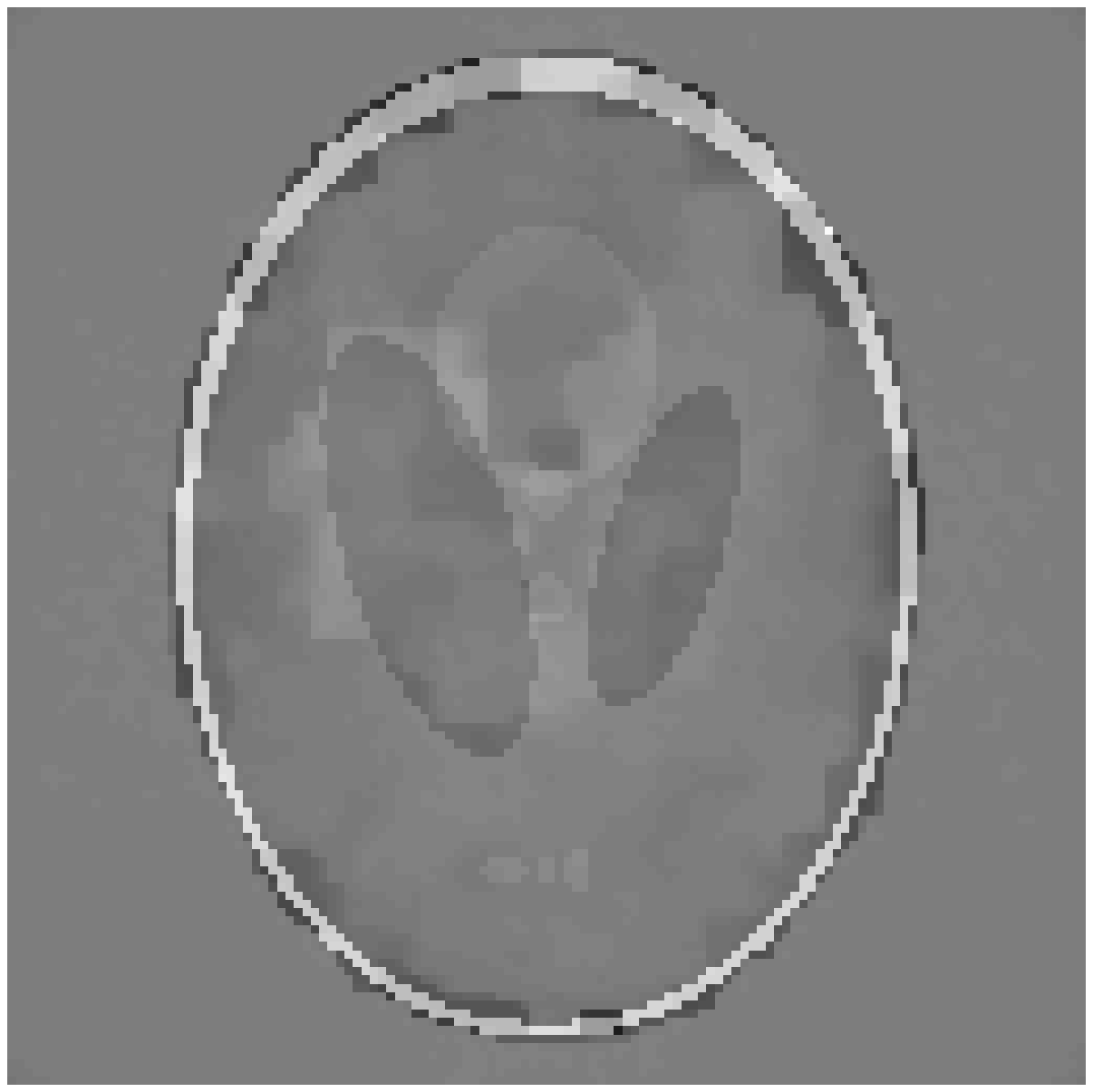} & \includegraphics[scale=0.2]{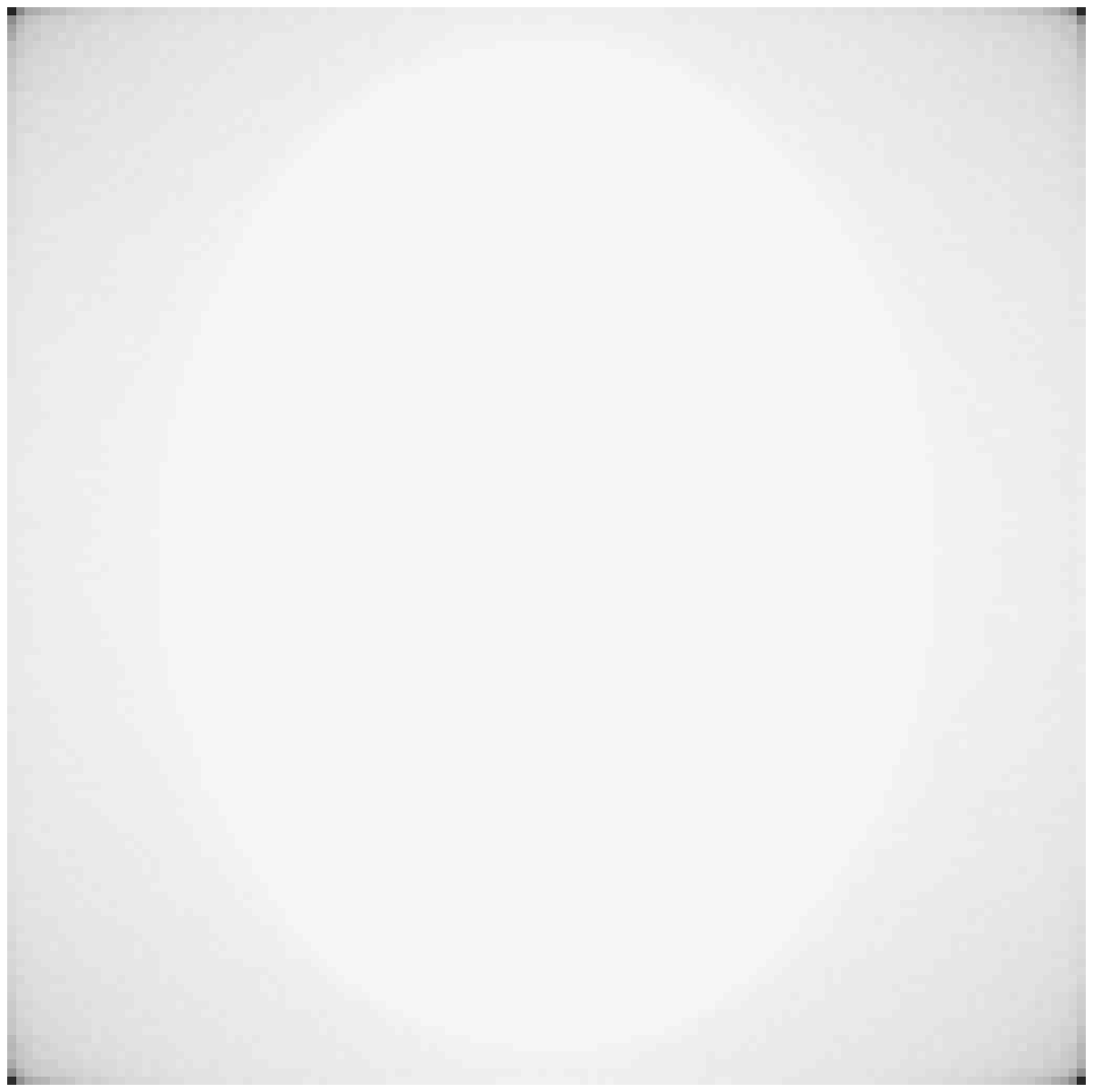}\\
&MAP & MAP error & EP mean & EP error & EP variance
\end{tabular}
\caption{MAP vs EP with anisotropic TV prior for the \texttt{Shepp-Logan} phantom, low count case.\label{fig:LC_SL_248}}	
\end{figure}

First, we take simulated data: the ground-truth images are \texttt{Shepp-Logan} and \texttt{PET}
\cite{ehrhardt2014joint} phantoms of size $128\times 128$; The map $A$ is a discrete Radon transform,
formed using \texttt{MATLAB} built-in function \texttt{radon} with $185$ projections per angle and
three different angle settings, i.e., $[0:2:179]$, $[0:4:179]$ and $[0:8:179]$, and accordingly, the
matrix $A$ is of size $A\in\mathbb{R}^{16650\times 16384}$, $A\in\mathbb{R}^{8325\times 16384}$ and
$A\in\mathbb{R}^{4255\times 16384}$. For each image, we consider two count levels:
the moderate count case is obtained from $A$, and the low count case from $A/3$ (so that the measured
counts are mostly below 10). The original image, sinogram and observed Poisson data are shown
in Figs. \ref{fig:SL_xby} and \ref{fig:ME_xby} for Shepp-Logan and PET, respectively. The numerical
results are summarized in Tables \ref{tab:SL_quant} and \ref{tab:ME_quant}, Figs.
\ref{fig:SL_248}--\ref{fig:LC_SL_248} and Figs. \ref{fig:ME_248}--\ref{fig:LC_ME_248}. The EP mean is mostly
comparable with MAP in all three metrics for both moderate count and low count cases, and the
reconstruction quality improves steadily as the number of projection angles increases. Interestingly,
the shape of the EP variance resembles closely the outer boundary of the phantom, whereas within the
boundary, there is little difference in the magnitudes. This might indicate that the
algorithm is rather certain in the cold regions where the error is close to zero and more uncertain
about the region where the error is potentially larger. It is observed that the computational complexity of
the EP grows with the amount of the data. This is attributed to the following fact: the number of sweeps is
fixed at four, and the complexity increases with the number of projection angles. Since the computing 
time is presented only for one reconstruction at each case, these numbers should be viewed
as a representative instead of an absolute measure for algorithmic performance. Roughly, EP is about two
orders of magnitude more expensive than the MAP approach (computed by limited memory BFGS \cite{LiuNocedal:1989}).

The Poisson model is especially useful for low count data, where a naive Gaussian approximation can fail
to give reasonable reconstructions. The EP results for the low-count case are shown in Figs.
\ref{fig:LC_SL_248} and \ref{fig:LC_ME_248}. Just as expected, the reconstruction accuracy deteriorates
as the count level decreases. Nonetheless, the EP means remain largely comparable with MAP results both
qualitatively and quantitatively. Note
that for the \texttt{PET} image, the reconstruction accuracy for both EP and MAP suffers significantly
in that the fine details such as vertical bars in the true image disappear, especially when the number of
projection angles is small. The computing times for the moderate count and low count cases are nearly the
same; see Tables \ref{tab:SL_quant} and \ref{tab:ME_quant}. Thus, EP is still feasible for the low-count case.

\begin{table}[hbt!]
\centering
\caption{Comparisons between EP mean and MAP for the \texttt{Shepp-Logan} phantom. The top and bottom
blocks refer to the moderate count and  low count cases, respectively.\label{tab:SL_quant}}
\begin{tabular}{|l|c|c|c|c|c|c|}
 \hline
 angle       & \multicolumn{2}{|c|}{[0:2:179]} & \multicolumn{2}{|c|}{[0:4:179]} & \multicolumn{2}{|c|}{[0:8:179]}\\
 \hline
$\alpha$ & \multicolumn{2}{|c|}{6e0} & \multicolumn{2}{|c|}{4e0} & \multicolumn{2}{|c|}{3e0}\\
 \hline
Method   & EP &  MAP & EP & MAP & EP & MAP\\
 \hline
$L^2$ error & 5.32  & 5.36  & 5.64  & 5.67  & 6.09  & 6.11\\
 \hline
SSIM     & 0.74  & 0.78  & 0.70  & 0.75  & 0.67  & 0.72\\
 \hline
PSNR     & 18.58 & 18.53 & 17.97 & 17.93 & 17.29 & 17.27\\
 \hline
CPU time (s) & 80187.88 & 124.44 & 46031.95 & 55.55 & 29274.16 & 27.23\\
 \hline
 \hline
 \hline
$\alpha$ & \multicolumn{2}{|c|}{1.3e0} & \multicolumn{2}{|c|}{2e0} & \multicolumn{2}{|c|}{1e0}\\
 \hline
Method   & EP &  MAP & EP & MAP & EP & MAP\\
 \hline
$L^2$ error & 4.07  & 4.09  & 6.15  & 6.24  & 6.14  & 6.19\\
 \hline
SSIM     & 0.57  & 0.79  & 0.51  & 0.72  & 0.48  & 0.70\\
 \hline
PSNR     & 19.50 & 19.47 & 17.53 & 17.42 & 17.18 & 17.15\\
 \hline
CPU time (s) & 82125.92 & 42.25 & 47110.50 & 29.69 & 29756.10 & 15.20\\
 \hline
\end{tabular}
\end{table}

\begin{figure}[htb!]
\centering
\begin{tabular}{ccccccccc}
\includegraphics[scale=0.2]{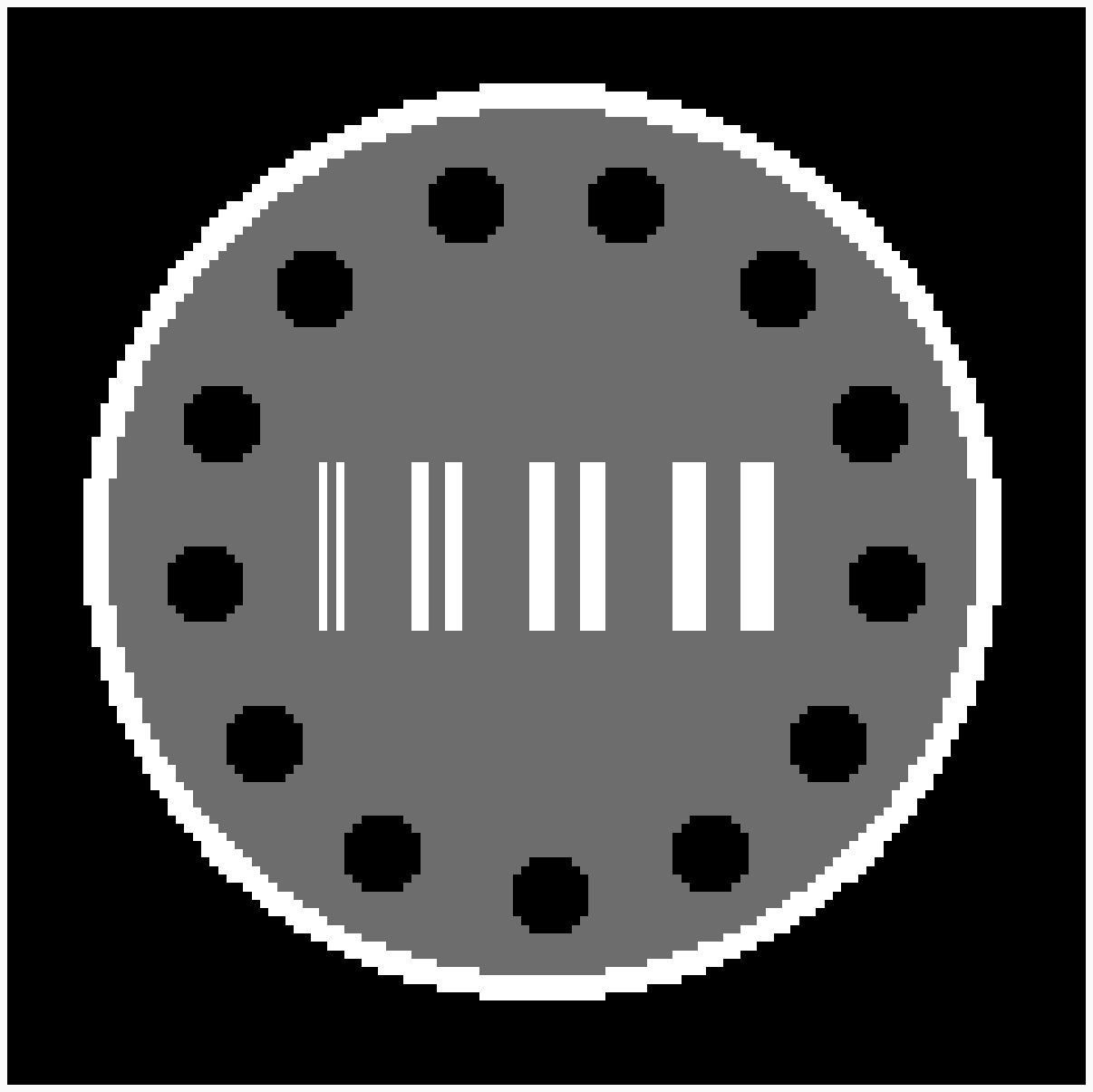}  \includegraphics[scale=0.2]{bar_1} & \includegraphics[scale=0.2]{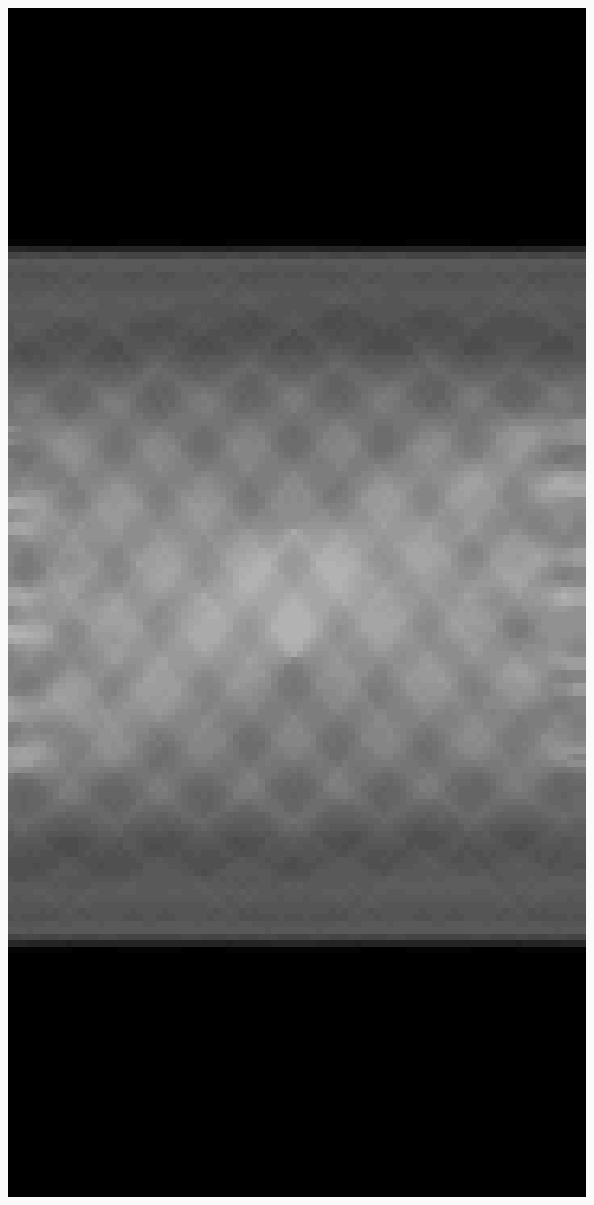}  \includegraphics[scale=0.2]{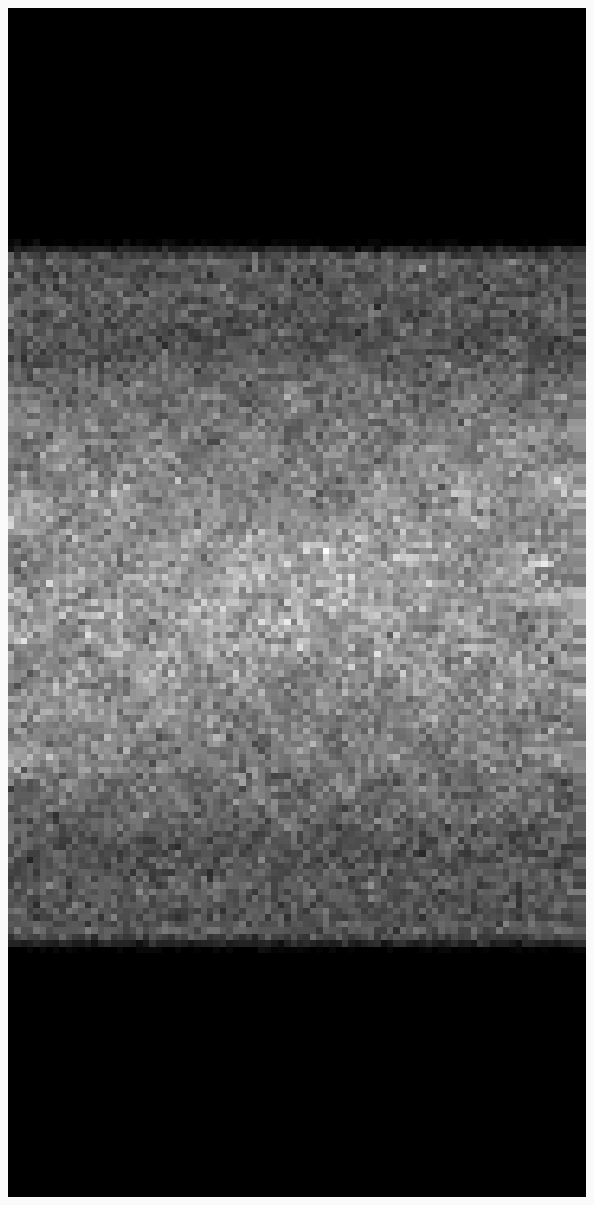} & \includegraphics[scale=0.2]{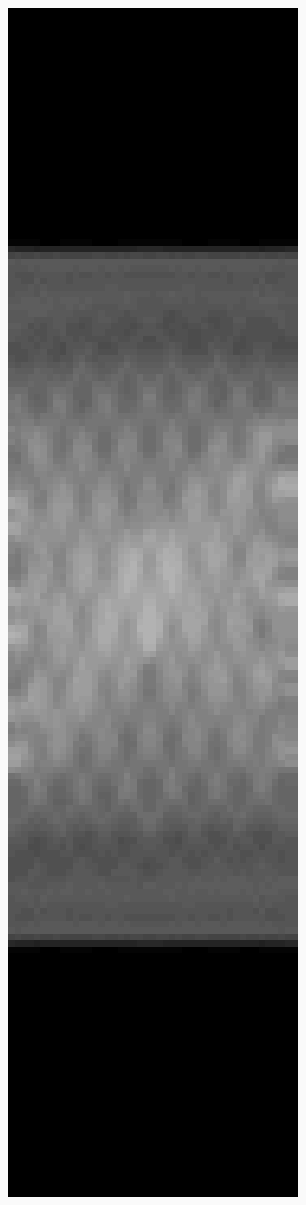}  \includegraphics[scale=0.2]{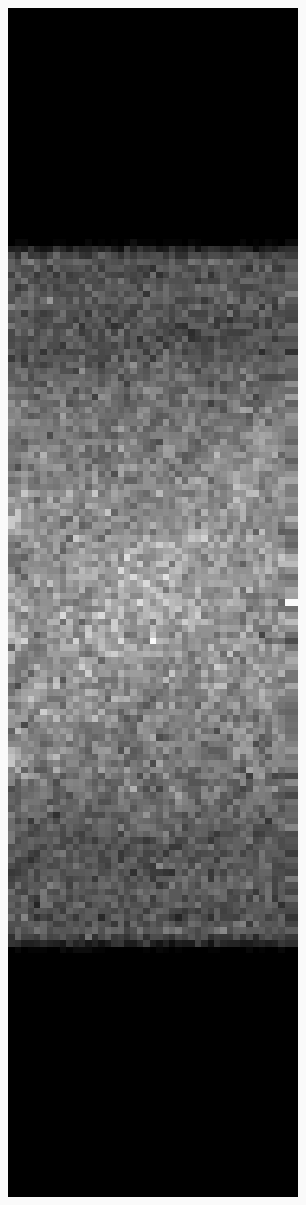} & \includegraphics[scale=0.2]{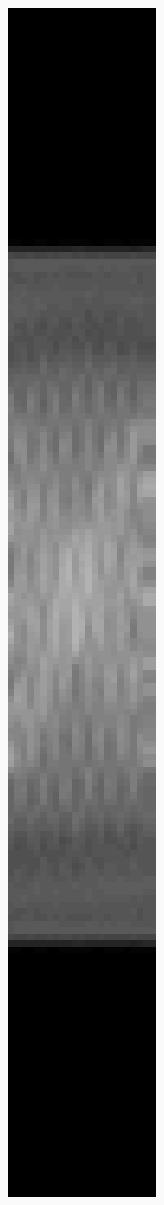}  \includegraphics[scale=0.2]{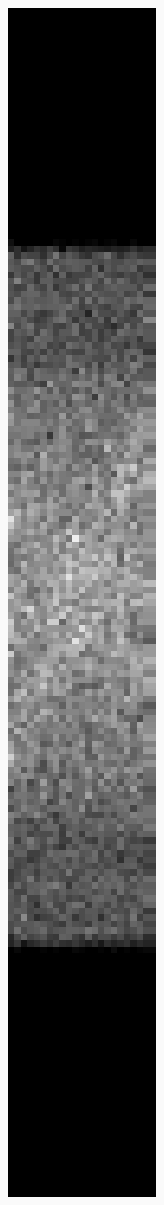} & \includegraphics[scale=0.2]{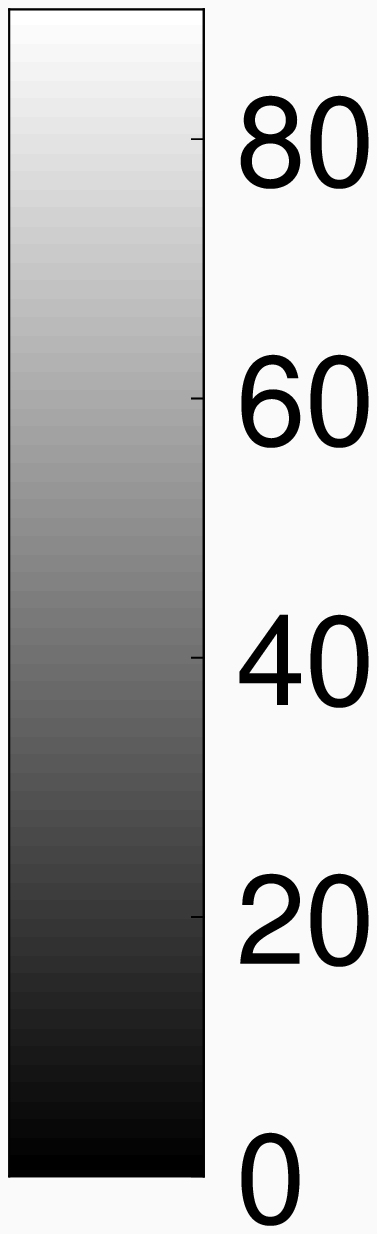}\\
\includegraphics[scale=0.2]{ME_true}  \includegraphics[scale=0.2]{bar_1} & \includegraphics[scale=0.2]{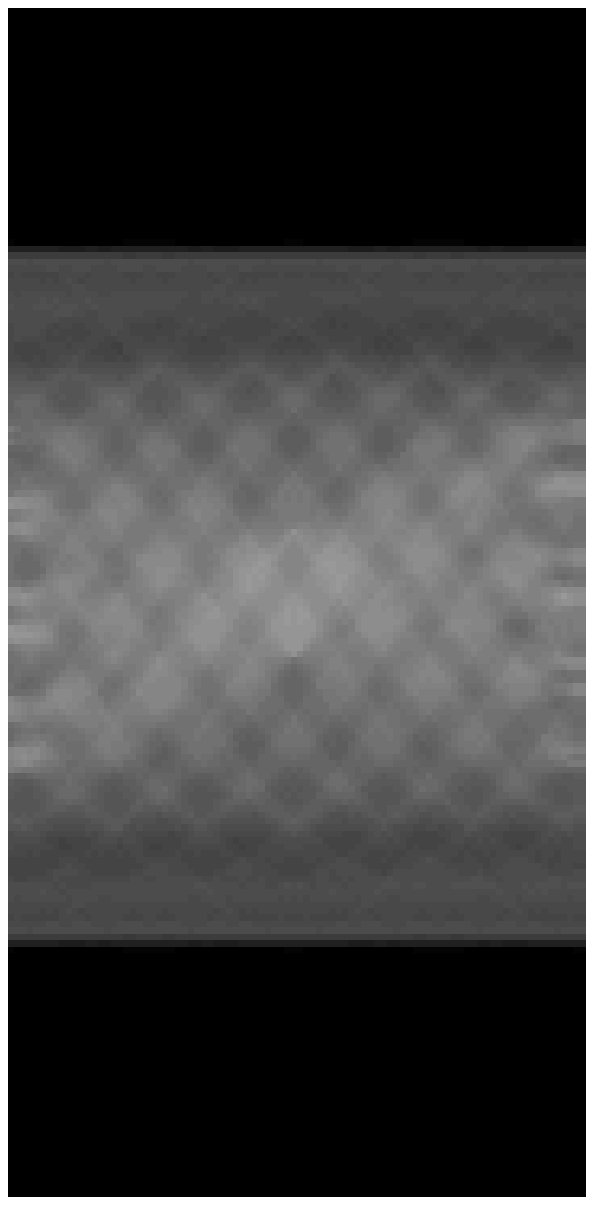}  \includegraphics[scale=0.2]{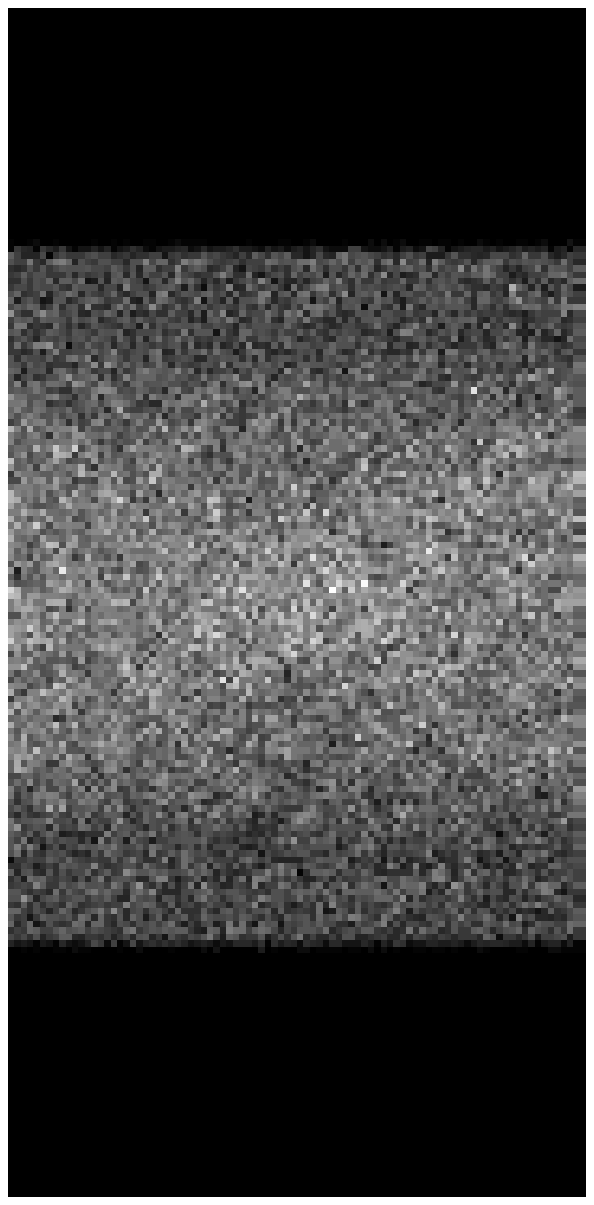} & \includegraphics[scale=0.2]{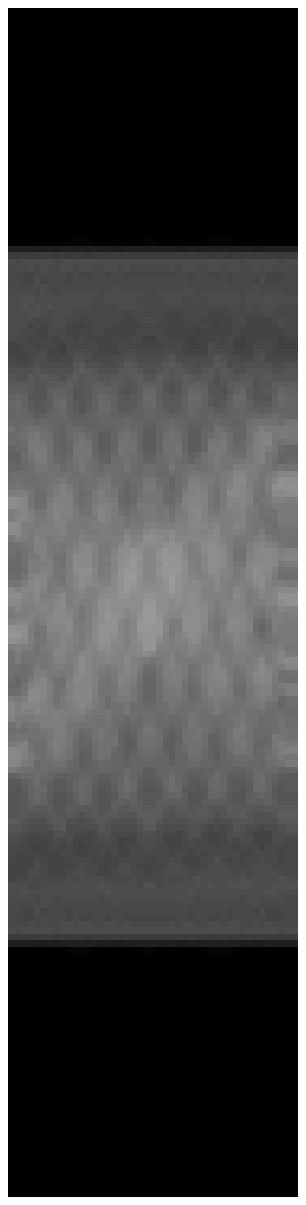}  \includegraphics[scale=0.2]{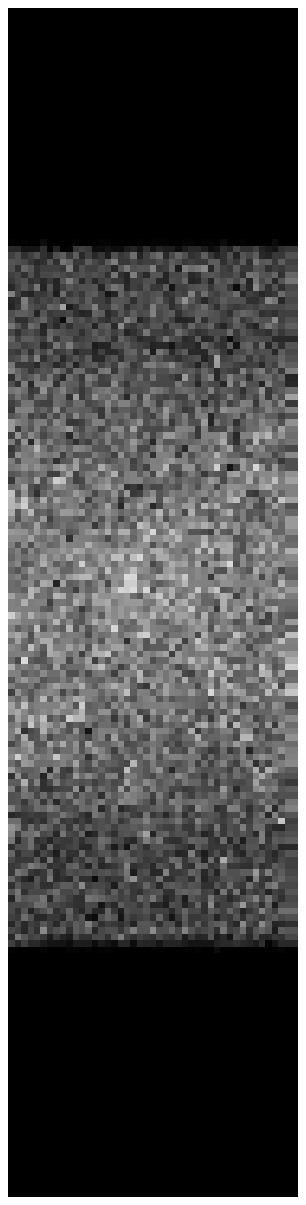} & \includegraphics[scale=0.2]{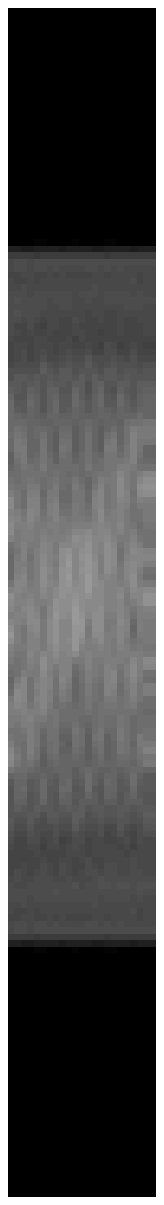}  \includegraphics[scale=0.2]{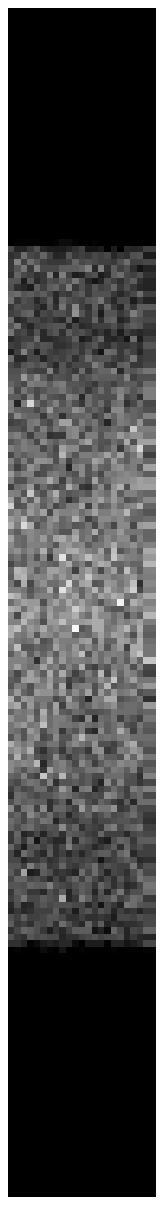} & \includegraphics[scale=0.2]{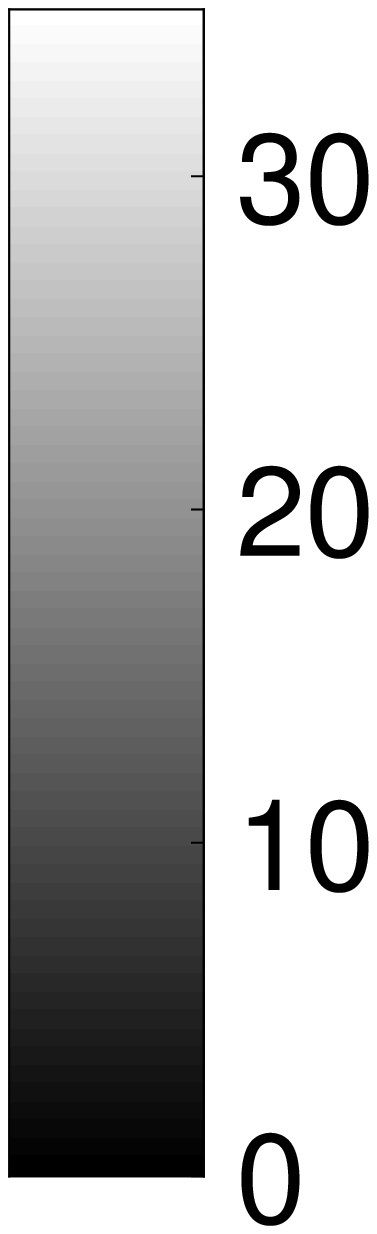}
\end{tabular}
\caption{The exact image, sinograms and observed data with three different $A$'s for the \texttt{PET} phantom. The top and bottom rows
 refer to the moderate count and low count cases, respectively. \label{fig:ME_xby}}	
\end{figure}

\begin{figure}[htb!]
\centering
\begin{tabular}{rcccccc}
 & \includegraphics[scale=0.2]{bar_1_h} & \includegraphics[scale=0.2]{bar_1_1_h} & \includegraphics[scale=0.2]{bar_1_h} & \includegraphics[scale=0.2]{bar_1_1_h} & \includegraphics[scale=0.2]{bar_001_h}\\
\rotatebox{90}{\quad[0:2:179]}&\includegraphics[scale=0.2]{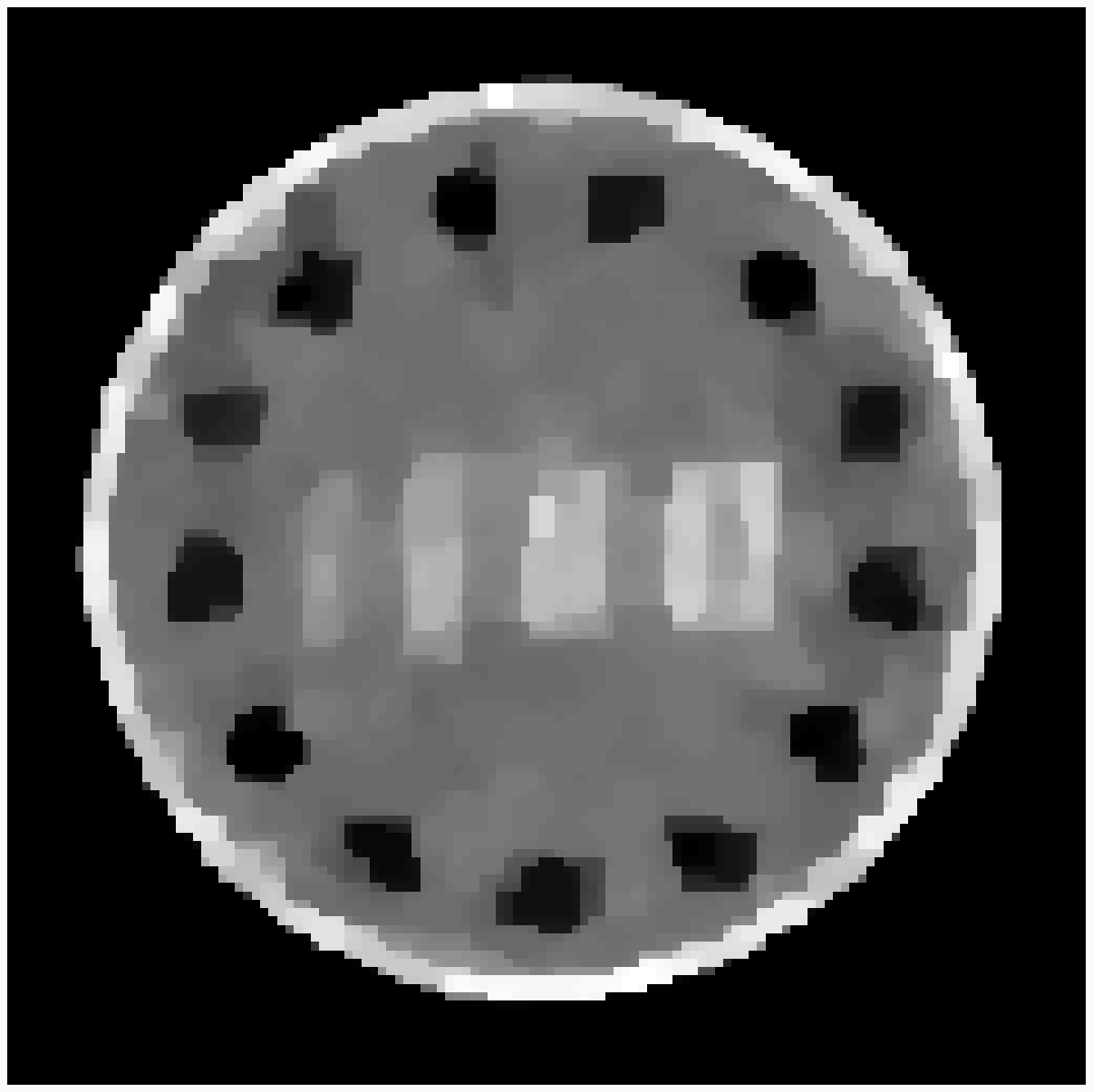} & \includegraphics[scale=0.2]{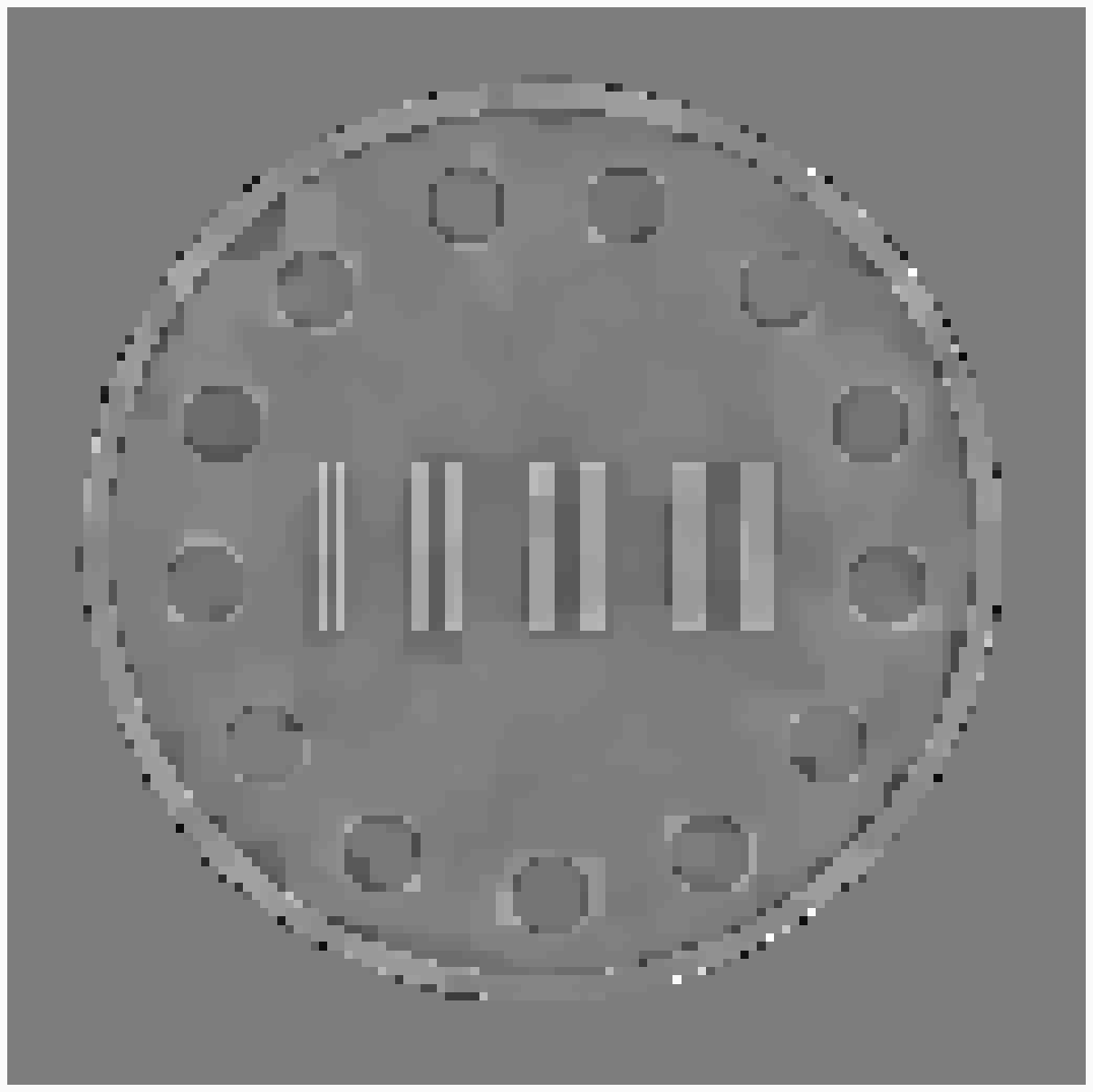} & \includegraphics[scale=0.2]{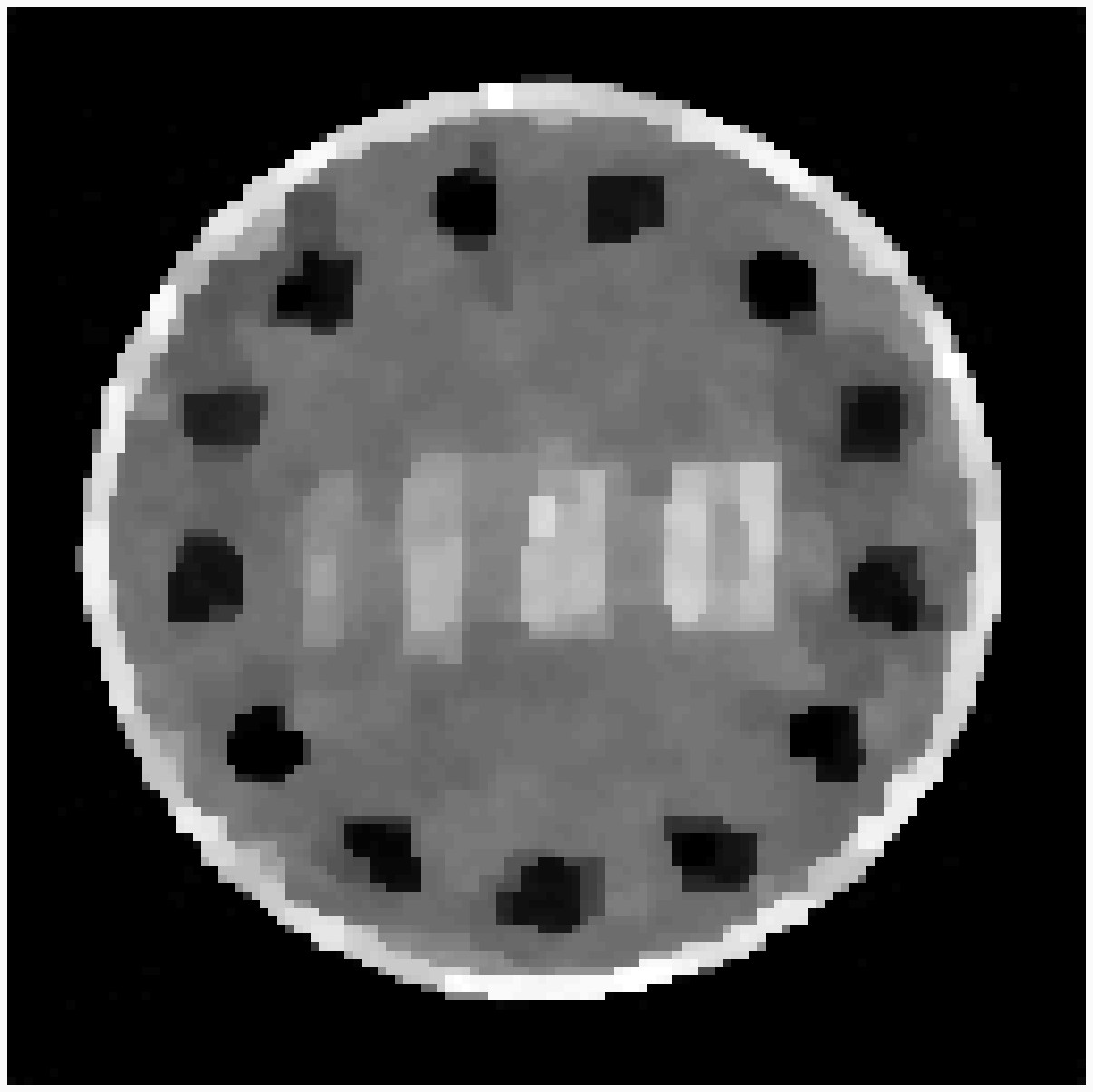} & \includegraphics[scale=0.2]{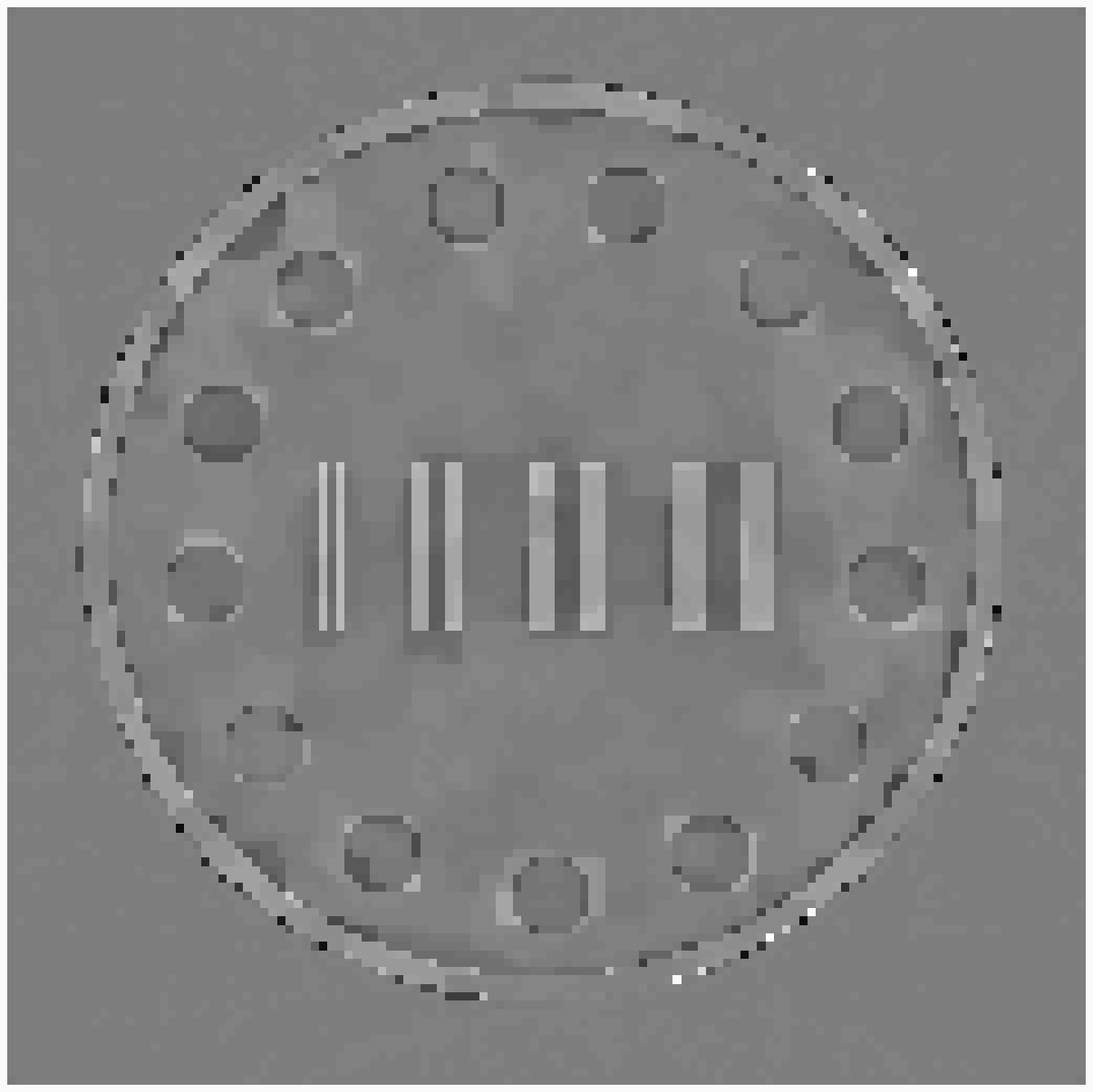} & \includegraphics[scale=0.2]{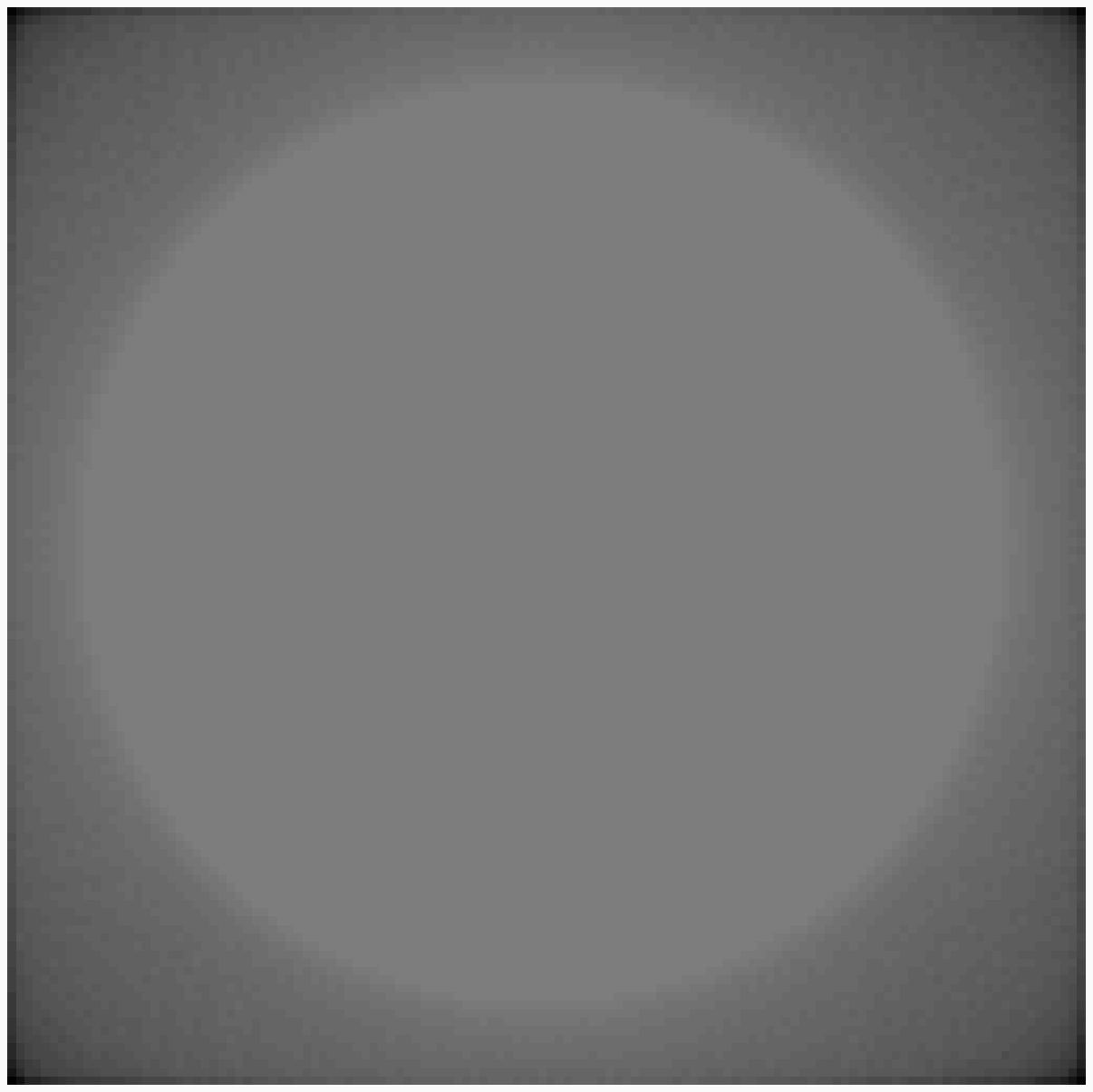}\\
\rotatebox{90}{\quad[0:4:179]}&\includegraphics[scale=0.2]{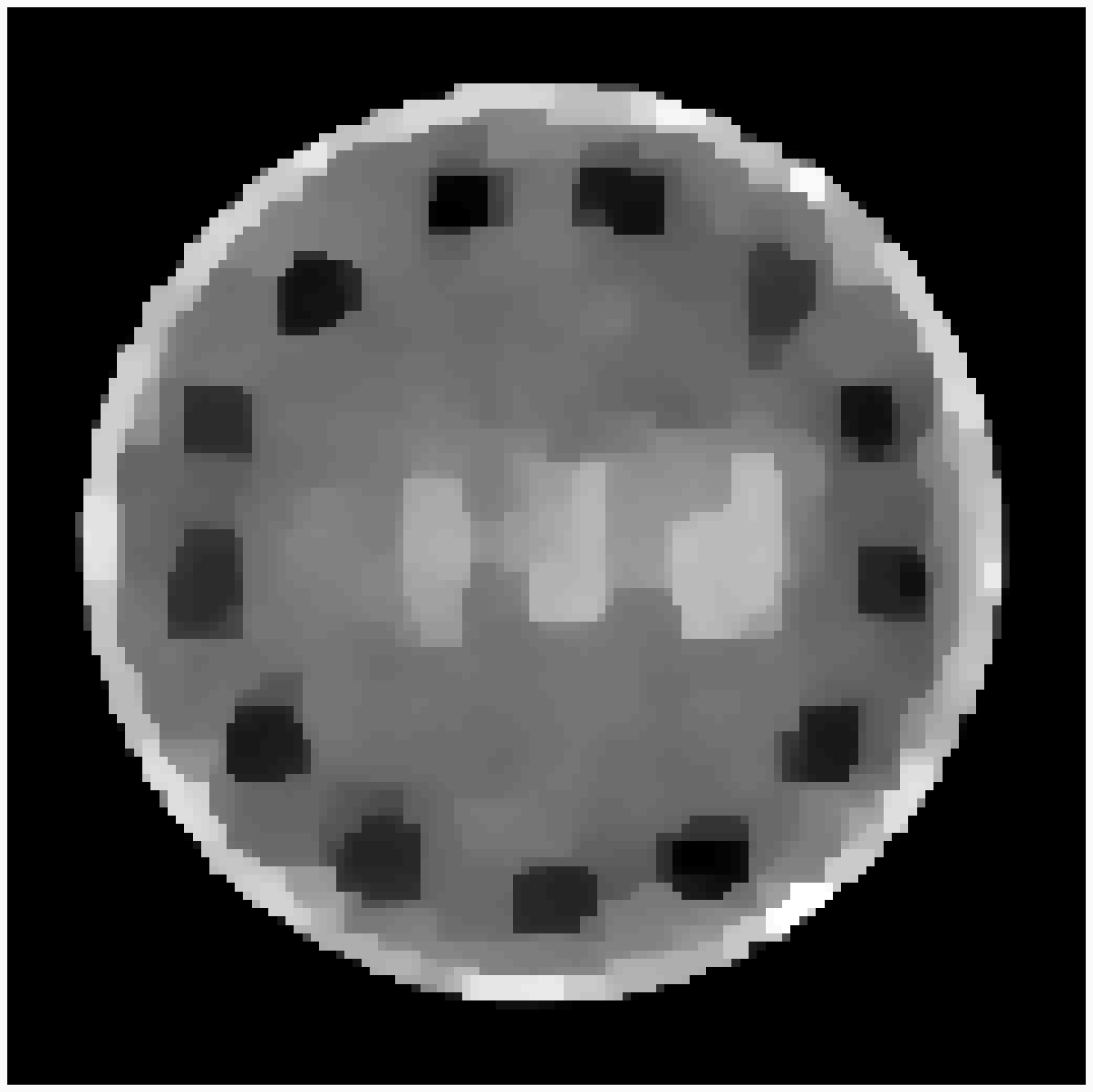} & \includegraphics[scale=0.2]{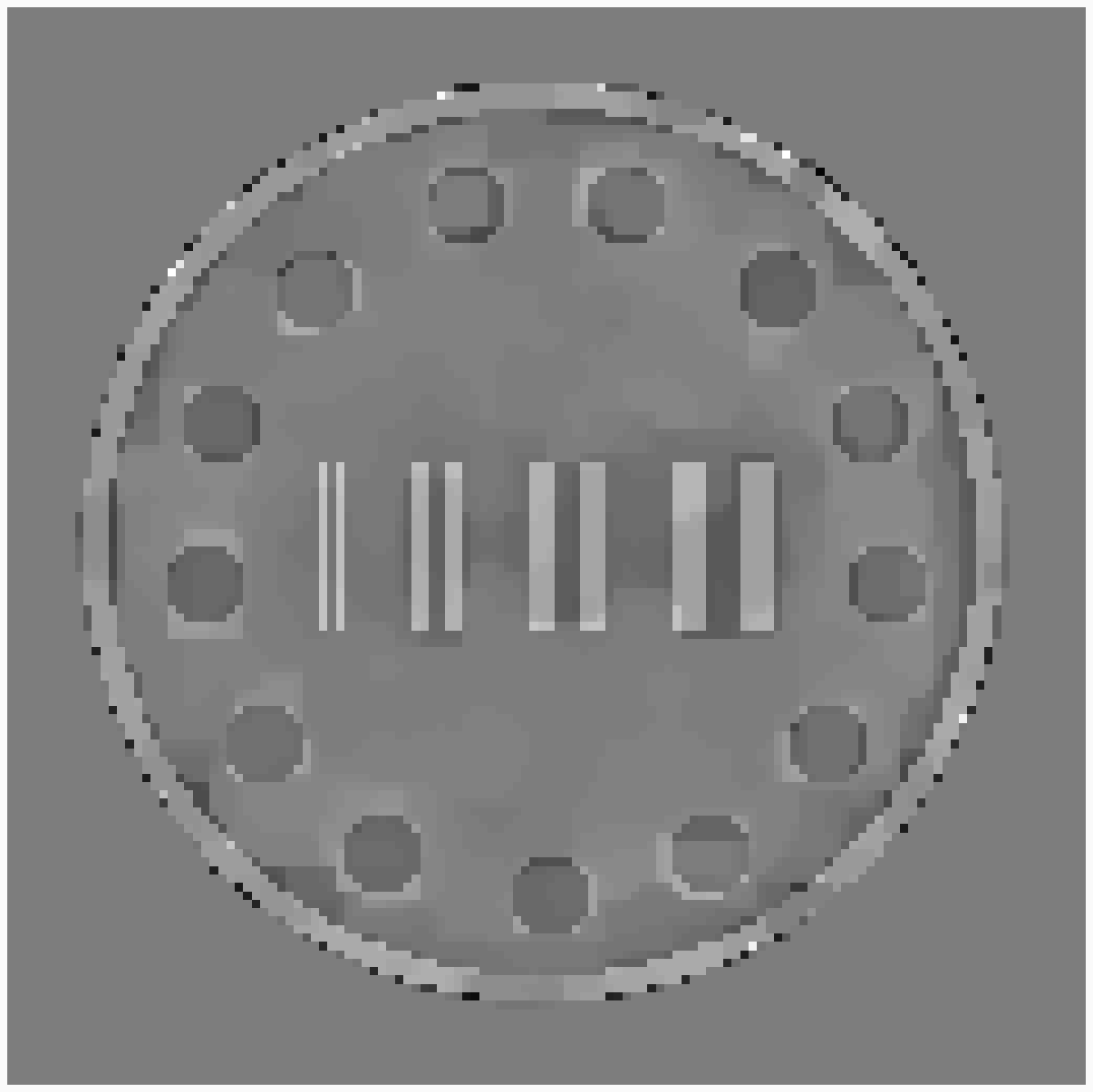} & \includegraphics[scale=0.2]{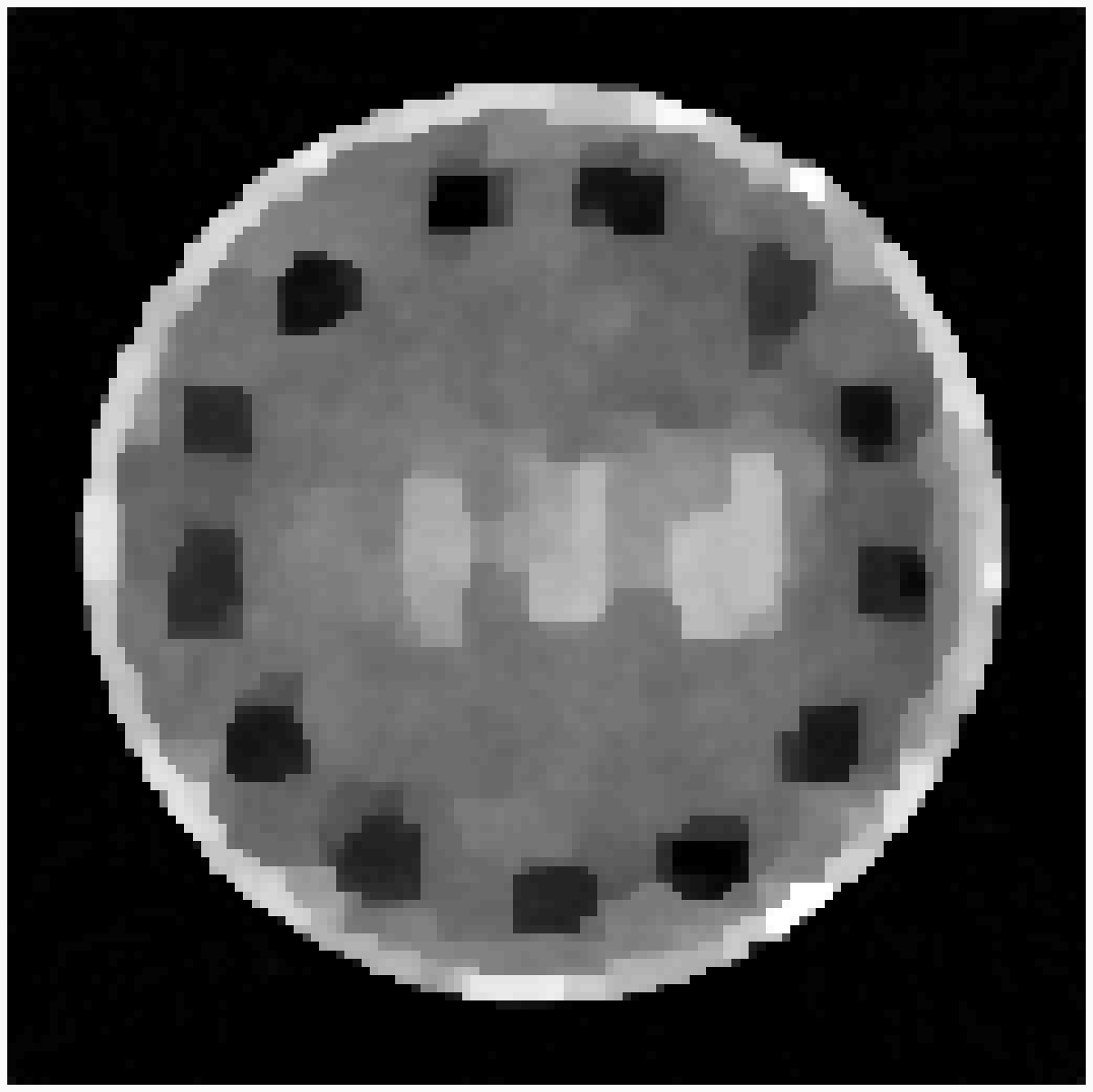} & \includegraphics[scale=0.2]{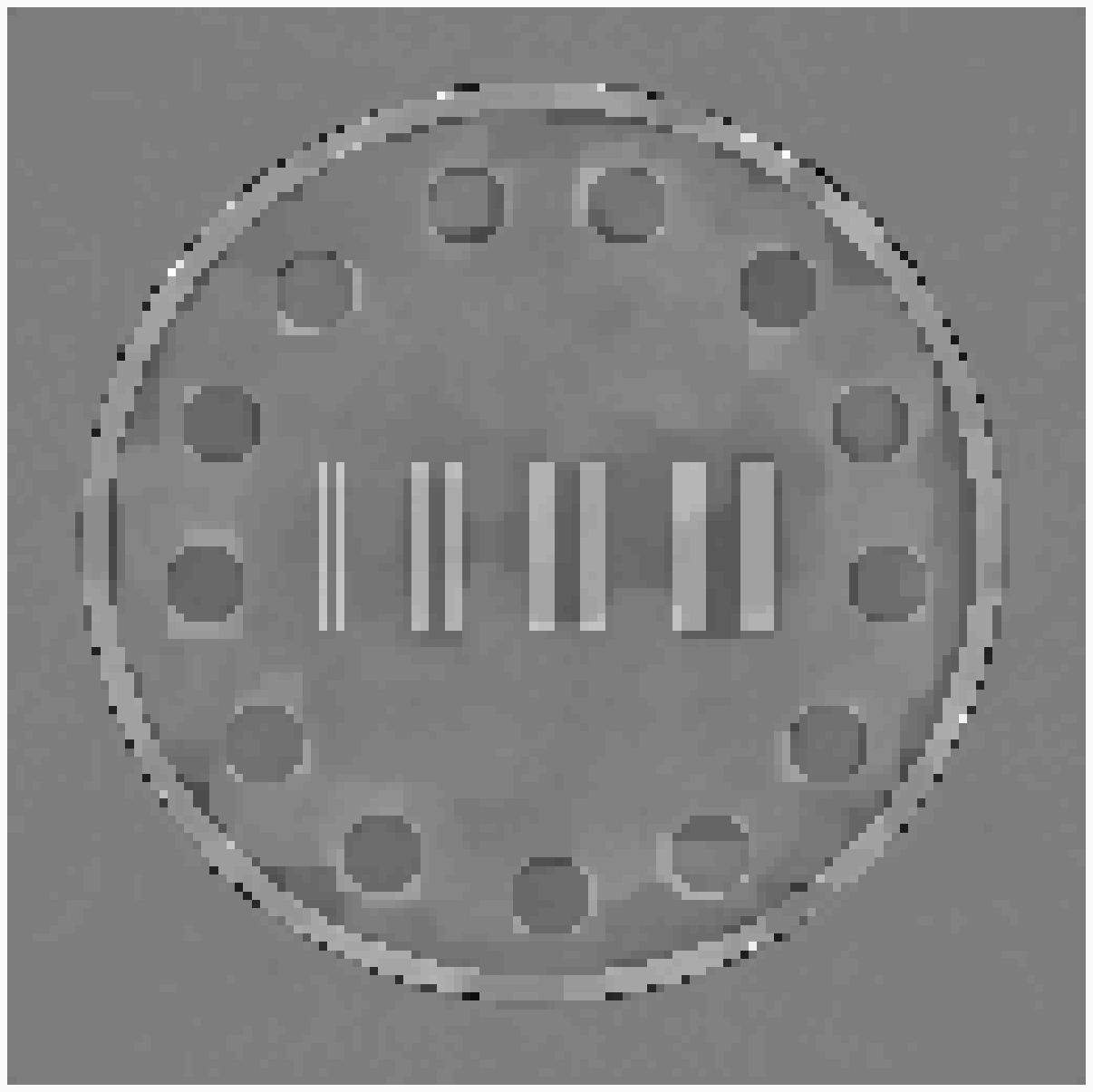} & \includegraphics[scale=0.2]{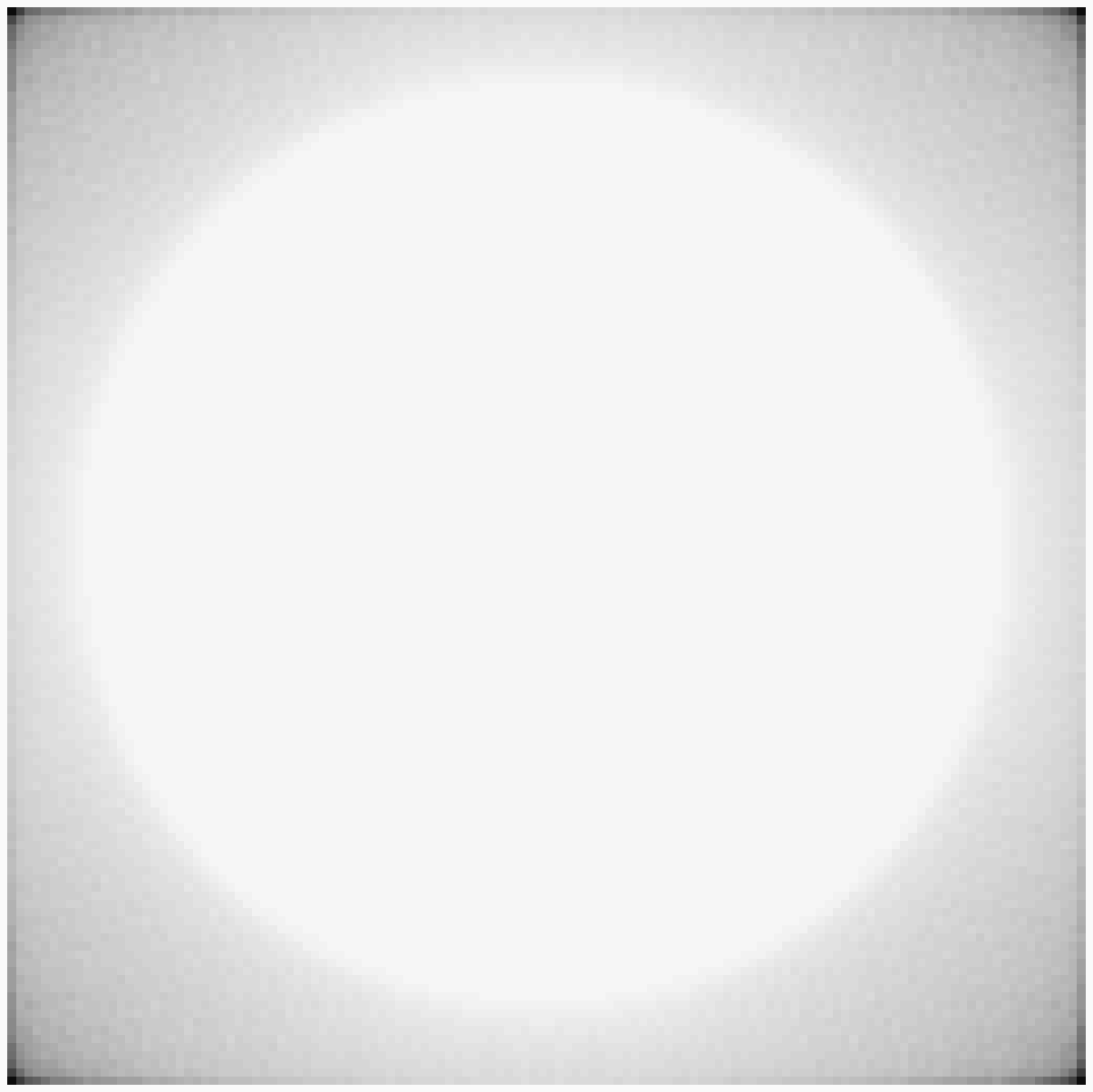}\\
\rotatebox{90}{\quad[0:8:179]}&\includegraphics[scale=0.2]{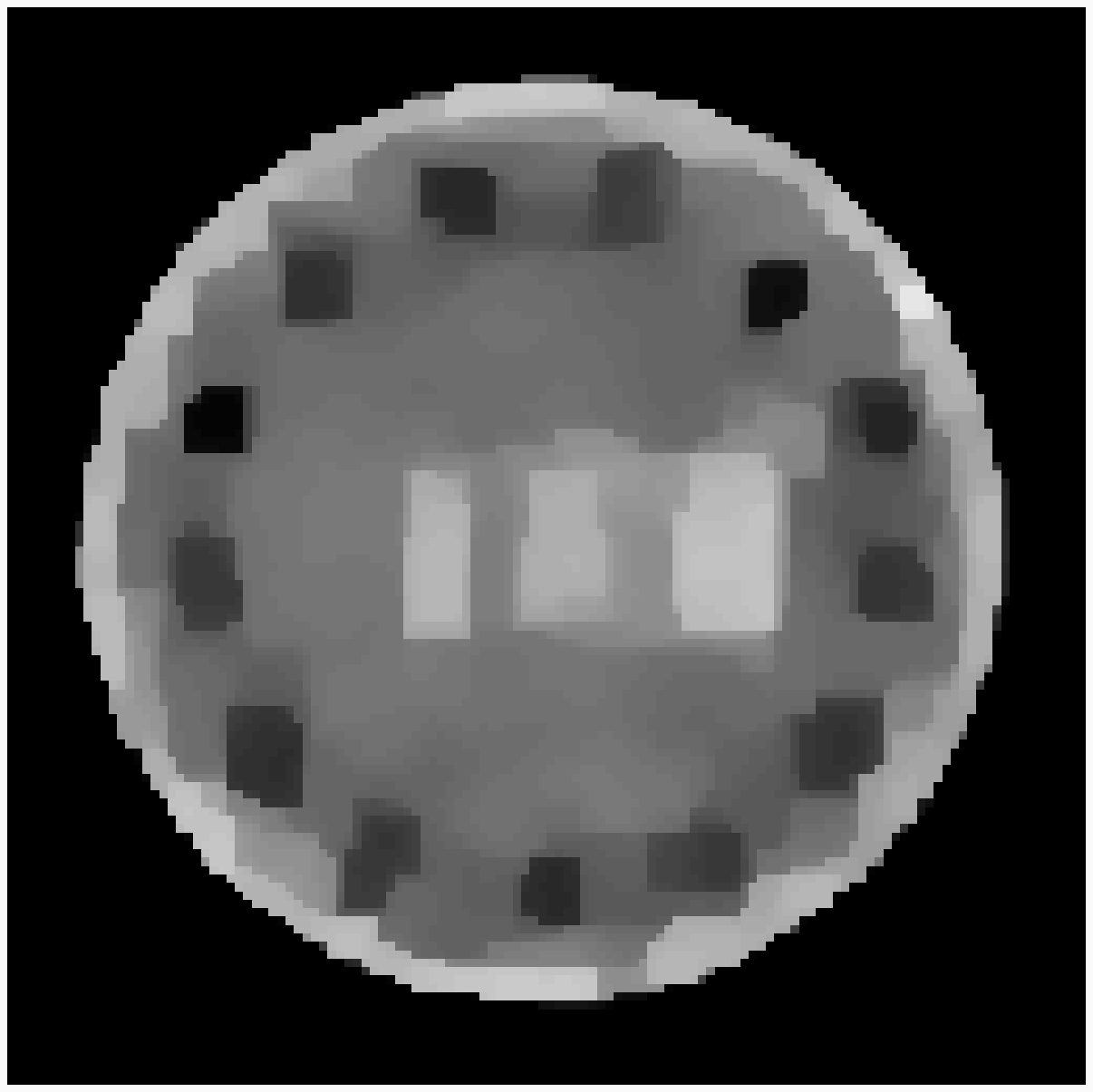} & \includegraphics[scale=0.2]{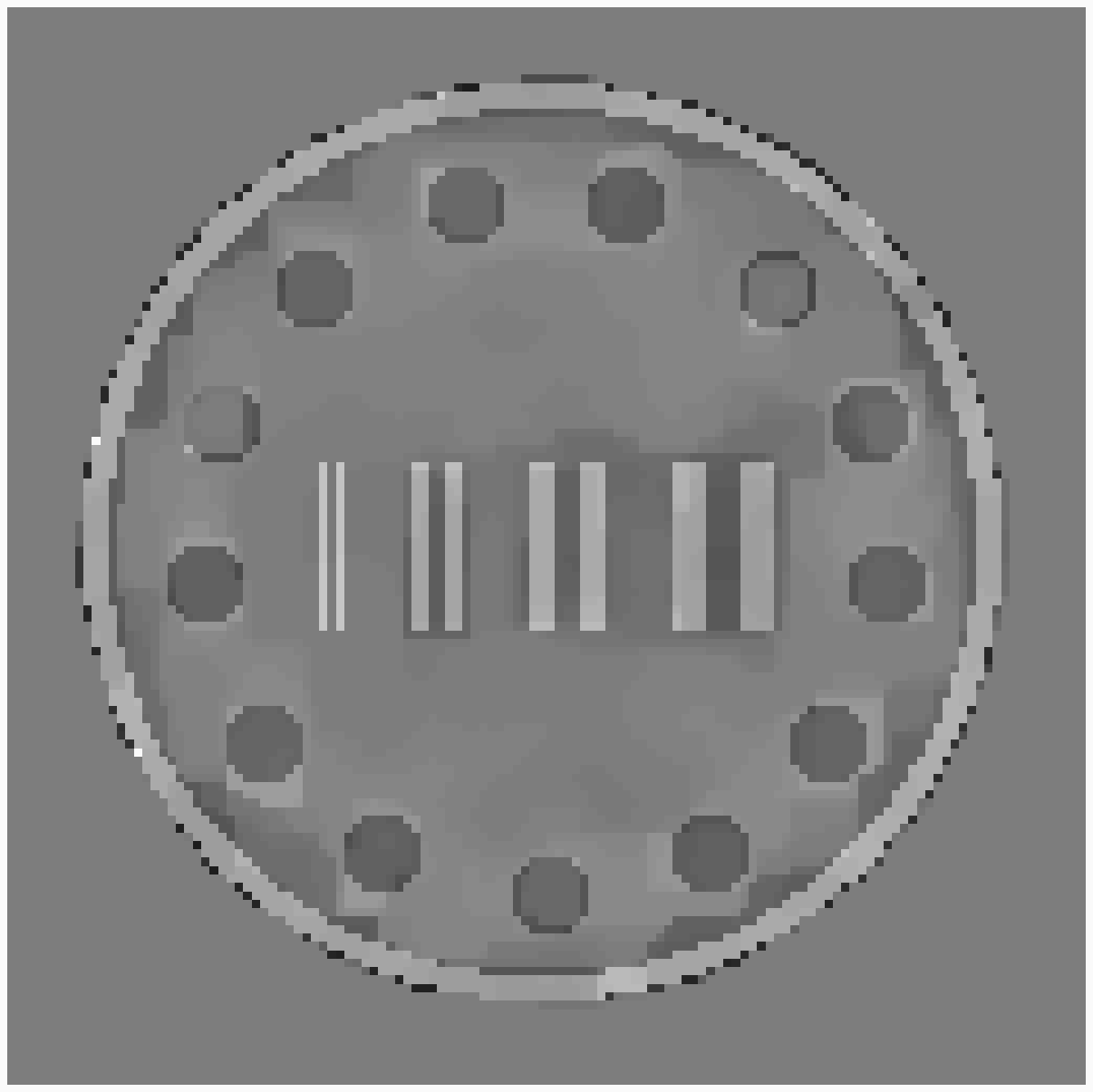} & \includegraphics[scale=0.2]{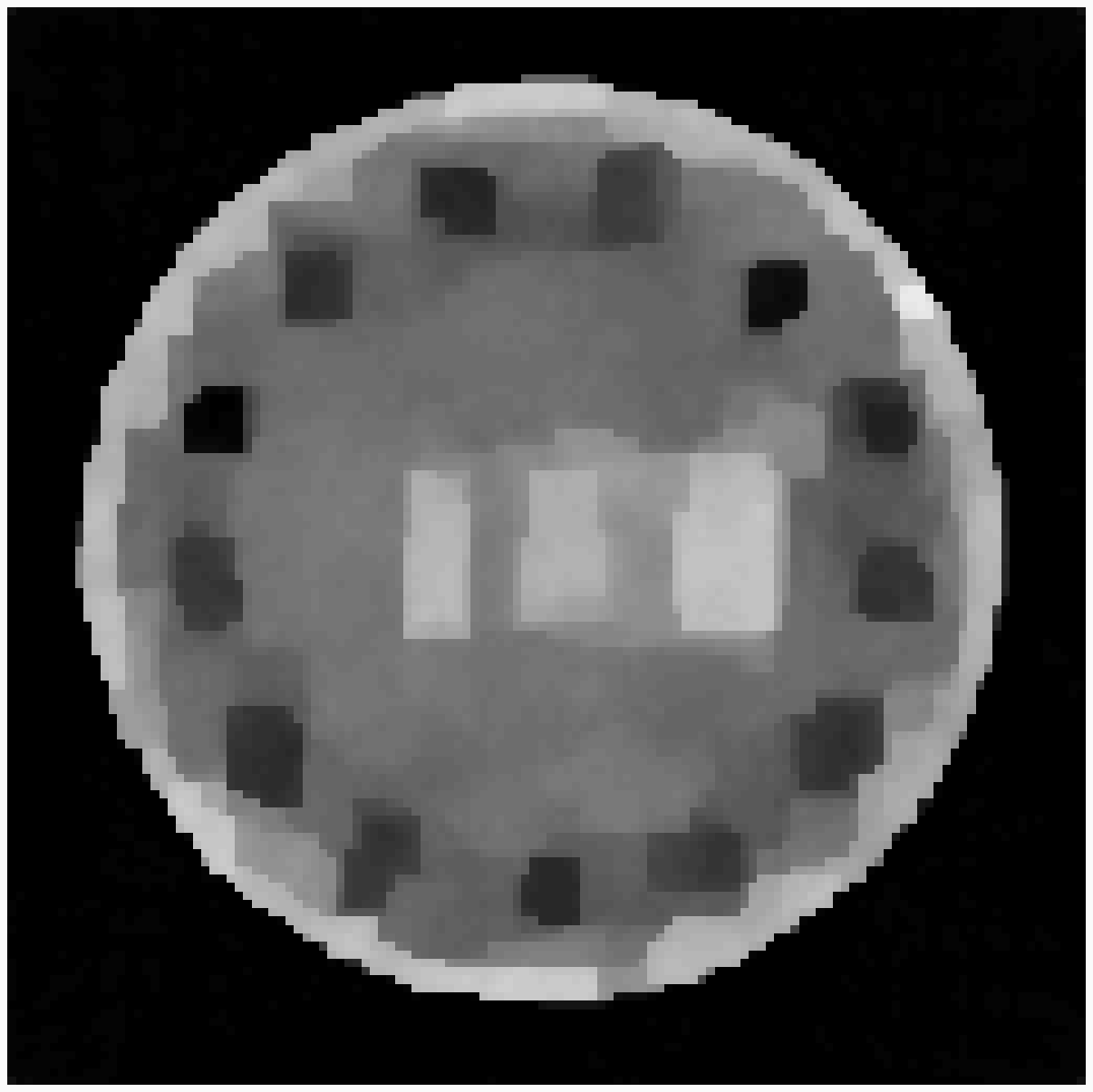} & \includegraphics[scale=0.2]{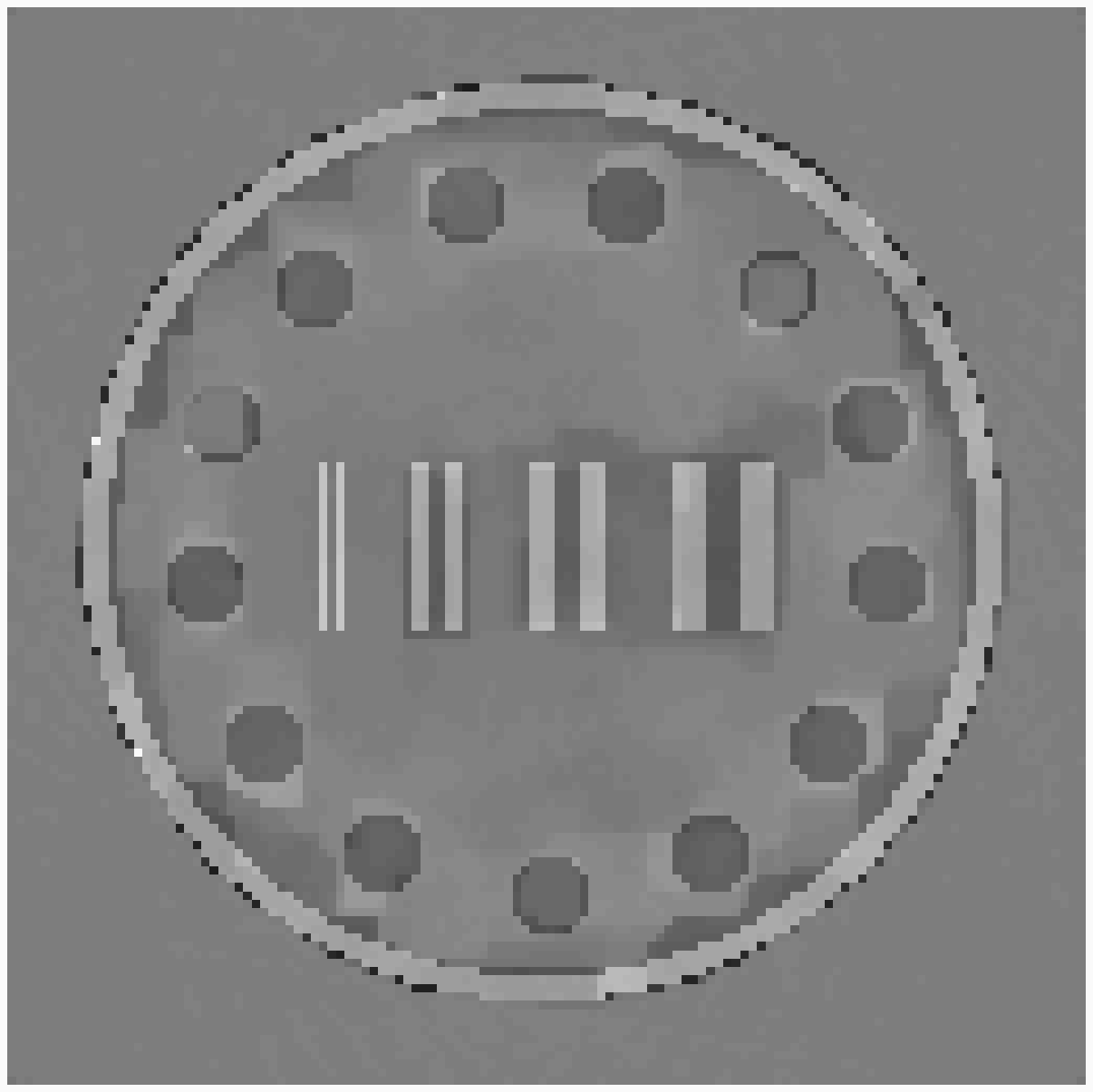} & \includegraphics[scale=0.2]{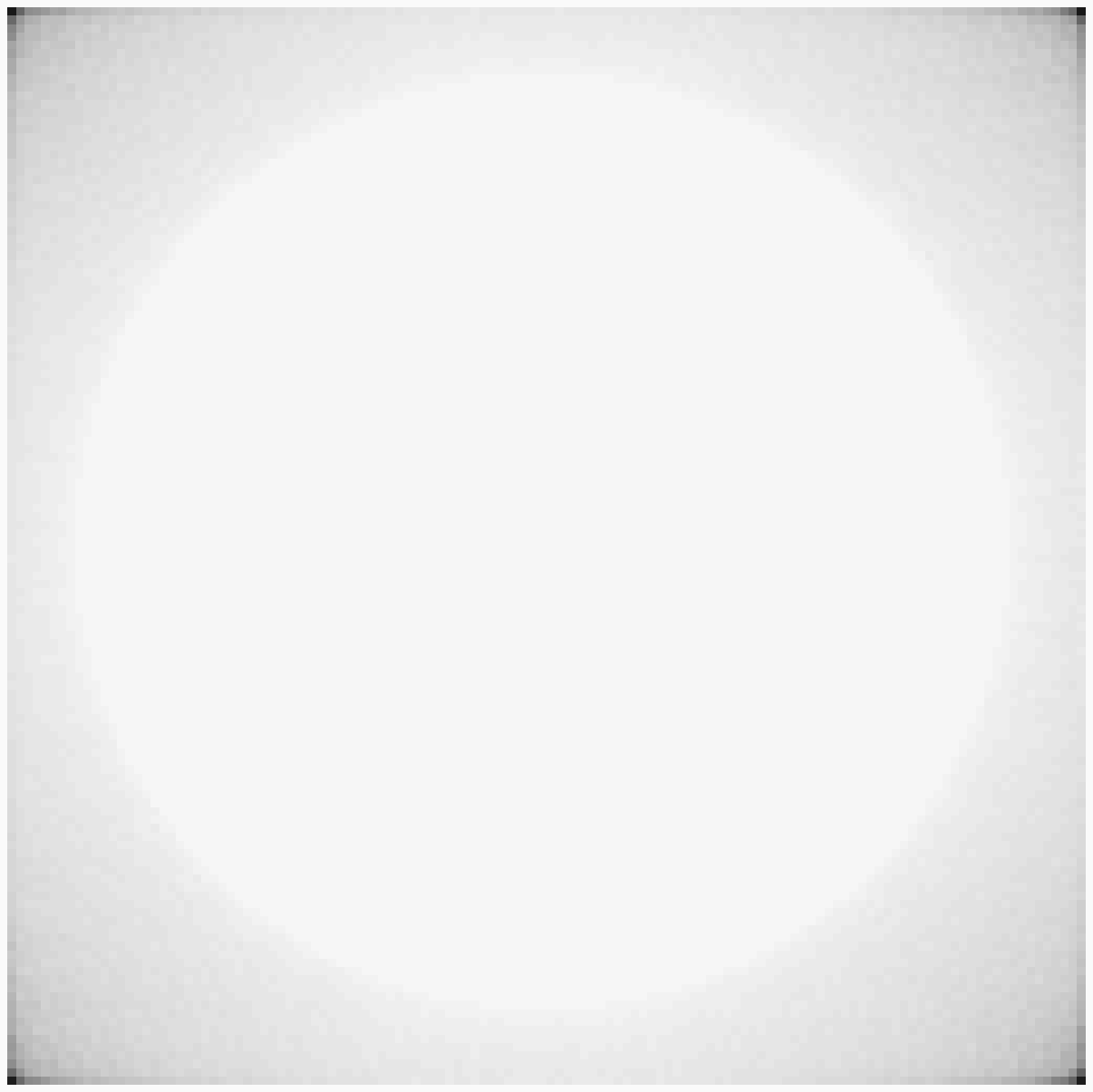}\\
&MAP & MAP error & EP mean & EP error & EP variance
\end{tabular}
\caption{MAP vs EP with anisotropic TV prior for the \texttt{PET} phantom, moderate count case.\label{fig:ME_248}}	
\end{figure}

\begin{figure}[htb!]
\centering
\begin{tabular}{rcccccc}
 & \includegraphics[scale=0.2]{bar_1_h} & \includegraphics[scale=0.2]{bar_1_1_h} & \includegraphics[scale=0.2]{bar_1_h} & \includegraphics[scale=0.2]{bar_1_1_h} & \includegraphics[scale=0.2]{bar_001_h}\\
\rotatebox{90}{\quad[0:2:179]}&\includegraphics[scale=0.2]{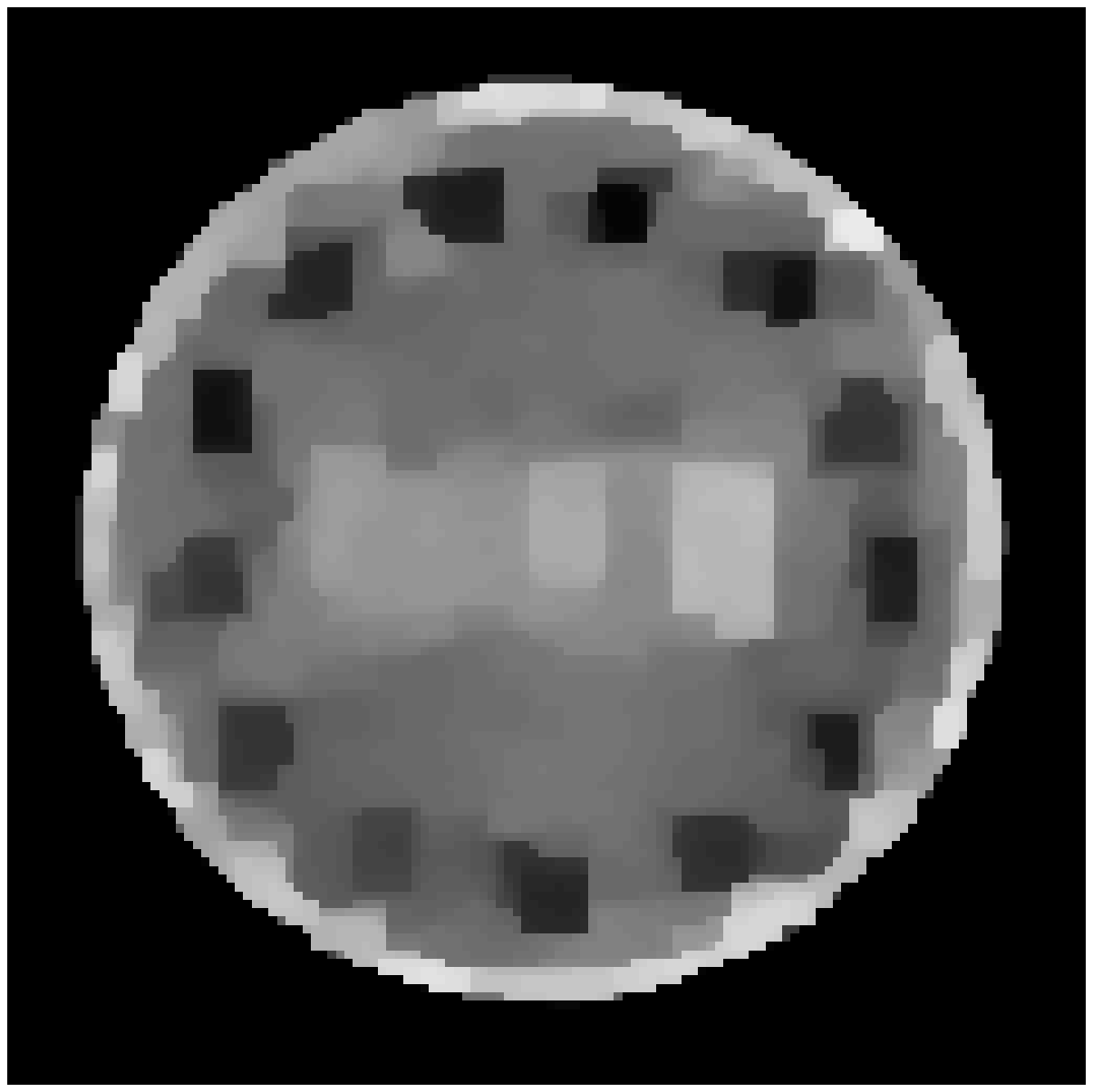} & \includegraphics[scale=0.2]{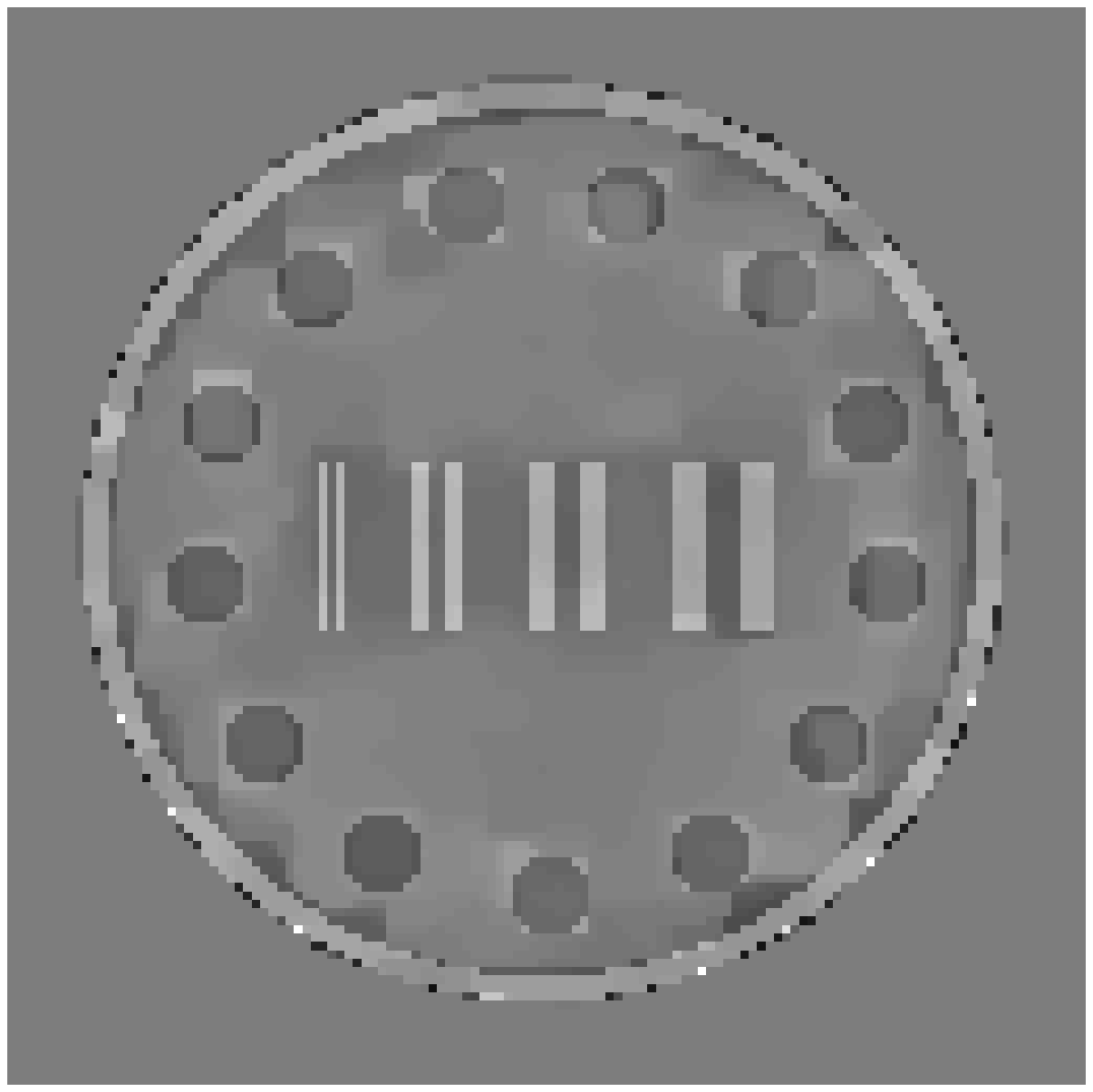} & \includegraphics[scale=0.2]{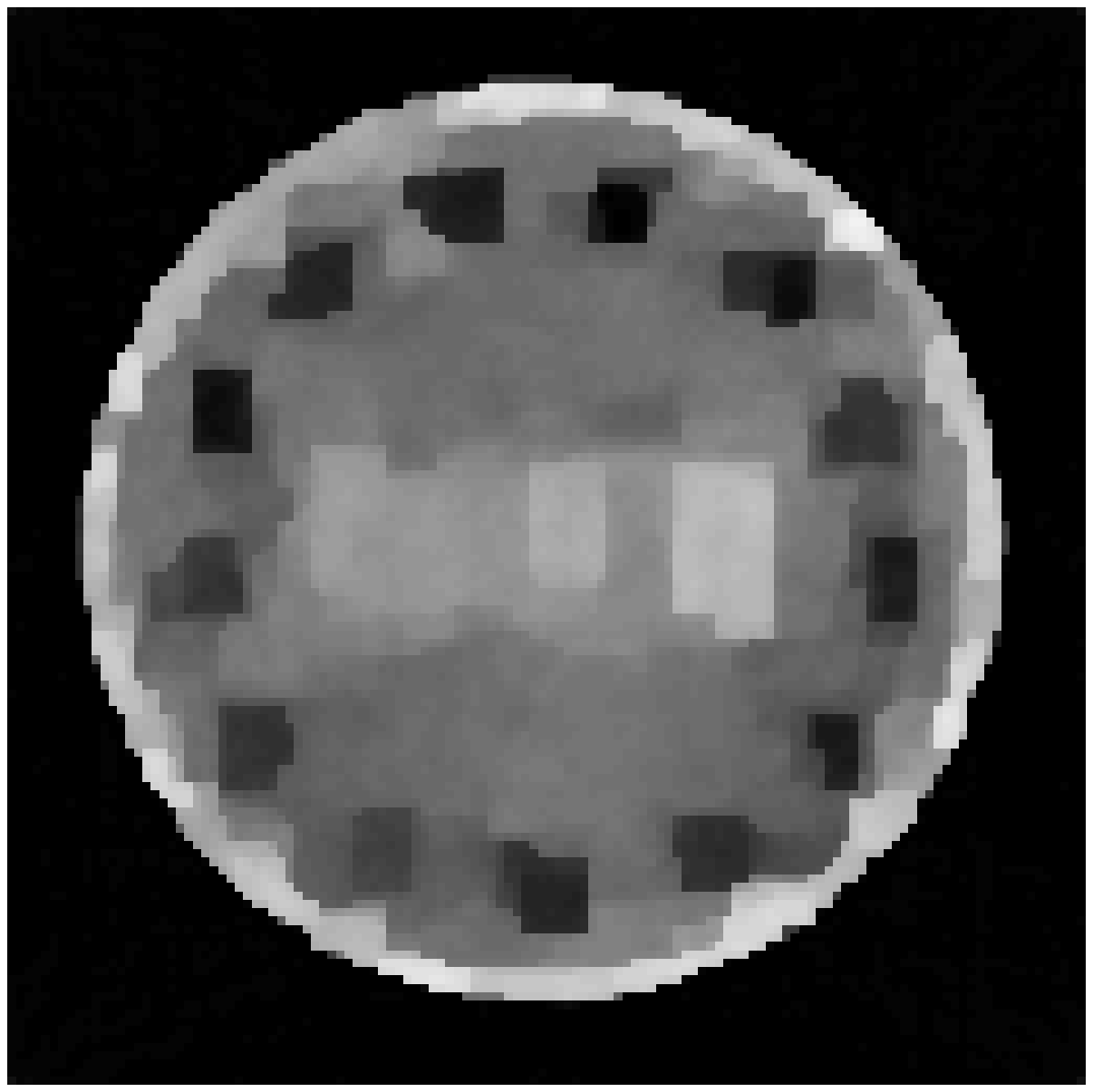} & \includegraphics[scale=0.2]{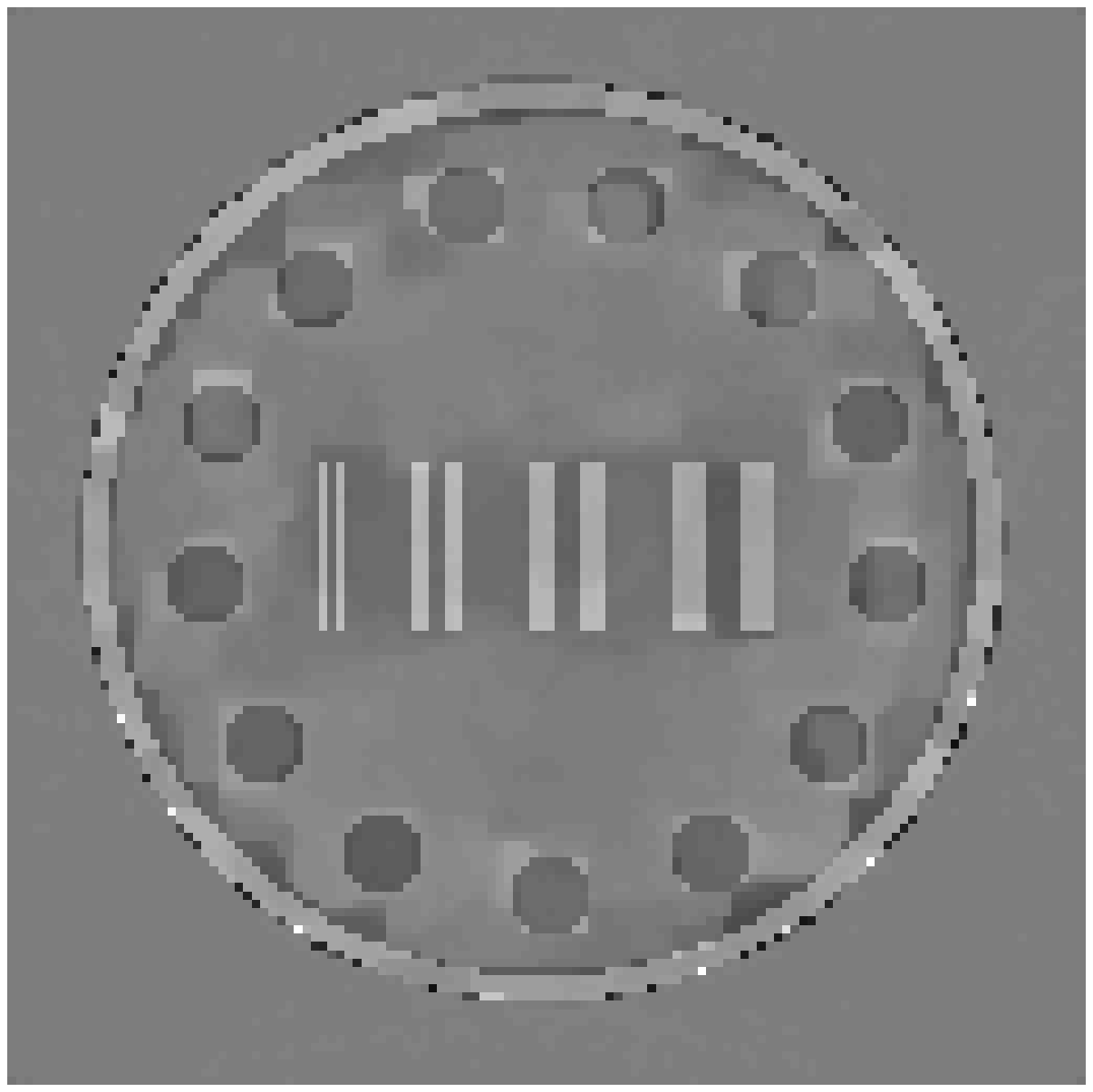} & \includegraphics[scale=0.2]{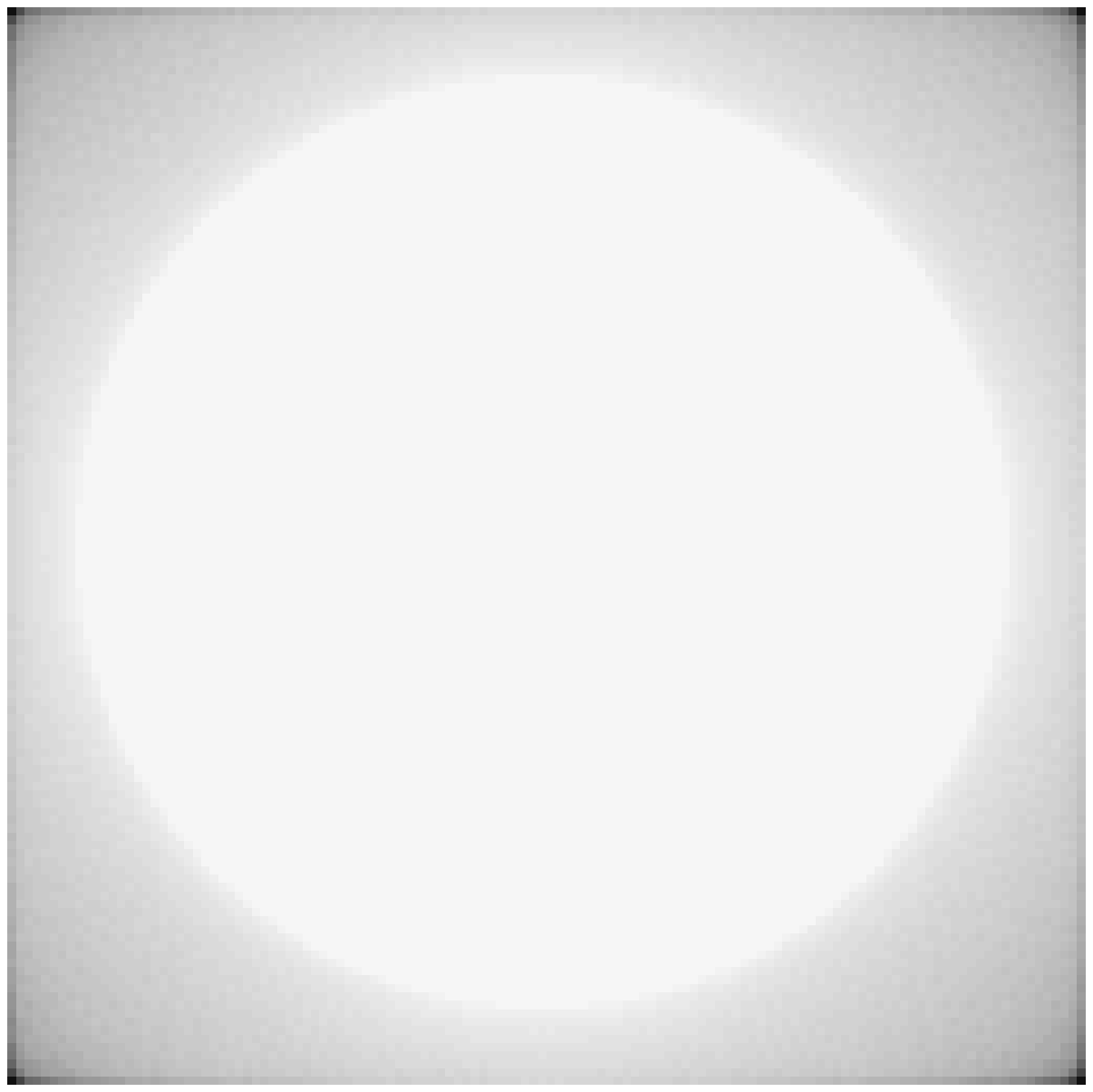}\\
\rotatebox{90}{\quad[0:4:179]}&\includegraphics[scale=0.2]{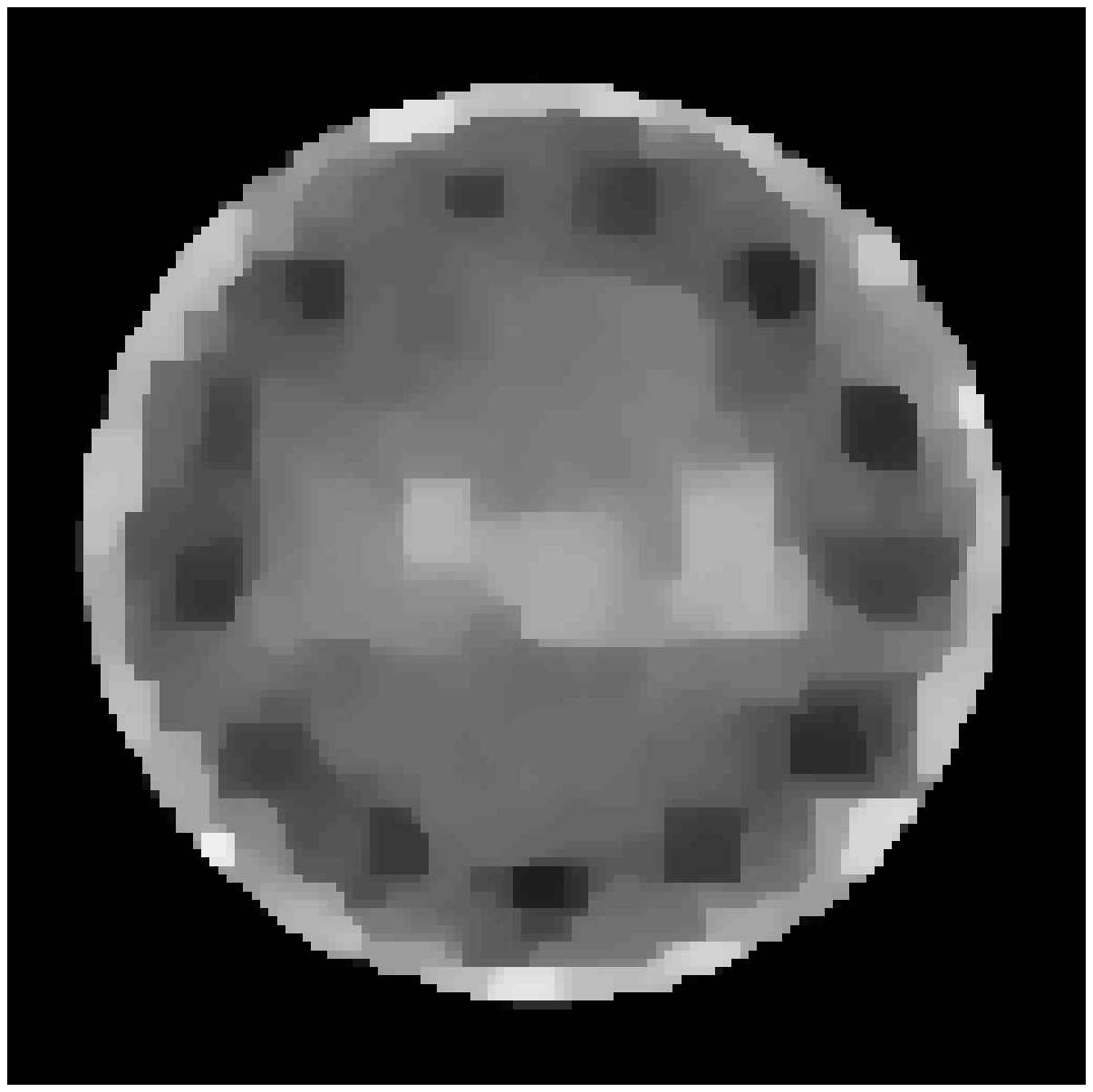} & \includegraphics[scale=0.2]{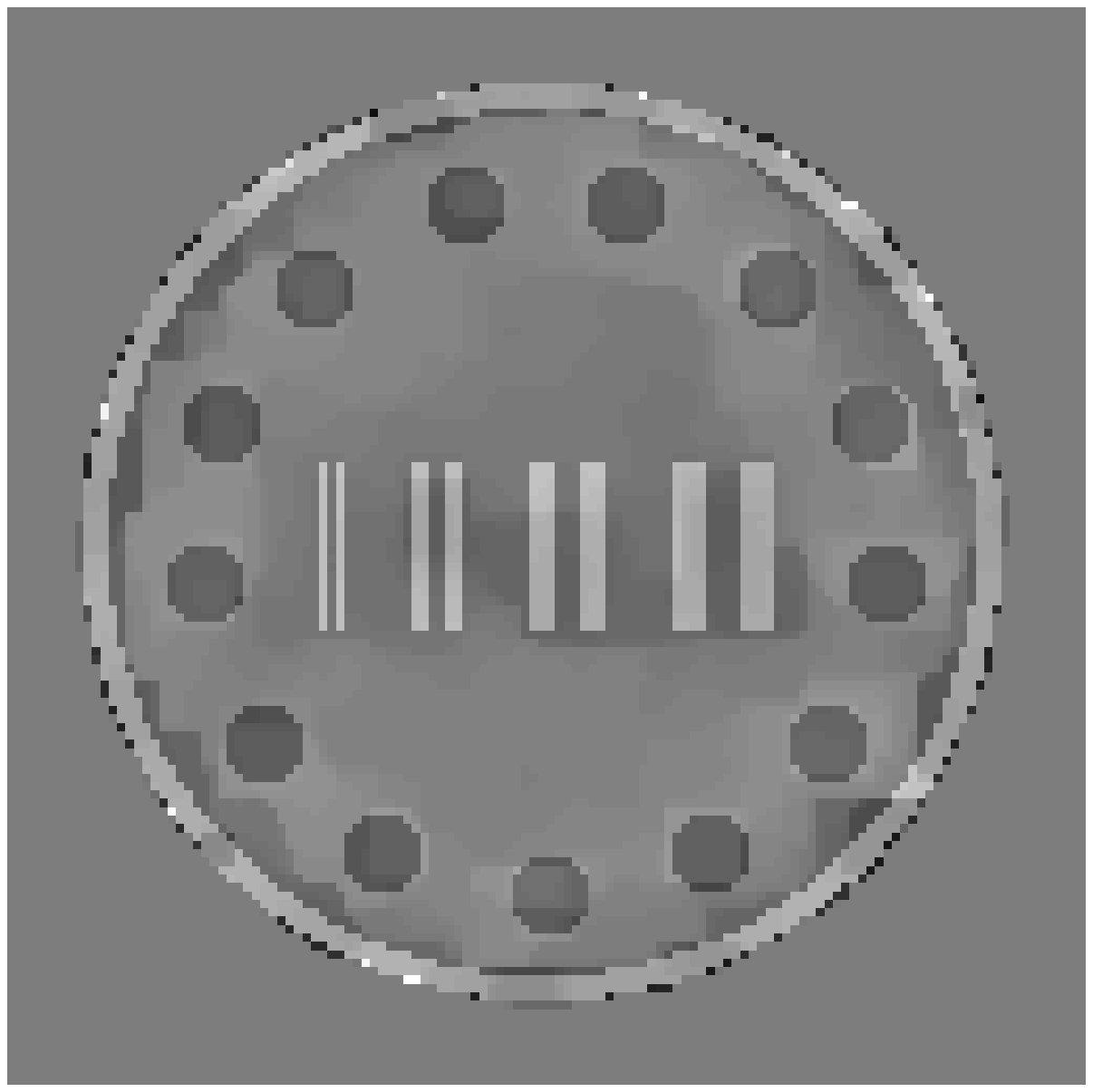} & \includegraphics[scale=0.2]{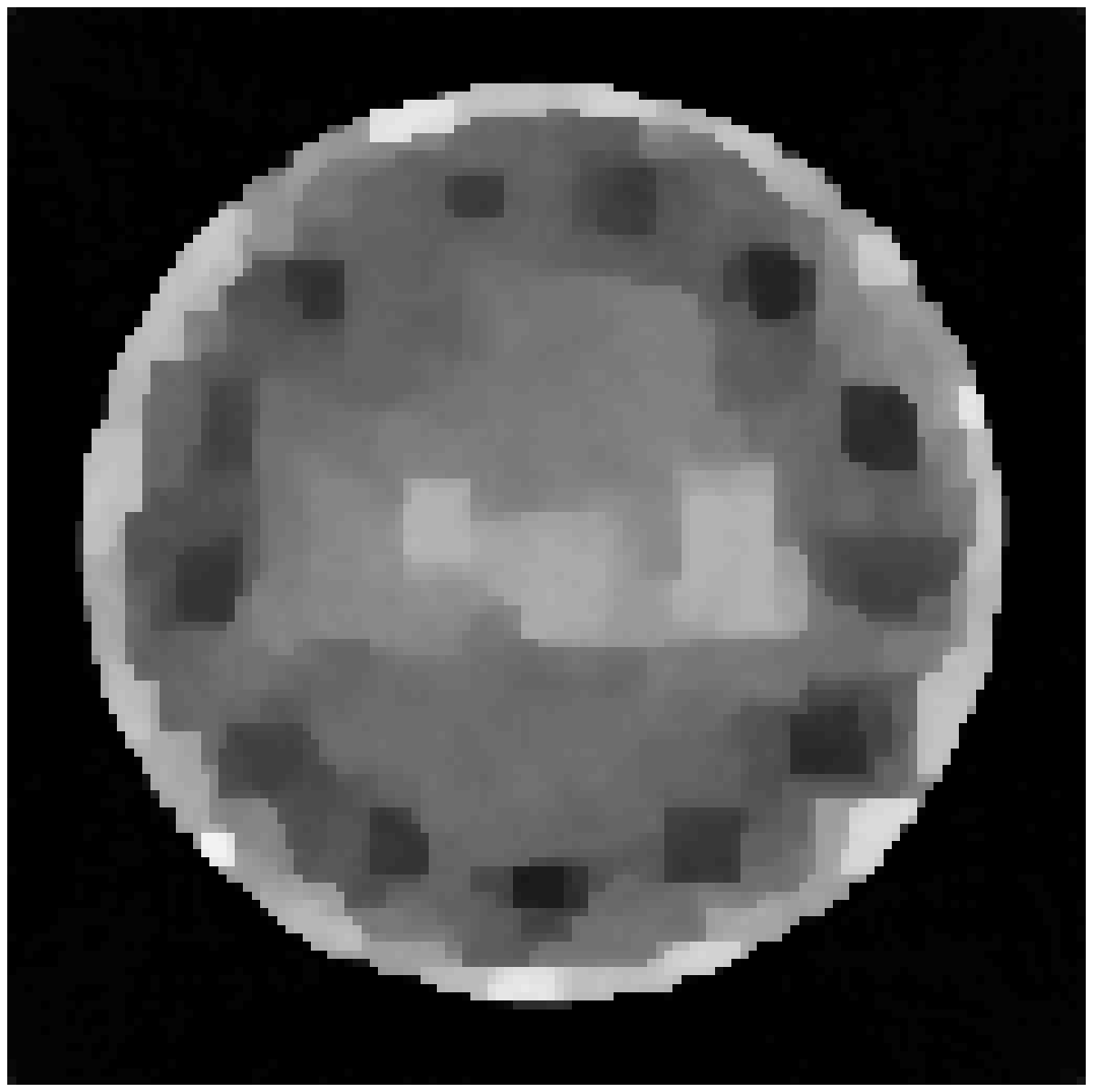} & \includegraphics[scale=0.2]{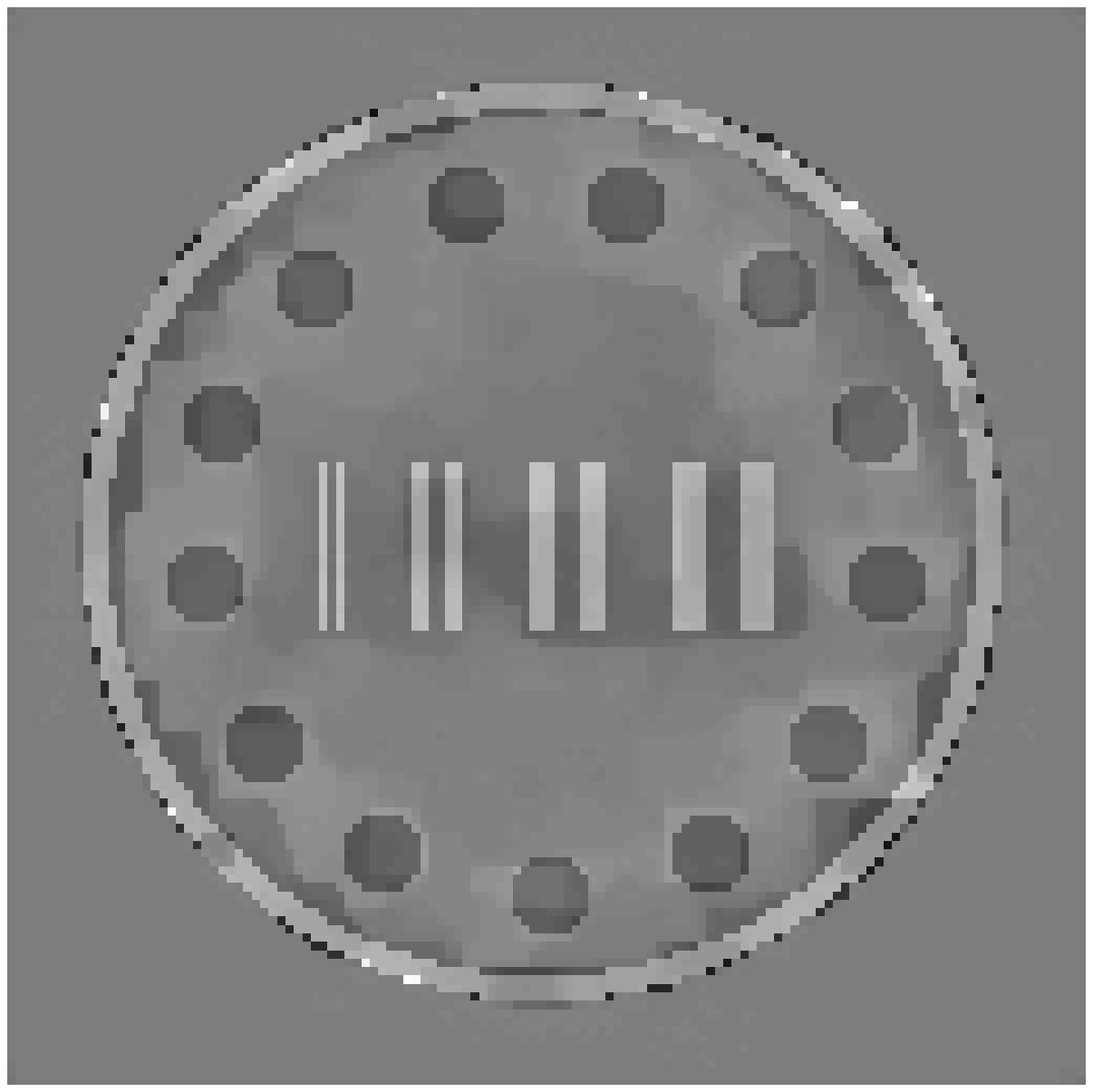} & \includegraphics[scale=0.2]{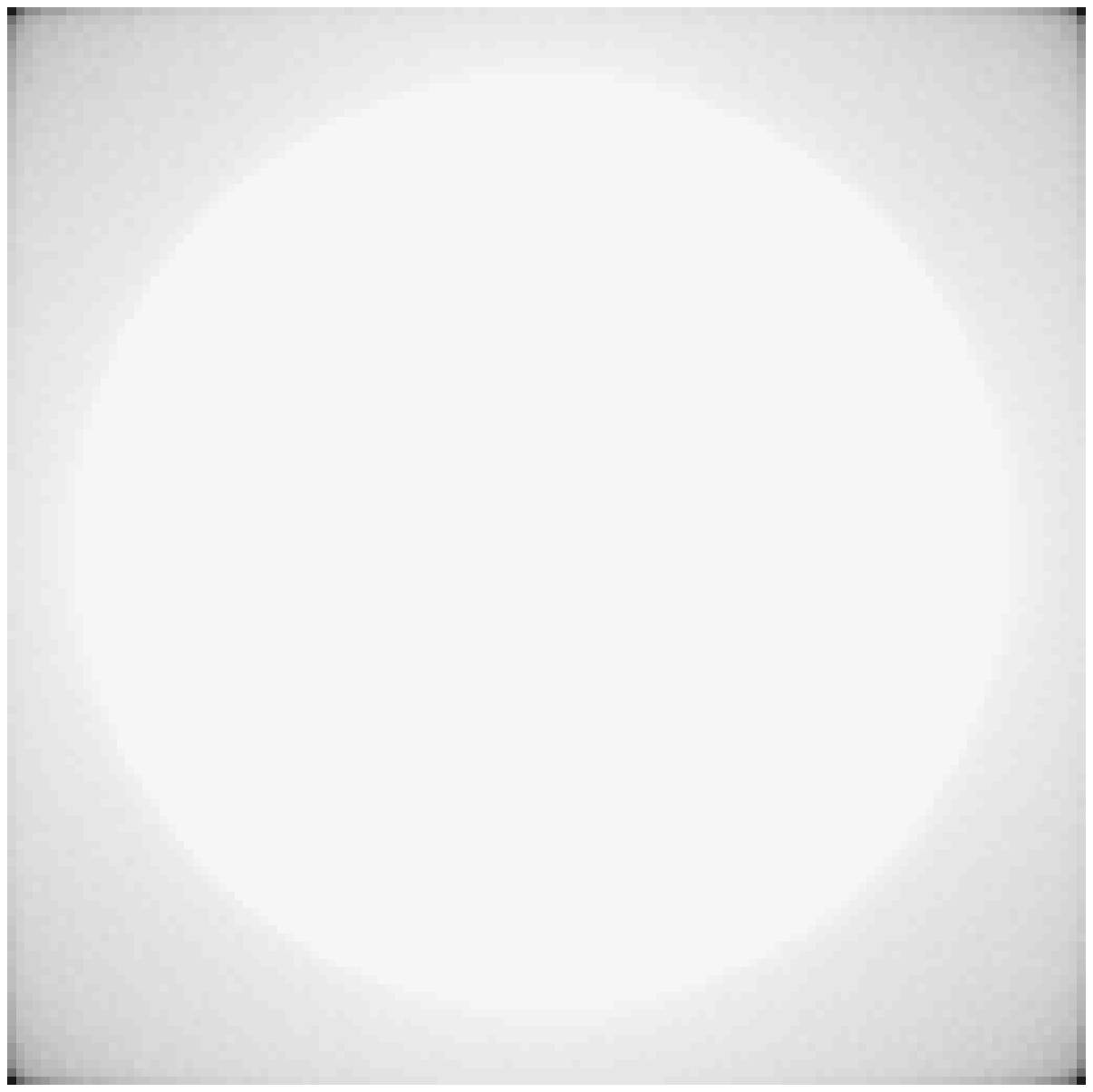}\\
\rotatebox{90}{\quad[0:8:179]}&\includegraphics[scale=0.2]{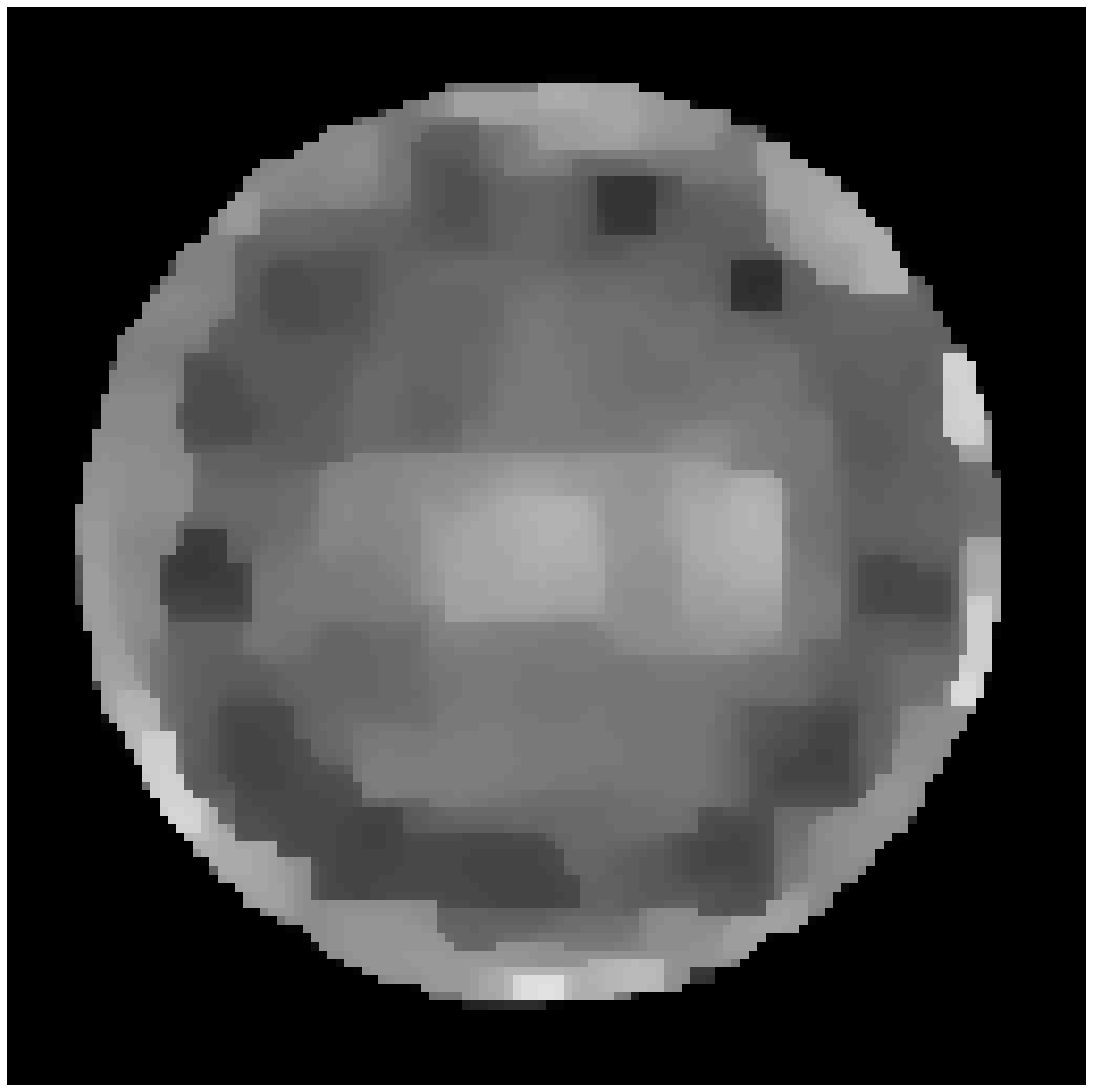} & \includegraphics[scale=0.2]{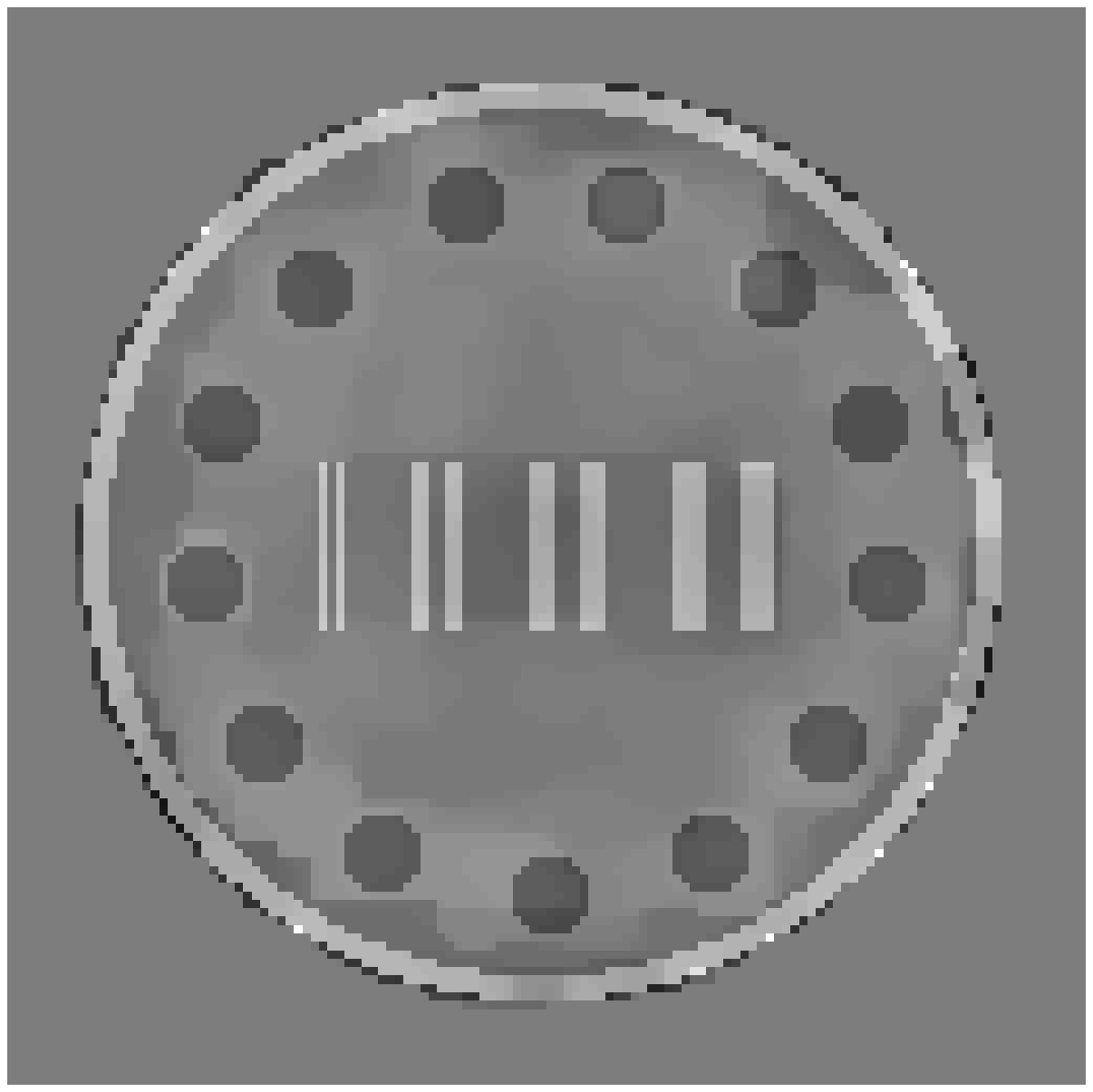} & \includegraphics[scale=0.2]{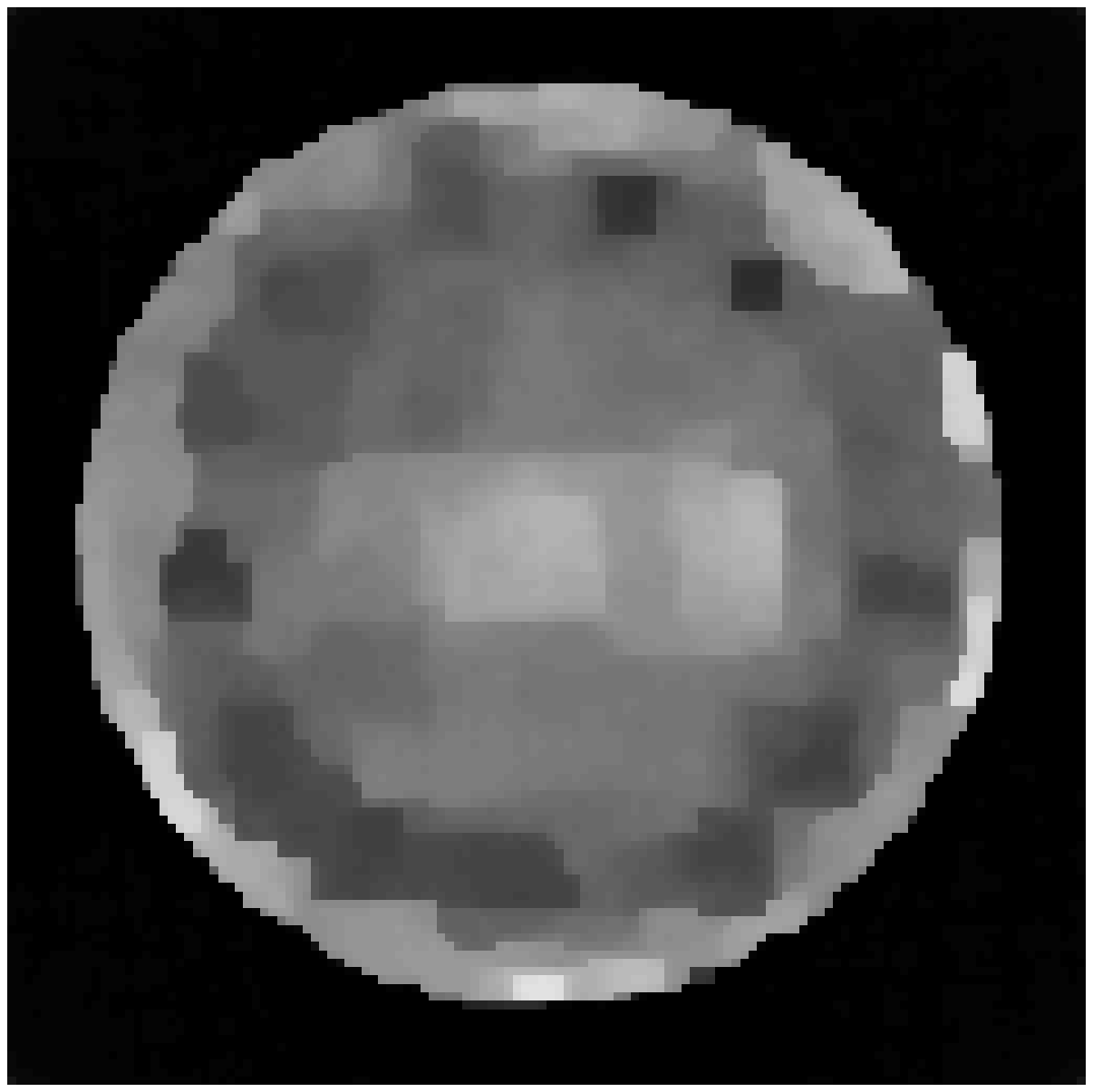} & \includegraphics[scale=0.2]{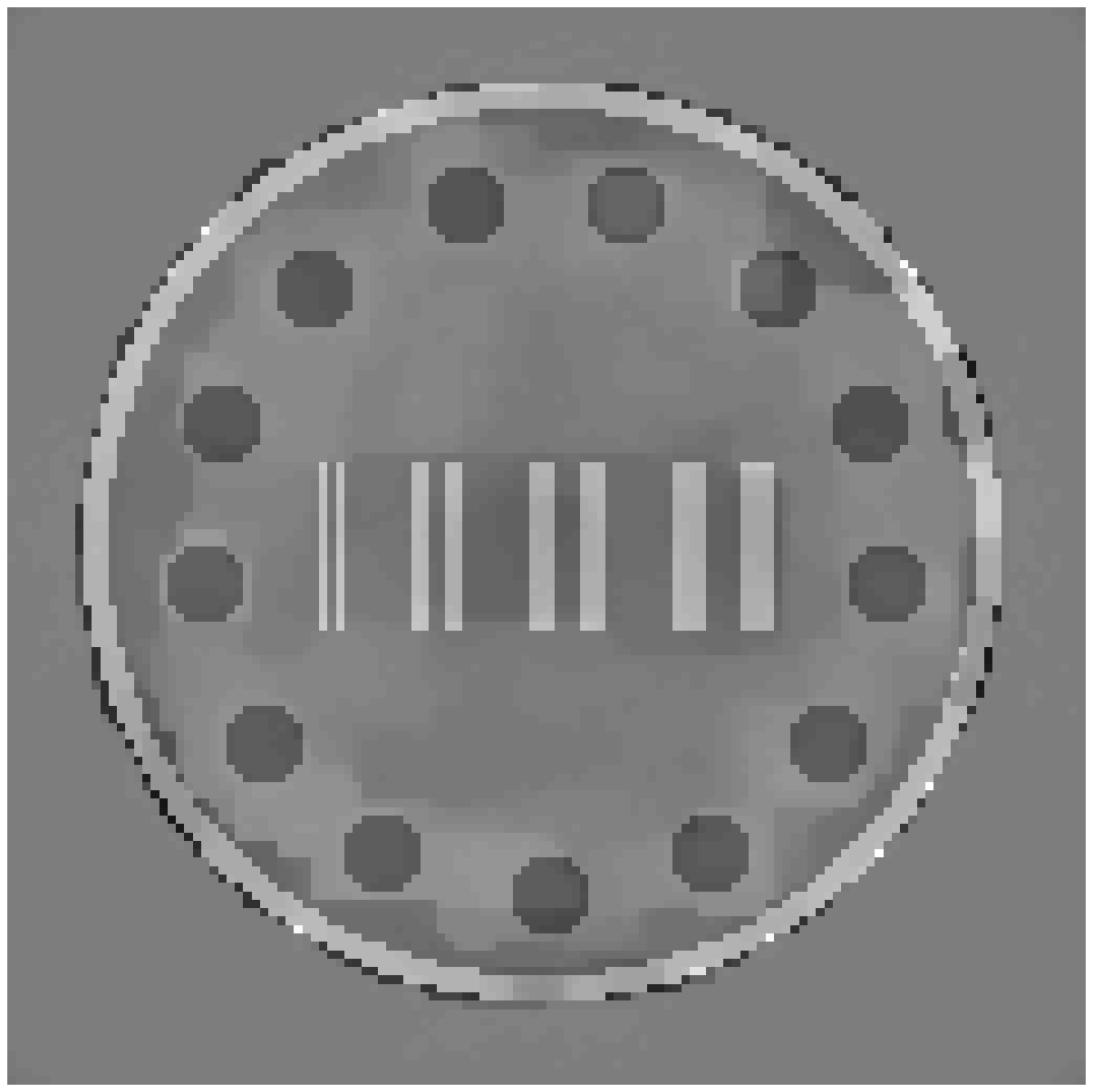} & \includegraphics[scale=0.2]{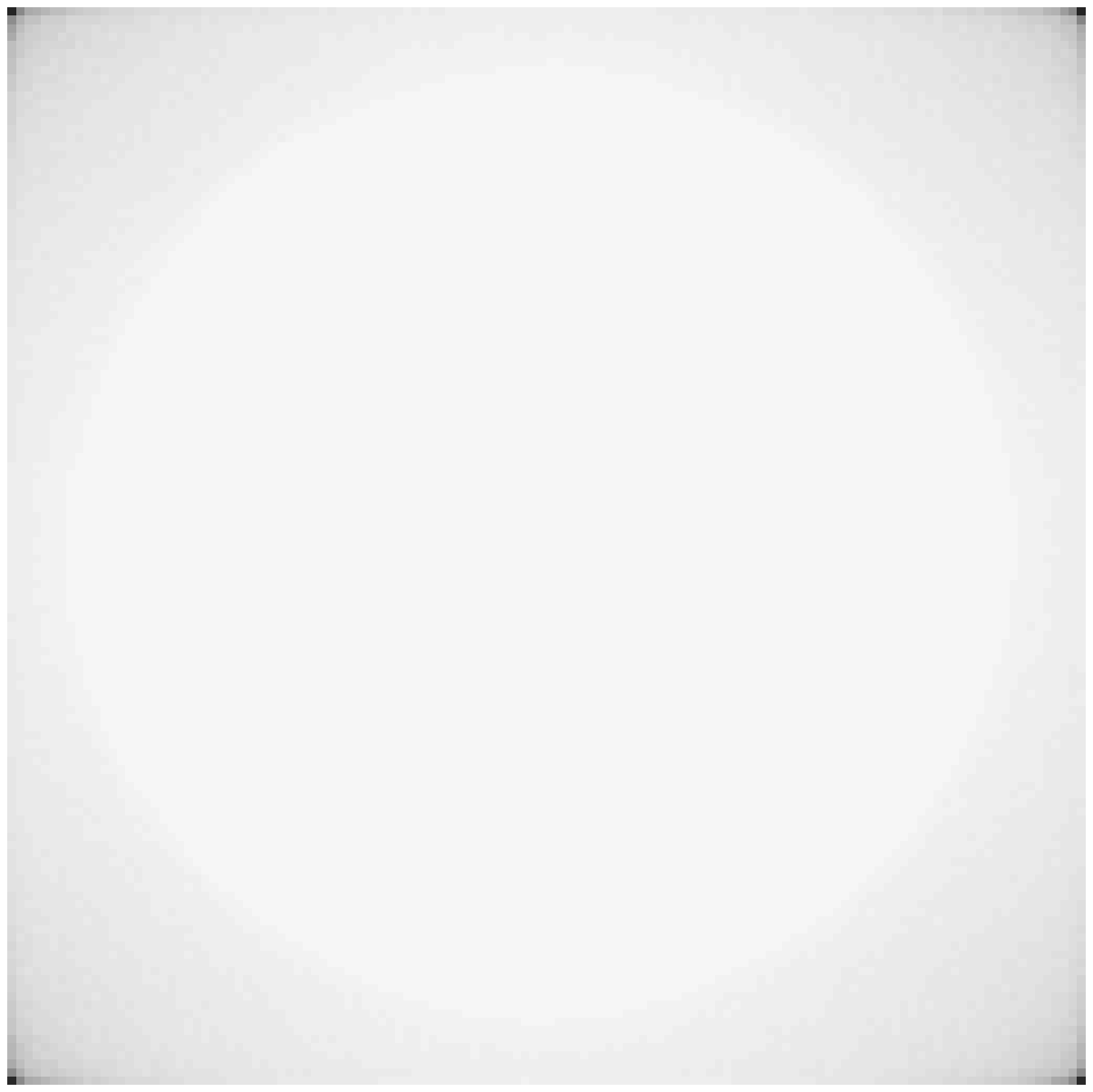}\\
&MAP & MAP error & EP mean & EP error & EP variance
\end{tabular}
\caption{MAP vs EP with anisotropic TV prior for the \texttt{PET} phantom, low count case.\label{fig:LC_ME_248}}	
\end{figure}

\begin{table}[hbt!]
\centering
\caption{The comparisons between EP mean and MAP for the \texttt{PET} phantom. The top
and bottom blocks refer to the moderate and low count cases, respectively. \label{tab:ME_quant}}
\begin{tabular}{|l|c|c|c|c|c|c|}
 \hline
 angle  & \multicolumn{2}{|c|}{[0:2:179]} & \multicolumn{2}{|c|}{[0:4:179]} & \multicolumn{2}{|c|}{[0:8:179]}\\
 \hline
$\alpha$ & \multicolumn{2}{|c|}{1.6e0} & \multicolumn{2}{|c|}{1.4e0} & \multicolumn{2}{|c|}{1.2e0}\\
 \hline
Method   & EP    & MAP   & EP    & MAP   & EP    & MAP\\
 \hline
$L^2$ error & 7.37  & 7.45  & 8.55  & 8.64  & 8.81  & 8.87\\
 \hline
SSIM     & 0.72  & 0.81  & 0.61  & 0.75  & 0.57  & 0.70\\
 \hline
PSNR     & 19.82 & 19.79 & 18.42 & 18.35 & 17.35 & 17.28\\
 \hline
CPU time (s) & 91263.00 & 110.05 & 53863.77 & 78.69 & 31537.05 & 28.20\\
 \hline
 \hline
 \hline
$\alpha$ & \multicolumn{2}{|c|}{1.2e0} & \multicolumn{2}{|c|}{9e-1} & \multicolumn{2}{|c|}{7.5e-1}\\
 \hline
Method   & EP    & MAP   & EP    & MAP   & EP    & MAP\\
 \hline
$L^2$ error & 8.96  & 9.04  & 9.30  & 9.35  & 10.13  & 10.17\\
 \hline
SSIM     & 0.55  & 0.72  & 0.49  & 0.67  & 0.43  & 0.62\\
 \hline
PSNR     & 17.66 & 17.61 & 16.93 & 16.89 & 15.84 & 15.81\\
 \hline
CPU time (s) & 82542.76 & 52.97 & 47263.64 & 32.43 & 29737.91 & 18.01\\
 \hline
\end{tabular}
\end{table}

\begin{figure}[htb!]
\centering
\begin{tabular}{cc}
\includegraphics[width=0.3\textwidth]{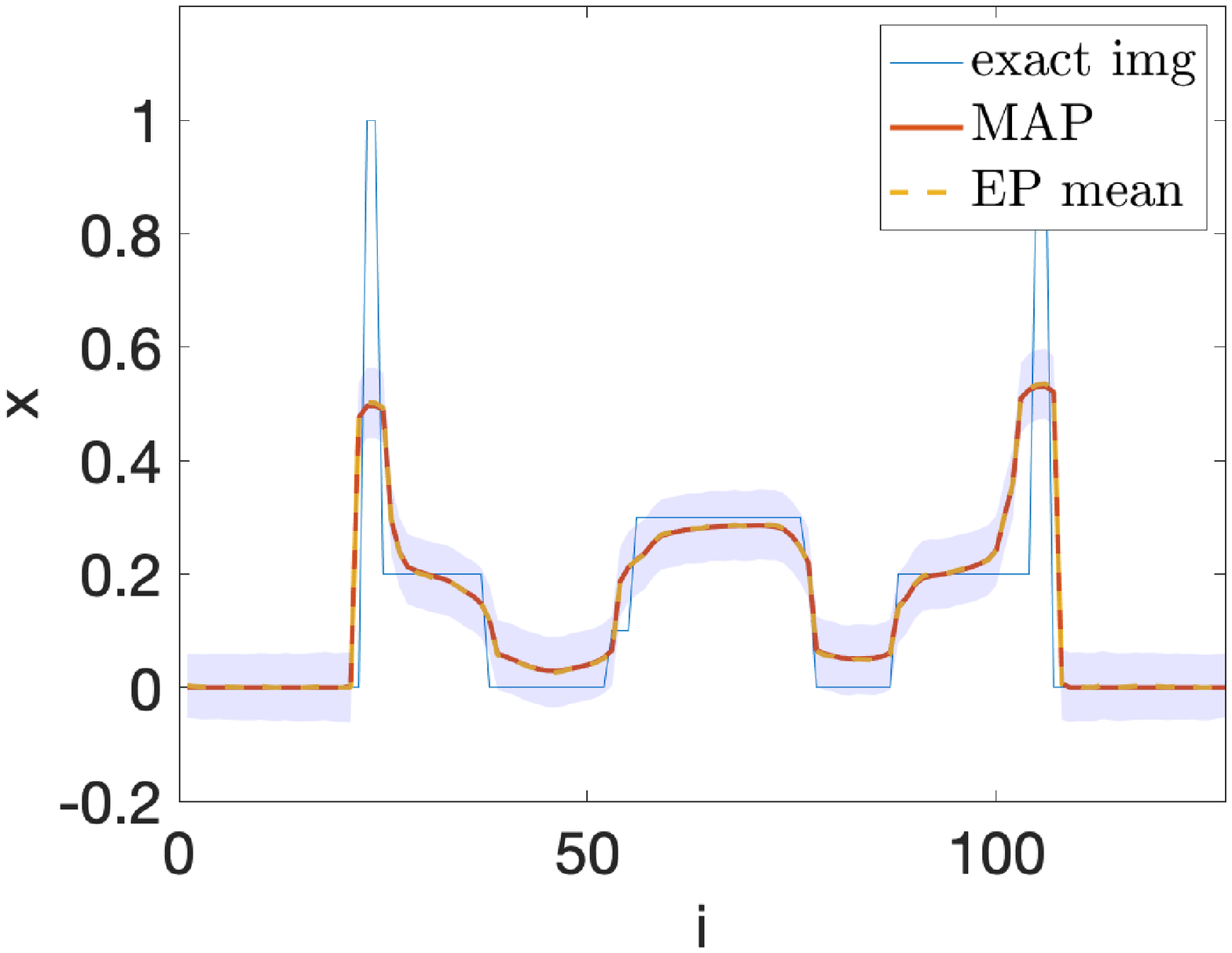} & \includegraphics[width=0.3\textwidth]{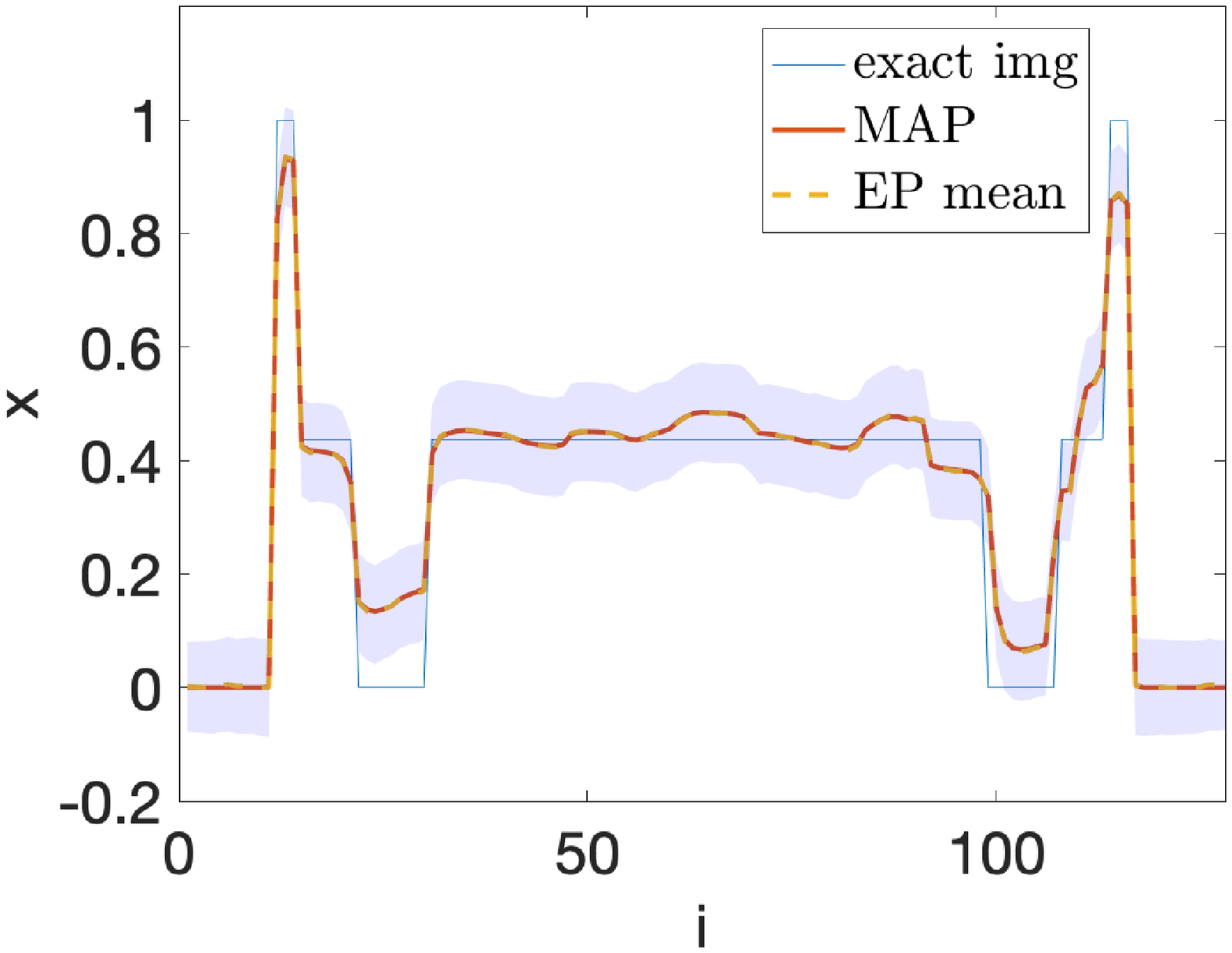}\\
(a) Shepp-Logan & (b) PET\\
\end{tabular}
\caption{The $50$-th cross-sections of the two phantoms and $0.95$-HPD regions, moderate count case.\label{fig:rev_128_slice}}	
\end{figure}

To further illustrate the approximation, we plot in Fig. \ref{fig:rev_128_slice} the cross-sections and 95\%
highest posterior density (HDP) region, which is estimated from the EP covariance. The EP mean is close to
MAP, and thus also suffers slightly from a reduced magnitude, as is typical of the total variation penalty in
variational regularization \cite{ChanShen:2005}. This also concurs with the error metrics in Tables \ref{tab:SL_quant} and
\ref{tab:ME_quant}. The thrust of EP is that it can also provide uncertainty estimates via covariance, which is
unavailable from MAP. In sharp contrast, the popular Laplace approximation (see Appendix \ref{app:Laplace})
can fail to yield a reasonable approximation for nonsmooth priors such as anisotropic total variation,
whereas MCMC tends to be prohibitively expensive for large images, though being asympotically exact; see
Appendix \ref{sec:1d} for further numerical results. So overall, EP represents a computationally feasible
approach to deliver uncertainty estimates for these benchmark images with Poisson data.

\subsection{Convergence of the EP algorithm}

Next, we present an experimental evaluation of the convergence of the EP algorithm, which is a long outstanding
theoretical issue, on the following experimental setup: \texttt{Shepp-Logan} phantom and Radon matrix $A\in \mathbb{R}^{4255\times
16384}$ (i.e., $185$ projections per angle and $[0:8:179]$, moderate count case). We denote the mean and covariance
after $k$ outer iterations (i.e., sweeps through all the sites) by $\mu^k$ and $C^k$, respectively, and the converged
iterate tuple by $(\mu^*,C^*)$. The EP mean $\mu^k$ converges rapidly, and visually it reaches convergence after
five iterations since thereafter the cross-sections graphically overlap with each other; see Fig.
\ref{fig:EP_conv_x_img}. Thus, in the numerical experiments, we have fixed the number of outer iterations to four, and
the complexity of the reconstruction algorithm is of order $O(mn^2)$. Fig. \ref{fig:EP_conv} shows the errors
of the iterate tuple $(\mu^k,C^k)$ with respect to $(\mu^*,C^*)$, where the errors
\begin{equation*}
\delta \mu=\mu^k-\mu^* \quad\mbox{ and }\quad \delta C=C^k-C^*
\end{equation*}
are measured by the $L^2$-norm and spectral norm, respectively. This phenomenon is also observed for all other
experiments, although not presented. Hence, both mean and covariance converge rapidly, showing the steady and fast
convergence of EP.

\begin{figure}[htb!]
\centering
\begin{tabular}{cccccc}
\includegraphics[scale=0.2]{bar_1_h} &  &  &  & \\
\includegraphics[scale=0.2]{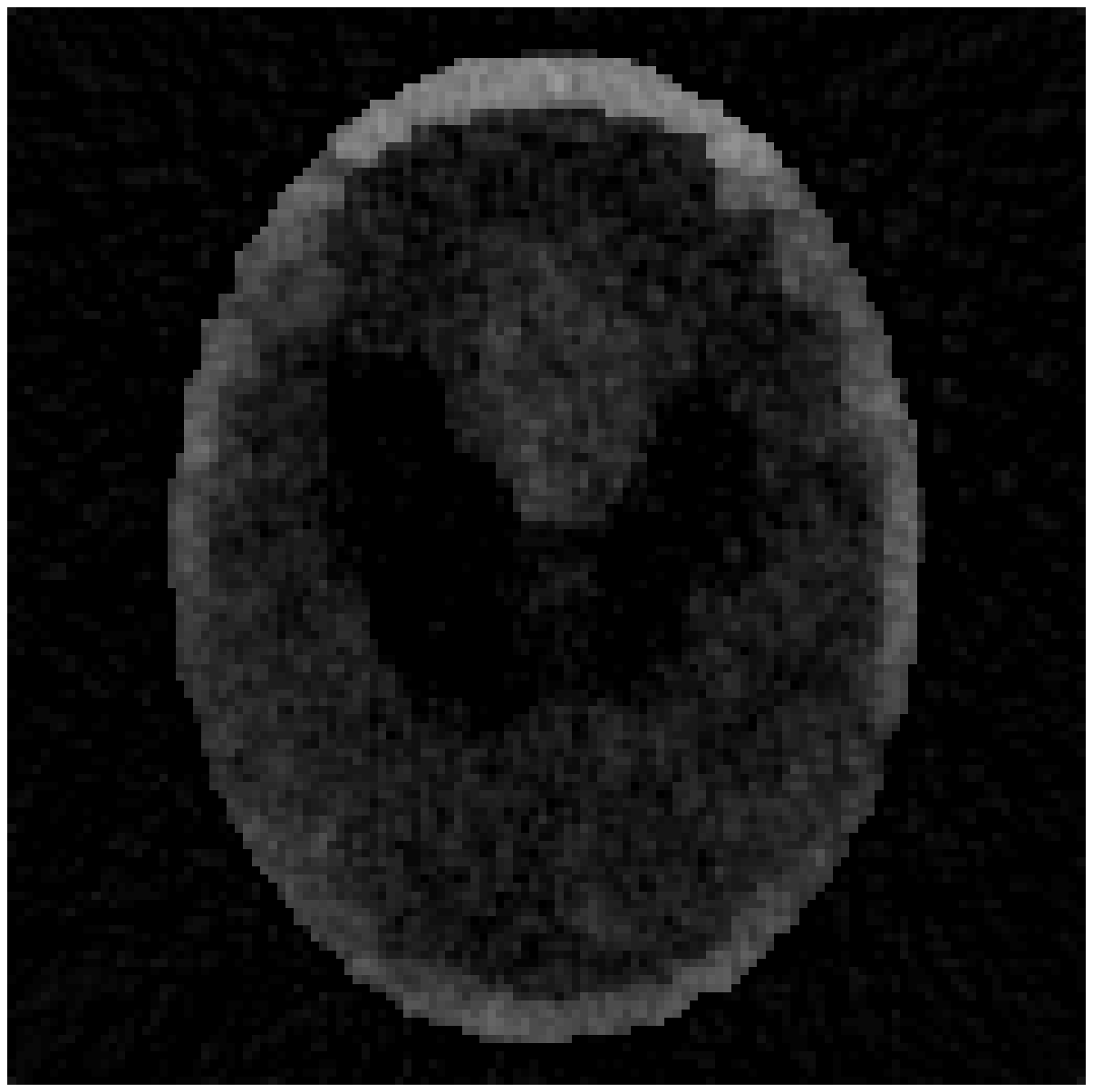} & \includegraphics[scale=0.2]{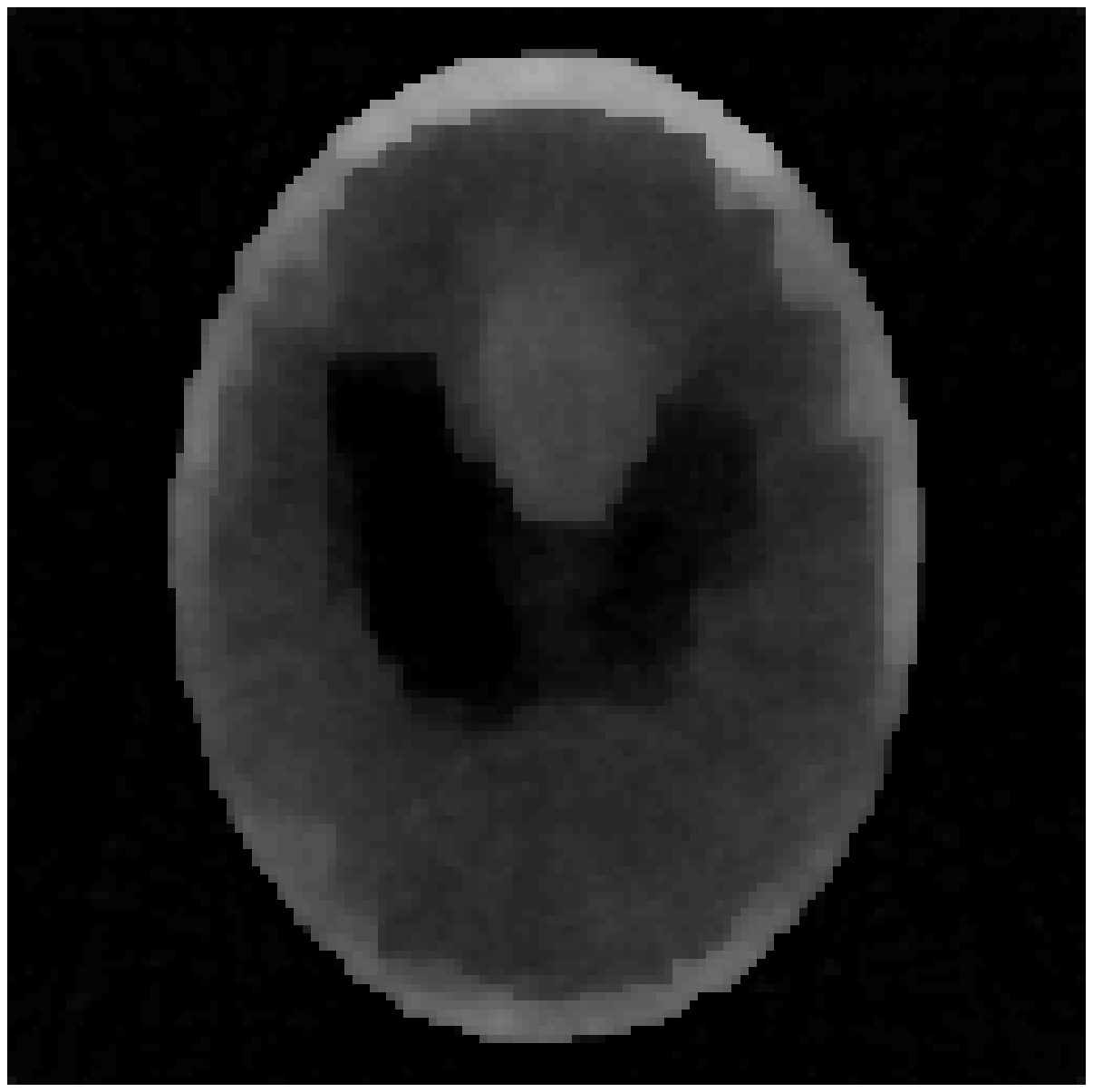} & \includegraphics[scale=0.2]{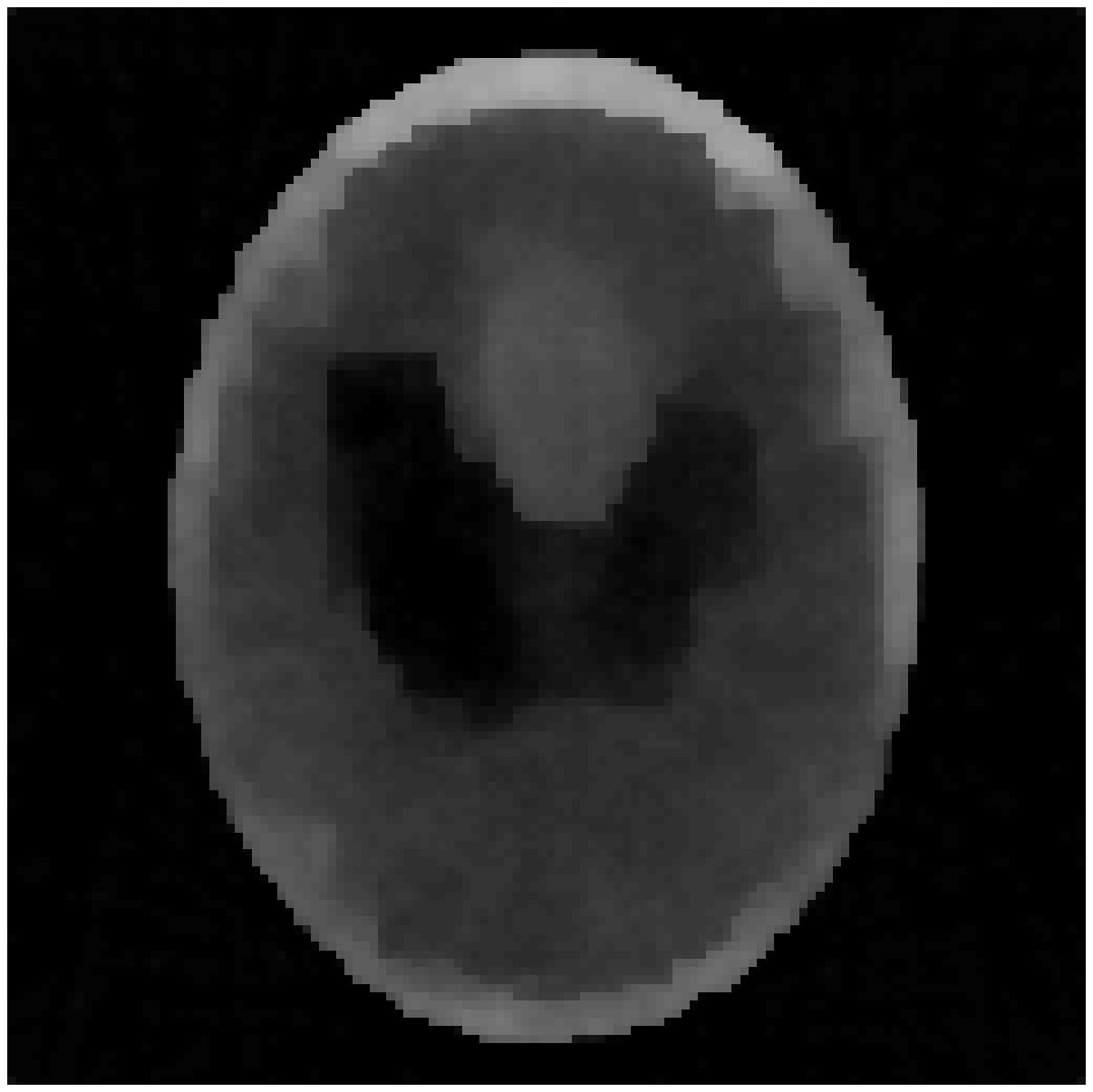} & \includegraphics[scale=0.2]{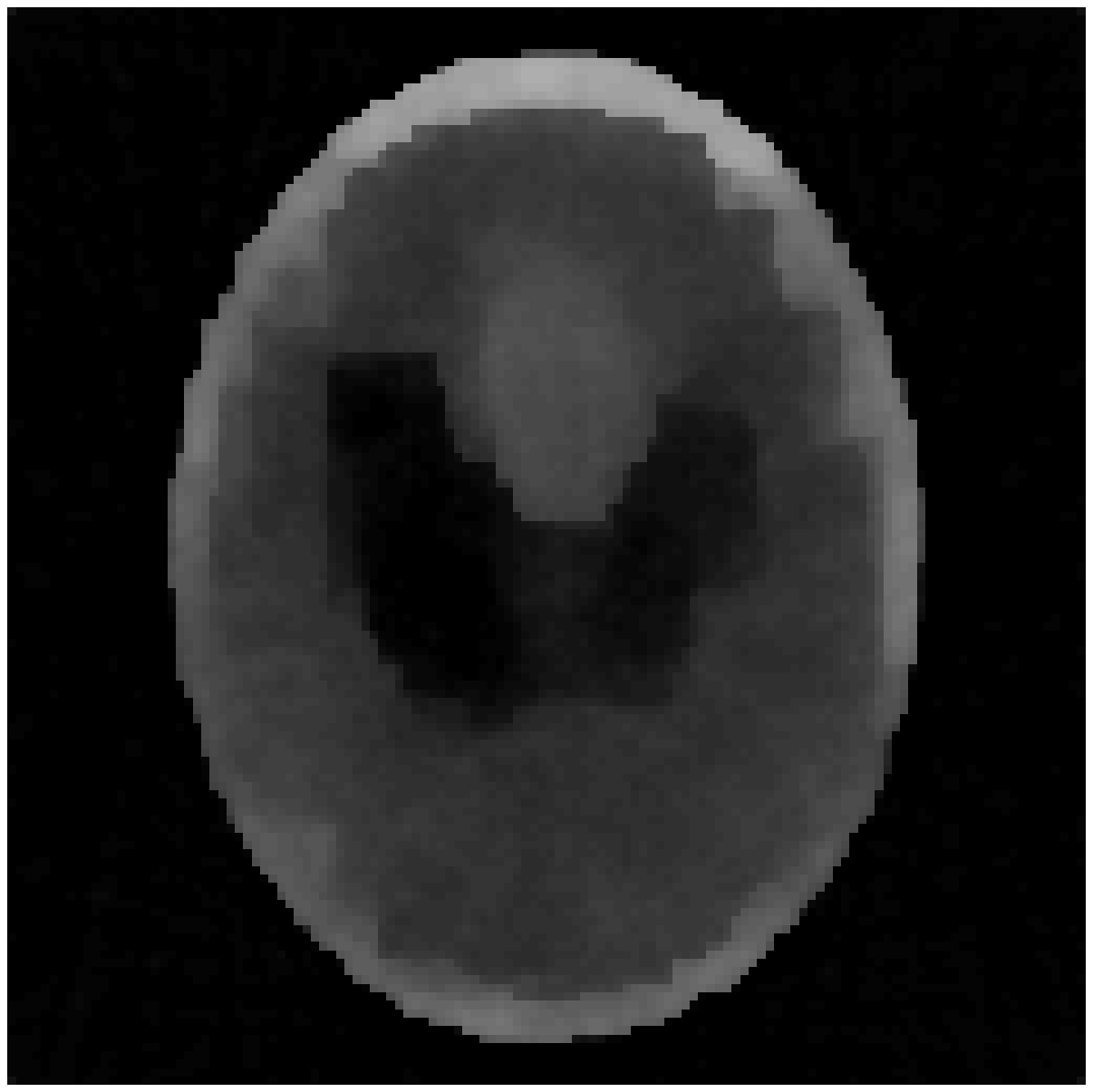} & \includegraphics[scale=0.2]{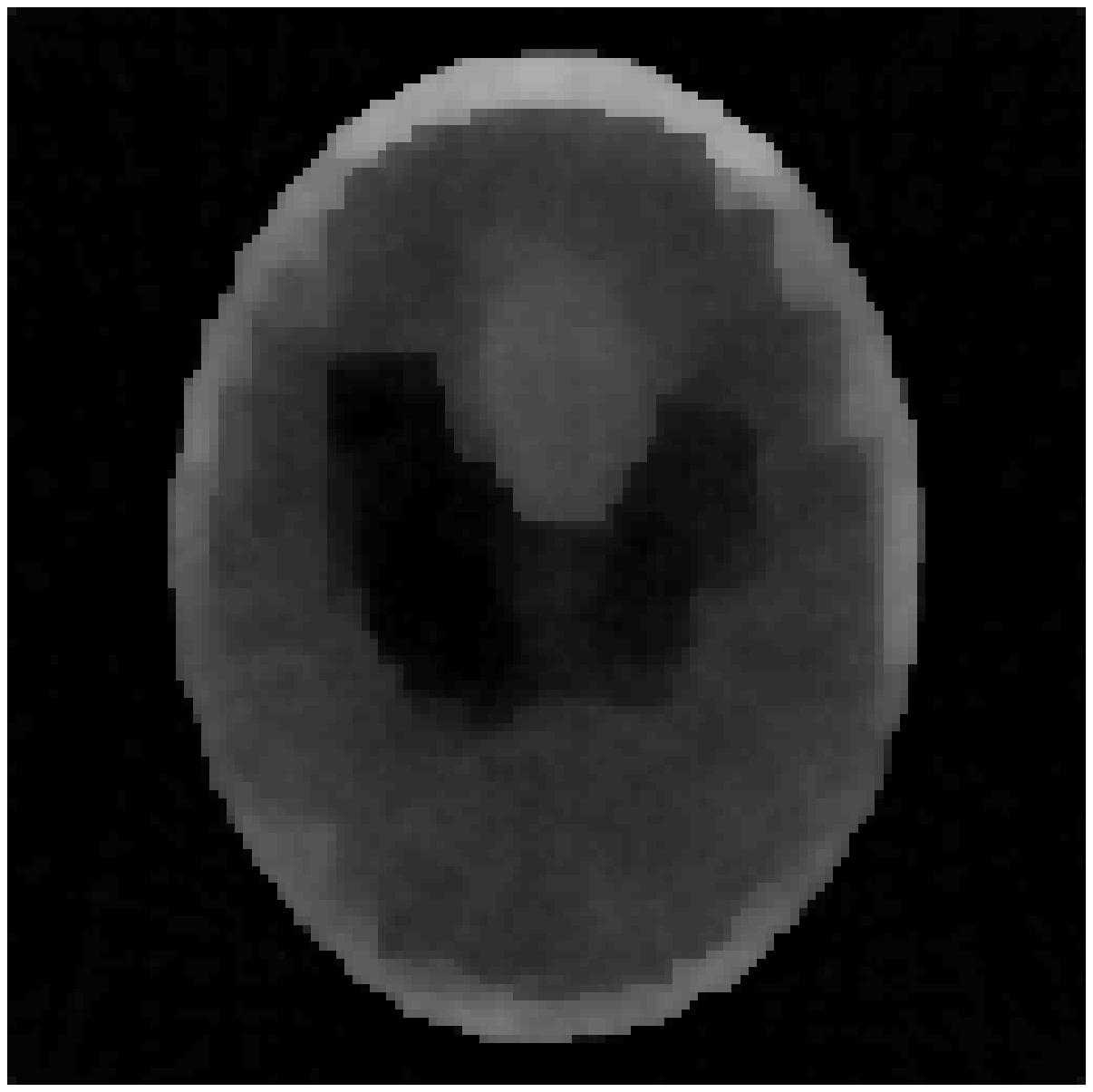}\\
$\mu^1$ & $\mu^2$ & $\mu^3$ & $\mu^4$ & $\mu^5$\\
\includegraphics[scale=0.2]{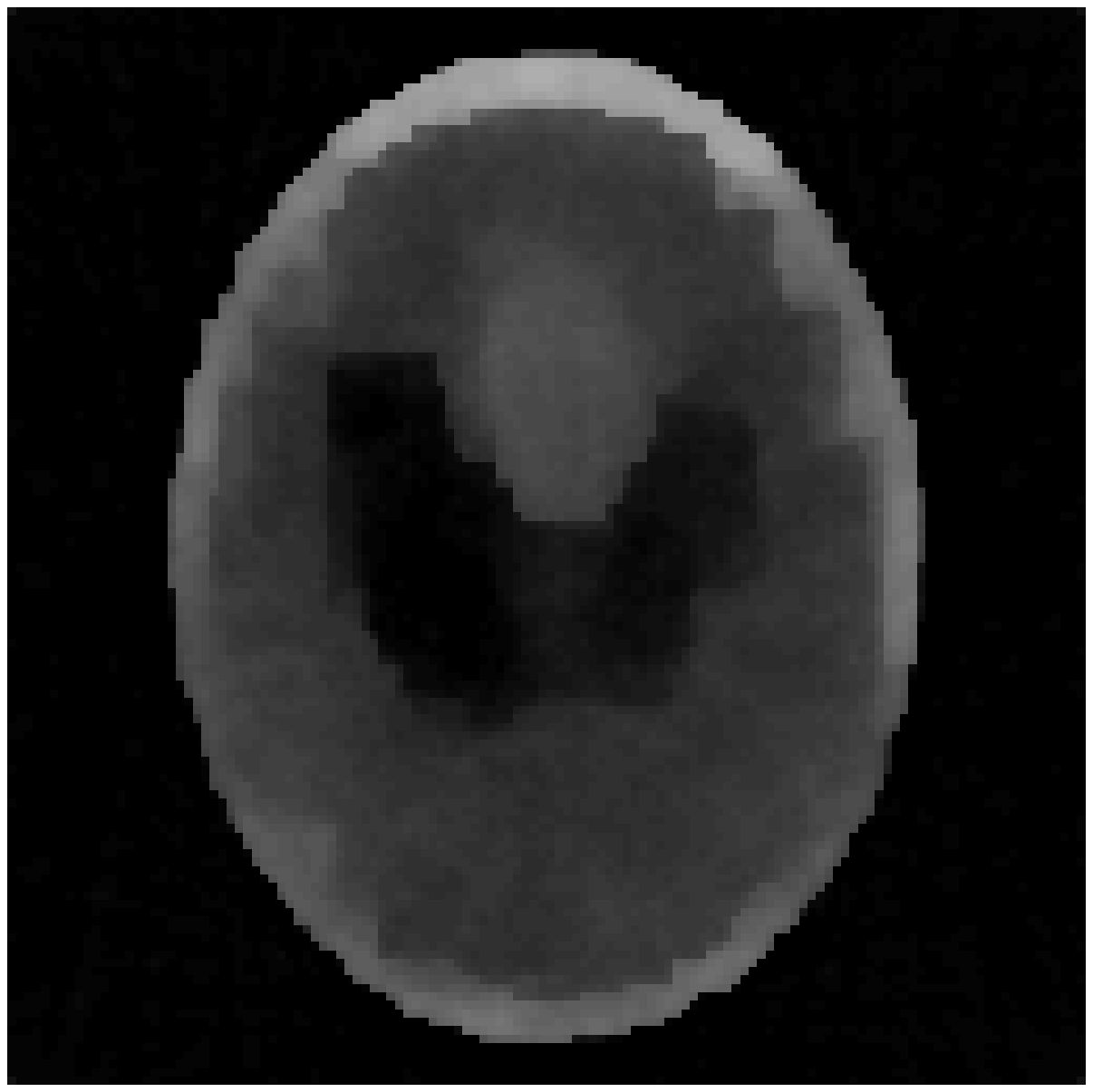} & \includegraphics[scale=0.2]{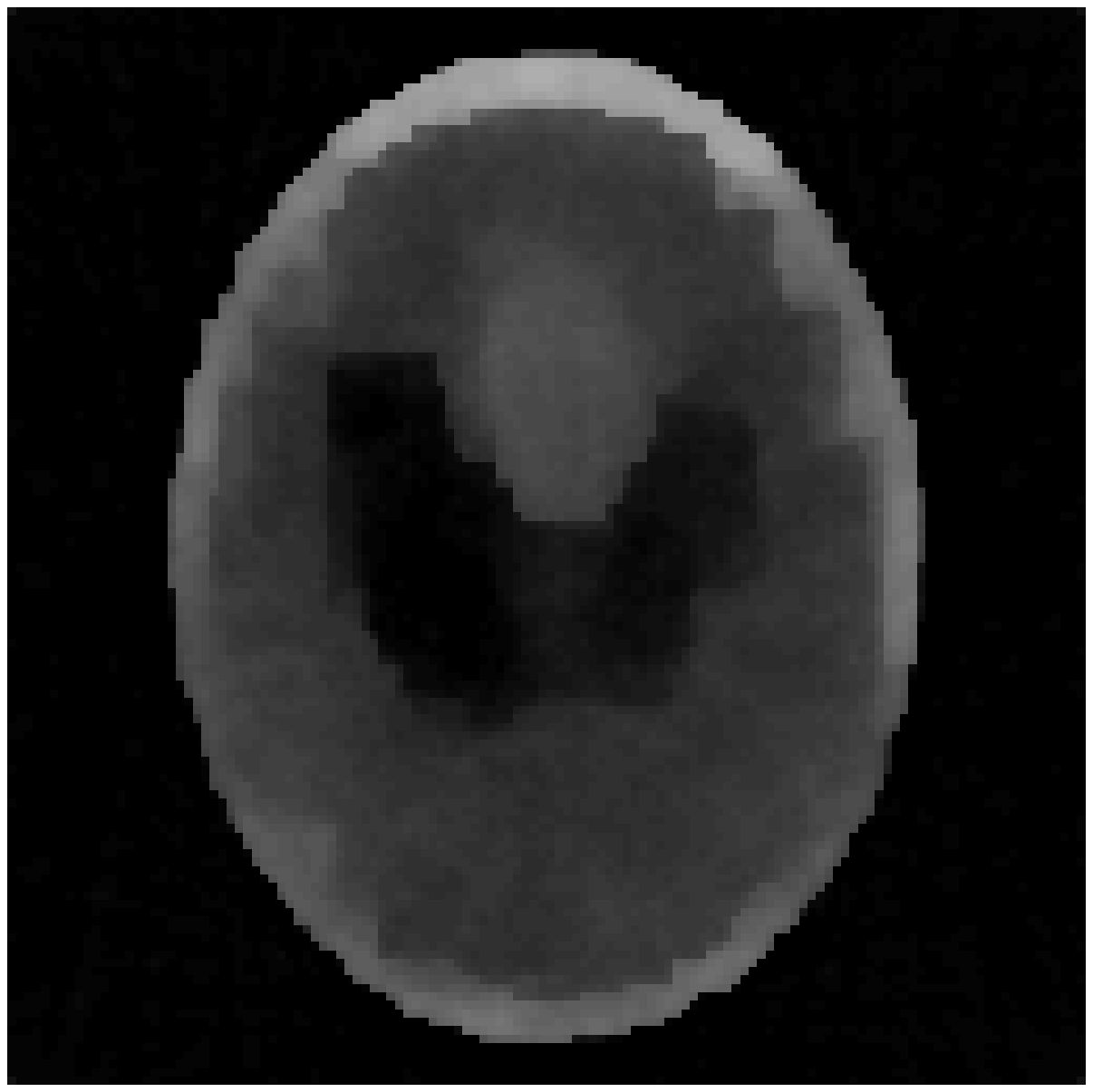} & \includegraphics[scale=0.2]{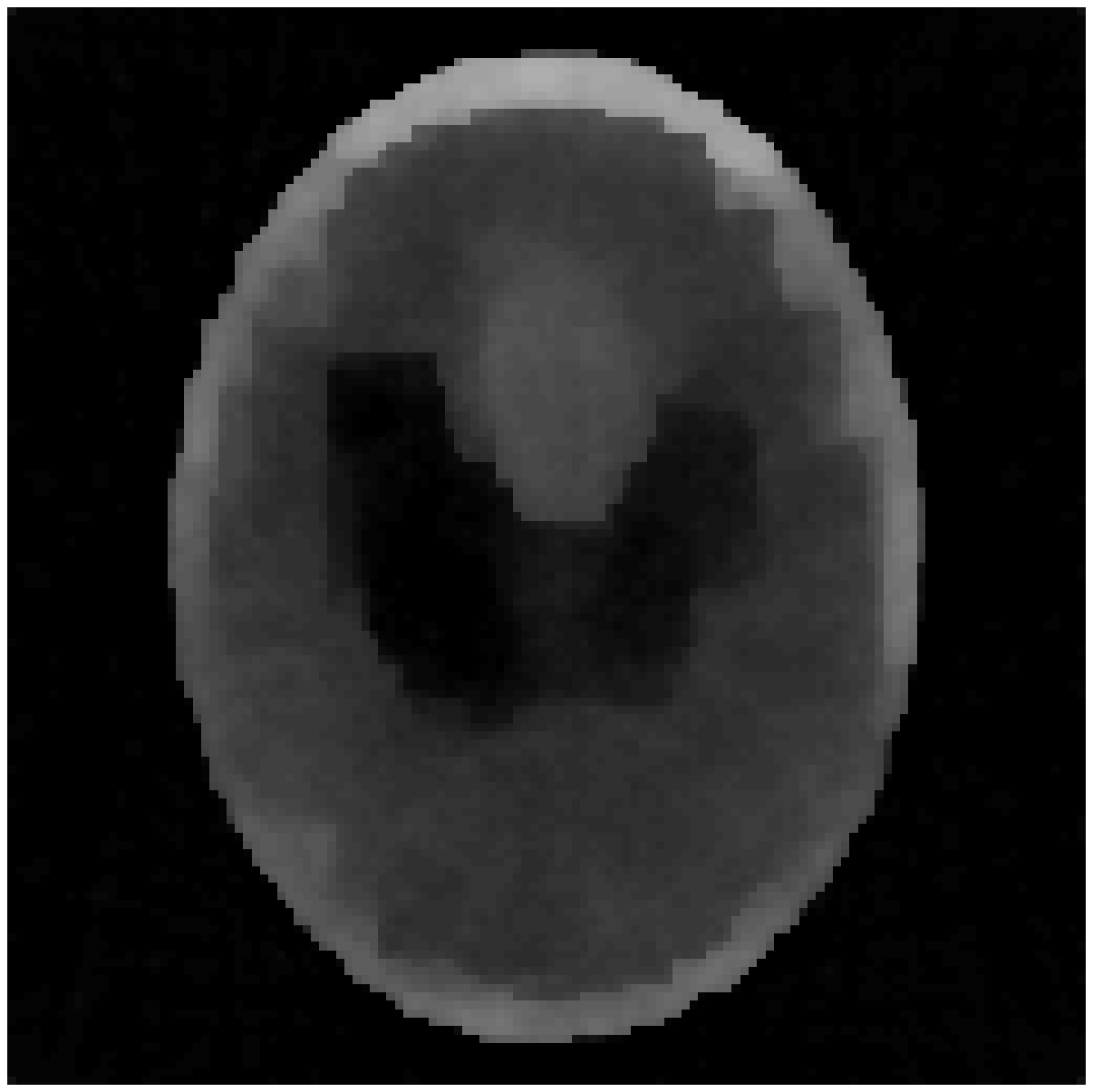} & \includegraphics[scale=0.2]{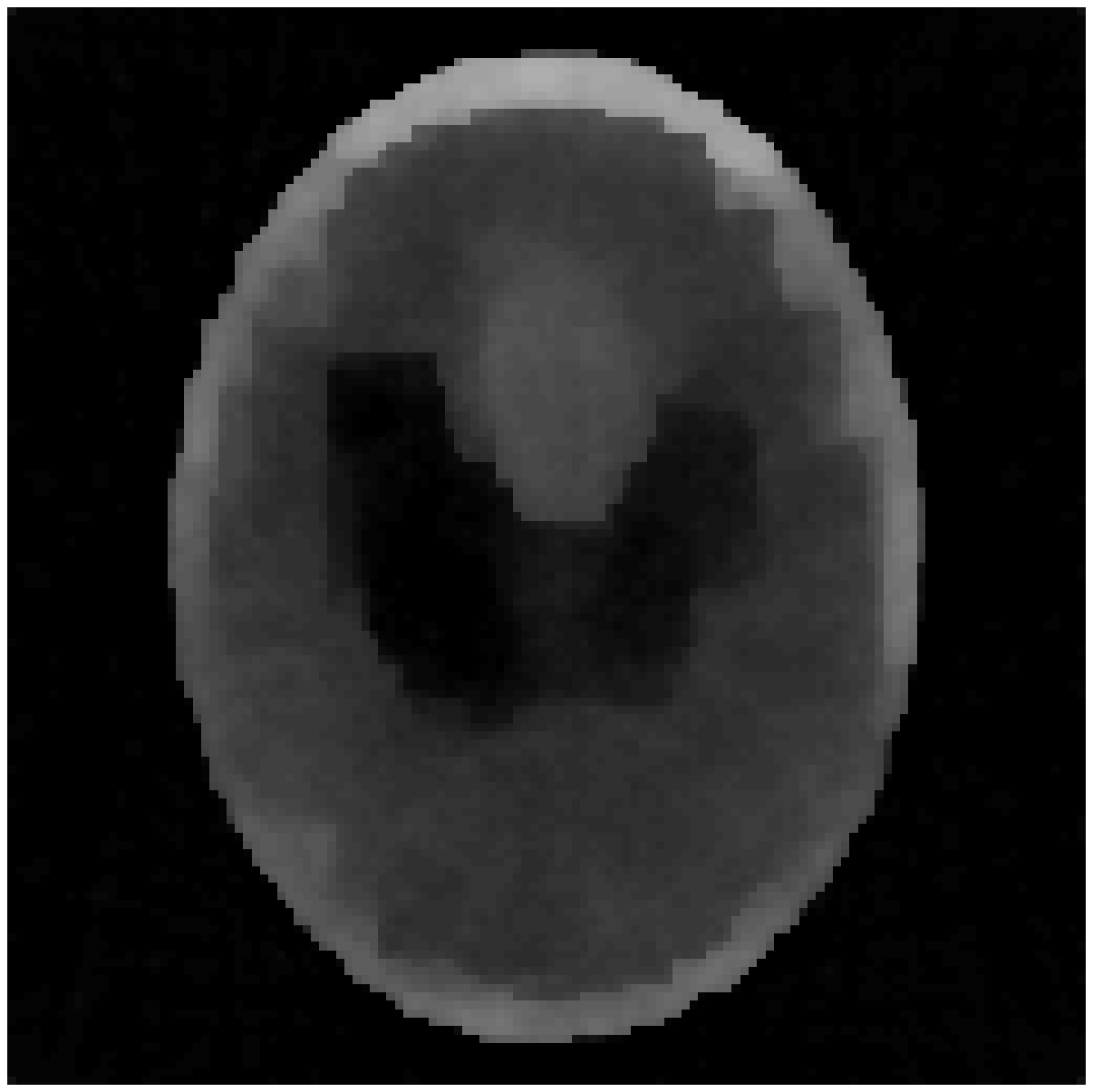} & \includegraphics[scale=0.2]{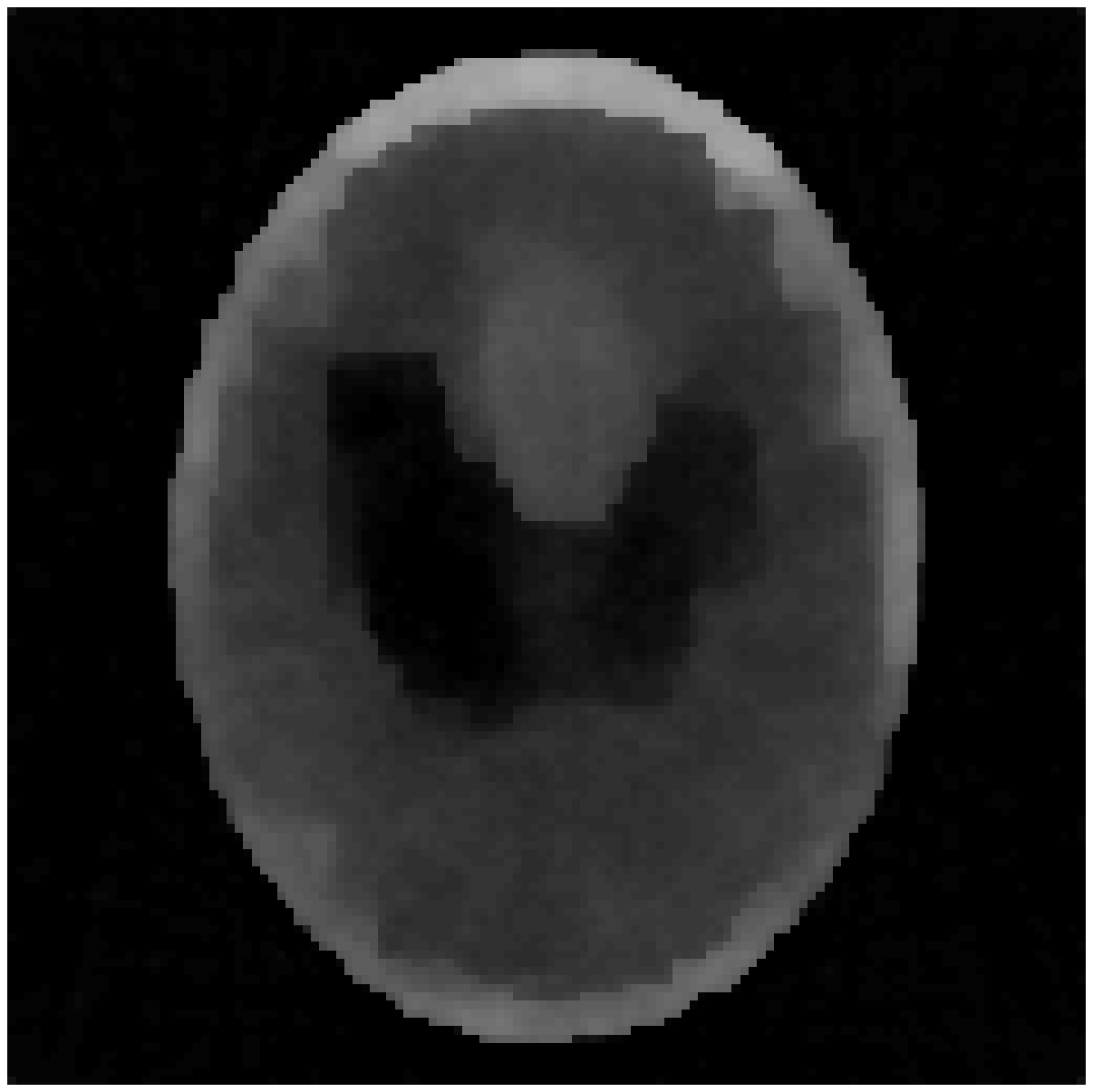}\\
$\mu^6$ & $\mu^7$ & $\mu^8$ & $\mu^9$ & $\mu^{10}$
\end{tabular}
\caption{The convergence of the mean $\mu^k$ by EP after $k$ outer iterations for the \texttt{Shepp-Logan} phantom, moderate count case.\label{fig:EP_conv_x_img}}	
\end{figure}

\begin{figure}[htb!]
\centering
\begin{tabular}{cc}
\includegraphics[width=0.3\textwidth]{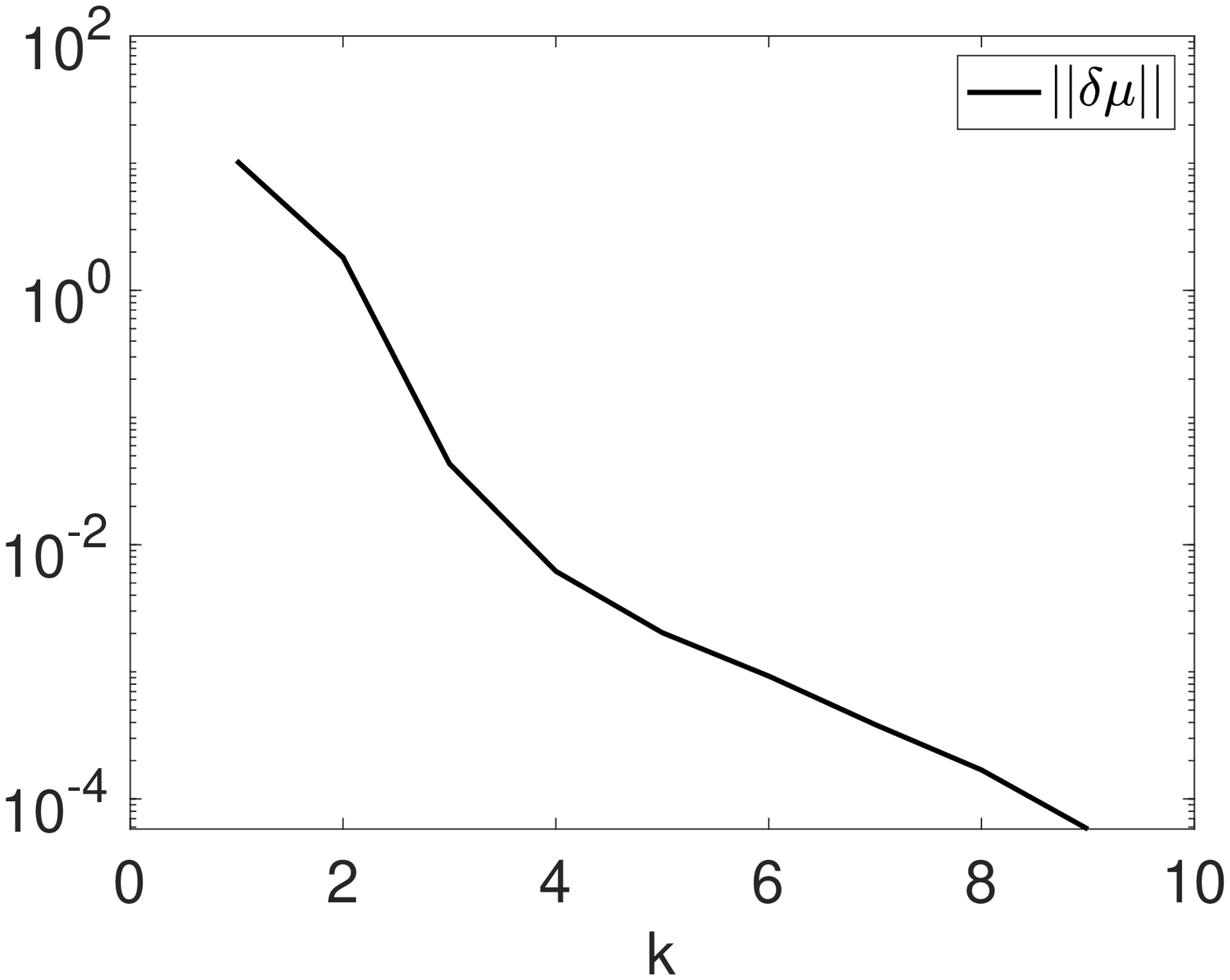} & \includegraphics[width=0.3\textwidth]{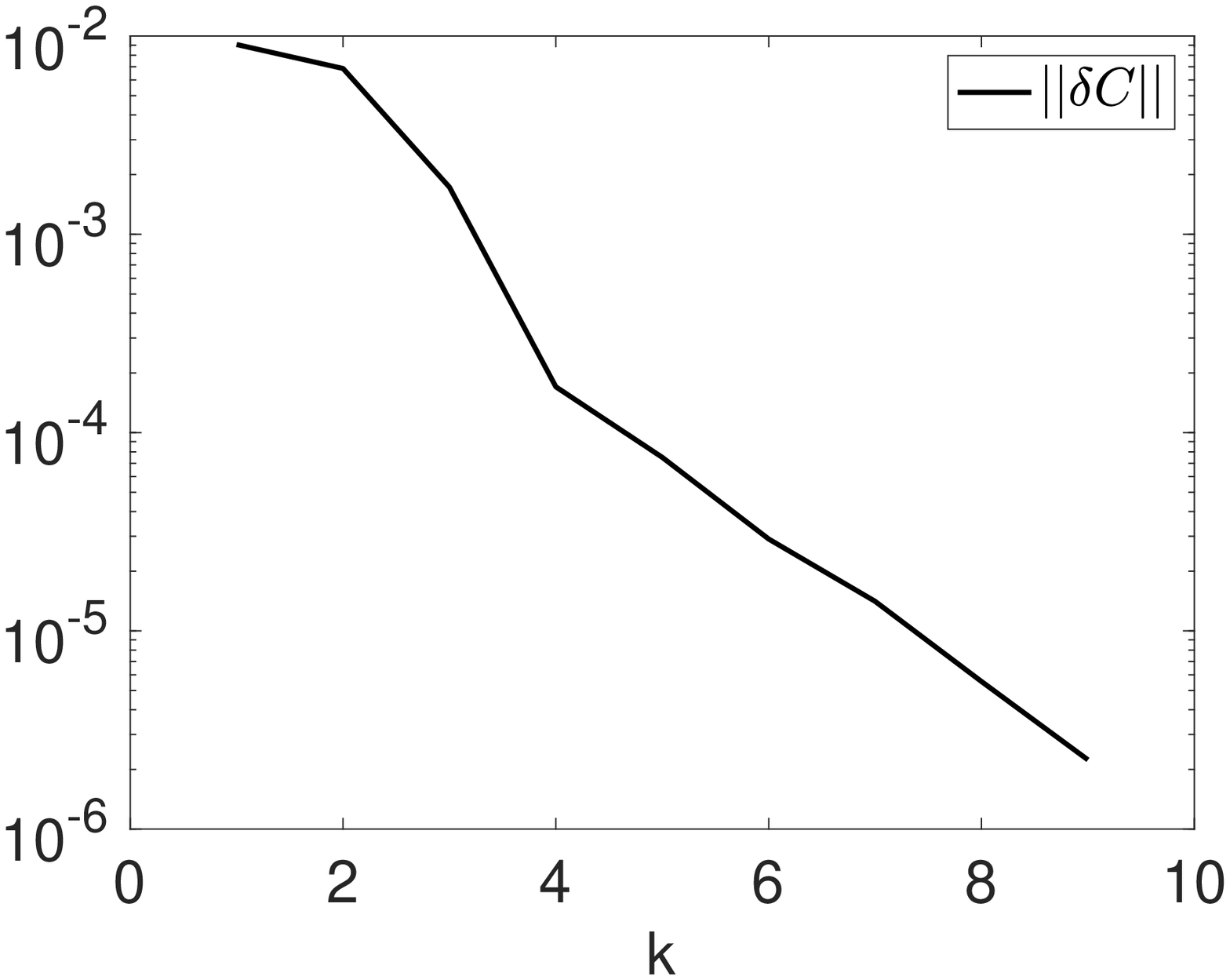}\\
(a) mean $\mu$ & (b) covariance $C$
\end{tabular}
\caption{The convergence of the mean $\mu$ and covariance $C$ after each outer iteration, moderate count case.\label{fig:EP_conv}}	
\end{figure}

\subsection{Real data}
Last, we illustrate the inference procedure with a dataset taken from \texttt{Michigan Image Reconstruction
Toolbox}\footnote{\url{https://web.eecs.umich.edu/~fessler/code/}, last accessed on July 30, 2018.}. The ground truth
image is denoted by \texttt{IRT}. The map  $A\in\mathbb{R}^{24960\times 16384}$ is assembled by $A=\text{diag}
(c_i)G$, where $G$ is the system matrix and $c_i$ is an attenuation vector, by setting the mask to all values
being unity and other parameters to default.
The exact image and data are shown in Fig. \ref{fig:irt_xby}; see Fig. \ref{fig:irt_248} for reconstructions,
obtained with a regularization parameter $\alpha=0.4$. The $L^2$ error, SSIM and PSNR for
EP and MAP are, respectively, $13.46$ and $13.48$, $0.62$ and $0.83$, and $25.66$ and $25.64$. Thus, the
EP results and MAP are comparable, and the preceding observations remain valid.

\begin{figure}[htb!]
\centering
\begin{tabular}{ccccc}
\includegraphics[scale=0.17]{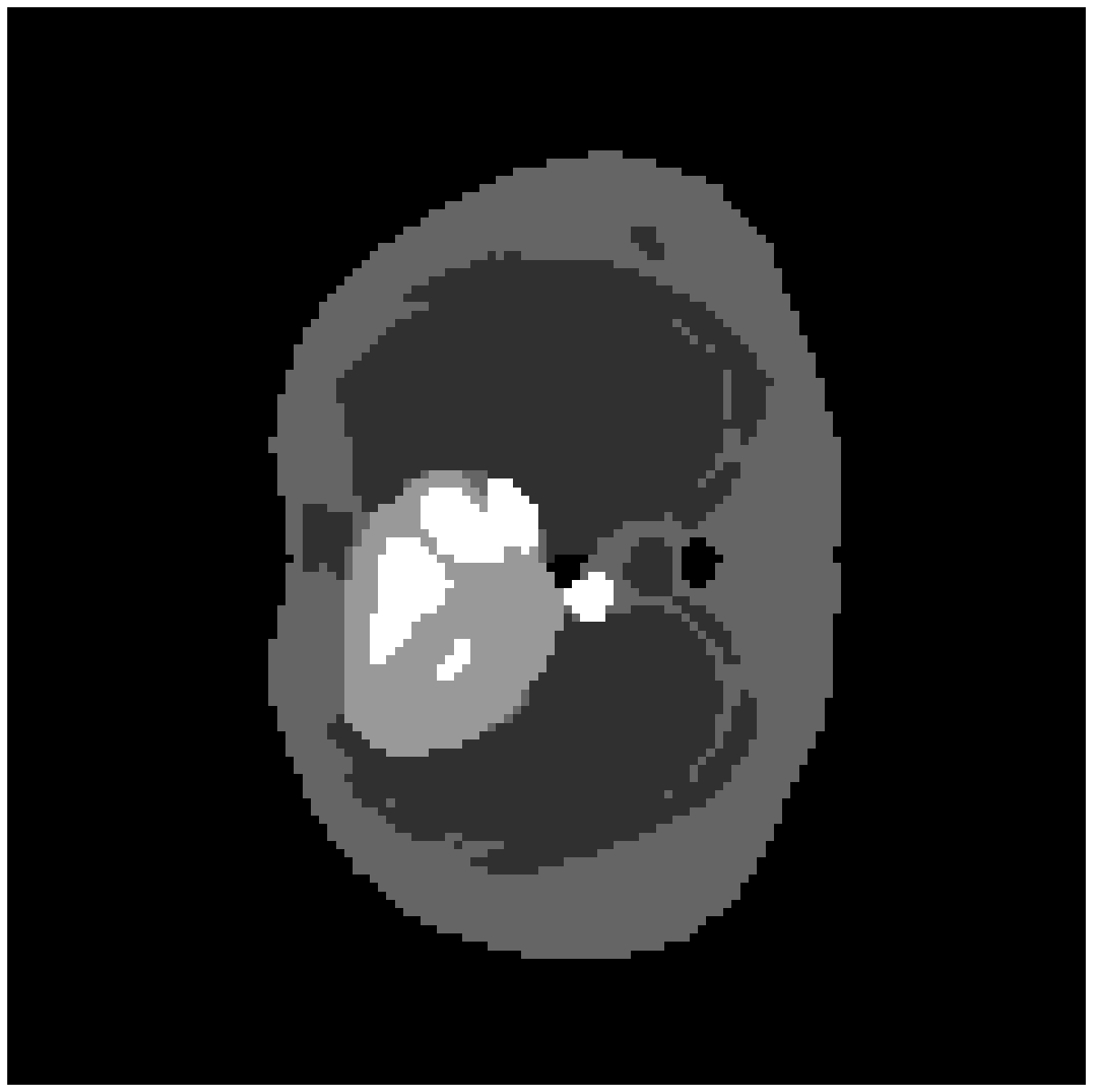}  \includegraphics[scale=0.17]{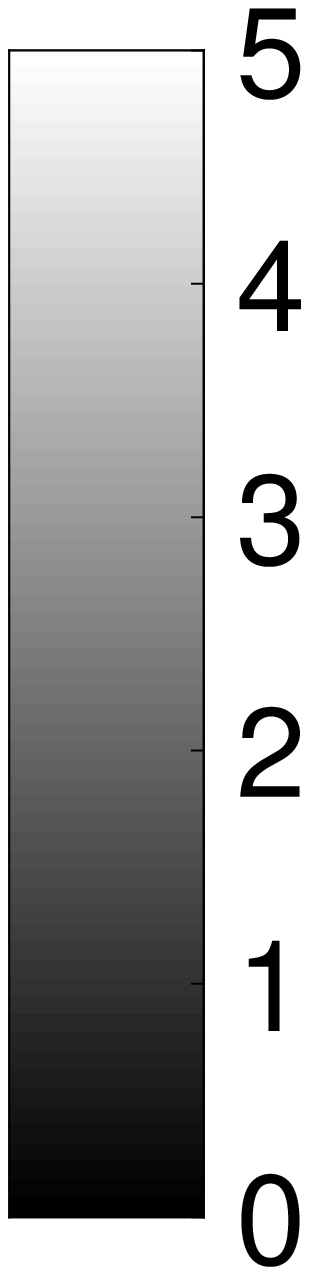} & \includegraphics[scale=0.2]{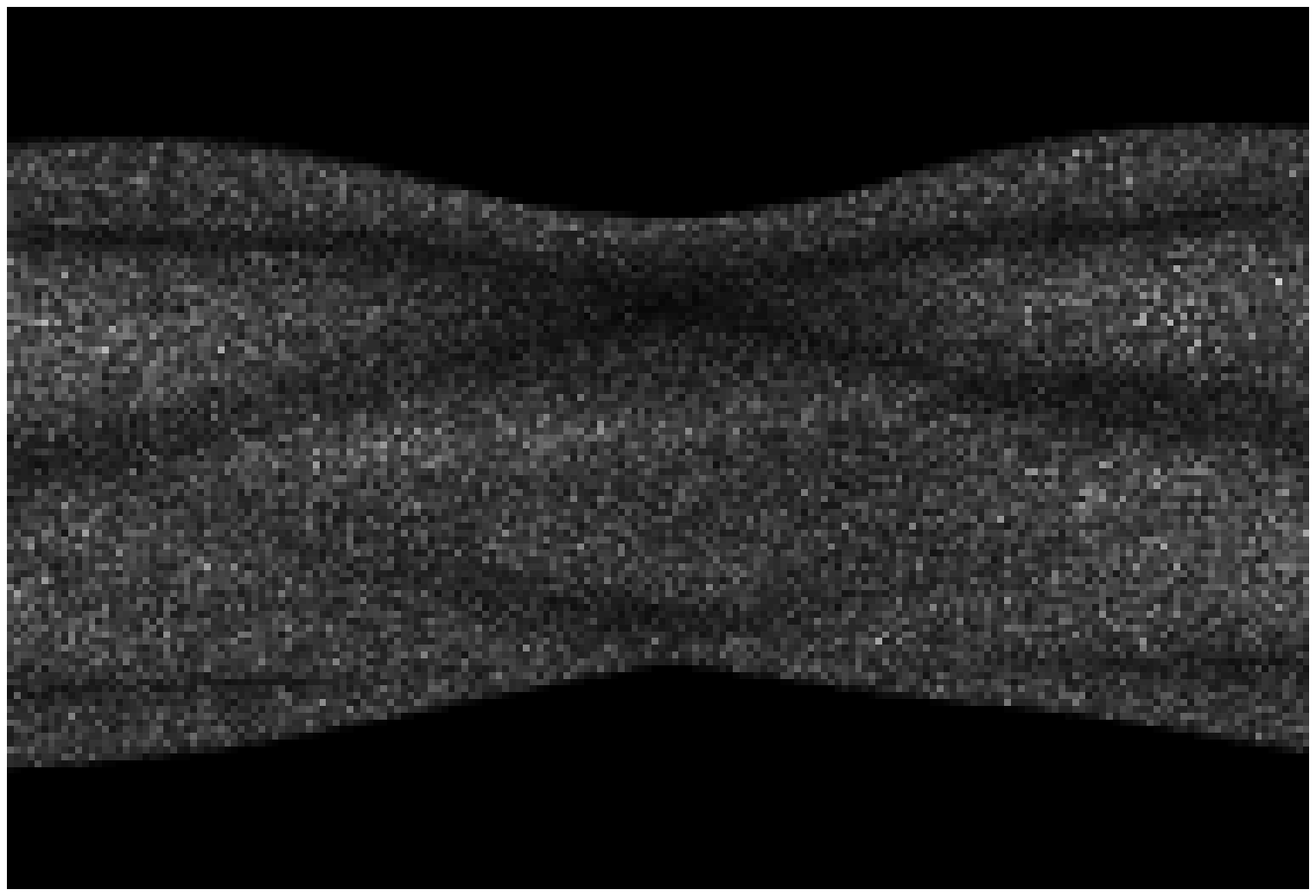}  \includegraphics[scale=0.2]{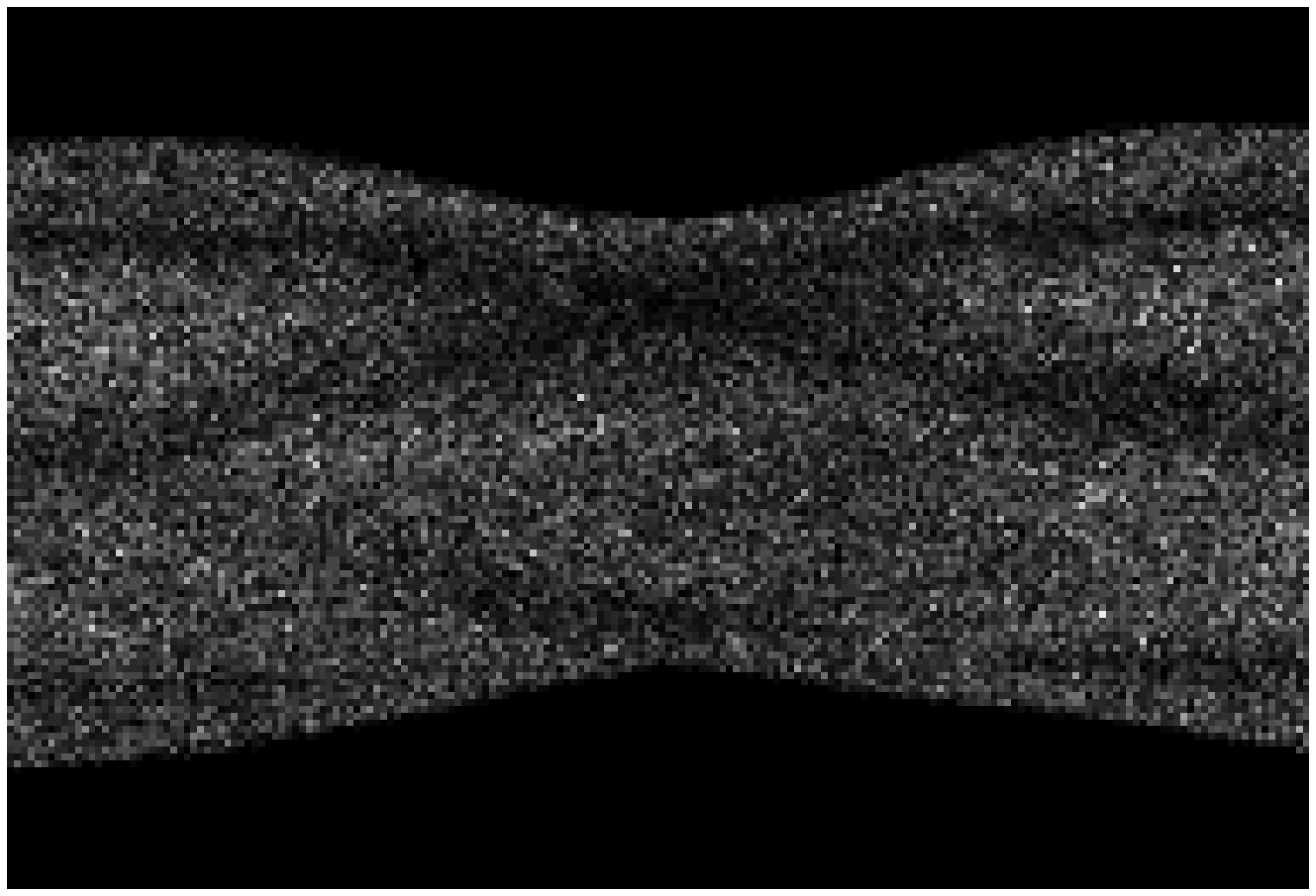} & \includegraphics[scale=0.17]{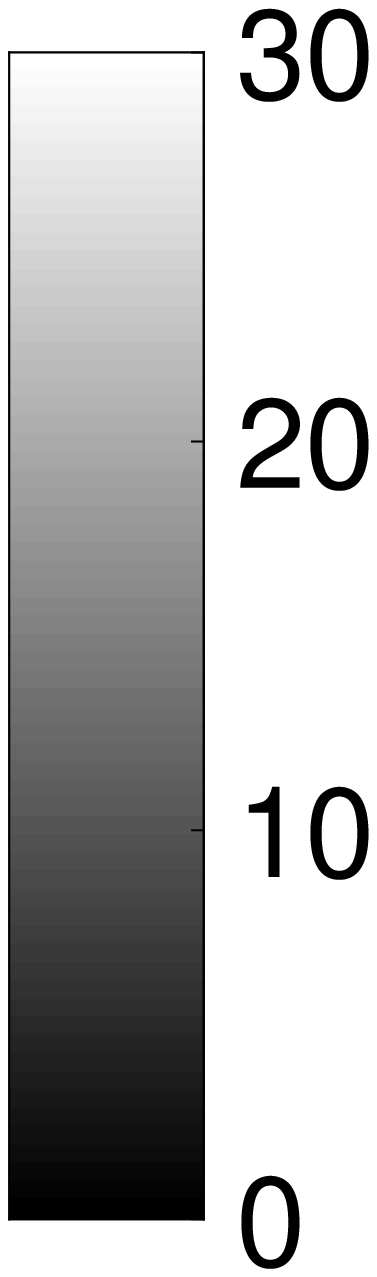}\\
\end{tabular}
\caption{The exact image, sinograms and observed data for \texttt{IRT} phantom. \label{fig:irt_xby}}	
\end{figure}

\begin{figure}[htb!]
\centering
\begin{tabular}{cccccc}
\includegraphics[scale=0.2]{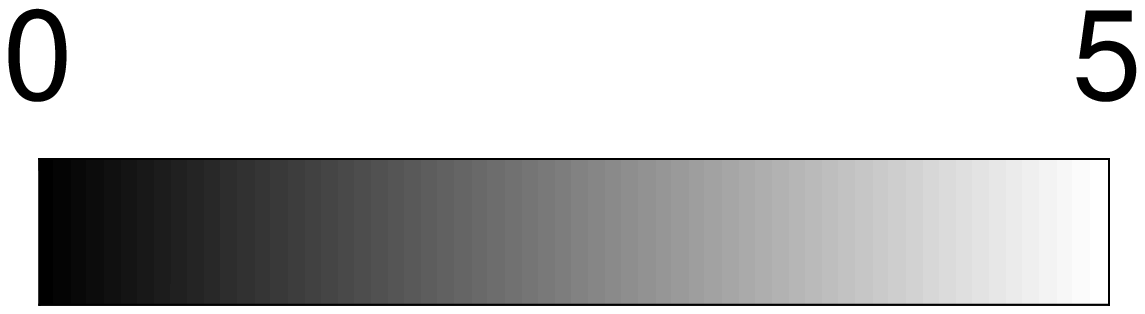} & \includegraphics[scale=0.2]{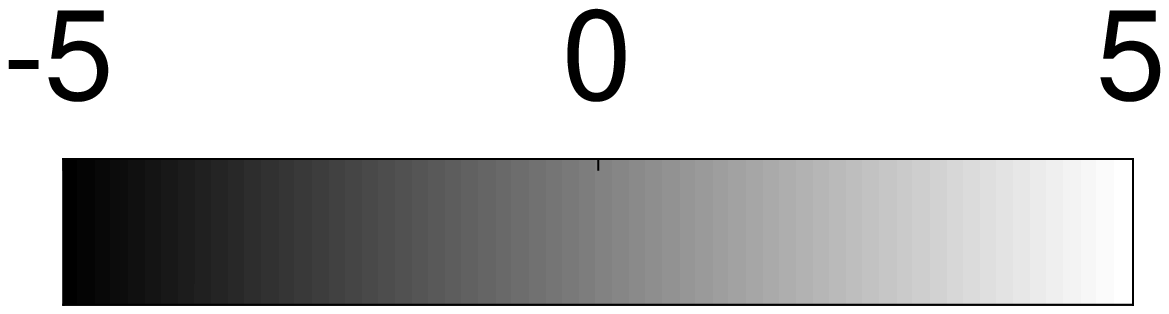} & \includegraphics[scale=0.2]{bar_5_h} & \includegraphics[scale=0.2]{bar_5_5_h} & \includegraphics[scale=0.2]{bar_001_h}\\
\includegraphics[scale=0.2]{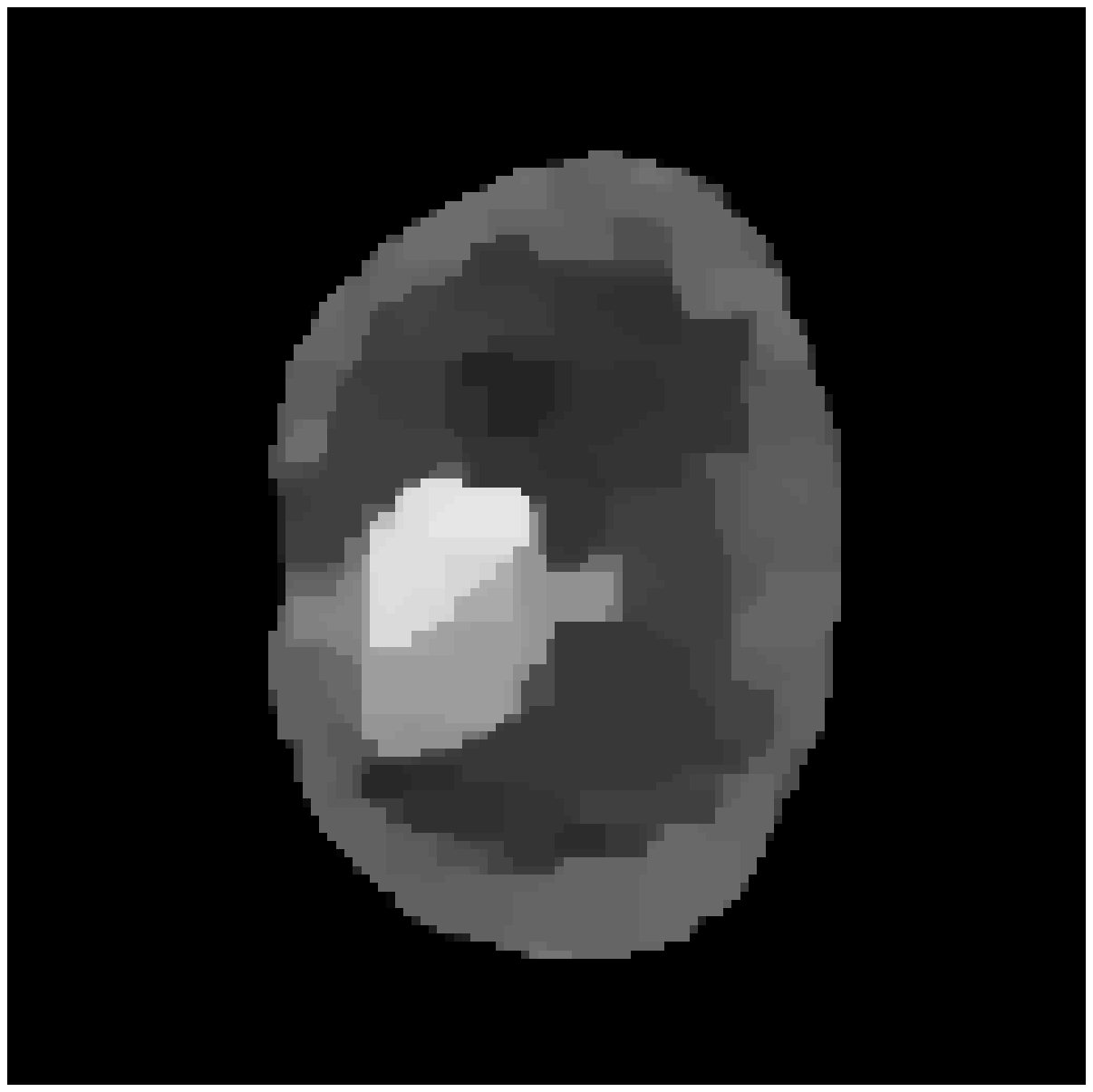} & \includegraphics[scale=0.2]{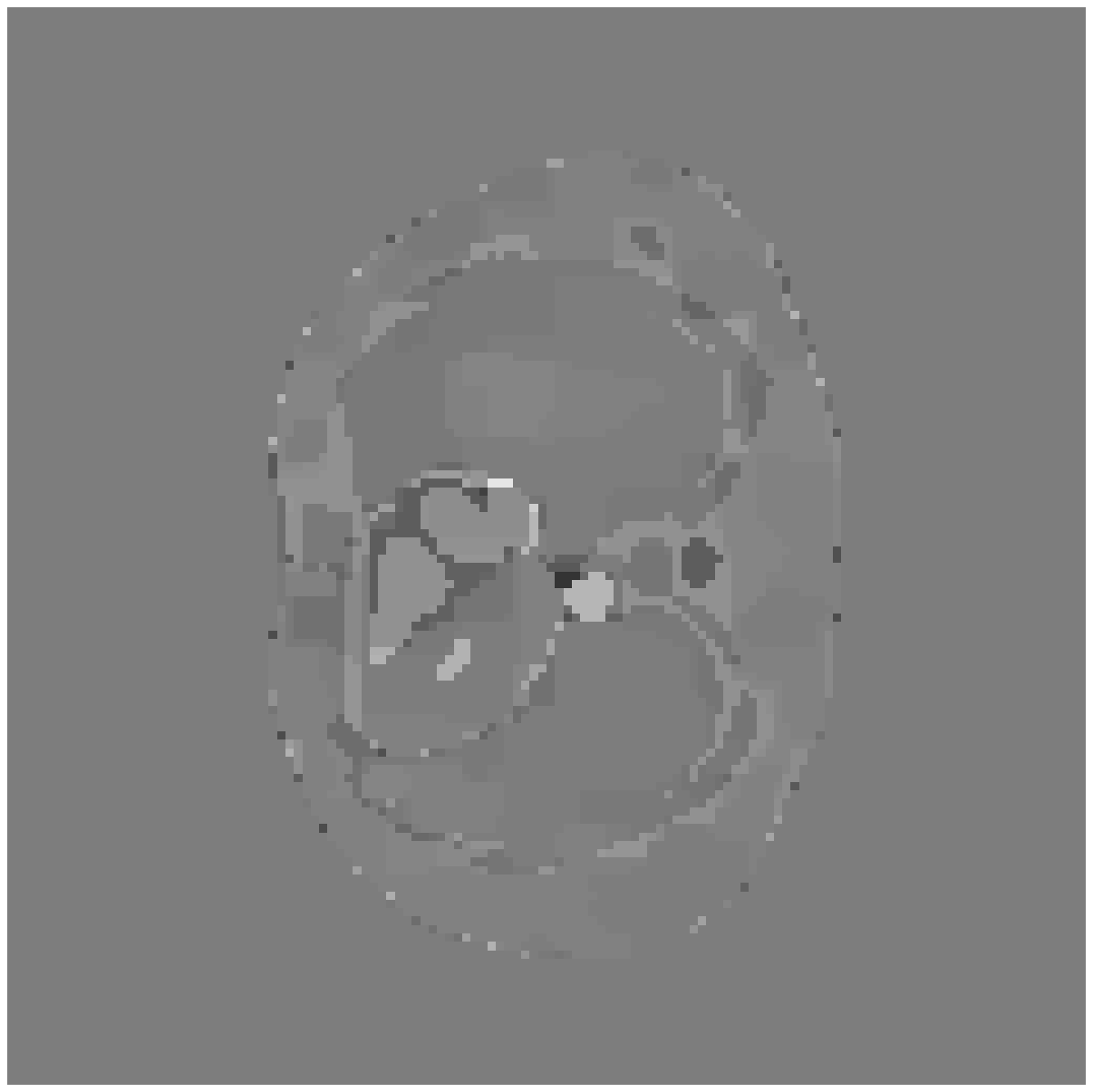} & \includegraphics[scale=0.2]{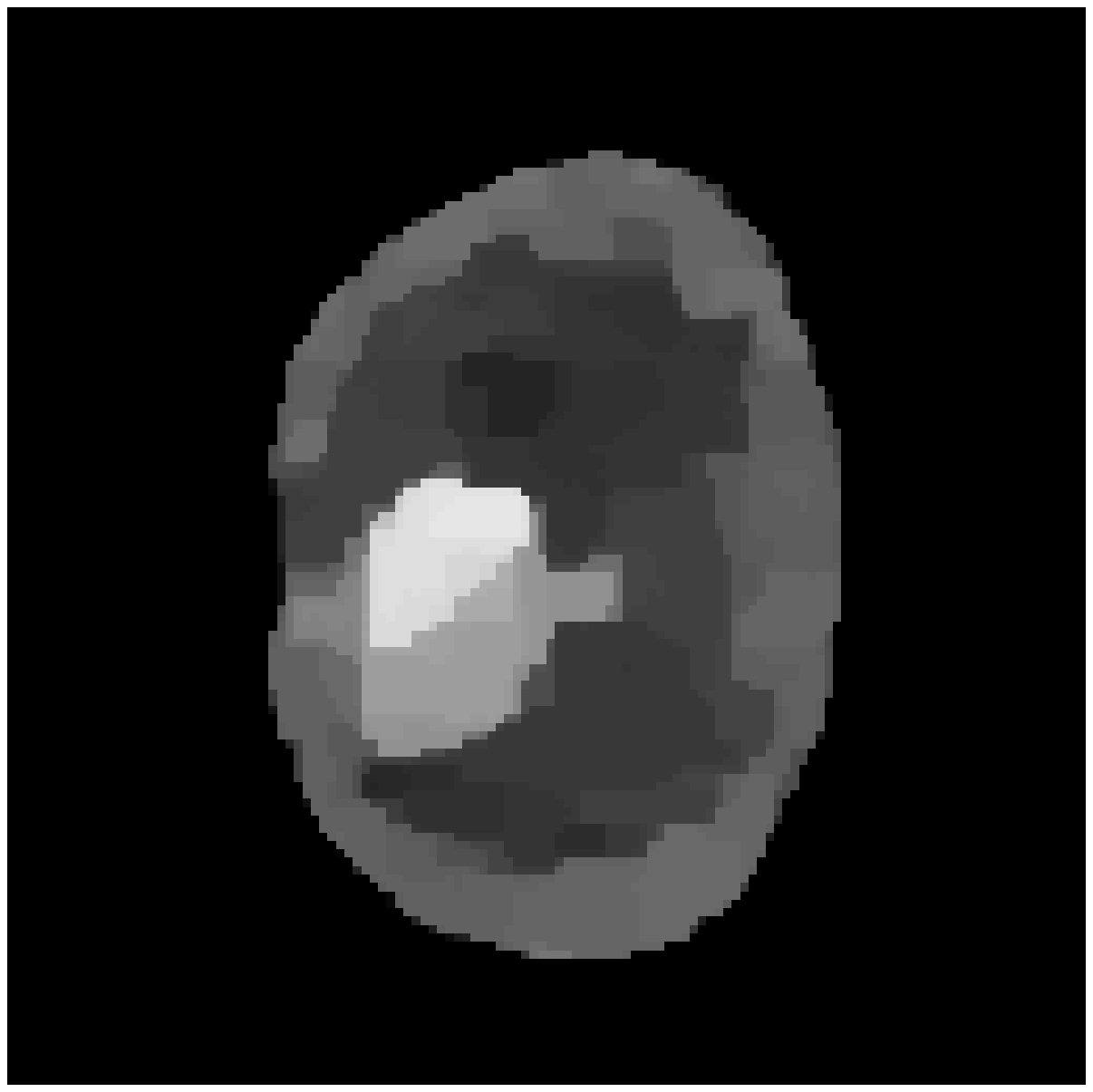} & \includegraphics[scale=0.2]{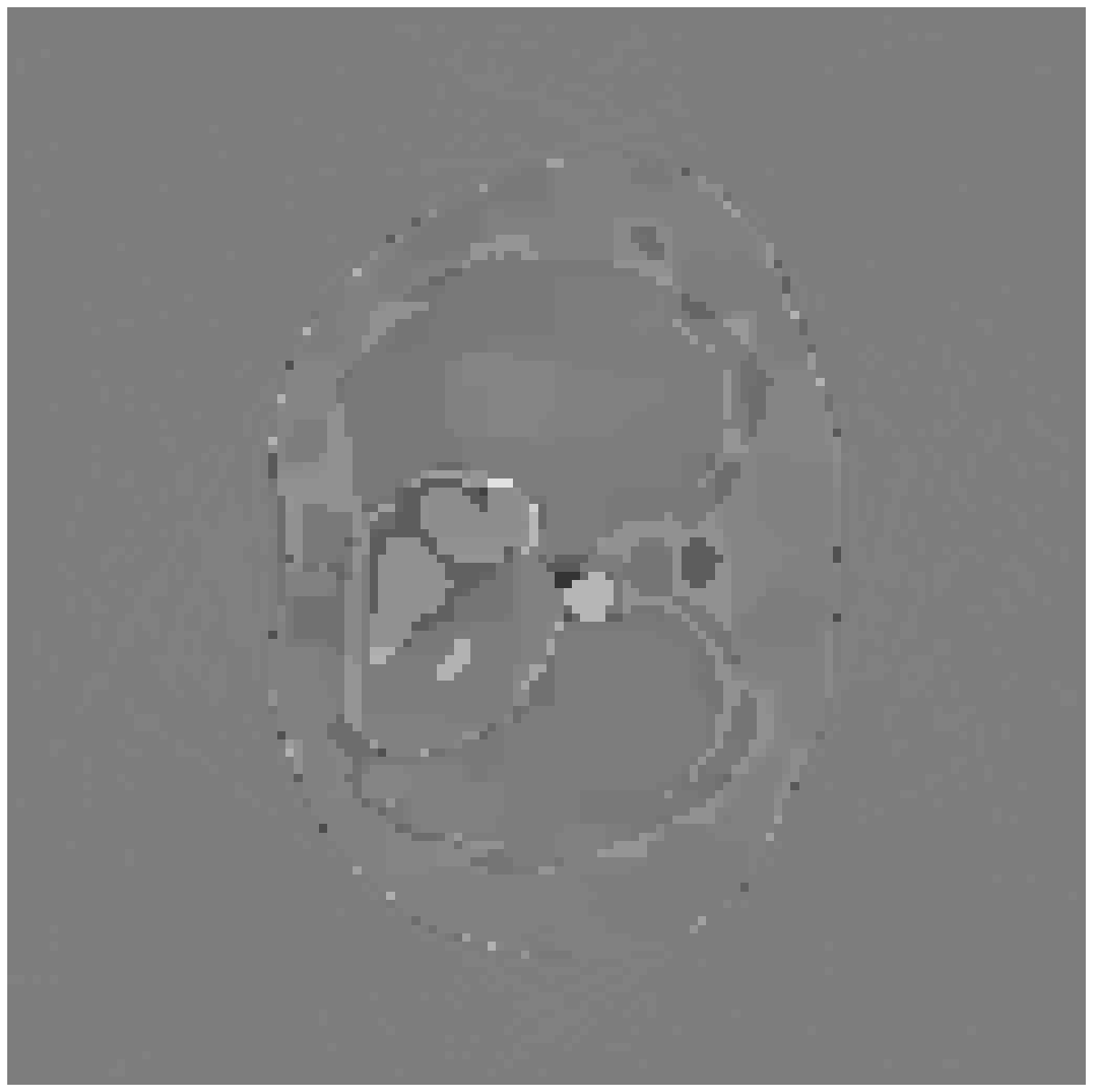} & \includegraphics[scale=0.2]{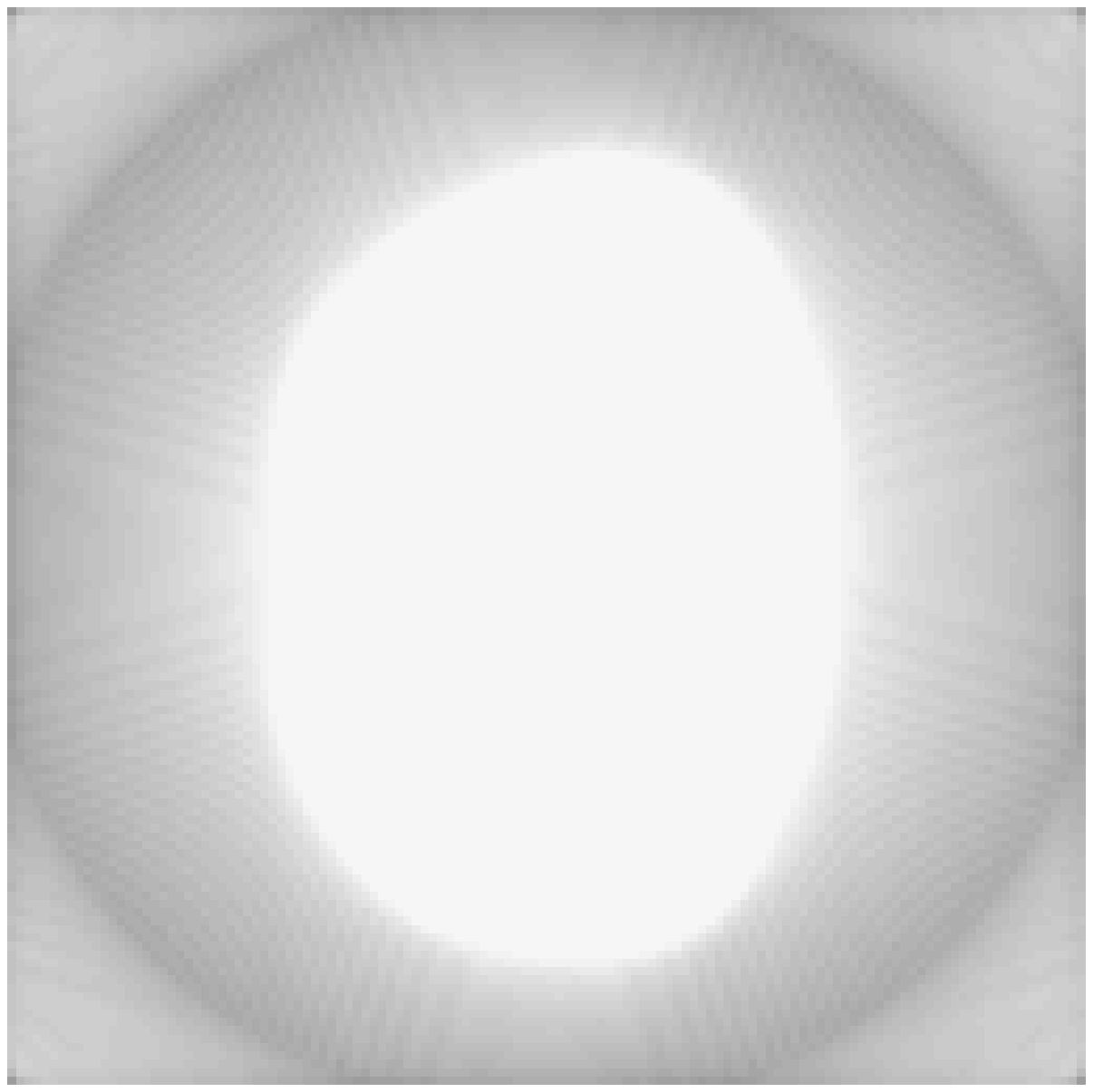}\\
MAP & MAP error & EP mean & EP error & EP variance
\end{tabular}
\caption{MAP vs EP with anisotropic TV prior for the \texttt{IRT} phantom.\label{fig:irt_248}}	
\end{figure}

These numerical results with different experimental settings show clearly that EP can provide comparable
point estimates with MAP as well as uncertainty information by means of the variance estimate.

\section{Conclusion}
In this work, we have developed inference procedures for the constrained Poisson likelihood
arising in emission tomography. They are based on expectation propagation developed in the
machine learning community. The detailed derivation of the algorithms, complexity and their
stable implementation are given for a Laplace type prior. Extensive numerical experiments
show that the EP algorithm (with natural parameters) converges rapidly and can deliver an
approximate posterior distribution with the approximate mean comparable with MAP, together
with uncertainty estimate, and can handle real images of medium size. Thus, the approach can be
viewed as a feasible fast alternative to the general-purposed but expensive MCMC for
rapid uncertainty quantification with Poisson data.

There are several avenues for future works. First, it is of enormous interest to analyze the convergence
rate and accuracy of EP, and more general approximate inference techniques, e.g., variational Bayes, which
have all achieved great practical successes but largely defied theoretical analysis. Second, it is important
to further extend the flexibility of EP algorithms to more complex posterior distributions, e.g., lack of
projection form. One notable example is isotropic total variation prior that appears frequently in practical
imaging algorithms. This may require introducing an additional layer of approximation, e.g.,
in the spirit of iteratively reweighed least-squares or (quasi-)Monte Carlo computation of
low-dimensional integrals. Third, many experimental studies show that EP converges very fast, with convergence
reached within five sweeps for the Poisson model under considerations. However, the overall $O(mn^2)$
computational complexity per sweep of all current implementations \cite{gelman2014expectation} is still
very high, and not scalable well to large images that are required in many real world applications.
Hence, there remains great demand to further accelerate the algorithms, e.g., via low-rank structure of the map
$A$ and diagonal dominance of the posterior covariance. Fourth and last, it is also important to derive
rigorous error estimates for the quadrature rules developed in Section \ref{sec:int}.

\appendix
\section{Parameterizing Gaussian distributions}\label{app:Gaussian}
For a Gaussian $\mathcal{N}(x|\mu,C)$ with mean $\mu\in \mathbb{R}^n$ and covariance $C\in \mathcal{S}^n_+$,
the density $\pi(x|\mu,C)$ is given by
\begin{equation*}
	\pi(x|\mu,C) = (2\pi)^{-\frac{n}{2}}|C|^{-\frac12}e^{-\frac{1}{2}(x-\mu)^tC^{-1}(x-\mu)}
						 = e^{\zeta+h^t x-\frac{1}{2}x^t\Lambda x},
\end{equation*}
where the parameters $\Lambda\in\mathcal{S}_+^n$, $h\in\mathbb{R}^n$ and $\zeta\in\mathbb{R}$ are respectively given by
\begin{equation*}
	\Lambda = C^{-1},\quad
	h=\Lambda\mu,\quad
\mbox{and}
\quad	\zeta=-\tfrac{1}{2}(n\log2\pi+\log|\Lambda|+\mu^t\Lambda\mu).
\end{equation*}
Thus, the density $\pi(x|\mu,C)$ is also uniquely defined by $\Lambda$ and $h$.
In the literature, $\Lambda$ is often referred to as the precision matrix and $h$ as the precision mean.
and the pair $(h,\Lambda)$ is called the natural parameter of a Gaussian
distribution.

It is easy to check that the product of $k$ Gaussians $\{\mathcal{N}(x|
\mu_k,C_k)\}_{k=1}^m$ is also a Gaussian $\mathcal{N}(x|\mu,C)$ after normalization,
and the mean $\mu$ and covariance $C$ of the product are given by
\begin{equation}
   \mu = C\sum_{k=1}^mC^{-1}_k\mu_k  \quad\mbox{and}\quad	C= \Big(\sum_{k=1}^mC^{-1}_k\Big)^{-1},
\end{equation}
or equivalently
\begin{equation}
   h = \sum_{k=1}^m h_k	\quad \mbox{and}\quad \Lambda = \sum_{k=1}^m \Lambda_k.
\end{equation}

\section{Laplace approximation}\label{app:Laplace}
In the engineering community, one popular approach to approximate the posterior distribution $p(x|y)$ is Laplace approximation
\cite{TierneyKadane:1986,Bishop:2006}. It constructs a Gaussian approximation by the second-order Taylor expansion of
the negative log-posterior $-\log p(x|y)$ around MAP $\hat x$. Upon ignoring the unimportant
constant and smoothing the Laplace potential, the negative log-posterior $J(x)$ is given by
\begin{equation}
	J(x) = \sum_{i=1}^{m_1} (-y_i\log(a^t_ix+r_i)+a^t_ix+r_i)+\alpha\sum_{i=1}^{m_2}((L^t_ix)^2+\epsilon^2)^{1/2},
\end{equation}
where $\epsilon>0$ is a small smoothing parameter to restore the differentiability.
The gradient $\nabla J(x)$ and Hessian $\nabla^2J(x)$ are given respectively by
\begin{align*}
	\nabla J(x) &= \sum_{i=1}^{m_1}(-\frac{y_i}{a^t_ix+r_i}+1)a_i + \alpha\sum_{i=1}^{m_2}((L^t_ix)^2+\epsilon^2)^{-1/2}(L_i^tx)L_i,\\
	\nabla^2J(x)&= \sum_{i=1}^{m_1} \frac{y_i}{(a^t_ix+r_i)^2}a_ia^t_i + \alpha{\epsilon^2}\sum_{i=1}^{m_2}{((L^t_ix)^2+\epsilon^2)^{-3/2}}L_iL^t_i.
\end{align*}
Since $\nabla J(\hat x)=0$, the Taylor expansion reads
\begin{equation}\label{eqn:Lap}
  J(x) \approx J(\hat x) + \tfrac{1}{2}(x-\hat x)^t\nabla^2J(\hat x)(x-\hat x),
\end{equation}
and $\nabla^2J(\hat x)$ approximates the precision matrix. When $\epsilon \ll |L_i^t\hat x|$, the second
term in $\nabla^2 J(x)$ can be negligible and thus the Hessian of the negative log-likelihood is dominating;
whereas for $\epsilon\gg |L_i^t\hat x|$, the second term is dominating. In either case, the approximation
is problematic. In practice, it is also popular to combine smoothing with an iterative weighted approximation
(e.g., lagged diffusivity approximation \cite{VogelOman:1996}) by fixing $((L_i^tx)^2+\epsilon^2)^{1/2}$ in
$\nabla J(x)$ at $((L_i^t\hat x)^2+\epsilon^2)^{1/2}$, which leads to a modified Hessian:
\begin{equation*}
  \widetilde \nabla^2J(x)= \sum_{i=1}^{m_1} \frac{y_i}{(a^t_ix+r_i)^2}a_ia^t_i + \alpha\sum_{i=1}^{m_2}{((L^t_i\hat x)^2+\epsilon^2)^{-1/2}}L_iL^t_i.
\end{equation*}
The Hessians $\nabla^2J(\hat x)$ and $\widetilde \nabla^2 J(\hat x)$ will be close to each other,
if $|L_i^t\hat x|$ are all small, which is expected to hold for truly sparse signals, i.e., $L_i^tx\approx0$
for $i=1,\ldots,m_2$. One undesirable feature of Laplace approximation is that the
precision approximation depends crucially on the smoothing parameter $\epsilon$.

\section{Comparison with MCMC and Laplace approximation}\label{sec:1d}
Numerically, the accuracy of EP has found to be excellent in several studies \cite{rasmussen2004gaussian,gehre2014expectation},
although there is still no rigorous justification. We provide an experimental evaluation of its accuracy with Markov chain
Monte Carlo (MCMC) and Laplace approximation. The true posterior distribution $p(x|y)$ can be explored by MCMC \cite{Liu:2001,
RobertCasella:2004}. However, usually a large number of samples are required to obtain reliable statistics. Thus, to obtain
further insights, we consider a one-dimensional problem, i.e., a Fredholm integral equation of the first kind
\cite{Phillips:1962} over the interval $[-6,6]$ with the kernel $K(s,t)=\phi(s-t)$ and exact solution $x(t)=\phi(t)$,
where $\phi(s) = 10+10\cos\frac{\pi}{3}s\chi_{[-3,3]}$. It is discretized by a standard piecewise constant Galerkin method,
and the resulting problem is of size $100$, i.e., $x\in\mathbb{R}^{100}$ and $A\in\mathbb{R}^{100\times 100}$. We implement a random
walk Metropolis-Hastings sampler with Gaussian proposals, and optimize the step size so that the acceptance ratio is close to
$0.23$ in order to ensure good convergence \cite{BrooksGelman:1998}. The hyperparameter $\alpha$ in the prior distribution is
set to $1$. The chain is run for a length of $2\times 10^7$, and the last $10^7$ samples are used for computing the mean and covariance.

To compare the Gaussian approximation by EP and MCMC results, we present the mean, MAP, covariance and $95\%$ HPD region. Both
approximations concentrate in the same region, and the shape and magnitude of $95\%$ HPD / covariance are mostly comparable;
see Figs. \ref{fig:1d_compare_HPD} and \ref{fig:1d_compare_mean_cov}, showing the validity of EP. However, there are noticeable
differences in the recovered mean: the EP mean is nearly piecewise constant, which differs from that by MCMC. So EP gives an
intermediate approximation between the MAP and posterior mean. In comparison with MAP, EP provides not only a point estimate,
but also the associated uncertainty, i.e., covariance. Interestingly, the covariance is clearly diagonal dominant, which suggests
the use of a banded covariance or its Cholesky factor for potentially speeding up the algorithm.

\begin{figure}[htb!]
\centering
\begin{tabular}{cc}
\includegraphics[width=0.35\textwidth]{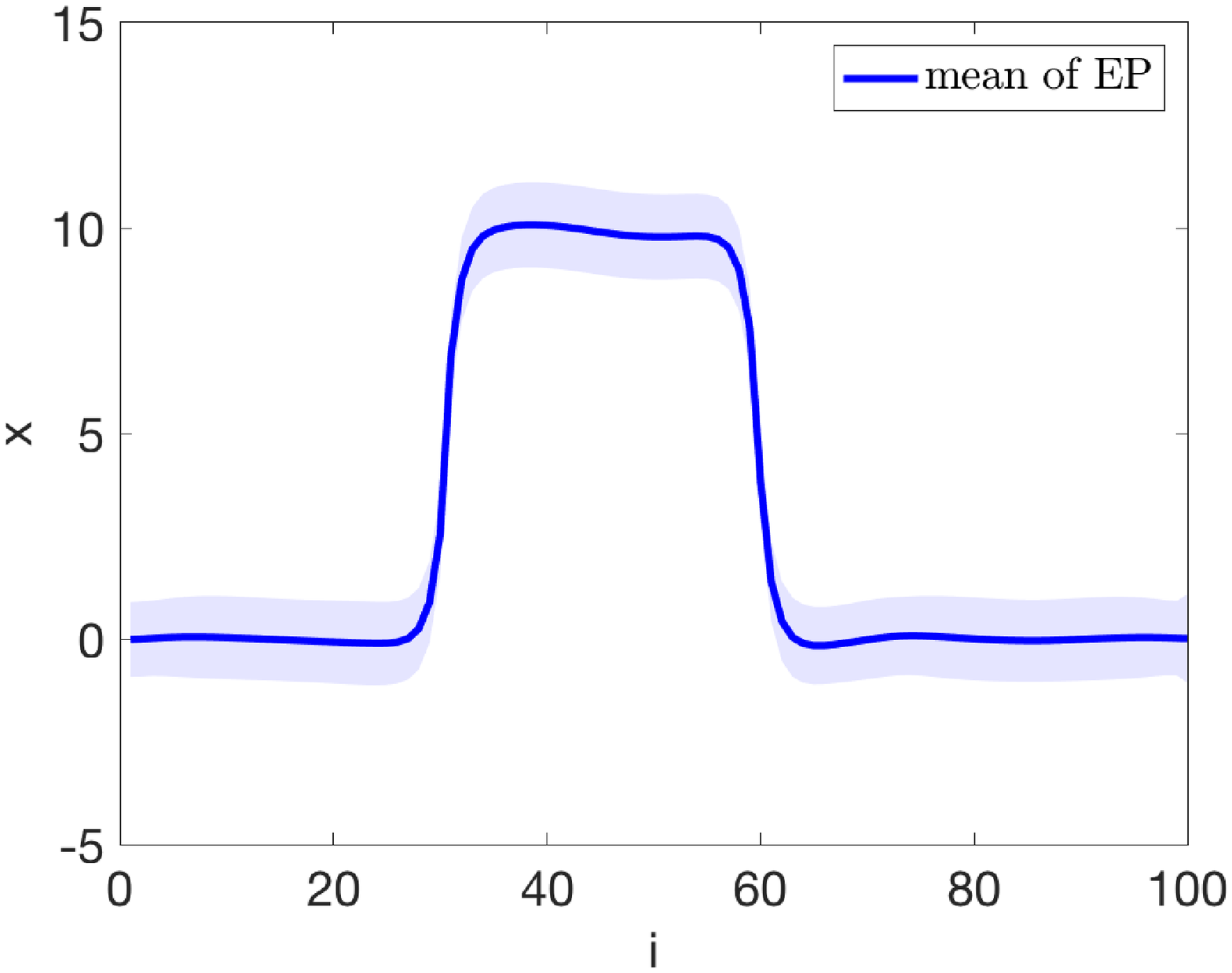} & \includegraphics[width=0.35\textwidth]{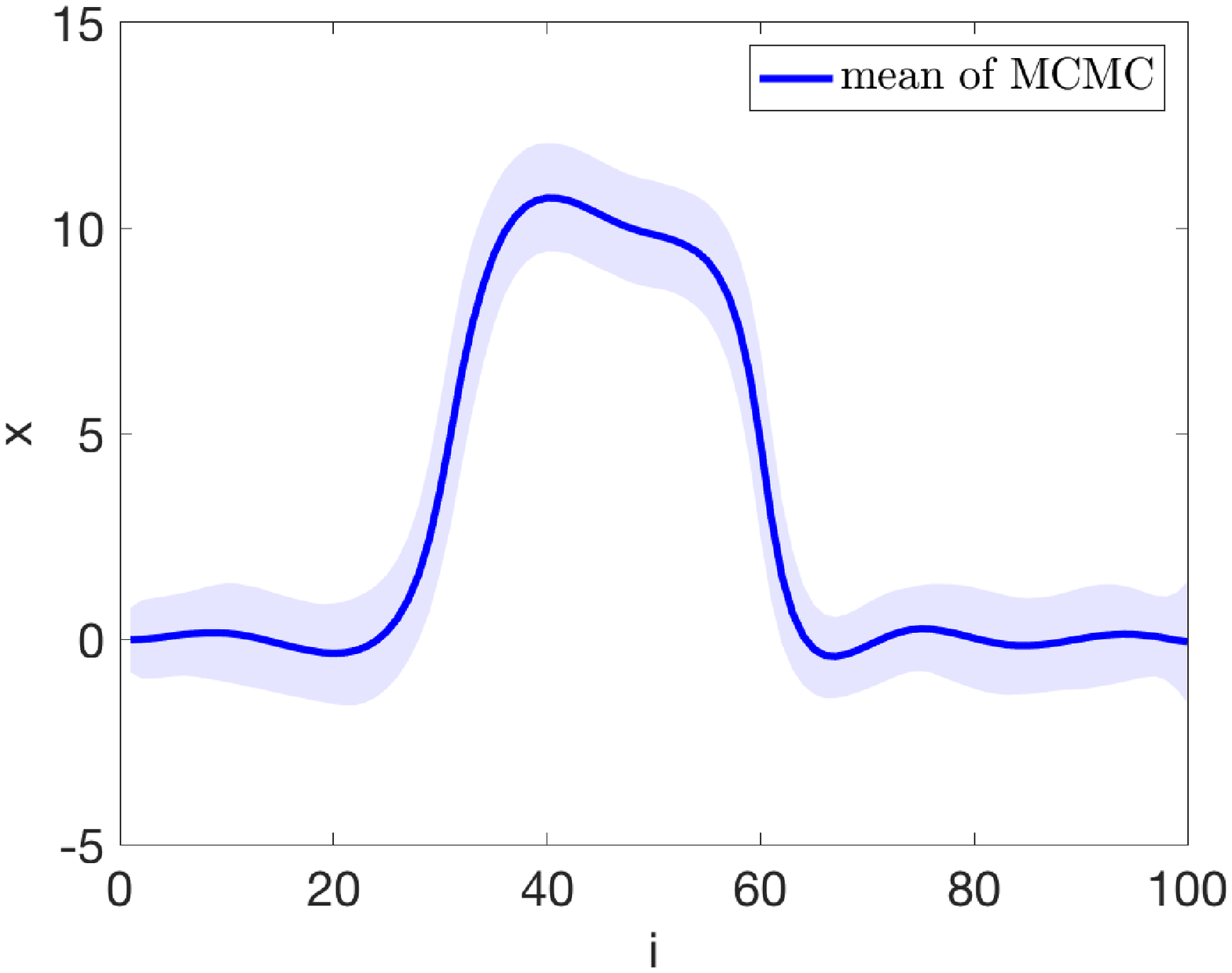} \\
(a) EP mean and $0.95$-HPD & (b) MCMC mean and $0.95$-HPD\\
\end{tabular}
\caption{Comparisons of mean and $0.95$ HPD between EP and MCMC for \texttt{Phillips} test\label{fig:1d_compare_HPD}}	
\end{figure}

\begin{figure}[htb!]
\centering
\begin{tabular}{ccc}
\includegraphics[scale=0.26]{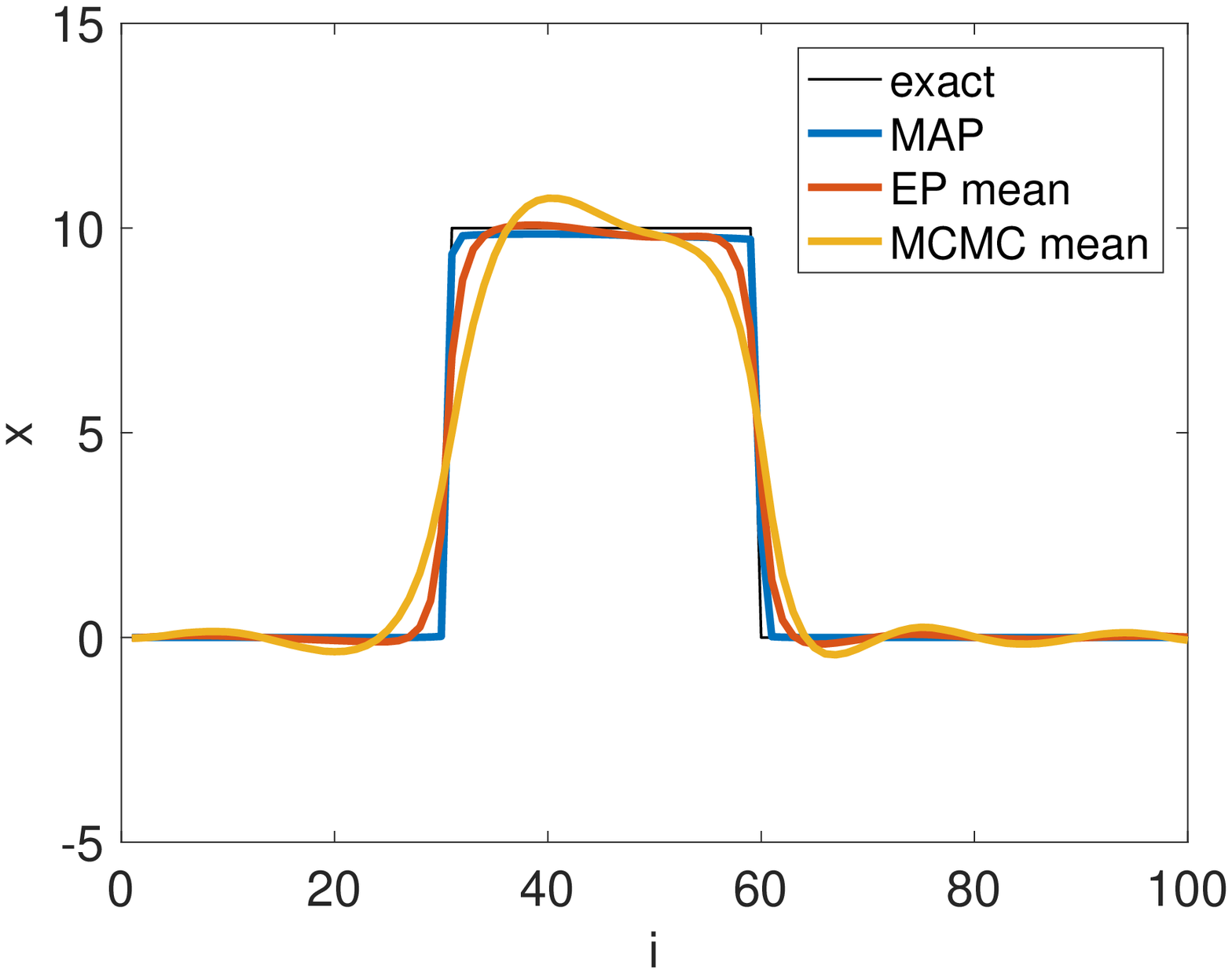} & \includegraphics[scale=0.3]{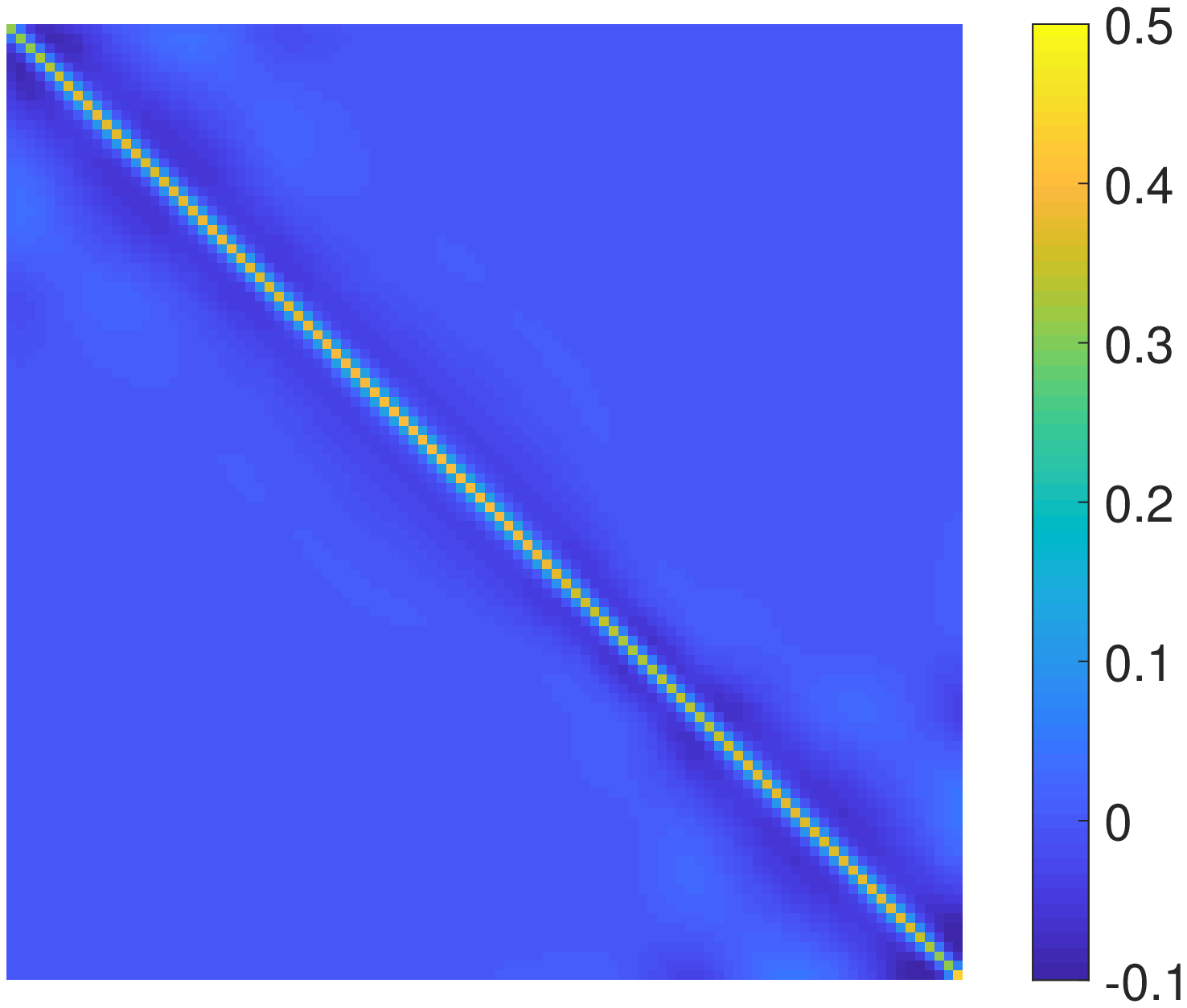} & \includegraphics[scale=0.3]{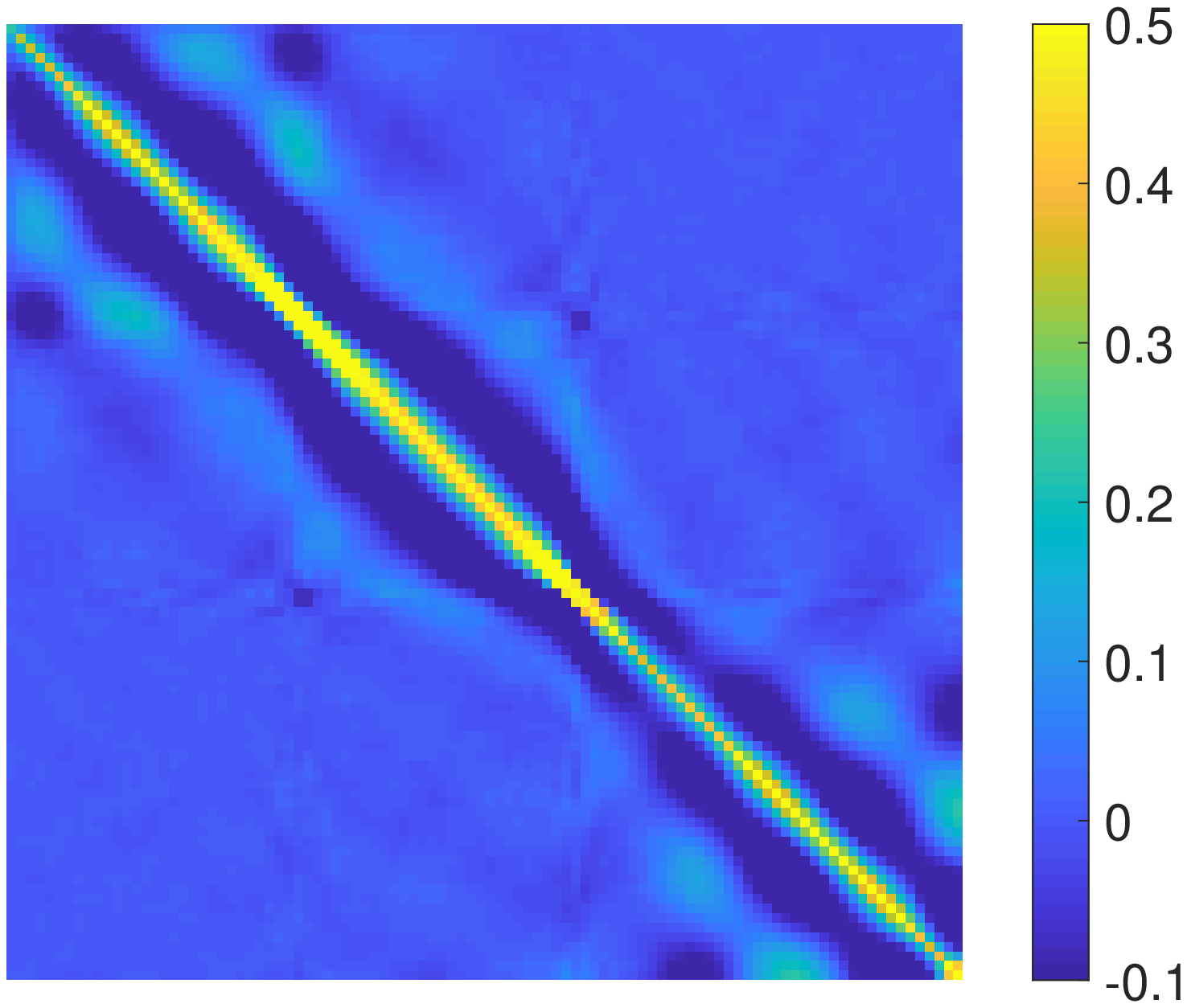}\\
(a) EP and MCMC means & (b) EP covariance & (c) MCMC covariance\\
\end{tabular}
\caption{Comparisons of mean and covariance of EP and MCMC for \texttt{Phillips} test\label{fig:1d_compare_mean_cov}}	
\end{figure}

The Laplace approximation described in Appendix \ref{app:Laplace} depends heavily on the smoothing parameter $\epsilon$,
and clearly there is a tradeoff between accuracy of MAP  and the variance approximation; see Fig. \ref{fig:map_smooth}
for the numerical results corresponding to four different smooth parameters $\epsilon$, based on the approximation \eqref{eqn:Lap}.
This tradeoff is largely attributed to the nonsmooth Laplace type prior, which pose significant challenges for
constructing the approximation. Thus, it is tricky to derive a reasonable approximation to the target posterior distribution.
In contrast, the EP algorithm only involves integrals, which are more amenable to non-differentiability, and thus can
handle nonsmooth priors naturally.

\begin{figure}[htb!]
\centering\setlength{\tabcolsep}{0pt}
\begin{tabular}{cccc}
\includegraphics[width=0.25\textwidth]{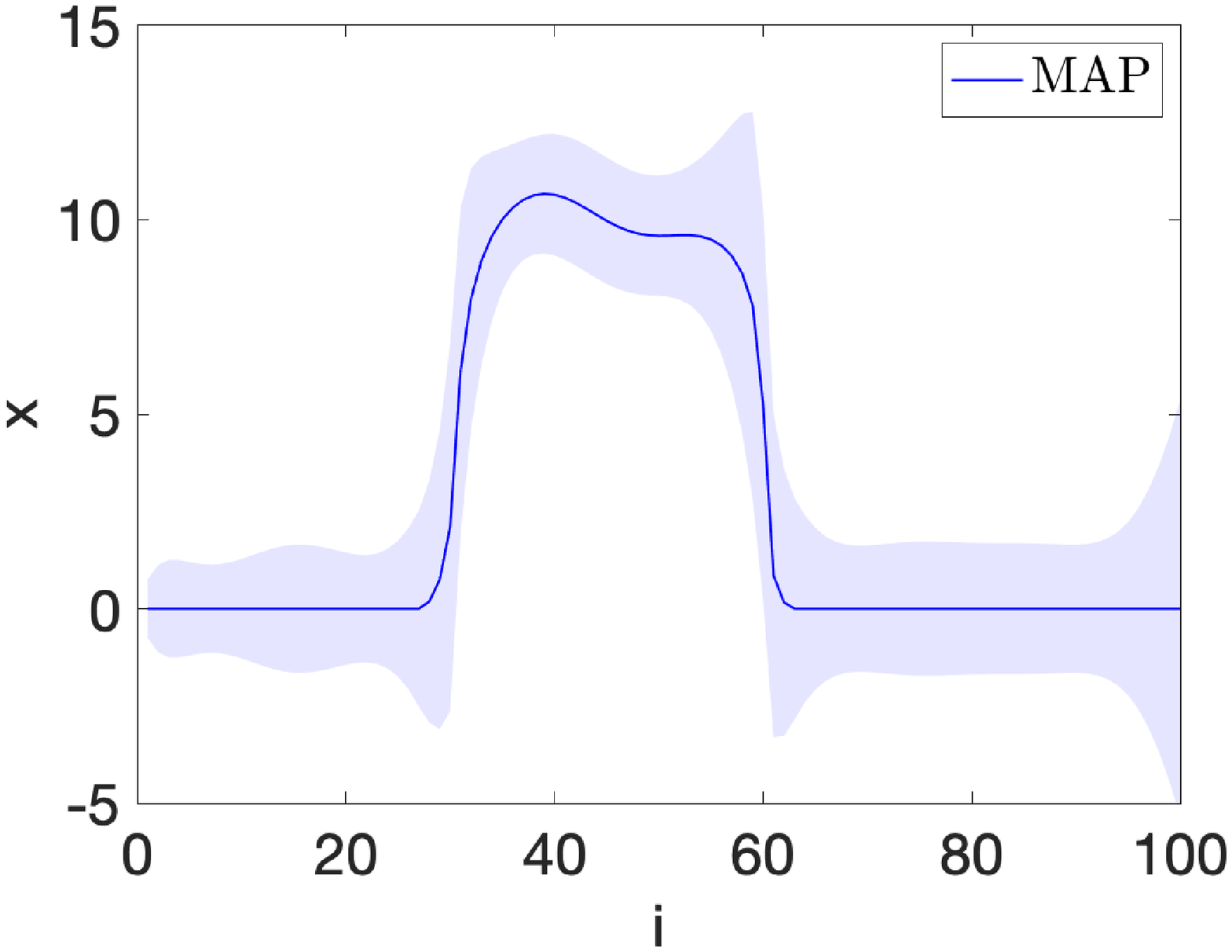} & \includegraphics[width=0.25\textwidth]{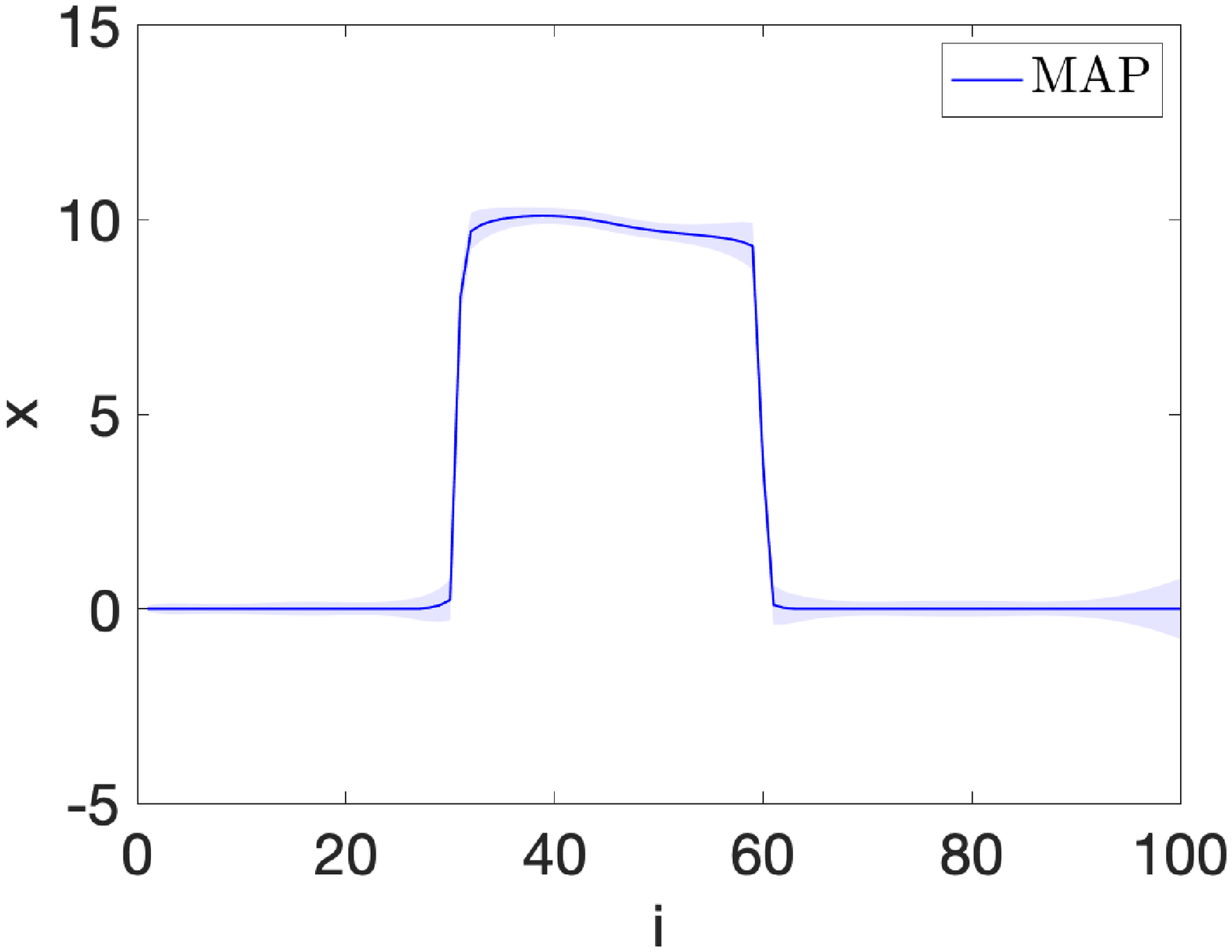}& \includegraphics[width=0.25\textwidth]{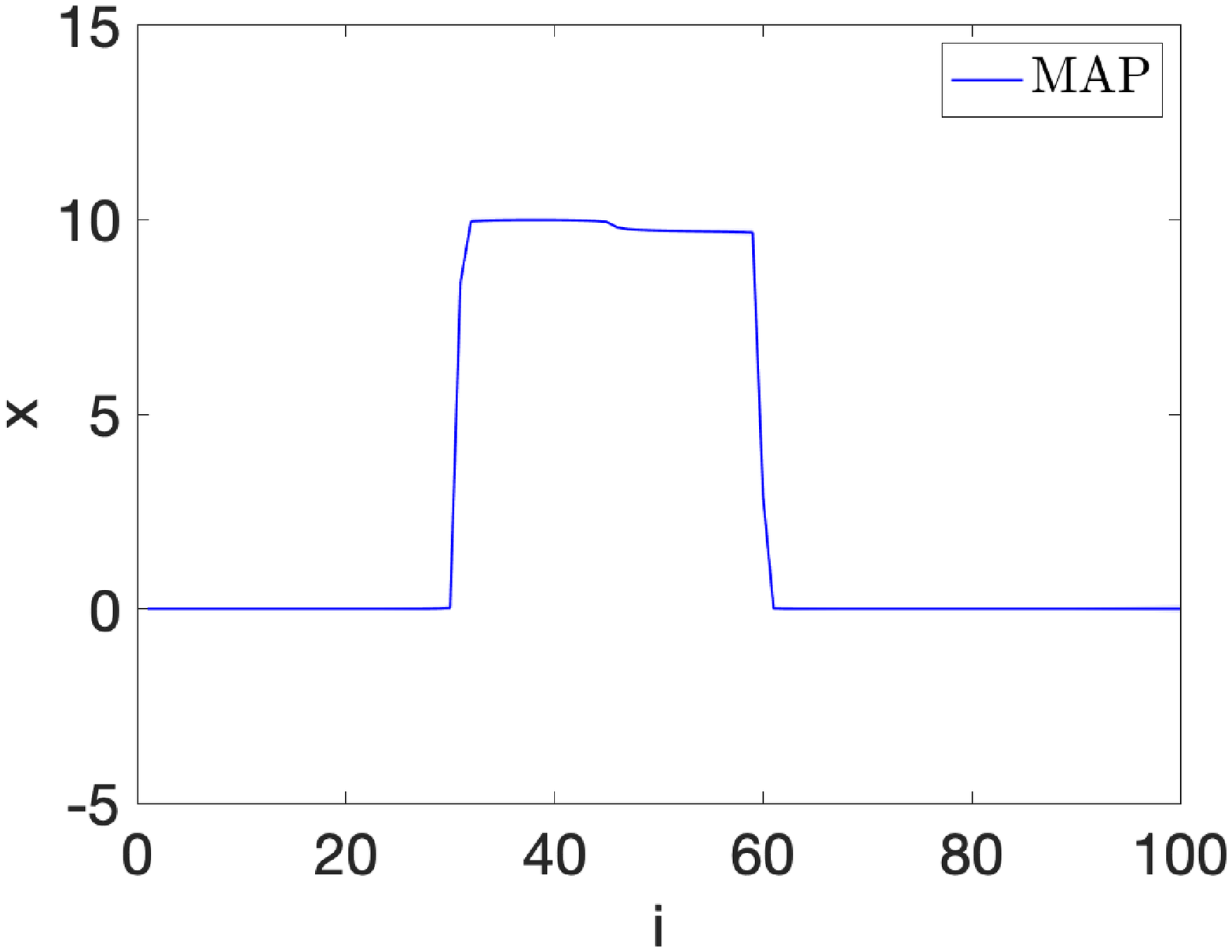} & \includegraphics[width=0.25\textwidth]{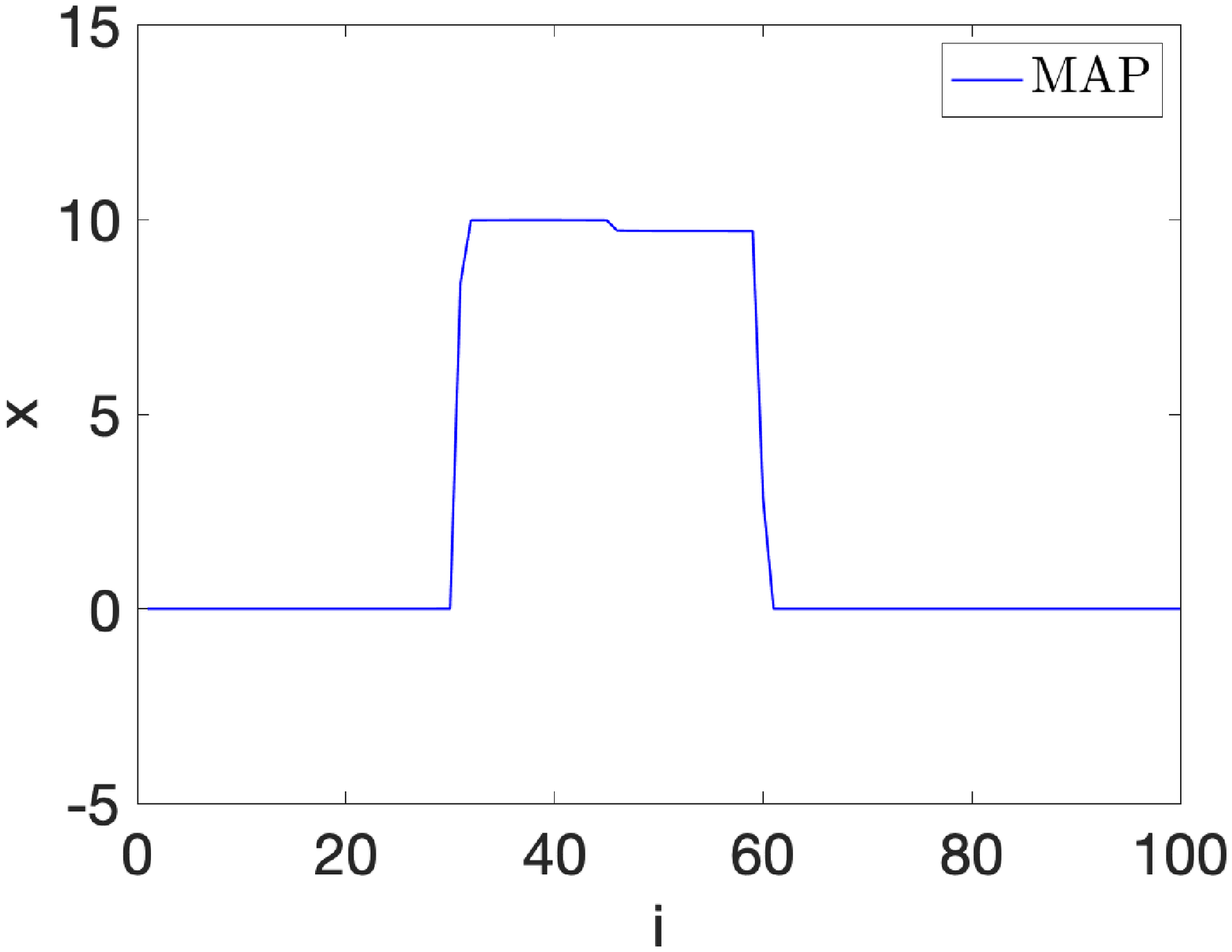}\\
(a) $\epsilon=5.18\text{e-}1$ & (b) $\epsilon=5.18\text{e-}2$& (c) $\epsilon=5.18\text{e-}3$ & (d) $\epsilon=5.18\text{e-}4$\\
\end{tabular}
\caption{Laplace approximation with different smoothing $\epsilon$.\label{fig:map_smooth}}	
\end{figure}

In passing, we note that the uncertainty estimate from the posterior probability distribution differs greatly from the concept of noise
variance \cite{Fessler:1996}, which is mainly concerned with the \textit{sensitivity} of the reconstruction with respect to
the noise in the input data $y$. It is derived using chain rule and implicit function theorem, under the assumptions of good
smoothness and local strong convexity of the associated functional \cite{Fessler:1996}. In contrast, the uncertainty in
the Bayesian framework as in this work originates from imprecise knowledge of the inverse solution encoded in the prior and the statistics of
the data. Thus, the results of these two approaches are not directly comparable.

\bibliographystyle{abbrv}
\bibliography{ep}

\end{document}